\documentclass[a4paper,english 12pt]{amsart}

\usepackage[utf8]{inputenc}
\usepackage[T1]{fontenc}
\usepackage[english]{babel}
\usepackage{multicol}
\usepackage{enumitem}
\usepackage{pifont}
 \usepackage{pgf,tikz,pgfplots}
\pgfplotsset{compat=1.15}
\usepackage{mathrsfs}
\usetikzlibrary{arrows}
\usepackage{mathtools}
\usepackage{float}
\usepackage{appendix}
\usepackage{tikz}
\usepackage{xcolor}
\usepackage[mathscr]{euscript}

\usepackage{hyperref}

\usepackage{amsmath,amssymb}
\usepackage{mathrsfs}
\usepackage{amsthm}
\usepackage{yhmath}
\usepackage{dsfont}
\numberwithin{equation}{section}
\usepackage{amsmath, mathtools, amsthm, amsfonts, amssymb,xcolor}

\usepackage{graphicx}
\usepackage{subcaption}

\usepackage[left=3cm,right=3cm,top=2.5cm,bottom=3cm]{geometry}

\DeclareMathOperator{\im}{Im}
\DeclareMathOperator{\re}{Re}

\DeclareMathOperator{\Id}{Id}

\DeclareMathOperator{\I}{I}

\DeclareMathOperator{\Diag}{Diag}

\newcommand{\R}{\mathbb{R}}

\newcommand{\N}{\mathbb{N}}
\newcommand{\Z}{\mathbb{Z}}
\newcommand{\C}{\mathbb{C}}

\DeclareMathOperator{\cinf}{\emph{C}^\infty}
\DeclareMathOperator{\cinfc}{\emph{C}_c^\infty}
\DeclareMathOperator{\supp}{supp}
\DeclareMathOperator{\ess}{ess}
\DeclareMathOperator{\gr}{Gr}

\DeclareMathOperator{\op}{Op_{\textit{h}}}
\DeclareMathOperator{\Gr}{Gr}
\DeclareMathOperator{\WF}{WF_{\textit{h}}}
\newcommand{\hinf}{O(h^\infty)}

\newtheorem{thm}{Theorem}
\newtheorem{corImp}{Corollary}
\newtheorem{ThmA}{Theorem}

\theoremstyle{definition}

\newtheorem{defi}{Definition}[section]
\newtheorem{prop}{Proposition}[section]
\newtheorem*{nota}{Notations}
\newtheorem{lem}{Lemma}[section]
\newtheorem{cor}{Corollary}[section]
\newtheorem*{rem}{Remark}

\makeatletter
\def\paragraph{\vspace{0.4cm} \@startsection{paragraph}{4}%
  \z@\z@{-\fontdimen2\font}%
  {\normalfont\bfseries}}
\makeatother

\makeatletter
\def\subparagraph{\vspace{0.3cm} \@startsection{subparagraph}{4}%
  \z@\z@{-\fontdimen2\font}%
  {\normalfont\bfseries}}
\makeatother

\title{Spectral gap for obstacle scattering in dimension 2}
\author{Lucas Vacossin}
\address{Universit\'e Paris-Saclay, Laboratoire de mathématiques d'Orsay, UMR 8628 du CNRS, B\^atiment 307, 91405 Orsay Cedex,}
\email{lucas.vacossin@universite-paris-saclay.fr}

\date{}

\makeindex
\begin{document}
\maketitle

\begin{abstract}
In this paper, we study the problem of scattering by several strictly convex obstacles, with smooth boundary and satisfying a non eclipse condition. We show, in dimension 2 only, the existence of a spectral gap for the meromorphic continuation of the Laplace operator outside the obstacles. The proof of this result relies on a reduction to an \emph{open hyperbolic quantum map}, achieved in \cite{NSZ14}. In fact, we obtain a spectral gap for this type of objects, which also has applications in potential scattering. The second main ingredient of this article is a fractal uncertainty principle. We adapt the techniques of \cite{NDJ19} to apply this fractal uncertainty principle in our context. 
\end{abstract}

\tableofcontents
\pagebreak{}

\section{Introduction}

\paragraph{Scattering by convex obstacles and spectral gap.}

In this paper, we are interested by the problem of scattering by strictly convex obstacles in the plane. Assume that $$\mathcal{O} = \bigcup_{j=1}^J \mathcal{O}_j$$ where $\mathcal{O}_j$ are open, strictly convex connected obstacles in $\R^2$ having smooth boundary and satisfying the \emph{Ikawa condition} : for $i \neq j \neq k$, $\overline{\mathcal{O}_i}$ does not intersect the convex hull of $ \overline{ \mathcal{O}_j }\cup \overline{\mathcal{O}_k}$. Let $$
\Omega = \R^2  \setminus \overline{\mathcal{O}}$$
\begin{figure}[h]
\begin{center}
\includegraphics[scale=0.3]{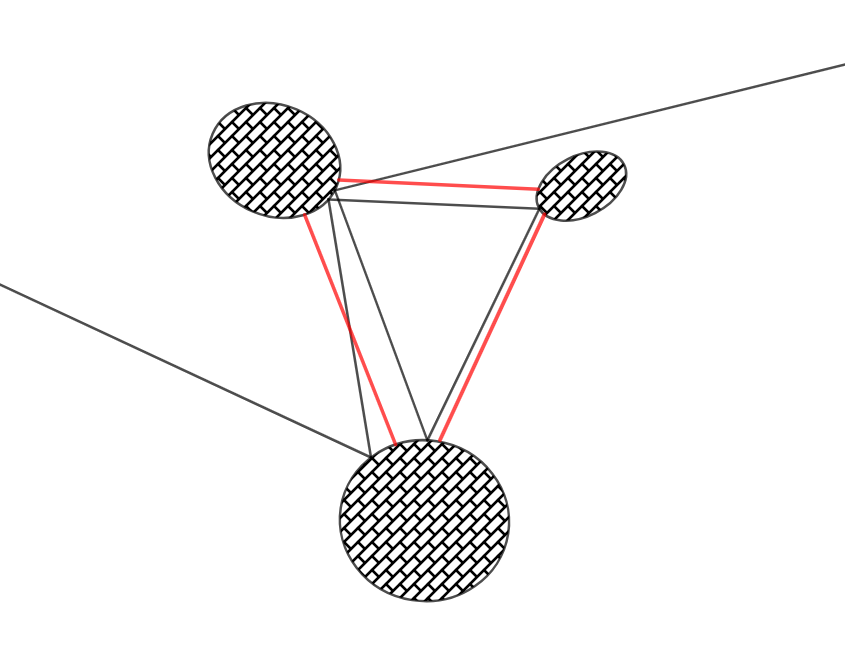}
\caption{Scattering by three obstacles in the plane}
\end{center}
\end{figure}

It is known that the resolvent of the Dirichlet Laplacian in $\Omega$ continues meromorphically to the logarithmic cover of $\C$ (see for instance \cite{DyZw18}).  More precisely, suppose that
$\chi \in \cinfc(\R^2)$ is equal to one in a neighborhood of $\overline{\mathcal{O}}$. 
$$\chi (-\Delta - \lambda^2 )^{-1} \chi  : L^2 (\Omega) \to L^2(\Omega)$$ 
is holomorphic in the region $\{ \im \lambda >0 \}$ and it continues meromorphically to the logarithmic cover of $\C$. Its poles are the \emph{scattering resonances}. We are interested in the problem of the existence of a spectral gap in the first sheet of the logarithmic cover (i.e. $\C \setminus  i \R^-$). We prove the following theorem : 

\begin{ThmA}\label{ThmA}
There exist $\gamma >0$ and $\lambda_0 >0$ such that there is no resonance in the region 
$$ [ \lambda_0, + \infty [ +i[- \gamma , 0 ]$$
\end{ThmA}
\vspace{0.5cm}

This problem has a long history in the physics and mathematics literature. The spectral gap has for instance been studied by \cite{Ik88} in dimension 3. For related problems concerning the distribution of scattering resonances for such systems, here is a non exhaustive list of papers in which the reader can find pointers to a larger litterature : \cite{GR} for the three-disks problem, \cite{Ge88}, \cite{Ik82} for the the two obstacles problem, \cite{PS10} for link with dynamical zeta functions, \cite{BLR}, \cite{HaLe} for the diffraction by one convex obstacle, \cite{ZwSja} among others papers of the two authors concerning the distribution of the scattering resonances. We will also widely use the presentation and the arguments of \cite{NSZ14}. 
\vspace{0.5cm}

The spectral gap problem is a high-frequency problem and justifies the introduction of a small parameter $h$, where $\frac{1}{h}$ corresponds to a large frequency scale. Under this rescaling, we are interested in the semiclassical operator 
$$ P(h)  = -h^2 \Delta - 1 \quad h \leq h_0$$ and spectral parameter $z \in D(0, Ch)$ for some $C >0$. 

In the semiclassical limit, the classical dynamics associated to this quantum problem is the billiard flow in $\Omega \times \mathbb{S}^1$, that is to say, the free motion outside the obstacles with normal reflection on their boundaries. A relevant dynamical object is the trapped set corresponding to the points $(x, \xi) \in \Omega \times \mathbb{S}^1$ that do not escape to infinity in the backward and forward direction of the flow. In the case of two obstacles, it is a single closed geodesic. As soon as more obstacles are involved, the structure of the trapped set becomes complex and exhibits a fractal structure. This is a consequence of the hyperbolicity of the billiard flow. It is known that the structure of the trapped set plays a crucial role in the spectral gap problem. 

A good dynamical object to study this structure is the topological pressure associated to the unstable Jacobian $\phi_u$. This dynamical quantity is a strictly decreasing function $s \mapsto P(s)$ which measures the instability of the flow (see Section \ref{Section_main_theorem} for definitions and references given there). In dimension 2, Bowen's formula shows that the  Hausdorff and upper box dimensions of the trapped set are $2s_0$ where $s_0$ is the unique root of the equation $P(s) = 0$. In \cite{NZ09}, the existence of a spectral gap for such systems has been proved under the pressure condition $$P\left(\frac{1}{2}\right) <0$$
 Their result holds in any dimension, with a quantitative spectral gap. Our result doesn't need this assumption anymore. In fact, it relies on the weaker pressure condition : 

$$ P(1) <0$$ 
It is known that this condition is always satisfied in the scattering problem we consider since the trapped set is not an attractor (\cite{BoRu}). Due to Bowen's formula, this condition can be interpreted as a fractal condition.  This is this fractal property that will be crucial in the analysis. 

\paragraph{Open hyperbolic systems and spectral gaps. }
The problem of scattering by obstacles falls into the wider class of spectral problems for open hyperbolic systems (see \cite{Nonnen11}).  In these open systems, the spectral problems concern the resonances : these are generalized eigenvalues which exhibit some resonant states. Among the problems which widely interest mathematicians and physicians, resonance counting and spectral gaps are on the top of the list. Spectral gaps are known to be important to give resonance expansion (see for instance \cite{DyZw}) and local energy decay (see for instance the works of Ikawa \cite{Ik82} and \cite{Ik88} concerning local energy decay in the exterior of 2 and several obstacles in $\R^3$).  It has been conjectured in \cite{Zw17} (Conjecture 3) that such systems might exhibit a spectral gap as soon as the trapped set has such a fractal structure.

\subparagraph{Convex co-compact hyperbolic surfaces. }

Another class of open hyperbolic systems exhibiting a fractal trapped set consists of the convex co-compact hyperbolic surfaces, which can be obtained as the quotient of the hyperbolic plane $\mathbb{H}^2$ by Schottky groups $\Gamma$. The spectral problem concerns the Laplacian on these surfaces and its classical counterpart is the geodesic flow on the cosphere bundle, which is known to be hyperbolic due to the negative curvature of theses surfaces. In this context, it is common to write the energy variable $\lambda^2 = s(1-s)$ and study 
$$ \left(- \Delta - s(1-s) \right)^{-1} $$
The trapped set is linked to the limit set of $\Gamma$ and the dimension $\delta$ of this limit set influences the spectrum. The Patterson-Sullivan theory (see for instance \cite{Bo}) tells that there is a resonance at $s=\delta$ and that the other resonances are located in $\{ \re(s) < \delta \}$. In particular, it gives an essential spectral gap of size $\max(0, 1/2- \delta)$. This is consistent with the pressure condition $P(s) <1/2$ since in that situation, $P(s)$ is simply given by $P(s) = \delta - s$. Results where obtained by Naud (\cite{Na}), where he improves the gap given by the Patterson-Sullivan theory in the case $\delta \leq 1/2$.  Recent results, initiated by \cite{DZ16}, have improved this gap. In \cite{BD18}, the authors show that there exists an essential spectral gap for any convex co-compact hyperbolic surfaces. In particular, the pressure condition $\delta < 1/2$ is no more a necessary assumption. The new idea in these papers is the use of a fractal uncertainty principle. It will be a crucial tool of our analysis. 

\subparagraph{Potential scattering.} Scattering by a compactly potential also falls in the class of open systems. It consists in studying the semiclassical operator $P(h) = -h^2 \Delta + V(x)$ where $V \in \cinfc(\R^2)$. 

\begin{figure}[h]
\centering
\includegraphics[scale=0.7]{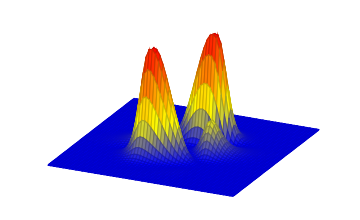}
\caption{Scattering by a smooth compactly supported potential $V$. }
\end{figure}

In this framework, the spectral gap problem consists in exhibiting bands in the complex plane of the form 
$$[a,b] - i  \times [0,h \gamma]$$
 where $P(h)$ has no resonance, for $h$ small enough. In the semiclassical limit, the behavior of $P(h)$ is linked to the classical flow of the system, that is the Hamiltonian flow generated by $p(x,\xi) = |\xi|^2 + V(x)$. Note that in potential scattering, one has to focus on some energy shell $\{p=E \}$ where $E \in \R$ is independent of $h$, with $\re z$ sufficiently close to $E$. This specification is not necessary in obstacle scattering (implicitly, we have already decided to work with $E=1$). The properties of the resonant states $u_h$, which are generalized solutions of the equation $(P(h) - z)  u_h =0$, are linked to the trapped set of the flow. This trapped set corresponds to all the trajectories which stay bounded for the  backward and forward evolution of the flow. When the flow is hyperbolic on the trapped set, this trapped set is known to exhibit a fractal structure.

\paragraph{Reduction to open hyperbolic quantum maps. }

 An important aspect of our analysis to prove Theorem \ref{ThmA} relies on previous results of \cite{NSZ14}. Their Theorem 5 (Section 6) reduces the study of the scattering poles to the study of the cancellation of
$$ z \mapsto \det( \I - M(z) ) $$ 
where 
\begin{equation}\label{M(z)}
M(z) : L^2 (\partial \mathcal{O}) \to L^2 (\partial \mathcal{O})
\end{equation}  is a family of \emph{hyperbolic open quantum map} (see below Section \ref{hyperbolic_open_quantum_map}). The family $z \mapsto M(z)$ depends holomorphically on $z \in D(0,Ch)$ for some $C>0$ and is sometimes called a \emph{hyperbolic quantum monodromy operator}.
The construction of this operator relies on the study of the operators $M_0(z)$ defined as follows : 
for $1 \leq j \leq J$, let $H_j(z) : \cinf(\partial\mathcal{O}_i ) \to \cinf (\R^2 \setminus \mathcal{O}_j ) $ be the resolvent of the problem 
$$ \left\{ \begin{array}{l}
(-h^2 \Delta - 1 - z) (H_j(z) v )  = 0 \\
H_j(z) v \text{ is outgoing} \\
H_j(z) v =v \text{ on } \partial \mathcal{O}_j
\end{array}\right. $$
Let $\gamma_j$ be the restriction of a smooth function $u \in \cinf(\R^2)$ to $\cinf(\partial \mathcal{O}_j )$ and define $M_0(z)$ by  : 

$$M_0(z) =  \left\{ \begin{array}{l}
0 \text{ if } i=j \\
- \gamma_i H_j(z) \text{ otherwise} 
\end{array}\right. $$
Due to results of Gerard (\cite{Ge88}, Appendix II), this matrix is a Fourier integral operator associated with a Lagrangian relation related to the billiard flow. \textit{A priori}, it does exclude neither the glancing rays nor the shadow region. Ikawa's conditon allows the authors to get rid of these embarrassing regions, since they do not play a role when considering the trapped set (see Section 6 in \cite{NSZ14}). A consequence of their analysis is that $M(z)$ is associated with a simpler Lagrangian relation $\mathcal{B}$, which is the restriction of the billiard map to a domain excluding the glancing rays. To be more precise, let us introduce 
 
 \begin{align*}
 &S^*_{\partial \mathcal{O}_j}=\{ (x, \xi) \in T^* \R^2, x \in \partial \mathcal{O}_j , |\xi| = 1 \} \\
 &B^* \partial \mathcal{O}_j = \{ (y, \eta) \in T^* \partial \mathcal{O}_j , |\eta | \leq 1 \}\\
 & \pi_j  : S^*_{\partial \mathcal{O}_j} \to B^* \partial \mathcal{O}_j  \text{ the orthogonal projection on each fiber}
 \end{align*}
 
 \begin{figure}[h]
 \begin{center}
  \includegraphics[scale=0.5]{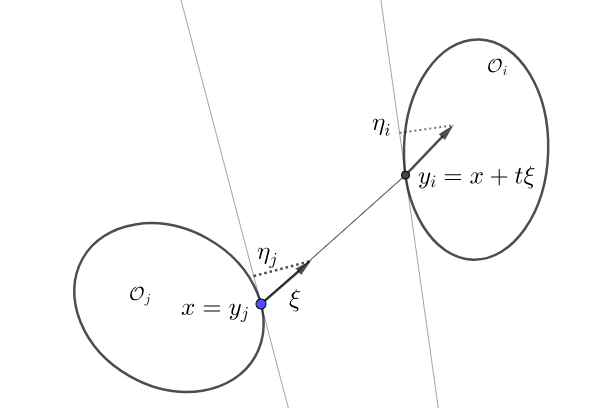}
 \end{center}
 \caption{Description of the Lagrangian relation $\mathcal{B}_{ij}$}
 \end{figure}
$\mathcal{B}$ is then the union of the relations $\mathcal{B}_{ij}$ corresponding to the reflection on two obstacles : for $ (\rho_i, \rho_j) \in B^* \partial \mathcal{O}_i \times B^* \partial \mathcal{O}_j$ : 
 
 \begin{align*}
   &(\rho_i, \rho_j)  \in \mathcal{B}_{ij} \iff  
   \exists t >0  , \xi \in \mathbb{S}^1, x \in \partial \mathcal{O}_j  \\
&\pi_j(x, \xi) = \rho_j , \pi_i(x + t \xi, \xi) = \rho_i ,  \nu_j(x) \cdot \xi >0 , \nu_i(x+t \xi) \cdot \xi  <0 
 \end{align*}
It is a standard fact in the study of chaotic billiards (see for instance \cite{Che}) that the billiard map is hyperbolic due to the strict convexity assumption. Ikawa's condition ensures that the restriction of the dynamical system to the trapped set has a symbolic representation (\cite{Mo91}).

\paragraph{Spectral gap for hyperbolic open quantum maps.} 
Using this reduction, Theorem \ref{ThmA} will be proved once we are able to show that the spectral radius of $M(z)$ is strictly smaller than $1$ for $z \in D(0,Ch) \cap \{ \im z \in [-   \delta h , 0]\}$, for some $\delta >0$. 
 This will be a consequence of the following statement, which will be demonstrated in this paper (see Section \ref{Section_main_theorem} below for a more precise version).

\begin{ThmA}\label{ThmB}
Let $(M(z))_z$ be the family introduced in (\ref{M(z)}), that is a hyperbolic quantum monodromy operator associated with the open Lagrangian relation $\mathcal{B}$. 
Then, there exist $h_0 >0$, $\gamma >0$ and $\tau_{\max} >0$  such that the spectral radius of $M(z)$, $\rho_{spec}(z)$, satisfies : for all $h \leq h_0$ and all $z \in D(0,Ch)$, 
$$ \rho_{spec}(z) \leq  e^{ - \gamma - \tau_{\max} \im z} $$
\end{ThmA}

When $z \in \R$, the operator $M(z)$ is microlocally unitary near the trapped set and its $L^2$ norm is essentially 1. Then, we have the trivial bound 
$$ \rho_{spec}(z) \leq 1$$ 
The bound given by the theorem is a spectral gap since we obtain $$\rho_{spec}(z) \leq e^{-\gamma} < 1$$ 
The dependence of the bound with the parameter $z$ is related to the symbol of the open quantum map $M(z)$.  

\vspace*{0.5cm}

The link between open quantum maps and the resonances of open quantum systems has also been established in \cite{NSZ11} for the case of potential scattering. As a consequence, we will also obtain a spectral gap in this context.  We review this reduction both in obstacle and potential scattering in Section \ref{Section_main_theorem} and show how it implies the spectral gap.  This correspondance between open quantum maps and open quantum systems leads to an heuristics :  
 to a resonance $z$ for the open quantum systems, it corresponds an eigenvalue $e^{-i \tau \frac{z}{h} }$ of an open quantum map. 
Here, $\tau$ is a return time associated with the classical dynamics of the open system. 
In particular, the spectral gap for open quantum maps given by the theorem heuristically implies that the resonances of the open systems might satisfy $\im z <- h \frac{\gamma}{\tau}$. 
\vspace*{0.5cm}

\paragraph{On the fractal uncertainty principle.} This is a recent tool in harmonic analysis in 1D developed by Dyatlov and several collaborators. For a large survey on this topic, we refer the reader to \cite{Dy18}. We do not enter into the details in this introduction and give the precise definitions and statements in Section \ref{Section_FUP}. We rather explain here the general idea of this principle in the spirit of our use. Roughly speaking, it says that no function can be concentrated both in frequencies and positions near a fractal set. 
Suppose that $X,Y \subset \R$ are fractal sets. To fix the ideas, let's say that $X$ and $Y$ have upper box dimension $\delta_X$ and $\delta_Y$ strictly smaller than one. For $c >0$, let's note $X(c) = X + [- c, + c]$ and the same for $Y$. Also denote $\mathcal{F}_h$ the $h$-Fourier transform : 

$$ \mathcal{F}_h u (\xi) = \frac{1}{(2 \pi h)^{1/2}} \int_\R e^{- i \frac{x \xi}{h}} u(x) dx$$
The fractal uncertainty principle then states that there exists $\beta >0$ depending on $X$ and $Y$ (See Proposition \ref{FUP} for the precise dependence) such that, for $h$ small enough, 
$$ || \mathds{1}_{X(h)}\mathcal{F}_h \mathds{1}_{Y(h)}   ||_{L^2 (\R) \to L^2(\R) } \leq h^\beta$$

\begin{figure}[h]
\centering
\includegraphics[scale=0.3]{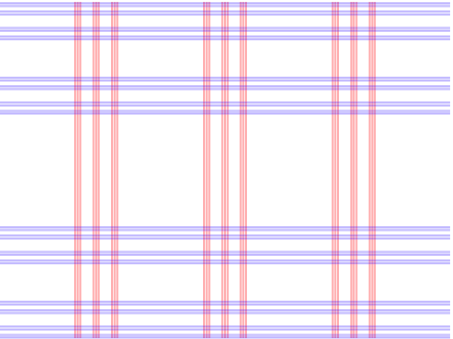}
\caption{The fractal uncertainty principle asserts that no state can be microlocalized both in frequencies (in blue) and positions (in red) near fractal sets.}
\end{figure}
Actually, one can change the scales and look for the sets $X(h^{\alpha_X} ) $ and $Y(h^{\alpha_Y})$ where $\alpha_X$ and $\alpha_Y$ are positive exponents. The result will stay true as soon as theses exponents satisfy a saturation condition : 
$$ \alpha_X + \alpha_Y > 1$$ 

It will be a key ingredient in the proof of the main theorem of this paper. It has been successfully used to show spectral gaps for convex co-compact hyperbolic surfaces (\cite{DZ16}, \cite{BoDy17},\cite{DyJ18}, \cite{DyZw18}). 
A discrete version of the fractal uncertainty principle is also the main ingredient of \cite{DJ17} where the author proved a spectral gap for open quantum maps in a toy model case. Their results concerning open baker's map on the torus $\mathbb{T}^2$ partly motivates our theorem on open quantum maps. 

\vspace*{0.5cm}

The fractal uncertainty principle has also given new results in quantum chaos on negatively curved compact surfaces. It has first been successfully used for compact hyperbolic surfaces in \cite{DJ17} where the authors proved that semiclassical measures have full support. The hyperbolic case was treated using quantization procedures developed in \cite{DZ16}, which allow to have a good semiclassical calculus for symbols very irregular in the stable direction, but smooth in the unstable one (or conversely). The existence of such quantization procedures relies on the smoothness of the horocycle flow. This smoothness is no more possible for general negatively curved surfaces. However, in \cite{NDJ19}, the authors bypassed this obstacle and succeeded to extend these results to the case of negatively curved surfaces. This is mainly from this paper that we borrow the techniques and we adapt them in our setting.

\paragraph{ A model example.}
To explain the main ideas of the proof of Theorem \ref{ThmB}, let us show how it works in an example where the trapped set is the smallest possible : a single point. In this context, we only need a simpler uncertainty principle. We focus on the case $z=0$ in Theorem \ref{ThmB} and focus on a single open quantum map. 

We consider the hyperbolic map 
$$ F : (x ,\xi) \in \R^2 \mapsto (2^{-1}x , 2\xi ) \in \R^2$$
It has a unique hyperbolic fixed point $\rho_0 = 0$ and the stable (resp. unstable) manifold at $0$ is given by $\{ \xi = 0 \}$ (resp. $\{ x = 0 \}$). 
The scaling operator 
$$U : v \in L^2(\R) \mapsto \sqrt{2} v(2x) $$
is a quantum map quantizing $F$. To open it, consider a cut-off function $\chi \in \cinfc(\R^2)$ such that $\chi \equiv 1$ in $B(0,1/2)$ and $\supp \chi \Subset B(0,1)$ and we consider the open quantum map 
$$ M= M(h)  = \op(\chi) U$$
where $\op$ is in this example (and only in this example) the left quantization : 

$$ \op(\chi) u(x) = \frac{1}{2 \pi h} \int_{\R^2} \chi(x, \xi) e^{ i \frac{(x-y) \xi}{h}} u(y) dy d\xi $$
One easily checks that Egorov's property for $U$ is true without remainder term : 
$$ U^* \op(\chi) U = \op\left( \chi \circ F\right) \quad ,\quad U \op\left(\chi  \right) U^* = \op\left( \chi \circ F^{-1} \right)$$
To show a spectral gap for $M$, we study $M^n$ with 
$$ n = n(h) \sim - \frac{3}{4} \frac{\log h}{\log 2} $$
This time is longer than the Ehrenfest time $- \frac{\log h}{\log 2} $. We write : 

$$ M^n = U^n \op\left(  \chi \circ F^{n} \right) \dots \op\left(  \chi \circ F^{1} \right) $$
The formula $[\op(a), \op(b)] = O(h^{1- 2 \delta})$ is valid for $a, b$ symbols in $S_\delta$ (we recall the definitions of symbol classes in section \ref{Section_preliminaries}) and $\delta <1/2$. The problem here is that for $1 \leq k \leq n$, $\chi \circ F^{k}$ are uniformly in $S_{3/4}$ : this is not a good symbol class. To bypass this difficulty, we observe that the symbols  $\chi \circ F^k$ are uniformly in $S_{3/8}$ for $k \in \{ - n/2, \dots,  n/2 \}$. As a consequence, for $j \in \{1, \dots, n \}$ we write: 
\begin{align*}
\left[ \op\left(  \chi \circ F^{n} \right) ,  \op\left(  \chi \circ F^{j} \right) \right] &= U^{-n/2} \left[ \op\left(  \chi \circ F^{n/2} \right) ,  \op\left(  \chi \circ F^{j-n/2} \right) \right]  U^{n/2} \\
&=  U^{-n/2} O\left( h^{1/4} \right) U^{n/2} \\
&= O\left( h^{1/4} \right) 
\end{align*}
where the constants in $O$ are uniform in $j$ and depend only on $\chi$. Applying this formula recursively to move the term $\op\left(  \chi \circ F^{n} \right)$ to the right, we get that 
$$ M^n = U^n \op\left(  \chi \circ F^{n-1} \right) \dots \op\left(  \chi \circ F^{1} \right) \op\left(  \chi \circ F^{n} \right) + O\left(  h^{1/4} \log h \right) $$
Similarly, we can write : 
$$ M^{n+1} =  \op\left(  \chi \circ F^{-n} \right) \op \left( \chi \right)  \dots \op\left(  \chi \circ F^{-n+1} \right)  U^{n+1} + O\left(  h^{1/4} \log h \right)$$
Hence, we have 
$$ M^{2n+1} =A \op\left(  \chi \circ F^{n} \right)  \op\left(  \chi \circ F^{-n} \right) B + O\left(  h^{1/4} \log h \right)$$
with $$A=A(h) = U^n \op\left(  \chi \circ F^{n-1} \right) \dots \op\left(  \chi \circ F^{1} \right) = O(1)$$
and $$B=B(h)=  \op \left( \chi  \right)  \dots \op\left(  \chi \circ F^{-n+1} \right) U^{n+1} = O(1)$$ 
We have the following properties on the supports 

$$ \supp \chi \circ F^{n} \subset \{ |\xi| \leq 2^{-n} \} \quad , \quad \supp \chi \circ F^{n} \subset \{ |x| \leq 2^{-n} \}$$
Assuming that $n(h) \geq - \frac{3}{4} \frac{\log h}{\log 2} $, we observe that 
$$  \op\left(  \chi \circ F^{n} \right) =  \op\left(  \chi \circ F^{n} \right) \mathds{1}_{[-h^{3/4},h^{3/4}]}(hD_x) $$
$$ \op\left(  \chi \circ F^{-n} \right)  = \mathds{1}_{[h^{-3/4},h^{3/4}]} (x) \op\left(  \chi \circ F^{-n} \right)  $$
 Finally, we have 
$$M^{2n+1} =A \op\left(  \chi \circ F^{n} \right)\mathds{1}_{[-h^{3/4},h^{3/4}]}(hD_x)   \mathds{1}_{[h^{-3/4},h^{3/4}]} (x)\op\left(  \chi \circ F^{-n} \right) B + O\left(  h^{1/4} \log h \right)$$
This is where we need an uncertainty principle : 
\begin{align*}
|| \mathds{1}_{[-h^{3/4},h^{3/4}]}(hD_x)  \mathds{1}_{[h^{-3/4},h^{3/4}]} (x) ||_{L^2 \to L^2} &= || \mathds{1}_{[-h^{3/4},h^{3/4}]}  \mathcal{F}_h \mathds{1}_{[-h^{3/4},h^{3/4}]} ||_{L^2 \to L^2} \\
&\leq || \mathds{1}_{[-h^{3/4},h^{3/4}]} ||_{L^\infty \to L^2 }  \times || \mathcal{F}_h||_{L^1 \to L^\infty} \times || \mathds{1}_{[-h^{3/4},h^{3/4}]} ||_{L^2 \to L^1 } \\
&\leq C h^{3/8} \times h^{-1/2} \times h^{3/8}  = C h^{1/4}
\end{align*} 
Here, the bound can be understood as a volume estimate : the box in phase space of size $h^{3/4}$ is smaller than a "quantum box".
Gathering all the computations together, we see that 
$$ ||M^{2n+1}||_{L^2 \to L^2} = O\left(h^{1/4} \log h \right) $$
Elevating this to the power $\frac{1}{2n+1}$, we see that for every $\varepsilon>0$, we can find $h_\varepsilon$ such that for $h \leq h_\varepsilon$, 
$$ \rho (M) \leq (1+ \varepsilon) 2^{-1/6}$$
\begin{rem}What matters in this example is the strategy we use, and not particularly the bound, which is in fact not optimal. 
\end{rem}

\paragraph{Sketch of proof. } The strategy presented in this simple model case is the guideline, but its direct application will encounter major pitfalls that we'll have to bypass. 

\begin{itemize}
\item The trapped set being a more complex fractal set, we'll need the general fractal uncertainty principle developed by Dyatlov and his collaborators. 
\item Even in small coordinate charts, the trapped set cannot be written has a product of fractal sets in the unstable and stable directions. To tackle this difficulty, we build adapted coordinate charts (see \ref{Adapted_charts}) in which we straighten the unstable manifolds. The existence of such coordinate charts is made possible by Theorem \ref{Thm_regularity}, in which we prove that the unstable (and stable) distribution can be extended in a neighborhood of the trapped set to a $C^{1+ \beta}$ vector field. 
\item In the model case,  there is only one point and hence one unstable Jacobian to consider which gives the Lyapouvov exponent of the map $\log J_u^1(0) =\log 2$.  Generally, the growth rate of the unstable Jacobian differs from one point to another (see \ref{Local_jacobian}) and the choice of the integer $n(h)$ is not as simple. In fact, we prefer to break the symmetry $2n(h) = n(h) + n(h)$ and split $2n(h)$ into a small logarithmic time $N_0(h)$ and a long logarithmic time $N_1(h)$ (see section \ref{Numerology}). The first one is supposed to be smaller than the Ehrenfest time and allows us to use semiclassical calculus to handle $M^{N_0}$. As a matter of fact, the major technical difficulties concerns the study of $M^{N_1}$. 
\item The study of $M^{N_1}$ requires fine microlocal techniques. The trick used in the model case to have the commutator estimate is no more possible and we have to use propagation results up to twice the Ehrenfest time. This is what we do in section \ref{section_propagation_Ehrenfest} but this study has to be made locally and we need to split $M^{N_1}$ into a sum of many terms $U_\mathbf{q}$.
\item We could use the fractal uncertainty principle to get the decay for single terms $M^{N_0}U_\mathbf{q}$. However, a simple triangle inequality to handle their sum will no more give a decay for $M^{N_0 + N_1}$ since the number of terms in the sum grows like a negative power of $h$. To bypass this problem, we need a more careful analysis and we gather them into clouds (see \ref{Clouds}). These clouds are supposed to interact with a few other ones, so that a Cotlar-Stein type estimate reduces the study of the norm of the sum, to the norm of each cloud.  The elements of a single cloud are supposed to be close to each other, so that the fractal uncertainty principle can be applied to all of them in the same time and gives the required decay for a single cloud. 
\end{itemize}

Our strategy follows the main lines of the proof of \cite{NDJ19}. In particular, their strategy allows us to apply the fractal uncertainty principle of \cite{BD18} in a case where the unstable foliation is not smooth (and in fact, a priori defined only in a fractal set). Their strategy relies on the existence of adapted charts based on $C^{2^-}$ regularity of the unstable foliations in negatively curved hyperbolic surfaces. It is based on results of \cite{KH} for Anosov flows. We needed to prove the existence of such adapted charts in this different context. To do so, we prove that the unstable lamination can be extended into a $C^{1 + \beta}$ foliation (see \ref{Adapted_charts}). Another aspect which changes from \cite{NDJ19} is the proof of porosity. In their study, the porous sets arise as iteration of artifical "holes" and they had to control the evolution of such holes. In our context,  this study is easier since we already know that the trapped set has a fractal structure, characterized by its Hausdorff dimension. In this paper, we will rather use the upper box dimension (but these two dimensions are equal in this context).  

\paragraph{Restrictions.} The main restriction of our theorem is that it only applies to quantum maps with two-dimensional phase space. In terms of open systems, it only concerns problems with physical space of dimension 2. Several points explain this restriction : 
\begin{itemize}
\item The fractal uncertainty principle works in dimension 1. In higher dimension, the result is currently not well understood and the only known cases require strong assumptions on the fractal sets (See \cite{Dy18}, Section 6). 
\item Our proof strongly relies on the regularity of the stable and unstable laminations. 
\item The growth of the unstable Jacobian controls the contraction (resp. expansion) rate in the unique stable (resp. unstable) direction. 
\end{itemize}

\paragraph{Plan of the paper. } The paper is organized as follows : 

\begin{itemize}
\item In Section \ref{Section_main_theorem}, we present the main theorem of this paper and show how it gives a spectral gap in some open quantum systems. 
\item In Section \ref{Section_preliminaries}, we give some background material in semiclassical analysis (pseudodifferential operators and Fourier integral operators). We also recall some standard facts about hyperbolic dynamical systems and give further results. In particular, in Theorem \ref{Thm_regularity}, we show that the unstable and stable distribution have $C^{1+ \beta}$ regularity. 
\item The proof of Theorem \ref{Spectral_radius} starts in Section \ref{Section_first_ingredient} where we introduce the main ingredients needed for the proof and give several technical results. 
\item In Section \ref{Section_reduction_to_FUP}, we use fine microlocal methods to microlocalize the operators we work with in small regions where the dynamic is well understood and we reduce the proof of Theorem \ref{Spectral_radius} to a fractal uncertainty principle with the techniques of \cite{NDJ19}. 
\item In Section \ref{Section_FUP}, we conclude the proof of this theorem by applying the fractal uncertainty principle of \cite{BD18}, and more precisely, the version stated in \cite{NDJ19}. 
\end{itemize}

\subsubsection*{Acknowledgment}
The author would like to warmly thank Stéphane Nonnenmacher for his careful reading and helpful discussions which contributed a lot in the achievement of this work.

\section{Main theorem and applications}\label{Section_main_theorem}

\subsection{Hyperbolic open quantum maps}\label{hyperbolic_open_quantum_map}
We introduce the main tools needed to state the main theorem of this paper. The following long definition is based on the definitions in the works of Nonnenmacher, Sjöstrand and Zworski in \cite{NSZ11} and \cite{NSZ14} specialized to the 2-dimensional phase space. 
Consider open intervals $Y_1, \dots, Y_J$ of $\R$ and set  : 
$$Y = \bigsqcup_{j=1}^J Y_j \subset \bigsqcup_{j=1}^J \R$$
and consider 
$$ U = \bigsqcup_{j=1}^J U_j \subset \bigsqcup_{j=1}^J T^*\R^d \quad ; \quad U_j \Subset T^*Y_j$$
The Hilbert space $L^2(Y)$ is the orthogonal sum $\bigoplus_{i=1}^J L^2(Y_i)$. 

Then, we introduce a smooth Lagrangian relation $F \subset U \times U$. It is a disjoint union of symplectomophisms. 
For $j=1, \dots, J$, consider open disjoint subsets $\widetilde{D_{i j } } \Subset U_j$, $1 \leq i \leq J$ and similarly, for $i= 1, \dots, J$ consider open disjoint subsets $\widetilde{A_{i j }} \Subset U_i$, $1 \leq j \leq J$. We consider a family of smooth symplectomorphisms 
\begin{equation}
F_{i j } : \widetilde{D_{i j }} \to F_{ij } \left(\widetilde{D_{i j }} \right) = \widetilde{A_{i j }} 
\end{equation} and define the relation 
$F$  as the disjoint union of the relation $F_{ij}$, namely, 
$$(\rho^\prime, \rho) \in F \iff \exists 1 \leq i, j  \leq J, \rho^\prime = F_{ij}(\rho)$$
In particular, $F$ and $F^{-1}$ are single-valued. We will identify $F$ with a smooth map and note by abuse $\rho^\prime = F(\rho)$ or $\rho = F^{-1}(\rho^\prime)$ instead of $(\rho^\prime, \rho) \in F$. 
\\
We note 
$$ \pi_L (F) = \widetilde{A} = \bigsqcup_{i=1}^J \bigcup_{j=1}^J \widetilde{A_{i j }} $$ 
$$ \pi_R (F) = \widetilde{D} = \bigsqcup_{j=1}^J \bigcup_{i=1}^J \widetilde{D_{i j }} $$

We define the outgoing (resp. incoming) tail by $\mathcal{T}_+ \coloneqq \{ \rho \in U ; F^{-n}(\rho) \in U , \forall n \in \N \} $ ( resp. $\mathcal{T}_- \coloneqq \{ \rho \in U ; F^{n}(\rho) \in U , \forall n \in \N \} $). We assume that they are closed subsets of $U$ and that the \emph{trapped set} 
\begin{equation}\label{trapped_set} 
\mathcal{T} = \mathcal{T}_+ \cap \mathcal{T}_-
\end{equation}
is compact. We note $f : \mathcal{T} \to \mathcal{T}$ the restriction of $F$ to $\mathcal{T}$. 
For $i , j \in \{1, \dots , J \}$, we note $\mathcal{T}_i = \mathcal{T} \cap U_i$, $$D_{i j } = \{ \rho \in \mathcal{T}_j ; f(\rho) \in \mathcal{T}_i \} \subset  \widetilde{D_{i j }}  $$ and $$
A_{i j } = \{ \rho \in \mathcal{T}_i ; f^{-1}(\rho) \in \mathcal{T}_j\} \subset \widetilde{A_{ij}}$$

\begin{rem}
$F$ is an open canonical transformation since $F$ (resp. $F^{-1}$) is defined only in $\widetilde{D}$ (resp. $\widetilde{A}$). The sets $U \setminus \widetilde{D}$ (resp. $U \setminus \widetilde{A}$) can be seen as holes in which a point $\rho$ can fall in the future (resp. in the past). 
\end{rem}

We then make the following hyperbolic assumption. 

\begin{equation}\label{Hyperbolicity_assumption}\tag{Hyp} 
\mathcal{T} \text{ is a hyperbolic set for }F  
\end{equation}
Namely, for every $\rho \in \mathcal{T}$, we assume that there exist stable and unstable tangent spaces $E^{s}(\rho)$ and $E^{u}(\rho)$ such that : 
\begin{itemize}
\item $\dim E^{s}(\rho) = \dim E^{u}(\rho) = 1$
\item $T_\rho U = E^{s}(\rho) \oplus E^{u}(\rho)$
\item there exist $\lambda >0$, $C >0$  such that for every $v \in E^{\star}(\rho) $ ($\star$ stands for $u$ or $s$) and any $n \in \N$, 
\begin{align}
v \in E^{s}(\rho)  \implies ||d_\rho F^n (v) || \leq C e^{-n \lambda} || v||  \label{hyp1}\\
v \in E^{u}(\rho) \implies ||d_\rho F^{- n} (v_\star) || \leq C e^{-n \lambda} \label{hyp2} || v||
\end{align}where  $|| \cdot ||$ is a fixed Riemannian metric on $U$. 
\end{itemize}
The decomposition of $T_\rho U$ into stable and unstable spaces is assumed to be continuous.

\begin{rem}\text{} \\
\begin{itemize}[label=-, nosep]
\item The definition is valid for any Riemannian metric and we can of course suppose that is it the standard Euclidean metric on $\R^2$. 
\item It is a standard fact (See \cite{Mather}) that there exists a smooth Riemannian metric on $U$, which is said to be adapted to the dynamics, such that (\ref{hyp1}) and (\ref{hyp2}) hold with $C=1$. 
\item  It is known that the map $\rho \mapsto E_{u/s}(\rho)$ is in fact $\beta$-Hölder for some $\beta >0$ (\cite{KH}). We will show further an improved regularity. This will be an essential property for the proof of the main theorem. 
\end{itemize}
\end{rem}

The last assumption we'll make on $\mathcal{T}$ is a  fractal assumption. To state it, we introduce the map $ \phi_u : \rho \in \mathcal{T} \mapsto -log \left| \left|d_\rho F|_{E_u(\rho)} \right| \right|$ associated with the bijection $f$. 
We suppose that 

\begin{equation}\label{Fractal} \tag{Fractal}
- \gamma_{cl} \coloneqq - P\left( -\log \left| \left|d_\rho F|_{E_u(\rho)} \right| \right| ,f \right) >0
\end{equation}
Here, in terms of thermodynamics formalism, $P$ denotes the topological pressure of the map $\phi_u$. The norm $||\cdot||$ is associated with any Riemannian metric on $U$. For instance, a possible formula for the definition of the pressure is 
$$ P(\phi) = \lim_{\varepsilon \to 0} \limsup_{n \to + \infty} \frac{1}{n} \log \sup_{E} \sum_{ \rho \in E} \exp^{\sum_{k=0}^{n-1}  \phi( f^k \rho) } $$
where the supremum ranges over all the $(n, \varepsilon)$ separated subsets $E \subset \mathcal{T}$ ($E$ is said to be $(n, \varepsilon)$ separated if for for every $\rho,\rho^\prime \in E$, there exists $k \in \{0, \dots, n-1 \}, d(f^k(\rho),f^k(\rho^\prime)) > \varepsilon $). 

\begin{rem}\text{} 
\begin{itemize}
\item $\gamma_{cl}$ is the classical decay rate of the dynamical system. It has the following physical interpretation : fix a point $\rho_0 \in \mathcal{T}$ and consider the set $B_m(\rho_0, \varepsilon)$ of points $\rho \in U$ such that $|F^k(\rho) - F^k(\rho_0)| < \varepsilon$ for $0 \leq k \leq m-1$. Then, its Lebesgue measure if of order $e^{- m \gamma_{cl}}$.
\item In Section \ref{upper-box dimension}, we recall arguments showing that $\mathcal{T}$ is indeed "fractal". More precisely, the trace of $\mathcal{T}$ along the unstable and stable manifolds (see Lemma \ref{classical_hyperbolic} for the definitions of these manifolds) have upper-box dimension strictly smaller than one. In fact, Bowen's formula (see for instance \cite{Bar} and referecences given there) gives that this upper-box dimension corresponds to the Hausdorff dimension $d_H$ and it is the unique solution of the equation 
$$ P(s\phi_u , f) = 0 , s \in \R$$
The Hausdorff dimension of the trapped set is then $2d_H$. 
\item This condition has to be compared with the pressure condition $P(\frac{1}{2} \phi_u)<0$ in \cite{NZ09} which ensured a spectral gap for chaotic systems. This condition required that $\mathcal{T}$ was sufficiently "thin", i.e. with Hausdorff dimension strictly smaller than one. Our condition allows to go up to the limit $\dim_H \mathcal{T} =2^-$. 
\end{itemize}
\end{rem}

We then associate to $F$ \textit{hyperbolic open quantum maps}, which are its quantum counterpart. 

\begin{defi}\label{def_FIO}
Fix $\delta \in [0,1/2)$. 
We say that $T=T(h)$ is a semi-classical Fourier integral operator associated with $F$, and we note $T=T(h) \in I_{\delta} ( Y \times Y , F^\prime )$ if : 
For each couple $(i,j) \in \{ 1, \dots , J \}^2$, there exists a semi-classical Fourier integral operator  $ T_{ i j }=T_{ i j } (h) \in I_{\delta} ( Y_j \times Y_i , F^\prime_{i j } )$ associated with $F_{ i j}$ in the sense of definition \ref{Def_FIO_local}, such that
$$ T = ( T_{i j } )_{ 1 \leq i,j \leq J }  : \bigoplus_{i=1}^J L^2 (Y_i) \to \bigoplus_{i=1}^J L^2 (Y_i)$$ 
In particular $\WF (T) \subset \widetilde{A} \times \widetilde{D}$. 
We note $I_{0^+}(Y \times Y , F^\prime)= \bigcap_{ \delta >0}  I_{\delta} ( Y \times Y , F^\prime )$. 
\end{defi}

We will say that $T$ is \emph{microlocally unitary near $\mathcal{T}$} if the two following conditions hold : 
\begin{itemize}
\item $||T T^*|| \leq 1 + O\left(h^{\varepsilon} \right)$ for some $\varepsilon >0$
\item there exists a neighborhood  $\Omega \subset U$ of $\mathcal{T}$ such that, for every $u =(u_1, \dots, u_J) \in \bigoplus_{j=1}^J L^2(Y_j)$, 
$$ \forall j \in \{1, \dots, J \}, \WF(u_j) \subset \Omega \cap U_j \implies TT^* u = u + \hinf ||u||_{L^2} , T^* T u = u +  \hinf ||u||_{L^2} $$
\end{itemize}

Let us now briefly see what the second condition implies for the components of $T^*T$. First focus on the off-diagonal entries. 
$$(T^*T)_{ij} = \sum_{k=1}^J (T^*)_{ik} T_{kj} =\sum_{k=1}^J (T_{ki})^*T_{kj}  $$
If $k \in \{1, \dots, J\}$ and $i \neq j$, $(T_{ki})^*T_{kj}  = \hinf$ since $$
\WF(T_{ki}^*) \subset \widetilde{D}_{ki} \times \widetilde{A}_{ki} \quad ; \quad \WF(T_{kj}) \subset \widetilde{A}_{kj} \times \widetilde{D}_{kj} \quad  \text{and} \quad \widetilde{A}_{kj} \cap \widetilde{A}_{ki}  = \emptyset$$
As a consequence, the off-diagonal terms are always $\hinf$. For the diagonal entries, $$(T^*T)_{ii} = \sum_{k=1}^J (T_{ki})^*T_{ki}  $$
Each term of this sum is a pseudodifferential operator with wavefront set  $$\WF(T_{ki}^* T_{ki} ) \subset \widetilde{D}_{ki}$$ 
Since the  $\widetilde{D}_{ki}$ are pairwise disjoint, $T^* T = \Id_{L^2(Y)} + \hinf$ microlocally near $\mathcal{T}$ if and only if for all $k,i$, $T^*_{ki}T_{ki} = \Id_{L^2(Y_i)} + \hinf$ microlocally near $D_{ki}$. The same computations apply to $TT^*$. As a consequence, $T$ is microlocally unitary near $\mathcal{T}$ if and only if for all $(k,i)$, $T_{ki}$ is a Fourier integral operator associated with $F_{ki}$, microlocally unitary near $D_{ki} \times A_{ki}$ (see the paragraphe below Definition \ref{Def_FIO_local}). 

\begin{nota}
An element of $S_\delta^{comp}(U)$ is a $J$-uple $\alpha=(\alpha_1, \dots, \alpha_J)$ where each $\alpha_j$ is an element of $S^\delta_{comp}(\R^2)$ such that $\ess \supp \alpha_j \subset U_j$ (this notation is recalled in the next section). \\
We fix a smooth function $\Psi_Y = (\Psi_1, \dots, \Psi_J)$ such that, for $1 \leq j \leq J$,  $\Psi_j \in \cinfc(Y_j, [0,1])$ satisfies $\Psi_j = 1$ on $\pi(U_j)$ (recall that $U_j \Subset T^*Y_j)$. \\
For $\alpha \in S_\delta^{comp}(U)$, we also note $\op(\alpha)$ the diagonal operator valued matrix: 
$$\op(\alpha) = \Diag(\Psi_1\op(\alpha_1)\Psi_1 , \dots, \Psi_J\op(\alpha_J) \Psi_J) :\bigoplus_{j=1}^J L^2(Y_j) \to \bigoplus_{j=1}^J L^2(Y_j)$$ 
Note that as operators on $L^2(\R)$, $\op(\alpha_j)$  and $\Psi_j \op(\alpha_j)\Psi_j$ are equal modulo $\hinf$.
\end{nota}

We can now state the main theorem of this paper, namely a spectral gap for hyperbolic open quantum maps.  We note $\rho_{spec}(A)$ the spectral radius of a bounded operator $ A : L^2(Y) \to L^2(Y)$.

\begin{thm}\label{Spectral_radius}
Suppose that the above assumptions on $F$ (\ref{Hyperbolicity_assumption}), (\ref{Fractal}) are satisfied. Then, there exists $\gamma >0$ such that the following holds : 

 Let $T=T(h) \in I_{0^+}(Y \times Y, F^\prime)$ be a semi-classical Fourier integral operator associated with $F$ in the sense of definition (\ref{def_FIO}) and $\alpha \in S_\delta^{comp}(U)$.  Assume that $T$ is microlocally unitary in a neighborhood of $\mathcal{T}$. 
Then, there exists $h_0>0$ such that
$$\forall 0 <h \leq h_0\quad , \quad \rho_{spec}( T(h)\op(\alpha) ) \leq e^{ -\gamma} ||\alpha||_\infty$$
$h_0$ depends on $(U,F)$, $T$ and semi-norms of $\alpha$ in $S_\delta$.
\end{thm} 

For applications, we will need the following corollary (it is in fact rather a corollary of the method used to prove Theorem \ref{Spectral_radius}) : 

\begin{corImp}\label{Cor_Thm} With the same notations and assumptions as in Theorem \ref{Spectral_radius}, 
if $R(h)$ is a family of bounded operators on $L^2(Y)$ satisfying $||R(h) || = O(h^\eta)$ for some $\eta >0$, then the there exists $\gamma^\prime$ depending only on $\gamma$ and $\eta$, such that for $0 < h \leq h_0$, 
$$  \rho_{spec}\big( T(h) \op(\alpha) + R(h) \big) \leq e^{-\gamma^\prime}||\alpha||_\infty$$
\end{corImp}

\begin{rem} \text{}
\begin{itemize}
\item If the value $h_0$ depends on $T$ and $\alpha$, this is not the case of $\gamma$ which depends on $(U,F)$. 
\item This is a spectral gap : it has to be compared with the easy bound we could have $$\rho_{spec}(T\op(\alpha) ) \leq ||\alpha||_{\infty} + o(1)$$ In particular, if $\alpha \equiv 1$ in a neighborhood of $\mathcal{T}$ and $|\alpha| \leq 1$ everywhere, $\rho_{spec} (T(h) ) \leq e^{-\gamma} < 1$.  
\item $T\op(\alpha)$ is the way we've chosen to write our Fourier integral operator with "gain" (or absorption depending on the modulus of $\alpha$) factor $\alpha$. $T\op(\alpha)$ transforms a wave packet $u_0$ microlocalized near $\rho_0$ lying in a small neighborhood of $\mathcal{T}$ into a wave packet microlocalized near $F(\rho_0)$, with norm essentially changed by a factor $|\alpha(\rho_0)|$.
\item The proof will actually show that if $\eta$ is strictly bigger than some threshold,  then $\gamma^\prime = \gamma$. 
\end{itemize}
\end{rem}

\begin{nota}
Throughout the paper, the meaning of the constants $C$ can change from line to line but these constants will only depend on our dynamical system $(U,F)$. If there is another dependence, it will be specified. 
\end{nota}

\subsection{Applications of the theorem }
This theorem has applications in the study of open quantum systems. We refer the reader to \cite{Nonnen11} for a survey on this topic. The spectral gap given by Theorem \ref{Spectral_radius} will actually give a spectral gap for the resonances of semiclassical operators $P(h)$ in $\R^2$, or for the resonances of the Dirichlet Laplacian in the exterior of strictly convex obstacles satisfying the Ikawa non-eclipse condition. We refer the reader to the review \cite{Zw17} for more background on scattering resonances or to the book \cite{DyZw}. The results we will obtain from Theorem \ref{Spectral_radius} give a positive answer (in dimension 2) to the Conjecture 3 in \cite{Zw17}, under a fractal assumption. 

\paragraph{Scattering by strictly convex obstacles in the plane} As already explained in the introduction the main problem motivating Theorem \ref{Spectral_radius}, is the problem of scattering by obstacles in the plane $\R^2$. It leads to 

\begin{thm}
Assume that $\mathcal{O} = \bigcup_{i=1}^J \mathcal{O}_j$ where $\mathcal{O}_j$ are open, strictly convex connected obstacles in $\R^2$ having smooth boundary and satisfying the \emph{Ikawa condition} : for $i \neq j \neq k$, $\overline{\mathcal{O}_i}$ does not intersect the convex hull of $ \overline{ \mathcal{O}_j }\cup \overline{\mathcal{O}_k}$. Let $$
\Omega = \R^2  \setminus \overline{\mathcal{O}}$$
There exist $\gamma >0$ and $\lambda_0 >1$ such that the Dirichlet Laplacian $-\Delta$ on $L^2(\Omega)$ has no scattering resonance in the region 
$$ [ \lambda_0, + \infty [ +i[- \gamma , 0 ]$$
\end{thm}

Let us give the arguments to see why Theorem \ref{Spectral_radius} implies this theorem. After a semiclassical reparametrization, is is enough to show that there exist $\delta>0$ and $h_0 >0$ such that $P(h) \coloneqq -h^2 \Delta -1$ has no resonance in $D(0,Ch) \cap \{ \im z \in [-\delta h , 0]\}$, for any $h \le h_0$. As already explained, the implication relies on \cite{NSZ14} (Theorem 5, Section 6).  They prove the existence of a family of \begin{equation}\label{monodromy}
(\mathcal{M}(z))_{z \in D(0,Ch)} = (\mathcal{M}(z,h))
\end{equation}
such that 
\begin{itemize}
\item $\mathcal{M}(z) = \Pi_h M(z) \Pi_h + O(h^L)$ where $\Pi_h$ is a finite rank projector, of rank comparable to $h^{-1}$, $L >0$ is a fixed constant (which can in fact be chosen as big as we want) and $M(z)$ is described below and satisfies  $\Pi_h M(z) \Pi_h = M(z) + O(h^L)$ ; 
\item $M(0)$ is an open quantum map associated with a Lagrangian relation $\mathcal{B}$ presented in the introduction, which is microlocally unitary near $\mathcal{T}$.  $\mathcal{B}$ and $M(0)$ play the role of $F$ and $T$ in Theorem \ref{Spectral_radius} and satisfy its assumptions ; 
\item $M(z) = M(0) \op \left( e^{ \frac{i z \tau}{h}} \right) + O\left( h^{1- \varepsilon} \right)$ uniformly in $D(0,Ch)$, where $\varepsilon>0$ can be chosen arbitrarily close to zero and $\tau \in \cinfc(U)$ is a smooth function (which has to be seen as a return time) ; 
\item  The resonances of $P(h)$ in $D(0,Ch)$, are the roots, with multiplicities, of the equation 
$$\det (I - \mathcal{M}(z) ) =0 $$ 
\end{itemize}
Hence, to prove the theorem, it is enough to show that the spectral radius of $\mathcal{M}(z)$ is strictly smaller than $1$ for $z \in D(0,Ch) \cap \{ \im z \in [-   \delta h , 0]\}$ for some $\delta >0$ and for $h$ small enough. To see that, we write 
$$ \mathcal{M}(z) = M(0)\op \left( e^{ \frac{i z \tau}{h}} \right)  + R(h)$$
with $R(h) = O(h^\eta)$ for any $\eta < \min(1,L)$. 
We apply Theorem \ref{Spectral_radius} and find some $\gamma^\prime$ such that 
$$\rho_{spec}(\mathcal{M}(z) ) \leq e^{-\gamma^\prime}  \left| \left| e^{i z \tau /h} \right| \right|_\infty  \leq e^{-\gamma^\prime} e^{ \delta \tau_{\max}}\quad , \quad  z \in D(0,Ch) \cap \{ \im z \in [-   \delta h , 0]\}  $$
where $\tau_{max} = || \tau||_\infty$. This ensures a spectral gap of size $$\delta < \frac{\gamma^\prime}{\tau_{\max}}$$

\paragraph{Schrödinger operators}

Actually, the obstacles, seen as infinite potential barriers, can be smoothened with a potential $V \in \cinfc(\R^2)$ and we can consider the Schrodinger operators $ P_0(h) = -h^2 \Delta + V(x) $

Unlike the obstacle problem, a simple rescaling does not allow to pass from energy $1$ to any energy $E$ and the behavior of the classical flow can drastically change from an energy shell to another. To study the problem at energy $E>0$, independent of $h$, we rather consider 
$$P(h) = P_0(h) -E$$
The resolvent $(P(h) - z)^{-1} $ continues meromorphically from $\im z >0$ to $D(0,Ch)$ (as previously in the sense that $ \chi (P(h) - z)^{1} \chi$ extends meromorphically with $\chi \in \cinfc(\R^2)$) and we are interested in the existence of a spectral gap. 

The classical Hamiltonian flow associated with $P(h)$ is the Hamiltonian flow $\Phi^t$ generated by $p_0(x,\xi) = |\xi|^2 + V(x)$ on the energy shell $p_0^{-1}(E)$. The trapped set is defined as above by 

$$ K_E \coloneqq \{ (x,\xi) \in T^* \R^2, p_0(x,\xi) = E, \Phi^t(x,\xi) \text{ stays bounded as } t \to \pm \infty \}$$ 
We assume that the flow is hyperbolic on $K_E$ and that the trapped set is topologically one-dimensional. Equivalently, we assume that transversely to the flow, $K_E$ is zero-dimensional. Under these assumptions, the authors proved (see Theorem 1 in \cite{NSZ11}) the existence of a family of monodromy operators associated with a Lagrangian relation $F_E$ which is a Poincaré map of the flow on different Poincaré sections $\Sigma_1, \dots, \Sigma_J \subset p_0^{-1}(E)$. The assumption on the dimension of $K_E$ implies that the assumption (\ref{Fractal}) is satisfied since $K_E$ cannot be an attractor (\cite{BoRu}). Hence, Theorem 1 applies and we can prove as done in the case of obstacles  

\begin{thm}
Under the above assumptions, there exists $\delta>0$ such that $P(h)$ has no resonances in $$D(0,Ch) \cap \{ \im z \in [-i \delta h ,0 ] \}$$ 
\end{thm}

\section{Preliminaries} \label{Section_preliminaries}

\subsection{Pseudodifferential operators and Weyl quantization}
We recall some basic notions and properties of the Weyl quantization on $\R^n$. We refer the reader to \cite{ZW} for the proofs of the statements and further considerations on semiclassical analysis and quantizations. We start by defining classes of $h$-dependent symbols. 

\begin{defi}
Let $0 \leq \delta \leq \frac{1}{2}$. We say that an $h$-dependent family $a \coloneqq \left( a (\cdot ; h) \right)_{0 < h \leqslant 1}$ is in the class $S_\delta(T^*\R^n)$ (or simply $S_\delta$ if there is no ambiguity) if for every $\alpha \in \N^{2n}$, there exists $C_\alpha >0$ such that : 
$$\forall 0< h \leq 1, \sup_{(x,\xi) \in \R^n} | \partial^\alpha a (x,\xi ; h) | \leq C_\alpha h^{-\delta |\alpha|}$$
\end{defi}
In this paper, we will mostly be concerned with $\delta <1/2$. We will also use the notation $S_{0^+} = \bigcap_{\delta >0} S_\delta$. \\
We write $a =  O(h^N)_{S_\delta}$ to mean that for every $\alpha \in \N^{2n}$, there exists $C_{\alpha,N}$ such that 
$$\forall 0< h \leq 1, \sup_{(x,\xi) \in \R^n}   | \partial^\alpha a (x,\xi ; h) | \leq  C_{\alpha,N} h^{-\delta |\alpha|} h^N $$ 
 If $a=O(h^N)_{S_\delta}$ for all $N \in N$, we'll write $a= O(h^\infty)_{S_\delta}$. \textit{A priori}, the constants $C_{\alpha,N}$ depend on the symbol $a$. However, in this paper, we will often make them depend on different parameters but not directly on $a$. This will be specified when needed. 

For a given symbol $a \in S_\delta(T^*\R^n)$, we say that $a$ has a compact essential support if there exists a compact set $K$ such : 
$$ \forall \chi \in \cinfc(\Omega), \supp \chi \cap K = \emptyset \implies \chi a =\hinf_{\mathcal{S}(T^*\R^n) }$$
(here $\mathcal{S}$ stands for the Schwartz space). We note $ \text{ess} \supp a \subset K$ and say that $a$ belongs to the class $S_\delta^{comp}(T^*\R^n) $. The essential support of $a$ is then the intersection of all such compact $K$'s. In particular, the class $S_{\delta}^{comp}$ contains all the symbols supported in a $h$-independent compact set and these symbols correspond, modulo $\hinf_{\mathcal{S}(T^*\R) }$, to all symbols of $S_\delta^{comp}$. For this reason, we will adopt the following notation :  $a \in S_\delta^{comp}(\Omega) \iff \text{ess} \supp a \Subset \Omega$.

For a symbol $a \in S_\delta(T^*\R^n)$, we'll quantize it using Weyl's quantization procedure. It is informally written as : 
$$(\op(a)u)(x) = (a^W u)(x)= \frac{1}{(2\pi h)^n }\int_{\R^{2n}} a\left( \frac{x+y}{2}, \xi \right)u(y) e^{i \frac{(x-y)\cdot \xi}{h}}dy d \xi$$

We will note $\Psi_\delta(\R^n)$ the corresponding classes of pseudodifferential operators. By definition, the wavefront set of $A = \op(a)$ is $\WF(A) = \text{ess} \supp a $. 
\vspace*{0.5cm}

We say that a family $u=u(h) \in \mathcal{D}^\prime(\R^n)$ is $h$-tempered if for every $\chi \in \cinfc(\R^n)$, there exist $C >0$ and $N \in \N$ such that $|| \chi u ||_{H_h^{-N}} \leq Ch^{-N}$. For a $h$-tempered family $u$, we say that a point $\rho \in T^*\R^n$ does \emph{not} belong to the wavefront set of $u$ if there exists $a \in S^{comp}(T^*\R^n)$ such that $a(\rho) \neq 0$ and $\op(a)u = \hinf_{\mathcal{S}}$. We note $\WF(u)$ the wavefront set of $u$. 

We say that a family of operators $B=B(h) : \cinfc(\R^{n_2}) \to \mathcal{D}^\prime(\R^{n_1})$ is $h$-tempered  if its Schwartz kernel $\mathcal{K}_B \in \mathcal{D}^\prime(\R^{n_1} \times \R^{n_2})$ is $h$-tempered. We define 
$$ \WF^\prime(B) = \{ ( x,\xi, y , -\eta) \in T^*\R^{n_1}  \times T^* \R^{n_2} ,  (x,\xi, y , \eta)  \in \WF(\mathcal{K}_B )\}$$

\vspace*{0.5cm}
Let us now recall standard results in semi-classical analysis concerning the $L^2$-boundedness of pseudodifferential operator and their composition. 
We'll use the following version of Calderon-Vaillancourt Theorem (\cite{ZW}, Theorem 4.23). 

\begin{thm}
There exists $ C_n>0$ such that the following holds. For every $0 \leq \delta < \frac{1}{2}$, and $a \in S_\delta(T^* \R^n)$, $\op(a)$ is a bounded operator on $L^2$ and 
$$ || \op(a) ||_{L^2 (\R^n) \to L^2 (\R^n) } \leq C_n \sum_{ |\alpha| \leq 8n } h^{|\alpha|/2} || \partial^\alpha a ||_{L^\infty} $$ 
\end{thm}

As a consequence of the sharp Gärding inequality (see \cite{ZW}, Theorem 4.32), we also have the precise estimate of $L^2$ norms of pseudodifferential operator, 

\begin{prop}\label{Garding}
Assume that $a \in S_\delta(\R^{2n})$. Then, there exists $C_a$ depending on a finite number of semi-norms of $a$ such that : 
$$ || \op(a) ||_{L^2 \to L^2} \leq ||a||_\infty + C_ah^{\frac{1}{2} - \delta}$$
\end{prop}

We recall that the Weyl quantizations of real symbols are self-adjoint in $L^2$. The composition of two pseudodifferential operators in $\Psi_\delta$ is still a pseudodifferential operator. More precisely (see \cite{ZW}, Theorem 4.11 and 4.18), if $a , b \in S_\delta$, $\op(a) \circ \op(b)$ is given by $\op(a \# b)$, where $a \# b$ is the Moyal product of $a$ and $b$. It is given by 
$$ a \# b (\rho) = e^{i h A(D) } (a \otimes b)|_{ \rho= \rho_1 = \rho_2} $$ 
where $a \otimes b (\rho_1, \rho_2) = a(\rho_1) b(\rho_2)$, $e^{i h A(D) }$ is a Fourier multiplier acting on functions on $\R^{4n}$ and, writing $\rho_i = (x_i, \xi_i)$,  $$A(D) = \frac{1}{2} \left( D_{\xi_1} \circ D_{x_2} - D_{x_1} \circ D_{\xi_2} \right) $$
We can estimate the Moyal product by a quadratic stationary phase and get the following expansion: for all $N \in \N$,  

\begin{equation*}
 a \# b (\rho) = \sum_{k=0}^{N-1} \frac{i^k h^k}{k!} A(D)^{k} (a \otimes b)|_{\rho= \rho_1= \rho_2} + r_N
\end{equation*}
where for all $\alpha \in \N^{2n}$, there exists $C_\alpha$, independent of $a$ and $b$, such that 
\begin{equation*}
||\partial^\alpha r_N ||_\infty \leq C_\alpha h^N ||a \otimes b ||_{ C^{2N+ 4n+1+ |\alpha| } }
\end{equation*}
As a consequence of this asymptotic expansion, we have the precise product formula :

\begin{lem}\label{Moyal_produc_op}
For every $N \in \N$, there exists $C_N >0$ such that, for every $a,b \in S_\delta(\R^n)$, 
\begin{equation}
\op(a) \circ \op(b) = \op \left( \sum_{k=0}^{N-1} \frac{i^k h^k}{k!} A(D)^{k} (a \otimes b)|_{\rho= \rho_1= \rho_2}  \right)  + R_N
\end{equation}
where \begin{equation}
||R_N||_{L^2(\R) \to L^2(\R) } \leq C_N h^N ||a \otimes b ||_{C^{2N+12n+1}}
\end{equation}
\end{lem}

\begin{rem}
It will be important in the sequel to understand the derivatives of $a$ and $b$ involved in the $k$-th term of the previous expansion. A quick recurrence using the precise form of the operator $A(D)$ shows that $A(D)^{k} (a \otimes b)(\rho_1,\rho_2) $ is of the form 
$$ \sum_{|\alpha| = k , |\beta| = k} \lambda_{\alpha, \beta} \partial^\alpha a(\rho_1) \partial^\beta b (\rho_2) $$ 
This can be rewritten $l_k \left( d^k a(\rho_1), d^k b(\rho_2) \right) $ where $l_k$ is a bilinear form on the spaces of $k$-symmetric forms on $\R^{2n}$. Of, course, we make use of the the identifications $T_{\rho_1} T^*\R^n  \simeq T_{\rho_2} T^*\R^n \simeq  \R^{2n}$
\end{rem}

As a simple corollary, we get an expression for the commutator of pseudodifferential operators. 
\begin{cor}
For every $N \in \N$, there exists $C_N >0$ such that, for every $a,b \in S_\delta(\R^n)$, 
$$ [\op(a), \op(b) ] = 
 \op \left( \frac{h}{i} \{ a,b\} + \sum_{k=2}^{N-1} h^k L_k (d^k a ,d^k b) \right)  + R_N
$$
where \begin{equation*}
||R_N||_{L^2(\R) \to L^2(\R) } \leq C_N h^N ||a \otimes b ||_{C^{2N+12n+1}}
\end{equation*}
where the $L_k$ are bilinear forms on the spaces of $k$-symmetric forms on $\R^{2n}$. 
\end{cor}

\subsection{Fourier Integral Operators}
We now review some aspects of the theory of Fourier integral operators. We follow \cite{ZW}, Chapter 11 and \cite{NSZ14}. We refer the reader to \cite{GuSt} for further details. Finally, we will give the precise definition needed to understand the definition \ref{def_FIO}. 

\subsubsection{Local symplectomorphisms and their quantization}
We momentarily work in dimension $n$. 
Let us note $\mathcal{K}$ the set of symplectomorphisms $\kappa : T^*\R^n \to T^* \R^n$ such that the following holds : there exist continuous and piecewise smooth families of smooth functions $(\kappa_t)_{t \in [0,1]} $, $(q_t)_{t \in [0,1]}$ such that : \begin{itemize}[nosep]
\item $\forall t \in [0,1]$, $\kappa_t : T^*\R^n \to T^*\R^n$ is a symplectomorphism ; 
\item $\kappa_0 = \Id_{T^* \R^n} , \kappa_1 = \kappa$ ; 
\item $\forall t \in [0,1], \kappa_t(0)=0 $ ; 
\item there exists $K \Subset T^*\R^n$ compact such that $\forall t \in [0,1], q_t : T^*\R^n \to \R$ and $ \supp q_t \subset K$ ; 
\item $\frac{d}{dt} \kappa_t = \left( \kappa_t \right)^* H_{q_t}$
\end{itemize}
 If $\kappa \in \mathcal{K}$, we note $C = Gr^\prime(\kappa) = \{ (x, \xi, y , - \eta), (x,\xi) = \kappa(y,\eta) \}$ the twisted graph of $\kappa$. 
We recall  \cite{ZW}, Lemma 11.4, which asserts that local symplectomorphisms can be seen as elements of $\mathcal{K}$, as soon as we have some geometric freedom. 

\begin{lem}\label{lemma_local_symp}
Let $U_0, U_1$ be open and precompact subsets of $T^*\R^n$. Assume that $\kappa : U_0 \to U_1$ is a local symplectomorphism fixing 0 and which extends to $V_0 \Supset U_0$ an open star-shaped neighborhood of 0. Then, there exists $\tilde{\kappa} \in \mathcal{K}$ such that $\tilde{\kappa}|_{U_0} = \kappa$. 
\end{lem}

 If $\kappa \in \mathcal{K}$ and  if $(q_t)$ denotes the family of smooth functions associated with $\kappa$ in its definition, we note $Q(t) = \op(q_t)$. It is a continuous and piecewise smooth family of operators. Then the Cauchy problem
 
 \begin{equation}\label{Cauchy_pb_Egorov}
\left\{ \begin{array}{c}
hD_t U(t) + U(t)Q(t) = 0 \\
U(0)=\Id
\end{array}
\right. 
\end{equation}
is globally well-posed.

Following \cite{NSZ14}, Definition 3.9, we adopt the definition : 

\begin{defi}
Fix $\delta \in [0,1/2)$. We say that $U \in I_\delta(\R^n \times \R^n; C )$ if there exist $a \in S_\delta(T^*\R^n)$ and a path $(\kappa_t)$ from $\Id$ to $\kappa$ satisfying the above assumptions such that $U = \op(a)U(1)$, where $t \mapsto U(t)$ is the solution of the Cauchy problem (\ref{Cauchy_pb_Egorov}). 

The class $I_{0^+}(\R \times \R, C)$ is by definition $\bigcap_{\delta >0} I_\delta(\R \times \R, C)$. 
\end{defi}

It is a standard result, known as Egorov's theorem (see \cite{ZW}, Theorem 11.1) that if $U(t)$ solves the Cauchy problem (\ref{Cauchy_pb_Egorov}) and if $a \in S_\delta$, then $U^{-1} \op(a) U$ is a pseudodifferential operator in $\Psi_\delta$ and if $b= a \circ \kappa$, then 
$U^{-1} \op(a) U - \op(b) \in h^{1- 2 \delta} \Psi_\delta$.

\begin{rem}\text{}
Applying Egorov's theorem and Beal's theorem, it is possible to show that if $(\kappa_t)$ is a closed path from $\Id$ to $\Id$, and $U(t)$ solves (\ref{Cauchy_pb_Egorov}), then $U(1) \in \Psi_0(\R^n)$. In other words, $I_\delta(\R \times \R , \Gr^\prime(\Id) ) \subset \Psi_\delta(\R^n)$. But the other inclusion is trivial. Hence, this in an equality : 
$$I_\delta(\R^n \times \R^n , \Gr^\prime(\Id) ) = \Psi_\delta(\R^n)$$
The notations $I(\R^n\times \R^n, C)$ comes from the fact that the Schwartz kernel of such operators are Lagrangian distributions associated with $C$, and in particular have wavefront set included in $C$. As a consequence, if $T \in I_\delta(\R^n \times \R^n, C)$, $\WF^\prime(T) \subset \Gr(T)$. 
\end{rem}

Let us state a simple proposition concerning the composition of Fourier integral operators : 

\begin{prop}\label{composition_FIO}
Let $\kappa_1, \kappa_2 \in \mathcal{K}$ and $U_1 \in I_\delta(\R \times \R, \Gr^\prime(\kappa_1) ) , U_2 \in I_\delta(\R \times \R, \Gr^\prime(\kappa_1) )$. Then, 
$$ U_1 \circ U_2 \in  I_\delta(\R \times \R, \Gr^\prime(\kappa_1 \circ \kappa_2) )$$ 
\end{prop}

\begin{proof}
Let's write $U_1 = \op(a_1) U_1(1)$, $U_2 = \op(a_2) U_2(1)$ with the obvious notations associated with the Cauchy problems (\ref{Cauchy_pb_Egorov}) for $\kappa_1$ and $\kappa_2$. Egorov's theorem asserts that $U_1(1) \op(a_2)U_1(1)^{-1} = \op(b_2)$ for some $b_2 \in S_\delta$ and $\op(a_1)\op(b_2) = \op(a_1 \# b_2)$. It is then enough to focus on the case $a_1= a_2 =1$. We set  $$
U_3(t) \coloneqq \left\{ \begin{array}{l} U_1(2t) \text{ for } 0 \leq t \leq 1/2 \\
U_1(1)\circ U_2(2t-1)  \text{ for } 1/2 \leq t \leq 1 \end{array}\right. $$
It solves  the Cauchy problem $$\left\{ \begin{array}{c}
hD_t U_3(t) + U_3(t)Q_3(t) = 0 \\
U(0)=\Id
\end{array}
\right. 
$$
with  $$
Q_3(t) \coloneqq \left\{ \begin{array}{l} 2 Q_1(2t) \text{ for } 0 \leq t \leq 1/2 \\
2Q_2 (2t-1) \text{ for } 1/2 \leq t \leq 1 \end{array}\right. $$
To conclude the proof, it is enough to notice that this Cauchy problem is associated with the path $\kappa_3(t)$ between $\kappa(0) = \Id$ and $\kappa_3(1) = \kappa_1 \circ \kappa_2$ where  $$
\kappa_3(t) \coloneqq \left\{ \begin{array}{l} \kappa_1(2t) \text{ for } 0 \leq t \leq 1/2 \\
\kappa_1 \circ \kappa_2(2t-1)  \text{ for } 1/2 \leq t \leq 1 \end{array}\right. $$
\end{proof}

\subsubsection{Precise version of Egorov's theorem}

We will need a more quantitative version of Egorov's theorem, similar to the one in \cite{NDJ19} (Lemma A.7). The result does not show that $U(1)^{-1} \op(a) U(1)$ is a pseudodifferential operator (one would need Beal's theorem to say that) but it gives a precise estimate on the remainder, depending on the semi-norms of $a$. We now specialize to the case of dimension $n=1$ but the following result holds in any dimension but changing the constant 15 in something of the form $Mn$. 

\begin{prop}\label{prop_Egorov}
Consider $\kappa \in \mathcal{K}$ and note $U(t)$ the solution of (\ref{Cauchy_pb_Egorov}). There exists a family of differential operators $(D_{j})_{j \in \N}$ of order $j$ such that for all $a \in S_\delta$ and all $N \in \N$, 
\begin{equation}
U(1)^{-1} \op(a) U(1) = \op \left(a \circ \kappa +  \sum_{j=1}^{N-1} h^j (D_{j+1} a) \circ \kappa \right) + O_{\kappa} \left( h^N ||a||_{C^{2N +15}} \right)
\end{equation}
\end{prop}

\begin{proof}
We keep the notations introduced previously. 
Let us first note $$A_0(t) = U(t) \op(a \circ \kappa_t  ) U(t)^{-1}$$ and compute 
\begin{align*}
U(t)^{-1} \partial_t A_0(t) U(t) &= -\frac{i}{h} [Q(t) , \op( a \circ \kappa_t) ] + \op \left( \{ q_t, a \circ \kappa_t \} \right) \\
&= \op \left( \{ q_t, a \circ \kappa_t \} \right) - \frac{i}{h} \left(  \op \left(  \frac{h}{i } \{ q_t, a \circ \kappa_t \} + \sum_{j=2}^{N} h^j L_j (d^j q_t, d^j (a \circ \kappa_t) ) \right) \right) 
\\&+ O\left(h^{N} ||q_t \otimes (a \circ \kappa_t)||_{C^{2(N+1) +13}}  \right)  \\ 
& =  \op \left( \sum_{j=1}^{N-1} -i h^j L_{j+1} (d^{j+1} q_t, d^{j+1} (a \circ \kappa_t) ) \right)
+ O_{\kappa_t} \left(h^{N} ||a||_{C^{2N +15}} \right) 
\end{align*}
We now define by induction a family of functions $a_j(t), j=0, \dots, N-1$ by 
\begin{equation*}
a_0(t) = a \; ; \; a_k(t) = \sum_{m=0}^{k-1}  \int_0^t i L_{k+1-m} \left( d^{k+1-m} q_s , d^{k+1-m} ( a_m(s) \circ \kappa_s) \right) \circ \kappa_s^{-1} ds
\end{equation*}
and set $A_k(t) = U(t) \op \left( \sum_{j=0}^k h^j a_j(t) \circ \kappa_t \right) U(t)^{-1}$. 
We first remark by an easy induction on $k$, that $a_k(t)$ is of the form $D_{k+1}(t) a$ where $D_{k+1}(t)$ is a differential operator of order at most $k+1$, with coefficients depending continuously on $t$ and on $(\kappa_t)_t$. We now check by induction the following : 
$$U(t)^{-1} \partial_t A_k(t) U(t) = -i \op \left( \sum_{j=k+1}^{N-1} \sum_{m=0}^k h^{j}L_{j+1-m} \left( d^{j+1-m} q_t, d^{j+1-m} (a_m(t) \circ \kappa_t ) \right)  \right) + O_{\kappa} \left( h^N ||a||_{C^{2N+15}}\right)$$
We've already done it for $k=0$. Let's assume that the equality holds for $k-1$ and let's prove it for $k \geq 1$. 

$$U(t)^{-1} \partial_t A_k(t) U(t) =U(t)^{-1} \partial_t A_{k-1}(t) U(t) + h^k U(t)^{-1} \partial_t \op\left(  a_{k}(t) \circ \kappa_t \right)U(t)$$
Let's compute the second part of the right hand side. 
\begin{align*}
& U(t)^{-1} \partial_t \op\left( a_{k}(t) \circ \kappa_t \right) U(t)  \\
&= - \frac{i}{h} [Q(t), \op (a_k(t) \circ \kappa_t)] + \op( \{ q_t, a_k(t) \circ \kappa_t \} ) + \op \left( \partial_t a_k(t)  \circ \kappa_t \right) \\
&= -i\op \left( \sum_{l=1}^{N-1-k} h^j L_{l+1} \left(d^{l+1} q_t, d^{l+1} (a_k(t) \circ \kappa_t) \right) \right)+ O_\kappa\left( h^{N-k} ||a_k(t)||_{C^{2(N+1 - k )+13}} \right)  +  \op \left( \partial_t a_k(t)  \circ \kappa_t \right)
\end{align*}
We can estimate the remainder by 
$$ O_\kappa\left( h^{N-k} ||a_k(t)||_{C^{2(N+1 - k )+13}} \right)  =  O_\kappa\left( h^{N-k} ||a||_{C^{2(N+1 - k )+13+k+1}} \right) = O_\kappa\left( h^{N-k} ||a||_{C^{2N+15}} \right)$$
We now combine this with the value of $$U(t)^{-1} \partial_t A_{k-1}(t) U(t) = -i \op \left( \sum_{j=k}^{N-1} \sum_{m=0}^{k-1} h^{j}L_{j+1-m} \left( d^{j+1-m} q_t, d^{j+1-m} (a_m(t) \circ \kappa_t ) \right)  \right) + O_{\kappa} \left( h^N ||a||_{C^{2N+15}}\right)$$
By definition of $a_k(t)$, the term  $h^k \op \left( \partial_t a_k(t)  \circ \kappa_t \right)$ cancels the term corresponding to $j=k$ in the sum. Moreover, for every $j >k$, writing $j= k+l, l \in \{ 1, \dots, N-1-k \}$, the term $h^{k+l} L_{l+1} \left(d^{l+1} q_t, d^{l+1} (a_k(t) \circ \kappa_t) \right) $, gives the missing term $h^j L_{j+1-k} \left(d^{j+1-k} q_t, d^{j+1-k} (a_k(t) \circ \kappa_t) \right) $. This gives the required equality for $A_k(t)$. 

In particular, $\partial_t A_{N-1}(t) = O_{\kappa} \left( h^N ||a||_{C^{2N+15}}\right)$. 
We now use the fact that at $t=0$, $a_0(0) = a , a_k(0) = 0, k=1, \dots, N-1$, $U(0) = \Id, \kappa_0 = \Id$,  and hence $A_{N-1}(0) = \op(a)$. Integrating between $0$ and $1$, we hence have 
$$ A_{N-1}(t) - \op(a) = O_{\kappa} \left( h^N ||a||_{C^{2N+15}}\right)$$
Conjugating by $U(1)$, we finally have 
$$U(1)^{-1} \op(a) U(1) = \op( a \circ \kappa + \sum_{k=1}^{N-1} h^k a_k(1) \circ \kappa ) + O_{\kappa} \left( h^N ||a||_{C^{2N+15}}\right)$$
which is the what we wanted, since $a_k(1) = D_{k+1}(1) a$. 
\end{proof}

\subsubsection{An important example}
Let us focus on a particular case of canonical transformations. 
Suppose that $\kappa: T^*\R^n \to T^*\R^n$ is a canonical transformation such that 
$$ (x,\xi, y, \eta) \in \Gr(\kappa) \mapsto (x, \eta) $$
is a local diffeomorphism near $(x_0, \xi_0 , y_0, \eta_0 )$. Then, there exists a phase function $\psi \in \cinf(\R^n \times \R^n)$, $\Omega_x, \Omega_\eta$ open sets of $\R^n$ and $\Omega$ a neighborhood of $(x_0, \xi_0 , y_0, \eta_0 ) $,  such that 
$$ \Gr^\prime(\kappa)  \cap \Omega= \{ (x, \partial_x \psi(x, \eta) , \partial_\eta\psi(x,\eta), -\eta) , x \in \Omega_x, \eta \in \Omega_\eta \} $$ 
One says that $\psi$ generates $\Gr^\prime(\kappa)$. 
Suppose that that $\alpha \in S_\delta^{comp} (\Omega_x \times \Omega_\eta)$. Then,  modulo a smoothing operator $\hinf$, the following operator $T$ is an element of $I_\delta^{comp} ( \R^n \times \R^n , \Gr^\prime(\kappa) ) $ : 

$$ Tu(x) = \frac{1}{(2\pi h)^n} \int_{\R^{2n}} e^{ \frac{i}{h} (\psi(x,\eta) - y \cdot \eta) } \alpha(x,\eta) u(y) dy d\eta $$
and if $T^*T = \Id$ microlocally near $(y_0,\eta_0)$ then $|\alpha(x, \eta)|^2 = | \det D^2_{ x \eta}  \psi (x, \eta) | + O(h^{1-2 \delta})_{S_\delta}$ near $(x_0, \xi_0, y_0, \eta_0)$. The converse statement holds : microlocally near $(x_0,\xi_0, y_0, \eta_0)$ and modulo $\hinf$, the elements of $I_\delta(\R^n \times \R^n, \Gr^\prime(\kappa) ) $ can be written under this form. 

\subsubsection{Lagrangian relations}
Recall that the Lagrangian relation $F$ we consider is the union of local Lagrangian relations $F_{ij} \subset U_i \times U_j$. We fix a compact set $W \subset \pi_L(F)$ containing some neighborhood of $\mathcal{T}$. Our definition will depend on $W$. 
Following \cite{NSZ14} (Section 3.4.2), we now focus on the definition of the elements of $I_\delta(Y \times Y; F^\prime)$. An element $T \in I_\delta(Y \times Y; F^\prime)$ is a matrix of operators  
$$T = (T_{ij})_{1 \leq i,j\leq J } : \bigoplus_{j=1}^J L^2(Y_j) \to \bigoplus_{i=1}^J L^2(Y_i)$$
Each $T_{ij}$ is an element of $I_\delta(Y_i \times Y_j , F_{ij}^\prime )$. Let's now describe the recipe to construct elements of $I_\delta(Y_i \times Y_j, F_{ij}^\prime )$. We fix $i,j \in \{1, \dots , J \}$. 

\begin{itemize}
\item Fix some small $\varepsilon >0$ and two open covers of $U_j$, $U_j\subset \bigcup_{l=1}^L \Omega_l$, $\Omega_l \Subset \widetilde{\Omega}_l$, with $\widetilde{\Omega}_l$ star-shaped and having diameter smaller than $\varepsilon$. We note $\mathcal{L}$ the sets of indices $l$ such that $\Omega_l \subset \pi_R(F_{ij})=\widetilde{D}_{ij} \subset U_j$ and we require (this is possible if $\varepsilon$ is small enough)  
$$ F^{-1}(W) \cap U_j \subset \bigcup_{l \in \mathcal{L}} \Omega_l $$

\item Introduce a smooth partition of unity associated with the cover $(\Omega_l)$, $(\chi_l)_{1 \leq l \leq L} \in \cinfc(\Omega_l, [0,1])$, $\supp \chi_l \subset \Omega_l$, $\sum_l \chi_l = 1$ in a neighborhood of $\overline{U_j}$.
\item  For each $l \in \mathcal{L}$, we denote $F_l$ the restriction to $\widetilde{\Omega}_l$ of $F_{ij}$, seen as a symplectomorphism $F_{ij} : \widetilde{D}_{ij} \subset U \to V$. By Lemma \ref{lemma_local_symp}, there exists $\kappa_l \in \mathcal{K}$ which coincides with $F_l$ on $\Omega_l$. 
\item  We consider $T_l=\op(\alpha_i) U_l(1)$ where $U_l(t)$ is the solution of the Cauchy problem (\ref{Cauchy_pb_Egorov}) associated with $\kappa_l$ and $\alpha_i \in S_\delta^{comp}(T^*\R)$. 
\item We set 

\begin{equation}\label{def_FIO_global}T^\R = \sum_{l \in \mathcal{L}} T_l \op(\chi_l) : L^2(\R) \to L^2(\R)
\end{equation} 
$T^\R$ is a globally defined Fourier integral operator. We will note $T^\R \in I_\delta(\R \times \R, F_{ij}^\prime)$. Its wavefront set is included in $\widetilde{A}_{ij} \times \widetilde{D}_{ij}$. 
\item Finally, we fix cut-off functions $(\Psi_i, \Psi_j) \in \cinfc(Y_i, [0,1]) \times \cinfc(Y_j, [0,1])$ such that $\Psi_i \equiv 1$ on $\pi(U_i)$ and $\Psi_j \equiv 1$ on $\pi(U_j)$(here, $\pi : (x,\xi) \in T^*Y_{\cdot} \mapsto  x \in Y_{\cdot}$ is the natural projection) and we adopt the following definitions : 
\end{itemize}

\begin{defi}\label{Def_FIO_local}
We say that $T : \mathcal{D}^\prime(Y_j) \to \cinf(\overline{Y_i})$ is a Fourier integral operator in the class $I_\delta(Y_i \times Y_j, F_{ij}^\prime)$  if there exists $T^\R \in I_\delta(\R \times \R, F^\prime)$ as constructed above such that
\begin{itemize}
\item $T - \Psi_i T \Psi_j = \hinf_{\mathcal{D}^\prime(Y) \to \cinf(\overline{Z})}$; 
\item $\Psi_i T \Psi_j =  \Psi_i T^\R \Psi_j $
\end{itemize}
\end{defi}

For $U^\prime_j \subset U_j$ and $U_i^\prime = F(U_j^\prime) \subset U_i$, we say that $T$ (or $T^\R$)  is microlocally unitary in $U_i^\prime \times U_j^\prime$ if $TT^* = \Id$ microlocally in $U_i^\prime$ and $T^*T = \Id$ microlocally in $U_j^\prime$. 

\begin{rem}
The definition of this class is not canonical since it depends in fact on the compact set $W$ through the partition of unity. 
\end{rem}

\paragraph{Another version of Egorov's theorem.} The precise version of Egorov's theorem in Proposition \ref{prop_Egorov} is only stated for globally unitary Fourier integral operator defined using the Cauchy problem \ref{Cauchy_pb_Egorov}. We extend it here to microlocally unitary and globally defined Fourier integral operators. We fix $i,j \in \{1, \dots ,J \}$. 

\begin{lem}\label{lem_precise_Egorov}
Let $T \in I_\delta(\R \times \R, F_{ij}^\prime)$. Suppose that $B(\rho, 4 \varepsilon) \subset U_j$ and that $T$ is microlocally unitary in $F_{ij}(B(\rho, 4 \varepsilon)) \times B(\rho, 4 \varepsilon)$. Then, there exists a family $(D_k)_{k \in \N}$ of differential operators of order $k$, compactly supported in $B(\rho, 3\varepsilon)$ such that the following holds : 
For every $N \in \N$ and for all $b \in \cinfc(B(\rho, 2 \varepsilon))$, 
$$ T\op(b) = \op\left( b \circ \kappa^{-1}   + \sum_{k=1}^{N-1} h^k(D_{k+1}b) \circ \kappa^{-1}  \right) T+ O\left( h^N ||b||_{C^{2N+15} } \right)_{L^2(\R) \to L^2(\R)}$$
The constants in $O$ depend on $T$ and $F$.
\end{lem}

\begin{proof}
First, introduce some cut-off function $\chi$ such that $\chi \equiv 1$ in a neighborhood of $B(\rho,2\varepsilon)$ and $\supp \chi \subset B(\rho, 3\varepsilon)$. Due to these properties and Proposition \ref{Moyal_produc_op}, we have 
$$ \op(b) = \op(\chi) \op(b) \op(\chi) \op(\chi) +  O\left( h^N ||b||_{C^{2N+13} } \right)_{L^2(\R) \to L^2(\R)}$$
Moreover, $\op(\chi) T^*T = \op(\chi) + \hinf$ and hence, 
$$T\op(b) = T\op(\chi)\op(b) \op(\chi)\op(\chi) T^* T + O\left( h^N ||b||_{C^{2N+13} } \right)_{L^2 \to L^2} + \hinf ||\op(b) ||_{L^2 \to L^2}$$
The term $\hinf ||\op(b) ||_{L^2 \to L^2}$ can be absorbed in $O\left( h^N ||b||_{C^{2N+13} } \right)_{L^2 \to L^2} $. Consider $\tilde{\kappa} \in \mathcal{K}$ extending $\kappa|_{B(\rho, 3 \varepsilon)}$ and construct $U = U(1)$ by solving the Cauchy problem (\ref{Cauchy_pb_Egorov}) associated with $\tilde{\kappa}$. Due to the properties on composition of Fourier integral operators (Proposition \ref{composition_FIO}), $T\op(\chi)U^{-1}$ and $U \op(\chi)T^*$ are pseudodiffferential operators, and we note them $\op(a_1), \op(a_2)$. Now write
\begin{align*}T\op(b) &=  \left[T\op(\chi)U^{-1}\right]  U\op(b) \op(\chi)U^{-1} \left[U\op(\chi) T^* \right] T + O\left( h^N ||b||_{C^{2N+13} } \right)_{L^2 \to L^2} \\
&= \op(a_1) \left[ U\op(b) \op(\chi)U^{-1}\right]  \op(a_2) T + O\left( h^N ||b||_{C^{2N+13} } \right)_{L^2 \to L^2}
\end{align*}
By using the precise version in Proposition \ref{prop_Egorov}, one can write  
$$U\op(b) \op(\chi)U^{-1} = \op \left( b \circ \kappa^{-1} + \sum_{k=1}^{N-1} (L_{k+1} b) \circ \kappa^{-1} \right) + O\left( h^N ||b||_{C^{2N+15} } \right)_{L^2 \to L^2}$$
Applying Lemma \ref{Moyal_produc_op}, we see that we can write 
$$T\op(b) =   \op \left( b_0 \circ \kappa^{-1} + \sum_{k=1}^{N-1} (D_{k+1} b) \circ \kappa^{-1} \right) T + O\left( h^N ||b||_{C^{2N+15} } \right)_{L^2 \to L^2}$$
where $b_0 = a_1 \times b \circ \kappa^{-1}\times  a_2$. $T$ being microlocally unitary in $B(\rho, 4 \varepsilon)$, the product $a_1 a_2$ is equal to $1$ in $B(\rho, 2\varepsilon)$, and hence, the lemma is proved. 
\end{proof}

\subsection{Hyperbolic dynamics}\label{section_hyperbolic} \label{Subsection_hyperbolic_dynamics}

We assumed that $F$ is hyperbolic on the trapped set $\mathcal{T}$. As already mentioned, we can fix an adapted Riemannian metric on $U$ such that the following stronger version of the hyperbolic estimates are satisfied for some $\lambda_0 >0$ : for every $\rho \in \mathcal{T}$, $n \in \N$, 
\begin{align}
v \in E_u( \rho) \implies ||d_\rho F^{-n}(v)|| \leq e^{-\lambda_0 n } ||v||\\
v \in E_s( \rho) \implies ||d_\rho F^{n}(v)|| \leq e^{-\lambda_0 n } ||v||
\end{align}

\begin{nota}
We now use the induced Riemannian distance on $U$ and denote it $d$. \\
We also use the same notation $|| \cdot||$ to denote the subordinate norm on the space of linear maps between tangent spaces of $U$, namely, if $F(\rho_1) = \rho_2,$ $$||d_{\rho_1}F || = \sup_{ v \in T_{\rho_1} U, ||v||_{\rho_1}=1} ||d_{\rho_1}F(v)||_{\rho_2}$$
\end{nota}

If $\rho \in \mathcal{T}$,  $n \in \Z$, we use this Riemannian metric to define the unstable Jacobian $J^u_n (\rho)$ and stable Jacobian $J^s_n (\rho)$ at $\rho$ by : 

\begin{align}\label{Def_jacobian}
v \in E_u( \rho) \implies ||d_\rho F^{n}(v)|| = J^u_n (\rho) ||v|| \\
v \in E_s( \rho) \implies ||d_\rho F^{n}(v)|| = J^s_n (\rho)||v||
\end{align}
These Jacobians quantify the local hyperbolicity of the map. 
\begin{nota}
Suppose that $f$ and $g$ are some real-valued functions depending on the same family of parameters $\mathcal{P}$. For instance, for $J^u_n(\rho)$, $\mathcal{P}= \{ n, \rho\}$.
We will note $ f \sim g$ to mean that there exist constant a $C\geq 1$ depending only on $(U,F)$, but not on $\mathcal{P}$, such that $ C^{-1}g \leq f \leq C g$. \\
For instance, if we define unstable and stable Jacobian $\tilde{J}^u_n$ and $\tilde{J}^s_n$ using another Riemannian metric, then, for every $n \in \Z$ and $\rho \in \mathcal{T}$, 
\begin{equation*}
\tilde{J}^u_n(\rho) \sim J^u_n(\rho)\quad ; \quad  \tilde{J}^s_n(\rho) \sim J^s_n(\rho)
\end{equation*}
\end{nota}
From the compactness of $\mathcal{T}$, there exist $\lambda_1 \geq \lambda_0$ which satisfies 

\begin{align}
&e^{n\lambda_0} \leq J^{u}_n (\rho) \leq  e^{n\lambda_1 } &\text{ and }  & e^{-n\lambda_1} \leq J^{s}_n (\rho) \leq  e^{-n\lambda_0} &\quad ; \quad  n \in \N , \rho \in \mathcal{T} \label{lamba0et1}\\
 &e^{n\lambda_0} \leq J^{s}_{-n} (\rho) \leq  e^{n\lambda_1 } &\text{ and } & e^{-n\lambda_1} \leq J^{u}_{-n} (\rho) \leq  e^{-n\lambda_0} &\quad ; \quad  n \in \N , \rho \in \mathcal{T} \label{lambdao0et1bis}
\end{align}
We cite here standard facts about the stable and unstable manifolds (see for instance \cite{KH}, Chapter 6). 
\begin{lem} \label{classical_hyperbolic}
For any $\rho \in \mathcal{T}$, there exist local stable and unstable manifolds $W_s(\rho), W_u( \rho) \subset U$ satisfying, for some $\varepsilon_1 >0$ (only depending on $F$) ($\star$ will denote a letter in $\{u ,s \}$ and the use of $\pm$ with $\star$ has to be read with the convention $ u \to - , s \to +)$ )
\begin{enumerate}[label = (\arabic*)]
\item $W_s(\rho), W_u( \rho)$ are $C^\infty$-embedded curves, with the $C^\infty$ norms of the embedding uniformly bounded in $\rho$. 
\item the boundary of $W_\star(\rho)$ do not intersect $\overline{B(\rho, \varepsilon_1)} $ \footnote{in other words, there exists a smooth curve $\gamma : [-\delta,\delta]  \to U$ such that $\overline{B(\rho, \varepsilon_1)} \cap W_\star(\rho) = \im \gamma$, with $\gamma(0) = \rho$ : it means that the size of the (un)stable manifolds is bounded from below uniformly. }
\item $W_s(\rho) \cap W_u(\rho) = \{ \rho \}$, $T_\rho W_\star (\rho) = E_\star (\rho) $ 

\item $F^{\pm}(W_\star(\rho) ) \subset W_\star \left(F(\rho) \right)$.
\item For each $\rho^\prime \in W_\star(\rho), d( F^{\pm n } (\rho), F^{ \pm n}(\rho^\prime) ) \to 0$. 
\item Let $\theta>0$ satisfying $ e^{-\lambda_0} < \theta < 1$. If $\rho^\prime \in U$ satisfies $d(F^{\pm i }(\rho) , F^{\pm i }( \rho^\prime) ) \leq \varepsilon_1$ for all $i = 0 , \dots, n$  then $d\left( \rho^\prime, W_\star(\rho) \right)\leq C \theta^n \varepsilon_1$ for some $C>0$. 
\item If $\rho, \rho^\prime \in \mathcal{T}$ satisfy $d(\rho, \rho^\prime) \leq \varepsilon_1$, then  $W_u (\rho) \cap W_s (\rho^\prime)$ consists of exactly one point in $\mathcal{T}$. 
\end{enumerate}
\end{lem}
\vspace{0.5cm}

Since we work with the local unstable and stable manifolds, we may assume that $W_\star(\rho) \subset B(\rho, 2 \varepsilon_1)$. 
 
For our purpose, we will need a more precise version of these results. The following lemmas are an adaptation of Lemma 2.1 in \cite{NDJ19} to our setting. 

\begin{lem}\label{Local_hpyerbolic_1}
There exists a constant $C >0$ depending only on $(U,F)$, such that for all $\rho, \rho^\prime \in U$, 
\begin{enumerate}[label = (\arabic*)]
\item if $\rho \in \mathcal{T}$ and $\rho^\prime \in W_s( \rho)$ then 
\begin{equation} 
d\left( F^{n}(\rho) , F^{n} (\rho^\prime) \right) \leq CJ_n^s(\rho) d( \rho, \rho^\prime) \quad , \quad \forall n \in \N 
\end{equation}
\item if $\rho \in \mathcal{T}$ and $\rho^\prime \in W_u( \rho)$ then 
\begin{equation}
d\left( F^{-n}(\rho) , F^{-n} (\rho^\prime) \right) \leq CJ_{-n}^u(\rho) d( \rho, \rho^\prime) \quad , \quad  \forall n \in \N 
\end{equation} 
\end{enumerate}
\end{lem}

\begin{proof}
We prove (1). (2) is proved in a similar way by inverting the time direction.  
Let $\rho \in \mathcal{T}, \rho^\prime \in W_s(\rho)$. Since $T_{\rho}(W_s(\rho) )  =E_s(\rho) $ and $d_\rho F \left(E_s(\rho)\right) = E_s(F(\rho))$, the Taylor development of $F$ along $W_s(\rho)$ gives : 
\begin{equation}
d(F(\rho), F(\rho^\prime)) \leq J^s_1(\rho) d(\rho, \rho^\prime) + C d(\rho, \rho^\prime)^2 \leq J^s_1(\rho) d(\rho, \rho^\prime)  ( 1 + Cd(\rho, \rho^\prime) )
\end{equation}
since $J^s_1 \geq C^{-1}$. Applying this inequality with $F^k (\rho)$ and $F^k (\rho^\prime)$ instead of $\rho$ and $\rho^\prime$, and recalling that, by lemma \ref{classical_hyperbolic}, $d(F^k(\rho), F^k (\rho^\prime)) \leq C\theta^k d(\rho, \rho^\prime)$, 
we can write, 
\begin{equation}
 d(F^{k+1}(\rho), F^{k+1} (\rho^\prime)) \leq J^s_1(F^k(\rho)) d(F^k(\rho), F^k (\rho^\prime)) (1 + C \theta^k)
\end{equation}
By this last inequality and the chain rule, we have 
\begin{equation}
d(F^n(\rho), F^n (\rho^\prime)) \leq J^s_n(\rho) d(\rho, \rho^\prime) \prod_{k=0}^{n-1} (1 + C \theta^k) \leq C J^s_n(\rho) d(\rho, \rho^\prime)
\end{equation}
\end{proof}

\begin{figure}\label{Figure_hyperbolic}
\begin{center}
\includegraphics[scale=0.6]{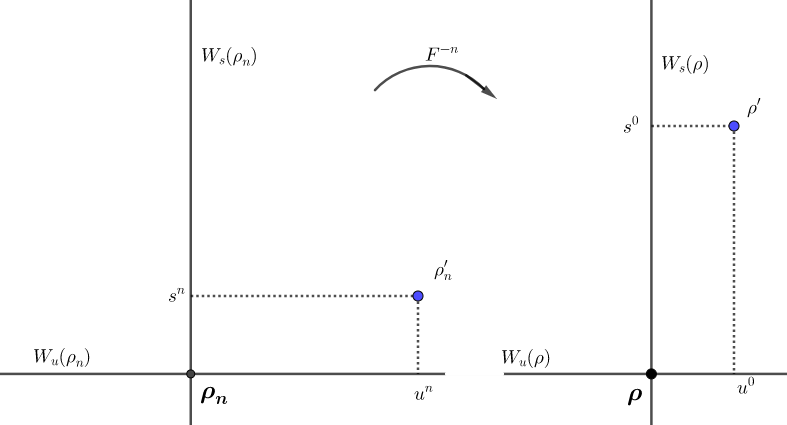}
\caption{Framework for the proof of Lemma \ref{Local_hpyerbolic_2} }
\end{center}
\end{figure}

The following lemma gives a stronger version of (6) in Lemma \ref{classical_hyperbolic}. 

\begin{lem}\label{Local_hpyerbolic_2}
There exist $C >0$ and $\varepsilon_1 >0$, depending only on $(U,F)$, such that for all $\rho, \rho^\prime \in U$ and $n \in \N$ : 
\begin{enumerate}[label = (\arabic*)]
\item if $\rho \in \mathcal{T}$ and $d\left( F^i(\rho) , F^i(\rho^\prime) \right) \leq \varepsilon_1$ for all $i \in \{0, \dots, n \}$ then 
\begin{equation} \label{close_to_leaves}
d\left( \rho^\prime , W_s (\rho) \right) \leq \frac{C}{J_n^u(\rho)} 
\end{equation}
and 
\begin{equation}\label{control_of_jacobian}
||d_{\rho^\prime}F^n|| \leq C J^u_n(\rho)
\end{equation}
\item if $\rho \in \mathcal{T}$ and $d\left( F^{-i}(\rho) , F^{-i}(\rho^\prime) \right) \leq \varepsilon_1$ for all $i \in \{0, \dots, n \}$ then 
\begin{equation} 
d\left( \rho^\prime , W_u (\rho) \right) \leq \frac{C}{J_{-n}^s(\rho)} 
\end{equation}
and 
 \begin{equation}
||d_{\rho^\prime}F^{-n}|| \leq C J^s_{-n}(\rho)
\end{equation}
\end{enumerate}
\end{lem}

\begin{proof}
We prove (1). (2) is proved in a similar way by inverting the time direction. Let $\rho \in \mathcal{T}$ and $\rho^\prime \in U$ such that $d(F^i(\rho), F^i(\rho^\prime)) \leq \varepsilon_1$ for $ 0 \leq i \leq n$ with $\varepsilon_1$ to be determined. 
Denote $\rho_k = F^k(\rho)$. The first condition on $\varepsilon_1$ is that it is smaller than the one of lemma \ref{classical_hyperbolic} so that we ensure the folowing estimates : for $ k \in \{0, \dots, n \}$
\begin{align}
d \Big( F^k(\rho^\prime), W_s(F^k(\rho)) \Big) \leq C \theta^{n-k} \varepsilon_1 \label{lh1}  \\
d\Big( F^k(\rho^\prime), W_s(F^k(\rho)) \Big) \leq C \theta^k \varepsilon_1 \label{lh2}
\end{align}
We will use coordinates charts $\kappa_k : \hat{\rho} \in U_k \mapsto (u^k,s^k) \in V_k$ adapted to the dynamical system (see \cite{KH}, Theorem 6.2.3, the explanations below and Theorem 6.2.8 for the existence of this chart). More precisely, we want these charts to satisfy 
\begin{itemize}[itemsep=0.2em]
\item $\kappa_k(\rho_k) = (0,0)$
\item $\kappa_k \left( W_s(\rho_k) \cap U_k \right) = \{ (0,s) , s \in \R \} \cap V_k$
\item $\kappa_k \left( W_u(\rho_k)  \cap U_k \right) = \{ (u,0) , u \in \R \} \cap V_k$
\item For $\hat{\rho} \in U_k$, $ |u^k| \sim d(\hat{\rho}, W_s (\rho_k) ) ; |s^k| \sim d(\hat{\rho}, W_u( \rho_k) ) ; |s^k|^2 + |u^k|^2 \sim d(\rho_k, \hat{\rho})^2$.
\item  $(\kappa_k)_{0 \leq k \leq n}$ are uniformly bounded in the $C^N$ topology for all $N$, with constant independant of $\rho_0$ and $n$. In particular, we may assume that $\varepsilon_1$ is chosen small enough so that $B(\rho_k, \varepsilon_1) \subset U_k$ for all $0 \leq k \leq n$. 
\item Up to changing the metric we work with (which is not problematic), we may assume that the restrictions of $d\kappa_k(\rho)$ to $E_s(\rho)$ and $E_u(\rho)$ are isometries for the metrics $|\cdot|_s$ and $|\cdot|_u$. 
\end{itemize}
If we note $\widetilde{F}_k = \kappa_{k} \circ F \circ \kappa_{k-1}^{-1}$, we can check that in this pair of coordinates charts, the action of $F^{-1}$ is given by 
\begin{equation}\label{lh3}
\widetilde{F}_k^{-1} (u^k, s^k) = (\pm J^u_{-1}(\rho_k) u^k + \alpha_k(u^k, s^k) , \pm J_{-1}^s(\rho_k)s^k + \beta_k(u^k,s^k) ) 
\end{equation}
where $\alpha_k, \beta_k$ are smooth functions, uniformly bounded in $k$ for the $C^2$ topology and such that $\alpha_k(0,s^k)=0, \beta_k(u^k,0) =0, d\alpha_k(0,0) = 0, d\beta_k(0,0) = 0$. \\
With these properties, one can check that 
\begin{equation}\label{lh4}
\alpha_k(u^k,s^k) \leq C|u^k| \;||(u^k,s^k)||
\end{equation}
Let's now denote $\rho^\prime_k = F^k(\rho^\prime)$ and $(u^k,s^k)=\kappa_k(\rho^\prime_k)$. By (\ref{lh1}), (\ref{lh2}), (\ref{lh3}), (\ref{lh4}), we can write 
\begin{align*}
|u^{k-1}| &\leq J^u_{-1}  (\rho_k) ) |u^k| +  C|u^k| ||(u^k,s^k)|| \\
& \leq J^u_{-1}  (F^k (\rho) ) |u^k| \left(1+ C \varepsilon_1 (\theta_1^k + \theta_1^{n-k} ) \right) \\
& \leq J^u_{-1}  (F^k (\rho) ) |u^k| \left( 1+ C \varepsilon_1 \theta^{\min (k, n-k)} \right) \\
\end{align*}
Then, using the chain rule, one has 
\begin{equation}
d(\rho^\prime, W_s(\rho)  ) \leq C |u^0| \leq CJ^u_{-n}(F^n(\rho)) \prod_{k=0}^{n-1}\left( 1+ C \varepsilon_1 \theta^{\min (k, n-k)} \right)
\end{equation}
Finally, we can estimate 
$$ \prod_{k=0}^{n}\left( 1+ C \varepsilon_1 \theta^{\min (k, n-k)} \right) \leq \prod_{k=0}^{\lceil n/2 \rceil}\left( 1+ C \varepsilon_1 \theta^{k} \right)^2 \leq C $$
which gives 
\begin{equation}
d(\rho^\prime, W_s(\rho)) \leq C J^u_{-n} (F^n(\rho)) = \frac{C}{J^u_{n} (\rho)}
\end{equation}
This proves (\ref{close_to_leaves}). 

To prove (\ref{control_of_jacobian}), we first construct a metric which simplifies the computations. If $\rho \in \mathcal{T}$, we pick $v_\star(\rho) \in E_\star(\rho)$\footnote{Here, we are not concerned by the orientation. It is simply a matter of direction. } such that $||v_\star(\rho)|| = 1$.  There exists a Riemannian metric $|\cdot|$ on $\mathcal{T}$ such that for every $\rho \in \mathcal{T}, (v_u(\rho), v_s(\rho))$ is an orthonormal basis of $T_\rho U$. This metric is $\gamma$-Hölder in $\rho \in \mathcal{T}$ since stable and unstable distributions are $\gamma$-Hölder for some $\gamma \in (0,1)$. \\
If $\rho \in \mathcal{T}$ and $n \in \Z$, we note $\tilde{J}^{u/s}_n(\rho) \in \R$ the numbers such that 
$$d_\rho(F^n)(v_u(\rho)) = \tilde{J}^u_n(\rho)  v_u(F^n(\rho)) \; ; \; d_\rho(F^{n})(v_s(\rho)) = \tilde{J}^s_n(\rho) v_s(F^n(\rho))$$ 
As already observed, $|\tilde{J}^u_n(\rho)| \sim J^u_n(\rho)$, for all $n$ (with constants independent of $n$). We can also assume that $|\tilde{J}^u_1(\rho)| >|\tilde{J}^s_1(\rho)|$ for all $\rho$.  In the orthonormal basis $(v_u(\rho), v_s(\rho))$ and $(v_u(F^n(\rho), v_s(F^n(\rho)))$, $d_\rho F^n$ has the form $$\left(\begin{array}{c c } 
\tilde{J}^u_n(\rho) & 0 \\
0  &\tilde{J}^s_n(\rho)
\end{array} \right)$$ 
Due to the ortonormality of these basis, we have that for the subordinate norms, $||d_\rho F^n || = |\tilde{J}^u_n(\rho)|$.  Hence, the chain rule implies the following equality for this particular Riemannian metric defined on $\mathcal{T}$ : 
\begin{equation}\label{nice_trick}
\forall \rho \in \mathcal{T}, ||d_\rho(F^n)|| = |\tilde{J}^u_n(\rho)| = \prod_{i=0}^{n-1} |\tilde{J}^u_1(F^i\left(\rho) \right)| = \prod_{i=0}^{n-1} ||d_{F^i(\rho)} F  ||
\end{equation}
We now claim that we can extend $| \cdot |$ to a relatively compact neighborhood $V$ of $\mathcal{T}$ such that $ \rho \in V \mapsto | \cdot |_\rho$ is still $\gamma$-Hölder. To do so, it is enough to extend the coefficients of the metric in a coordinate chart in a $\gamma-$Hölder way, which is possible (for instance, in virtue of Corollary 1 in \cite{McShane}), which still defines a non-degenerate 2-form in a sufficiently small neighborhood of $\mathcal{T}$. \\
We now aim at proving (\ref{control_of_jacobian}) for this particular metric. (\ref{control_of_jacobian}) will hold in the general case since two continuous metric are always uniformly equivalent in a compact neighborhood of $\mathcal{T}$. \\
In the following, we assume that $\varepsilon_1$ is small enough so that $\rho$ belongs to the neighborhood of $\mathcal{T}$ in which $| \cdot|$ is defined.  
Since $\rho \mapsto ||d_\rho F||_{T_\rho U \to T_{F(\rho)}U }$ is $\gamma$-Hölder (in the following, we will drop the subscript  in the norm) we have, for all $i \in \{0, \dots, n \}$ 
\begin{equation}
\Big| \, ||d_{F^i(\rho^\prime)}F|| - ||d_{ F^i(\rho)}F ||  \, \Big|  \leq C d(F^i (\rho^\prime) , F^i(\rho) )^\gamma \leq C \varepsilon_1 \theta^{\gamma \min(i,n-i)} 
\end{equation}
Using the chain rule and the submultiplicativity of $|| \cdot ||$, we have 
\begin{equation}
|| d_{\rho^\prime}F^n || \leq \prod_{i=0}^n || d_{F^i(\rho^\prime)} F || \leq \prod_{i=0}^n || d_{F^i(\rho)} F || \left( 1 + C \varepsilon_1 \theta^{\gamma \min(i,n-i)}  \right)
\end{equation}
Eventually, by (\ref{nice_trick}) and the fact that $\prod_{i=0}^n  \left( 1 + C \varepsilon_1 \theta^{\gamma \min(i,n-i)}  \right)$ is convergent, (\ref{control_of_jacobian}) holds. 
\end{proof}

As an immediate consequence of this lemma, we get : 
\begin{cor}\label{cor_control_jacobian}
There exist $C >0$ and $\varepsilon_1 >0$ (depending only on $(U,F)$) such that for all $\rho, \rho^\prime \in \mathcal{T}$ and $n \in \N$ : 
\begin{enumerate}[label = (\arabic*)]
\item if $d\left( F^i(\rho) , F^i(\rho^\prime) \right) \leq \varepsilon_1$ for all $i \in \{0, \dots, n \}$ then 
\begin{equation}
C^{-1} J^u_n(\rho) \leq  J^u_n(\rho^\prime) \leq CJ^u_n(\rho)
\end{equation}
\item if $d\left( F^{-i}(\rho) , F^{-i}(\rho^\prime) \right) \leq \varepsilon_1$ for all $i \in \{0, \dots, n \}$ then 
\begin{equation} 
  C^{-1} J^s_{-n}(\rho) \leq J^s_{-n}(\rho^\prime) \leq CJ^s_{-n}(\rho)
\end{equation}
\end{enumerate}
\end{cor}
\begin{proof}
This is a consequence of the previous lemma and of the fact that uniformly in $\rho$ and $n \in \N$,  \begin{align*}
||d_\rho F^n|| \sim J^u_n(\rho) \\ ||d_\rho F^{-n}|| \sim J^s_{-n}(\rho)
\end{align*} 
\end{proof}

\subsection{Regularity of the invariant splitting}
It is known for Anosov diffeomorphisms that stable and unstable distributions are in fact $C^{2 - \varepsilon}$ in dimension 2 (see \cite{HuKa}).  For our purpose, we need to extend this result to our setting, where the hyperbolic invariant set $\mathcal{T}$ is not the full phase space, but a fractal subset of it. In fact, we will show that one can extend the stable and unstable distributions to an open neighborhood of $\mathcal{T}$ and that theses extensions are $C^{1, \beta}$ for some $\beta >0$. Actually, since what happens outside a fixed neighborhood of $\mathcal{T}$ is irrelevant (one can always use cut-offs), we will prove the following theorem which might be of independent interest.

\begin{thm}\label{Thm_regularity}
Let us denote $\mathcal{G}_1(U)$ the Grassmanian bundle of $1$-plane in $TU$. 
There exists $\beta >0$ and sections $E_u, E_s :  U \to \mathcal{G}_1(U)$ such that : 
\begin{itemize}
\item For every $\rho \in \mathcal{T}, E_u(\rho)$ (resp. $E_s(\rho)$) is the unstable (resp. stable) distribution at $\rho$ ; 
\item $E_u$ and $E_s$ have regularity $C^{1, \beta}$ 
\end{itemize}
\end{thm}

 \begin{rem}
It is likely that one can improve this regularity using the method of \cite{HuKa}. Our proof relies on the techniques of \cite{HiPu}. In fact, in \cite{KH} 19.1.d, the authors show how one can obtain $C^1$ regularity of the map $\rho \in \mathcal{T} \mapsto E_u(\rho)$ and explains how to prove $C^{1, \beta}$ regularity. Their notion of differentiability on the set $\mathcal{T}$ (which is clearly not open in our case) relies on the existence of linear approximations. Here, we choose to show a slightly different version of this regularity by proving that $\rho \in \mathcal{T} \mapsto E_u(\rho)$ can be obtained as the restriction of a $C^{1, \beta}$ map defined in an open neighborhood of $\mathcal{T}$. 
\end{rem}

 \subsubsection{Proof of the $C^{1,\beta}$ regularity}

\paragraph{Preliminaries.}

We recall that $\mathcal{T}$ is an invariant hyperbolic set for $F$. Hence, there exists a continuous splitting of $T_{\mathcal{T}} U$, into stable and unstable spaces $ \rho \in \mathcal{T} \mapsto E_s(\rho) , \rho \in \mathcal{T} \mapsto E_u(\rho)$.  We use a continuous Riemannian metric on $T_{\mathcal{T}} U$ such that $d_\rho F$ is a contraction from $E_s(\rho) \to E_s (F(\rho))$ and expanding from $E_u(\rho) \to E_u(F(\rho))$, and making $E_u(\rho)$ and $E_s(\rho)$ orthogonal. 

Let $ \rho \in \mathcal{T} \mapsto e_u(\rho) \in TU$ and $ \rho \in \mathcal{T} \mapsto e_s(\rho) \in TU$  be two continuous sections \footnote{Note that there is no problem of orientation to construct such global sections. Indeed, $\mathcal{T}$ is totally disconnected and hence, one can cover $\mathcal{T}$ by a disjoint union of open sets small enough so that it is possible to construct local sections in each such sets. Since these open sets are disjoint, these local sections allow us to build a global continuous section.  } such that, for every $\rho \in \mathcal{T}$, 
\begin{itemize}[nosep]
\item $e_u(\rho)$ spans $E_u(\rho)$, 
\item $e_s(\rho)$ spans $E_s(\rho)$, 
\item $||e_u(\rho)||=1, ||e_s(\rho)||=1$
\end{itemize} 

The matrix representation of $d_\rho F$\footnote{The definition of $\tilde{J}^{u/s}$ may differ from the one of $J^{u/s}_1$ above since we don't work \emph{a priori} with the same metric.} in theses basis is 
$$ d_\rho F = \left( \begin{matrix}
\tilde{J}^u(\rho) & 0 \\
0 & \tilde{J}^s(\rho) 
\end{matrix} \right)$$
with $ \nu \coloneqq \sup_{\rho \in \mathcal{T}} \max \left[ \left(\left|\tilde{J}^u(\rho)\right|\right)^{-1}, \left|\tilde{J}^s(\rho)\right| \right] < 1$. 

We can extend $e_u$ and $e_s$ to $U$ to continuous functions, still denoted $e_u$ and $e_s$. Let us consider smooth vector fields $v_u$ and $v_s$ on $U$ approximating $e_u$ and $e_s$ and a smooth Riemannian metric approximating the one considered above. By slightly modifying this vector fields, we can assume that for this new metric, $(v_u(\rho), v_s(\rho))$ is an orthonormal basis for all $\rho \in U$. In these new basis, we now write 

$$ d_\rho F = \left( \begin{matrix}
a(\rho) & b(\rho) \\
c(\rho) & d(\rho) 
\end{matrix} \right)$$
We assume that $v_u$ and $v_s$ are sufficiently close to $e_u$ and $e_s$ to ensure that, for some $\eta>0$ small enough, 
\begin{align*}
\sup_{ \rho \in \mathcal{T} } \max \left(  |b(\rho)|, |c(\rho) |\right) \leq \eta \\
\sup_{ \rho \in \mathcal{T} }  |d(\rho)| \leq \nu + \eta \leq 1 - 4\eta \\
\inf_{ \rho \in \mathcal{T} } |a(\rho) | \geq \nu^{-1} -\eta \geq 1 + 4\eta 
\end{align*}
We consider an open neighborhood $\Omega$ of $\mathcal{T}$ such that the following holds : 
\begin{align*}
\sup_{ \rho \in \Omega} \max \left(  |b(\rho)|, |c(\rho) |\right) \leq 2\eta \\
\sup_{ \rho \in \Omega }  |d(\rho)| \leq \nu + 2 \eta \leq 1 - 3\eta \\
\inf_{ \rho \in \Omega } |a(\rho) |  \geq \nu^{-1} - 2 \eta \geq 1 + 3\eta \\
\end{align*}

Our method relies on different uses of the Contraction Map Theorem. We state the  Fiber Contraction Theorem of \cite{HiPu} (Section 1), which will be used further. We recall that a fixed point $x_0$ of a continuous map $f : X \to X$ is  said to be \emph{attractive} if for every $x \in X, f^n(x) \to x_0$. 

\begin{thm}\textbf{Fiber Contraction Theorem} \label{Fiber_Contraction_Theorem}\\
Let $(X,d)$ be a metric space and $h : X \to X$ a map having an attractive fixed point $x_0$. Let us consider $Y$ another metric space and a family of maps $(g_x : Y \to Y)_{x \in X}$ and denote by $H$ the map 
$$ H : (x,y) \in X \times Y \mapsto (h(x) , g_x (y) ) \in X \times Y$$
Assume that 
\begin{itemize}[nosep]
\item $H$ is continuous ; 
\item For all $x \in X$, $\limsup_{n \to + \infty} L\left( g_{h^n(x)} \right)  <1$ where $L\left( g_{h^n(x)} \right) $ denotes the best Lipschitz constant for $g_{h^n(x)}$  ; 
\item $y_0$ is an attractive fixed point for $g_{x_0}$.
\end{itemize}
Then $(x_0,y_0)$ is an attractive fixed point for $H$. 
\end{thm}

In the following, we study the regularity of the unstable distribution. The same holds for the stable distribution by changing the roles of $F^{-1}$ and $F$.

\paragraph{$E_u$ is a fixed point of a contraction.}
By our assumption on $v_u$ and $v_s$, there exists a continuous function $\lambda : U \to \R$ such that 
$$\R e_u(\rho) = \R ( v_u(\rho) + \lambda(\rho) v_s(\rho) ) $$
Hence, we will represent the extension of the unstable distribution by a continuous map $\lambda : \Omega \to \R$.  Our aim is to show that we can find $\lambda$ regular enough such that for $\rho \in \mathcal{T}$, 
$$ E_u(\rho) = \R( v_u(\rho) + \lambda(\rho) v_s(\rho) ) $$
To do so, we will start by constructing $\lambda$ as a fixed point of a contraction in a nice space. This contraction will be related to invariance properties of the unstable distribution. \\
First of all, if $\rho^\prime = F(\rho) \in \Omega \cap F(\Omega)$, and if $v = v_u(\rho) + \lambda v_s(\rho)$, $d_\rho F$ maps $v$ to $$w = \big( a(\rho) + \lambda b(\rho) \big)v_u( \rho^\prime ) + \big(c(\rho) + \lambda d(\rho) \big) v_s(\rho^\prime)  $$
Hence, the line of $T_\rho U$ represented by $\lambda$ is sent to the line represented by $t(\rho,\lambda)$ in $T_{\rho^\prime} U$ where 
\begin{equation}
t(\rho,\lambda) = \frac{\lambda d(\rho) + c(\rho) }{ a(\rho) + \lambda b(\rho) }
\end{equation}
Set $\Omega_1 = \Omega \cap F( \Omega)$ and let us consider a cut-off function $\chi \in \cinfc(\Omega_1)$ such that $0 \leq \chi \leq 1$ and $\chi \equiv 1$ in a neighborhood of $\mathcal{T}$. 
Let us introduce the complete metric space $$ X = \{ \lambda \in C( \Omega, \R), ||\lambda||_\infty \leq 1 \}$$
and consider the map $T : X \to X$ defined, for $\lambda \in X$ and $\rho^\prime \in \Omega$, 
\begin{equation}\label{def_of_T}
(T \lambda)(\rho^\prime)  = \chi(\rho^\prime) t\left( F^{-1}(\rho^\prime) , \lambda \big( F^{-1}(\rho^\prime) \big)  \right)
\end{equation}
To see that this is well defined, first note that $F^{-1}$ is well defined on $\supp \chi$ and $F^{-1} ( \supp \chi) \subset \Omega$. It is clear that if $\lambda \in X$, $T \lambda$ is continuous. To see that $||T \lambda ||_\infty \leq 1$, it is enough to note that if $\rho \in \Omega$ and $|\lambda| \leq 1$, 
$$ \left| t(\rho, \lambda ) \right| \leq  \frac{|d(\rho)| + |c(\rho)|}{|a(\rho)| - |b(\rho) | }  \leq \frac{1- 3 \eta + 2 \eta}{1 + 3 \eta - 2 \eta} \leq \frac{1- \eta}{1 + \eta}< 1$$
Let us now prove the following 
\begin{prop} \text{}
\begin{itemize} 
\item $T$ is a contraction. 
\item If $\lambda_u$ denotes its unique fixed point, then, for every $\rho \in \mathcal{T}$, 
$E_u(\rho) = \R \Big(  v_u(\rho) + \lambda_u(\rho) v_s(\rho) \Big)$
\end{itemize}
\end{prop}

\begin{proof}
Let $ \lambda, \mu \in X$. If $\rho^\prime \in \Omega \setminus \supp \chi$, we have  $T\mu(\rho^\prime) = T\lambda(\rho^\prime) = 0$. Now assume that $\rho^\prime \in \supp \chi$ and write $\rho^\prime = F(\rho)$ with $\rho \in \Omega$. 
$$ |T\lambda(\rho^\prime) - T\mu(\rho^\prime)| = |\chi(\rho^\prime)| | t(\rho, \lambda(\rho) ) - t(\rho, \mu(\rho) )| \leq | t(\rho, \lambda(\rho) ) - t(\rho, \mu(\rho) )| $$ 
The map $ \lambda \in [-1,1] \mapsto t(\rho,\lambda)$ is smooth, so that we can write 
$$||T\lambda- T \mu||_\infty \leq \sup_{ \rho^\prime \in \supp \chi} |T\lambda(\rho^\prime) - T\mu(\rho^\prime) | \leq \sup_{ \Omega \times [-1, 1]}|\partial_\lambda t| \times ||\lambda- \mu||_\infty$$
It is then enough to show that $\sup_{ \Omega \times [-1, 1]}|\partial_\lambda t|  <1$. For $(\rho, \lambda) \in \Omega \times [-1,1]$, we have
\begin{equation}\label{partial_lambda_T}
 \partial_{\lambda} t(\rho,\lambda) = \frac{d(\rho)}{a(\rho)  + \lambda b(\rho)} - b(\rho)   \frac{\lambda d(\rho) + c(\rho) }{\left( a(\rho)) + \lambda b(\rho) \right)^2}
\end{equation}
Hence, we can control 
\begin{align*}
|\partial_{\lambda} t(\rho,\lambda)| \leq \frac{1-3 \eta}{1 + \eta} + \eta \frac{1- \eta}{\left( 1+ \eta \right)^2} = \kappa_\eta < 1
\end{align*}
if $\eta$ is small enough. This demonstrates that $T$ is a contraction. 

As a consequence, $T$ has a unique fixed point, $\lambda_u$. We note $v(\rho) = v_u(\rho) + \lambda_u(\rho) v_s(\rho)$. We want to show that $v(\rho) \in \R e_u(\rho)$ for $\rho \in \mathcal{T}$ (recall that $e_u : U \to TU$ is continuous and that $e_u(\rho)$ spans $E_u(\rho)$ if $\rho \in \mathcal{T}$). Since $\chi = 1$ on $\mathcal{T}$, we see by definition of $T$ that for every $\rho \in \mathcal{T}$, 
\begin{equation}\label{fixed_point_prop}
d_\rho F (v(\rho)) \in \R v (F(\rho) )
\end{equation}
If $v_u$ is sufficiently close to $e_u$, we can find a continuous and bounded function $\mu$ such that 
$$ \R v(x) = \R \left( e_u(x) + \mu(x) e_s(x) \right)$$
From (\ref{fixed_point_prop}), if $\rho^\prime = F(\rho) \in \mathcal{T}$, 

$$d_\rho F \Big( e_u(\rho) + \mu(\rho) e_s(\rho) \Big) = \tilde{J}^u_1 (\rho) \left( e_u(\rho^\prime) + \mu(\rho) \frac{\tilde{J}^s_1(\rho)}{\tilde{J}^u_1(\rho)} e_s(\rho^\prime) \right) \in \R \Big( e_u(\rho^\prime) + \mu(\rho^\prime) e_s(\rho^\prime) \Big)$$
This implies the equality 
\begin{equation}
\mu(\rho^\prime) = \mu(\rho) \frac{\tilde{J}^s_1(\rho)}{\tilde{J}^u_1(\rho)}
\end{equation}
This equality implies that $\mu =0$ on $\mathcal{T}$ and hence, $v = e_u$ on $\mathcal{T}$, as expected. 
\end{proof}

\begin{rem}
As long as $\rho^\prime \in \{ \chi =1 \}$, the vector field $v(\rho^\prime) = v_u(\rho^\prime) + \lambda(\rho^\prime) v_s(\rho^\prime)$ is invariant by $dF$. When $\rho^\prime \in W_u(\rho) \cap \{ \chi =1 \}$ for some $\rho \in \mathcal{T}$, we will see below that the direction given by $v(\rho^\prime)$ coincides with the tangent space to $W_u(\rho)$, namely $T_{\rho^\prime}W_u(\rho) = \R v(\rho^\prime)$. When $\rho^\prime \not \in \bigcup_{\rho \in \mathcal{T}}W_u(\rho)$, there exists $n \in \N$ such that $F^{-n}(\rho^\prime) \not \in \supp \chi$. Hence, $\lambda_u(\rho^\prime)$ is given by an explicit expression obtained by iterating the fixed point formula. 
\end{rem}

\paragraph{Differentiability of $\lambda_u$.} We go on by showing that $\lambda$ is $C^1$ by adapting the method of \cite{HiPu}. 
We now introduce the Banach space $Y$ of bounded continuous sections $\alpha : \Omega \to T^* \Omega$. We will use the norm on $T^* \Omega$ adapted to the metric on $T\Omega$, namely 
if $\alpha \in Y$, $$||\alpha||_Y = \sup_{ \rho \in \Omega} \sup_{ v \in T_\rho \Omega , v \neq 0} \frac{| \alpha(\rho) (v) |}{||v||_{T_\rho \Omega}}$$
For $\lambda \in X$, let us introduce the map $G_{\lambda} : Y \to Y$, defined as follows. For $\alpha \in Y$ and $\rho^\prime \in \Omega$, 

\begin{align}\label{def_of_G_lambda}
\left( G_\lambda \alpha \right)(\rho^\prime) = &\chi(\rho^\prime) \Big[ d_\rho t \left( \rho  , \lambda (\rho) \right) +   \partial_\lambda t \left(\rho  , \lambda (\rho) \right) \alpha \left(\rho \right) \Big] \circ \left( d_\rho F \right)^{-1}+ t \left( \rho  , \lambda (\rho)\right) d_{\rho^\prime} \chi
\end{align}
with $\rho=F^{-1}(\rho^\prime)$, which is well defined  since $ \rho \in \Omega$ if $\rho^\prime \in \supp(\chi)$.
$G_\lambda$ is constructed to satisfy : for $\lambda \in X$, if $\lambda$ is $C^1$, then the following relation holds : 
\begin{equation}\label{rec_for_G}
G_\lambda (d \lambda)   = d \left( T \lambda \right)
\end{equation}
Let us first state the key tool to show the differentiability of $\lambda_u$.  
\begin{prop}\label{G_lambda_contraction}
For every $\lambda \in X$, $G_\lambda$ is a contraction with Lipschitz constant $L_\lambda$ satisfying 
$$ \sup_{ \lambda \in X} L_\lambda <1$$
\end{prop}
Before proving it, let us show how it leads us to 
\begin{prop}
$\lambda_u$ is $C^1$. 
\end{prop}

\begin{proof}
We use the Contraction Fiber Theorem. Let $\alpha_u$ be the unique fixed point of $G_{\lambda_u}$. The map $$H : (\lambda, \alpha) \in X \times Y \mapsto (T \lambda, G_\lambda \alpha )\in X \times Y $$ is continuous and the previous proposition shows that for every $\lambda \in X$, $ \sup_n L \left( G_{T^n \lambda} \right) <1$. The Contraction Fiber Theorem implies that $(\lambda_u, \alpha_u)$ is an attractive fixed point for $H$.\\
Let $\lambda \in X$ be $C^1$. Hence, $H^{n} ( \lambda, d \lambda) \to (\lambda_u, \alpha_u)$. But $H^{n} (\lambda, d \lambda) = ( T^n\lambda , \alpha_n)$
with $$\alpha_n = G_{ T^{n-1} \lambda} \circ \dots  \circ G_{\lambda} d \lambda$$ It is clear that if $\lambda \in C^1$, so is $T \lambda$ and an iterative use of (\ref{rec_for_G}) implies that $\alpha_n = d \left( T^n \lambda \right) $. This shows that $\lambda_u$ is $C^1$ and $d \lambda_u = \alpha_u$. 
\end{proof}

Let us now prove Proposition \ref{G_lambda_contraction}. 
\begin{proof}
Let $\lambda \in X$ and fix $\alpha, \beta \in Y$.  It is of course enough to control $||G_\lambda \alpha(\rho^\prime) - G_\lambda \beta(\rho^\prime)||$ for $\rho^\prime \in \supp (\chi)$ since both $G_\lambda \alpha$ and $G_\lambda \beta$ vanishes outside. Let us fix $\rho^\prime = F(\rho) \in \supp (\chi)$. \\
$G_\lambda \alpha(\rho^\prime) - G_\lambda \beta(\rho^\prime)$ is given by 
$$\chi(\rho^\prime) \partial_\lambda t (\rho,\lambda(\rho) ) [\alpha(\rho)- \beta(\rho) ] \circ (d_\rho F)^{-1}$$ so it is enough to control  $\partial_\lambda t (\rho,\lambda(\rho) ) \gamma(\rho) \circ (d_\rho F)^{-1}$ for $\gamma = \alpha- \beta$. 
With the precise expression of $\partial_\lambda t (\rho,\lambda(\rho) )$ given by (\ref{partial_lambda_T}), we can estimate 
$$|\partial_\lambda t (\rho,\lambda(\rho) ) | =  \frac{|d(\rho)|}{|a(\rho)  + \lambda(\rho) b(\rho)|} +O_\nu(\eta) = \frac{|d(\rho)|}{|a(\rho)|}  + O_\nu(\eta)   $$
(By the notation $O_\nu(\eta)$, we mean that this term is bounded by $C \eta$ where $C$ is a constant depending only on $\nu$ and $(F,U)$). \\
Moreover, we have $|| (d_\rho F)^{-1}||  = \max\left( \frac{1}{a(\rho)}, \frac{1}{d(\rho)} \right)+ O_\nu(\eta) =  \frac{1}{d(\rho)} + O_\nu(\eta)$. 
Hence, 
$$||\partial_\lambda t (\rho,\lambda(\rho) ) \gamma(\rho) \circ (d_\rho F)^{-1}|| \leq \left( \frac{1}{a(\rho)} + O_\nu(\eta) \right) || \gamma(\rho)|| \leq \left( \nu + O_\nu (\eta)  \right) ||\gamma||_Y$$
Hence, if $\eta$ is small enough, the proposition is proved.
\end{proof}

\paragraph{Hölder regularity of $\alpha_u$.} 
In fact, as explained at the end of 19.1.d in \cite{KH}, we can improve the $C^1$ regularity. 

To deal with Hölder regularity of sections $\alpha : \Omega \to T^*\Omega$ , we will simply evaluate the distance between $\alpha(\rho_1)$ and $\alpha(\rho_2)$ for $\rho_1, \rho_2 \in \Omega$ using the natural identification  $T^*\Omega = \Omega \times (\R^2)^*$, where we see $\alpha(\rho_1)$ as an element of $(\R^2)^*$. This allows us to write $\alpha(\rho_1)- \alpha(\rho_2)$ and compute $||\alpha(\rho_1)- \alpha(\rho_2)||$ where $|| \cdot ||$ is a norm on $(\R^2)^*$. There exists $C >0$ such that for every $\alpha \in Y$, 
$\sup_{\rho \in \Omega} ||\alpha(\rho) || \leq C ||\alpha||_Y$. 

Let us introduce $\mu$ a Lipschitz constant for $F^{-1}$ on $\Omega$ and an exponent $\beta >0$ such that 
\begin{equation}
\nu \mu^\beta  < 1
\end{equation}
This condition is called a \emph{bunching condition} in \cite{KH} (19.1.d). Such a $\beta$ exists. 
We will then show the following, which finally concludes the proof of Theorem \ref{Thm_regularity}. 

\begin{prop}
$\alpha_u$ is $\beta$-Hölder, that is to say, $\lambda_u$ is $C^{1, \beta}$. 
\end{prop}

\begin{proof}
Let us introduce 
\begin{align*}
Y^\beta \coloneqq \{ \alpha \in Y ; \alpha \text{ is } \beta\text{- Hölder} \}
\end{align*}
Let us consider some $\varepsilon >0$ to be determined later and we equip $Y^\beta$ with the norm 
$$ || \alpha ||_{Y^\beta} = ||\alpha ||_Y + \varepsilon ||\alpha||_\beta\; ; \; || \alpha||_\beta= \sup_{ \rho_1 \neq \rho_2 } \frac{||\alpha(\rho_1) - \alpha(\rho_2)||}{d(\rho_1,\rho_2)^\beta}$$
The map $T: X \to X$ defined by (\ref{def_of_T}) actually maps $X \cap C^1(\Omega , \R)$ to  $X \cap  C^1(\Omega , \R)$. Moreover, our previous results have proved that $\lambda_u$ is an attractive fixed point for $T$ in $X \cap C^1 (\Omega, \R)$, where $X \cap C^1 (\Omega, \R)$ is now equipped with the $C^1$ norm.
For $\lambda \in X$ and $\alpha \in Y$, we can write, 
$$G_\lambda \alpha = \gamma_\lambda + \tilde{G}_\lambda \alpha$$ where for $\rho^\prime = F(\rho) \in \supp \chi$, 
\begin{align*}
\gamma_\lambda(\rho^\prime) = \chi(\rho^\prime) d_\rho t(\rho, \lambda(\rho) ) + t(\rho, \lambda(\rho)) d_{\rho^\prime} \chi  \\
\tilde{G}_\lambda \alpha (\rho^\prime) = \chi(\rho^\prime) \partial_\lambda t(\rho, \lambda(\rho)) \alpha(\rho) \circ (d_\rho F )^{-1}
\end{align*}
We state here some obvious facts on $\gamma_\lambda$ and $\tilde{G}_\lambda$
\begin{itemize}
\item $C_1 \coloneqq \sup_{ \lambda \in X} ||\gamma_\lambda||_\infty < + \infty$ ;
\item if $\lambda \in X \cap C^1 (\Omega, \R)$, $\gamma_\lambda$ is also $C^1$; 
\item According to Proposition \ref{G_lambda_contraction}; $\tilde{G}_\lambda : Y \to Y$ is a contraction with Lipschitz constant $L_\lambda$  and $\nu_1 \coloneqq \sup_{\lambda \in X} L_\lambda < 1$ ; 
\item if $\lambda \in X \cap C^1 (\Omega, \R)$ and $\alpha$ is $\beta$-Hölder, $\tilde{G}_\lambda \alpha$ is $\beta$-Hölder. 
\end{itemize}
If $M > \frac{C_1}{1- \nu_1}$ and $\lambda \in X \cap C^1(\Omega, \R)$, then $||d \lambda||_Y \leq M \implies ||d (T\lambda) ||_Y \leq M$. 
Indeed, we have 
$$ ||d (T \lambda)||_Y = ||G_\lambda (d\lambda)||_Y =|| \gamma_\lambda + \tilde{G}_\lambda d \lambda||_Y \leq C_1 + \nu_1 M \leq M $$
Hence, we introduce the complete metric space 
\begin{equation}
X^\prime = \{ \lambda \in X \cap C^1 (\Omega, \R), ||d\lambda||_Y \leq M \}
\end{equation}
$T(X^\prime) \subset X^\prime$ and $\lambda_u$ is an attractive fixed point for $(X^\prime,T)$. \\
We now wish to apply the Fiber Contraction Theorem to 
$$H_\beta : (\lambda, \alpha) \in X^\prime \times Y^\beta \mapsto (T \lambda, G_\lambda \alpha) \in X^\prime \times Y^\beta$$
To do so, we need to show that for every $\lambda \in X^\prime$, $G_\lambda : Y^\beta \to Y^\beta$ is a contraction and find a uniform estimate for the Lipschitz constants. 

Let's consider $\alpha_1, \alpha_2 \in Y^\beta$ and set $\gamma = \alpha_1 - \alpha_2$. We want to estimate the $Y^\beta$ norm of $\tilde{G}_\lambda \gamma$. We already know that $||\tilde{G}_\lambda \gamma||_Y \leq \nu_1 ||\gamma||_Y$. Take $\rho^\prime_1, \rho^\prime_2 \in \Omega$ and let's estimate $||\tilde{G}_\lambda \gamma (\rho^\prime_1) -\tilde{G}_\lambda \gamma(\rho^\prime_2)||$. We distinguish 3 cases  : 

\begin{itemize}[label = -]
\item $\rho^\prime_1, \rho^\prime_2 \not \in \supp \chi$ : there is nothing to write. 
\item $\rho^\prime_1 \in \supp \chi , \rho^\prime_2  \not \in \Omega \cap F( \Omega) $. In this case, $d(\rho^\prime_1, \rho^\prime_2) \geq \delta >0$ where $\delta$ is the distance between $\supp \chi$ and $\left(\Omega \cap F( \Omega) \right)^c$. Hence, 

$$\frac{||\tilde{G}_\lambda \gamma (\rho_1^\prime) - \tilde{G}_\lambda(\rho_2^\prime)||}{d(\rho^\prime_1, \rho^\prime_2)^\beta}  \leq \delta^{-\beta} || \tilde{G}_\lambda \gamma (\rho_1^\prime) || \leq  \delta^{-\beta} C ||\tilde{G}_\lambda \gamma||_Y \leq \nu_1 \delta^{-\beta} C ||\gamma||_Y$$
\item $\rho^\prime_1, \rho^\prime_2 \in \Omega \cap F(\Omega)$. Let's write $\rho^\prime_1 = F(\rho_1), \rho_2^\prime = F(\rho_2)$ and note that $d(\rho_1,\rho_2) \leq \mu d(\rho^\prime_1, \rho^\prime_2)$. 
\begin{align*}
\tilde{G}_\lambda \gamma (\rho_1^\prime) - \tilde{G}_\lambda \gamma (\rho_2^\prime) &= \chi(\rho_1^\prime) \partial_\lambda t(\rho_1,\lambda(\rho_1) ) [ \gamma(\rho_1) - \gamma(\rho_2) ] \circ (d_{\rho_1} F)^{-1}  \left. \right\} (1) \\
&+ \left[\chi(\rho_1^\prime) \partial_\lambda t(\rho_1,\lambda(\rho_1) ) - \chi(\rho_2^\prime)\partial_\lambda t(\rho_2,\lambda(\rho_2) )\right] \gamma(\rho_2)  \circ (d_{\rho_1} F)^{-1}   \left. \right\} (2) \\ 
&+ \chi(\rho_2^\prime) \partial_\lambda t(\rho_2,\lambda(\rho_2)) \gamma(\rho_2) \circ \left[(d_{\rho_1} F)^{-1} - (d_{\rho_2} F)^{-1} \right]  \left. \right\} (3) 
\end{align*}
\end{itemize}
To handle the last two terms (2) and (3), we notice that $\rho^\prime \in \Omega \cap F(\Omega) \mapsto \chi(\rho^\prime) \partial_\lambda t(\rho, \lambda(\rho) ) $ is Lipschitz since $\lambda$ is $C^1$, with Lipschitz constant which can be chosen uniform for $\lambda \in X^\prime$. The same is true for $\rho \mapsto d_{\rho} F^{-1}$. Hence, there exists a uniform constant $C>0$ such that 
$$ ||(2) + (3)|| \leq C d(\rho_1^\prime,\rho_2^\prime)^\beta ||\gamma||_Y$$
To deal with the first term (1), we recall that by previous computations, 
$$ | \chi(\rho^\prime) \partial_\lambda t (\rho, \lambda(\rho) ) | \cdot ||(d_\rho F)^{-1} || \leq \nu + O_\nu(\eta)$$
As consequence, we have 
$$ ||(1)|| \leq ( \nu + O_\nu(\eta) ) ||\gamma||_\beta d(\rho_1,\rho_2)^\beta \leq ( \nu + O_\nu(\eta) )\mu^\beta ||\gamma||_\beta d(\rho_1^\prime,\rho_2^\prime)^\beta $$
Henceforth, if $\eta$ is small enough, so that $ \nu_2 \coloneqq (\nu + O_\nu(\eta) )\mu^\beta <1$, 
\begin{equation*}
||H_\lambda \gamma||_\beta \leq \nu_2 ||\gamma||_\beta + C ||\gamma||_Y
\end{equation*}
Eventually, \begin{align*}
||\tilde{G}_\lambda \gamma||_{Y^\beta} &\leq \nu_1 || \gamma||_Y + \varepsilon \left( \nu_2 ||\gamma||_\beta + C ||\gamma||_Y \right) \\
& \leq \left( \nu_1 + \varepsilon C \right) ||\gamma||_Y +  \nu_2 \varepsilon ||\gamma||_\beta \\
& \leq \nu_3 ||\gamma||_{Y^\beta}
\end{align*}
where $\nu_3 = \max \left(\nu_1 + \varepsilon C, \nu_2 \right) <1$ if $\varepsilon$ is small enough. 

The Fiber Contraction Theorem applies and says that $(\lambda_u, \alpha_u)$ is an attractive fixed point for $H_\beta$. We conclude as previously : consider $\lambda \in C^{1, \beta} (\Omega, \R) \cap X^\prime$ so that $(\lambda, d \lambda) \in X^\prime \times Y^\beta$. $H_\beta^n (\lambda, d \lambda) = (T^n \lambda, d T^n \lambda) \to (\lambda_u, \alpha_u)$ in $X^\prime \times Y^\beta$. That ensures that $\alpha_u$ is $\beta$-Hölder.  
\end{proof}

\subsubsection{Regularity of the stable and unstable leaves}

Once we've extended the unstable distribution to a an open neighborhood of $\mathcal{T}$, we take advantage of the fact that these distribution are 1-dimensional to integrate the vector field defined by their unit vector. 

We set $E_u(\rho) = \R (v_u(\rho) + \lambda_u(\rho) v_s(\rho)) $. Recall that in a compact neighborhood of $\mathcal{T}$, the relation $d_\rho F(E_u(\rho) ) = E_u(F(\rho))$ is valid due to the definition of $\lambda_u$ as the fixed point of $T$ defined in (\ref{def_of_T}). $T^*U$ is equipped with a smooth Riemannian metric such that $dF^{-1}$ is a contraction on $E_u(\rho)$ for $\rho \in \mathcal{T}$ and hence, in a compact neighborhood of $\mathcal{T}$, this is also true. We can consider the vector field 
$$ \rho  \in U \mapsto e_u(\rho) $$ 
where $e_u(\rho)$ is a unit vector spanning $E_u(\rho)$. By our previous result, this vector field is $C^{1, \beta}$ and if $\rho$ lies in a sufficiently small neighborhood of $\mathcal{T}$, $d_\rho (F^{-1})( e_u(\rho)) = \tilde{J}^u(\rho) e_u(F^{-1}(\rho) ) $ where $|\tilde{J}^u(\rho) | \leq \nu < 1$. 

We denote by $\varphi^t_u(\rho)$ the flow generated by $e_u(\rho)$ and we will show that one can identify the unstable manifolds and the flow lines of $e_u$ in a small neighborhood of $\mathcal{T}$. 

\begin{prop}\label{equivalence_flow_leaves}
There exists $t_0$ such that for every $\rho \in \mathcal{T}$, 
$ \{ \varphi^t_u (\rho) , |t| \leq t_0 \} \subset W_u(\rho) $ 
\end{prop}
\begin{proof}
 Consider $t_0$ is sufficiently small such that $|\tilde{J}^u(\varphi^t_u(\rho))| \leq \nu < 1 $ for $\rho \in \mathcal{T}$, $t \in [-t_0,t_0]$. 
 For $(t,\rho) \in \R \times U$, set $\mu(t,\rho) = \int_0^t \tilde{J}^u ( \varphi^s_u(\rho) ) ds$ and we claim that for $t_0$ small enough, if $|t| \leq t_0$, 
  $$F^{-1} (\varphi^t_u (\rho) ) = \varphi_u^{\mu(t,\rho)} (F^{-1} (\rho) )$$
Indeed, in $t =0$, both are equal to $F^{-1}(\rho)$ and a quick computation shows that both satisfy the ODE 
 $$ \frac{d}{dt}Y(t) = J^u (\varphi^t_u(\rho) ) e_u(Y(t)) $$
 As a consequence, by induction, we see that one can write for $n \in \N$, 
 $$ F^{-n} (\varphi^t_u(\rho) ) = \varphi_u^{\mu_n(t,\rho) } (F^{-n} (\rho) )$$
 where $\mu_n$ is defined by induction by $\mu_{n+1}(t,\rho) = \mu ( \mu_n(t,\rho) , F^{-n}(\rho)) $. Hence, if $|t| \leq t_0$ and $\rho \in \mathcal{T}$,  we see that $\mu_n(t,\rho)$ stays in $[-t_0, t_0]$ and moreover $|\mu_n(t,\rho) | \leq \nu^n |t|$. 
We then see that if $|t| \leq t_0$ and $\rho \in \mathcal{T}$,
$$d(F^{-n} ( \varphi^t_u(\rho) ) , F^{-n}(\rho) ) = d( \varphi_u^{\mu_n(t,\rho)} (F^{-n}(\rho)), F^{-n}(\rho) ) \leq C|\mu_n(t,\rho) | \leq C \nu^n$$
This shows that $\varphi^t_u(\rho)$ belongs to the global unstable manifold at $\rho$, and hence, if $t_0$ is small enough, $\varphi_u^t(\rho)$ belongs to the local manifold $W_u(\rho)$ and $t_0$ can be chosen uniformly with respect to $\rho \in \mathcal{T}$. 
\end{proof}

Since the regularity of the unstable distributions implies the same regularity for the flow $\varphi^t_u$ (see Lemma \ref{flow_regularity} in the Appendix), we deduce that, up to reducing the size of the local unstable manifolds, these local unstable manifolds $W_u(\rho)$ depend $C^{1, \beta}$ on the base point $\rho \in \mathcal{T}$. We'll also use this proposition to show the same regularity for holonomy maps. Suppose that $\varepsilon_0$ is small enough. We know that if $\rho_1, \rho_2 \in \mathcal{T}$ satisfy $d(\rho_1,\rho_2) \leq \varepsilon_0$, then $W_u(\rho_2) \cap W_s(\rho_1)$ consists of exactly one point. Let's note it $H_{\rho_1}^u(\rho_2)$.  

Finally, we define the holonomy map $$ H^u_{\rho_1,\rho_2} : \rho_3 \in W_s(\rho_2) \cap \mathcal{T} \mapsto H^u_{\rho_1}(\rho_3) \in W_s(\rho_1)\cap \mathcal{T}$$

\begin{figure}[h!]
\centering
\includegraphics[scale=0.4]{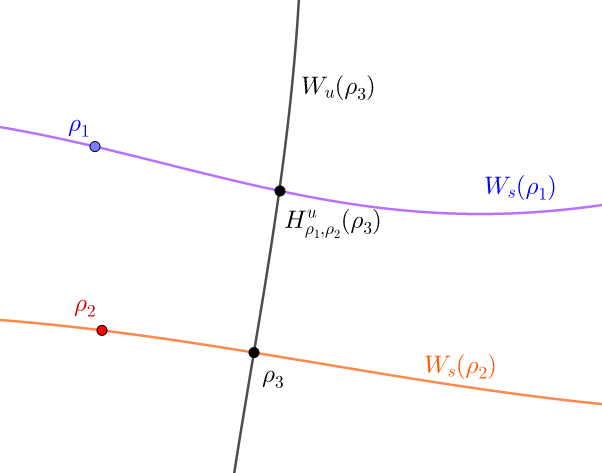}
\end{figure}
\begin{lem}\label{regularity_holonomy_maps}
If $\varepsilon_0$ is small enough, for every $\rho_1 \in \mathcal{T}$, the map 
$$H_{\rho_1}^u : \mathcal{T} \cap B(\rho_1, \varepsilon_0) \to W_s(\rho_1)\cap \mathcal{T}$$
is the restriction of a map $\tilde{H}_{\rho_1}^u : B(\rho_1, \varepsilon_0) \to W_s(\rho_1)$ which is $C^{1, \beta}$. 
\end{lem}

\begin{proof}
Let $\rho_1 \in \mathcal{T}$. As in the proof of Lemma \ref{Local_hpyerbolic_2}, consider a smooth chart $\kappa : U_1 \to V_1 \subset \R^2$ , $\rho_1 \in U_1, 0 \in V_1$ such that : 
\begin{itemize}[nosep]
\item $\kappa (\rho_1) = (0,0)$
\item $\kappa \left( W_s(\rho_1) \cap U_1 \right) = \{ (0,s) , s \in \R \} \cap V_1$
\item $\kappa \left( W_u(\rho_1)  \cap U_1 \right) = \{ (u,0) , u \in \R \} \cap V_1$
\item $d_{\rho_1} \kappa(e_u(\rho_1) ) = (1,0)$.
\end{itemize}
We now work in this chart $V_1$ and note $\Phi^t = \kappa \circ \varphi^t_u \circ \kappa^{-1}$ the flow in this chart, well defined for $t$ small enough. Consider the map 
$$ \psi(u,s) = \Phi^{u}(0,s)$$
$\psi$ is $C^{1,\beta}$ and $d_0\psi=\I_2$. By the Inverse Function Theorem, $\psi$ is a local diffeomorphism between neighborhoods of $0$ : 
$$ \psi : V_2 \to V_2^\prime$$
Since $d_{(u,s)} \left(\psi^{-1} \right)   =  \left(  d_{ \psi^{-1}(u,s)}\psi\right)^{-1}$, $\psi^{-1}$ is $C^{1, \beta}$. We now consider $$\kappa_0 = \psi^{-1} \circ \kappa : \kappa^{-1} (V_2) \coloneqq U_2 \to V^\prime_2$$ and observe that : 
\begin{itemize}[nosep]
\item $\kappa_0 (W_s(\rho_1) \cap U_2) = \{ (0,s) , s \in \R \} \cap V_2^\prime$ ; 
\item$ \kappa_0 \circ \varphi^t_u \circ \kappa_0^{-1} (u,s)) = (u + t, s)$. In other words $\kappa_0 $ rectifies the unstable manifolds.
\end{itemize}
Armed with theses facts, we define 
$$ \tilde{H}_{\rho_1}^u  : U_2 \to W_s(\rho_1) \quad ; \quad  \tilde{H}_{\rho_1}^u = \kappa_0^{-1} \circ \pi_s \circ \kappa_0 $$
where $\pi_s(u,s) = (0,s)$.  $\tilde{H}_{\rho_1}^u$ is $C^{1, \beta}$.
We assume that $B(0,\varepsilon_0) \subset U_1$.  Let us check that $ \tilde{H}_{\rho_1}^u$ extends the holonomy map in $B(\rho_1, \varepsilon_0)$ (if $\varepsilon_0$ is small enough).  Let $\rho_2 \in \mathcal{T} \cap B(\rho_1,\varepsilon_0)$ and note $\rho^\prime_2 =  \tilde{H}_{\rho_1}^u(\rho_2)$. By definition of $ \tilde{H}_{\rho_1}^u$, $\rho_2^\prime$ can be written $\rho_2^\prime=\varphi^t_u(\rho_1)$ and hence, if $\varepsilon_0$ is small enough, $\rho^\prime_2 \in W_u(\rho_1)$. Since, $\rho_2^\prime \in W_s(\rho_2)$, we see that $\rho_2^\prime = H_{\rho_1}^u(\rho_2)$.  
\end{proof}

Note that by compactness, $\varepsilon_0$ can be chosen uniformly in $\rho_1 \in \mathcal{T}$ and the $C^{1, \beta}$ norms of  $\tilde{H}_{\rho_1}^u$ are uniform.
As a corollary, we get the following : 
\begin{cor}
Suppose that $\varepsilon_0$ is small enough. Then, the holonomy maps, defined for $\rho_1,\rho_2 \in \mathcal{T}$ with $d(\rho_1,\rho_2) \leq \varepsilon_0$, 
$$ H_{\rho_1,\rho_2}^u : W_s(\rho_2) \cap \mathcal{T} \to W_s(\rho_1) \cap \mathcal{T}$$
are the restrictions of $C^{1,\beta} : \tilde{H}_{\rho_1,\rho_2}^u : W_s(\rho_1) \to W_s(\rho_2)$ with $C^{1,\beta}$ norms uniform in $\rho_1, \rho_2$. 
\end{cor}

\subsection{Adapted charts}\label{Adapted_charts}

We construct charts in which the unstable manifolds are close to horizontal lines. These charts will be used at different places and their existence relies on the $C^{1+ \beta}$ regularity of the unstable distribution. 

\paragraph{Weak version. }
We start with a weak version of these charts. 

\begin{lem}\label{weak_adapted_chart}
Suppose that $C>0$ is a fixed global constant and $\varepsilon_0$ is chosen small enough. For every $\rho_0 \in \mathcal{T}$, there exists a canonical transformation 
\begin{equation*}
\kappa_0 : U_{\rho_0}^\prime  \to V_{\rho_0}^\prime \subset \R^2
\end{equation*}
satisfying (we note $(y,\eta)$ the variable in $\R^2$) : 
\begin{enumerate}[label=(\arabic*), nosep]
\item $B(\rho_0, C \varepsilon_0) \subset U_{\rho_0}^\prime$; 
\item $\kappa_0(\rho_0) = 0$ , $d_{\rho_0} \kappa_0 (E_u(\rho_0) ) = \R \times \{ 0 \} ; d_{\rho_0} \kappa_0 (E_s(x) ) = \{0\} \times \R $  ;
\item The image of the unstable manifold $W_u(\rho_0) \cap U_{\rho_0}^\prime$ is exactly $\{ (y,0) , y \in \R \} \cap V_{\rho_0}^\prime$.
\end{enumerate}
Moreover, for every $N$, the $C^N$ norms of $\kappa_0$ are bounded uniformly with respect to $\rho_0 \in \mathcal{T}$. 
\end{lem}

\begin{rem}
The difference with the charts used in the proof of Lemma \ref{Local_hpyerbolic_2} is that we require $\kappa_0$ to be a smooth canonical transformation.
\end{rem}

\begin{proof}
 $W_u(\rho_0)$ is a $\cinf$ manifold, hence there exists a  $\cinf$ defining function $\eta$ defined in a neighborhood $\rho_0$ : namely $d_{\rho_0} \eta  \neq 0$ and $W_u(\rho_0) = \{ \eta = 0 \}$ locally near $\rho_0$. Darboux's theorem gives a function $y$ defined in a neighborhood of $\rho_0$ such that $(y,\eta)$ forms a system of symplectic coordinates. We can assume that $y(\rho_0) = 0$. If $\kappa(\rho) = (y, \eta)$, the third point is satisfied by assumption on $\eta$ and we need to ensure that $d_{\rho_0} \kappa (E_s(\rho_0) ) = \{0\} \times \R$ by modifying $\eta$ in a symplectic way. \\
Assume that $d_{\rho_0} \kappa (E_s(\rho_0) ) = \R {}^t(a,1)$. 
 The sympletic matrix $$A = \left( \begin{matrix}
1 & -a \\
0 & 1
\end{matrix}
 \right)$$ maps the basis $\left({}^t (1,0) , {}^t(a,1) \right)$ to the canonical basis of $\R^2$ and we can set $\kappa_0 \coloneqq A \circ \kappa$ which is the required canonical transformation, defined in a small neighborhood $U_{\rho_0}^\prime$ of $\rho_0$. 

We can ensure that $B(\rho_0, C\varepsilon_0) \subset U_{\rho_0}^\prime$ for $\varepsilon_0$ small enough and the uniformity of the $C^N$ norms of $\kappa$ thanks to the compactness of $\mathcal{T}$ and the fact that the unstable distribution depends continuously on $\rho_0 \in \mathcal{T}$. 
\end{proof}

\paragraph{Straightened version. }

We now straighten the unstable manifolds in a stronger version of the previous charts.  The construction and the use of these charts is similar to \cite{NDJ19} (Lemma 2.3).

\begin{lem}\label{Existence_of_adapated_charts}
Suppose that $\varepsilon_0$ is chosen small enough. For every $\rho_0 \in \mathcal{T}$ there exists a canonical transformation 
\begin{equation*}
\kappa = \kappa_{\rho_0}  : U_{\rho_0} \subset U \to V_{\rho_0} \subset \R^2
\end{equation*}
satisfying (we note $(y,\eta)$ the variable in $\R^2$) : 
\begin{enumerate}[label=(\arabic*), nosep]
\item $B(\rho_0, 2\varepsilon_0) \subset U_{\rho_0}$  ; 
\item $\kappa(\rho_0) = 0$ , $d_{\rho_0} \kappa  (E_u(\rho_0) ) = \R \times \{ 0 \} ; d_{\rho_0} \kappa (E_s(\rho_0) ) = \{0\} \times \R $ 
\item The images of the unstable manifolds $W_u(\rho), \rho \in U_{\rho_0} \cap  \mathcal{T}$, are described by \begin{equation}
\kappa   \left( W_u(\rho) \cap U_{\rho_0} \right)  = \Big\{ \Big( y, g(y, \zeta(\rho)) \Big) , y \in \Omega \Big\} 
\end{equation}
where $\Omega \subset \R$ is an open set, $\zeta : U_{\rho_0} \to \R$ is $C^{1 + \beta}$, $g : \Omega \times I\to \R$ is $C^{1+ \beta}$ (where $I$ is a neighborhood of $ \zeta(U_{\rho_0} ) $) and they satisfy 
\begin{enumerate}[label = (\roman*)]
\item $\zeta$ is constant on the unstable manifolds ; 
\item $\zeta(\rho_0)  =0 $, $g(y,0) = 0$ ; 
\item $g(0, \zeta) = \zeta$ ; 
\item $\partial_\zeta g(y,0) = 1$ 
\end{enumerate}
\end{enumerate}
The derivatives of $\kappa_{\rho_0}$ and the $C^{1+\beta}$ norms of $g, \zeta$ are bounded uniformly in $\rho_0$. 
\end{lem}

\begin{rem}
The most important condition, which will be used later on, is the last one : it makes the unstable manifolds very close to horizontal lines. The model situation we expect is when the unstable distribution is constant and horizontal. 
\end{rem}

\begin{proof}
Around a point $\rho_0 \in \mathcal{T}$, we work in the charts given by Lemma \ref{weak_adapted_chart} : $\kappa_0 : U_{\rho_0}^\prime \to V_{\rho_0}^\prime$. We recall that the unstable distribution is given by the restriction of a $C^{1+ \beta}$ vector field $e_u$. If $U_{\rho_0}^\prime$ is a sufficiently small neighborhood of $\rho_0$, we can write, for $\rho \in  U^\prime_{\rho_0}$, 
\begin{equation}\label{reparametrage}
d_\rho \kappa_0 (e_u(\rho) )  \in \R \tilde{e}_u(\rho) \quad \text{with }\tilde{e}_u (\rho) = {}^t (1 , f_0(\rho) ) 
\end{equation} 
where $f_0  : U^\prime_{\rho_0} \to \R$ is a $C^{1+ \beta}$ function which is nothing but the slope of the unstable direction in the chart $\kappa_0$.  
In the $(y,\eta)$ variable, we still note $f_0(\rho) = f_0(y,\eta)$ and we observe that due to the assumption on $\kappa_0$, we have 
$$ f_0(y, 0)  = 0 \quad , \quad (y,0) \in V_{\rho_0}^\prime$$ 
We consider $\Phi^t (y,\eta) $ the flow generated by the vector field $\tilde{e}_u$. Due to the form of $\tilde{e}_u$, we can write, 
$$ \Phi^t(y,\eta) = (y+t , Z^t(y, \eta) )$$
The reparametrization made in (\ref{reparametrage}) does not change the flow lines of the vector field $(\kappa_0)_* e_u$. In particular, in virtue of Proposition \ref{equivalence_flow_leaves}, they coincide locally with the unstable manifolds. More precisely, if we set, $$g_0(y,\eta) \coloneqq Z^y(0,\eta)$$ then, for $ (0, \eta)  = \kappa_0(\rho) \in  \kappa_0(\mathcal{T} \cap W_s(\rho_0))$,
$$ \kappa_0 \Big( W_u(\rho) \Big) \cap \{ |y| < y_0 \} = \left\{ \Big(y, g_0(y, \eta) \Big) , |y| < y_0 \right\}$$
for some $y_0$ small enough (which can be chosen uniformly in $\rho_0$). 
To define $\zeta$, we go back up the flow : 
suppose that $\rho \in U^\prime_{\rho_0}$ and write $\kappa_0(\rho) = (y, \eta)$ and assume $|y| < y_0$. We set 
$$ \zeta(\rho) \coloneqq Z^{-y} (y, \eta)$$ 
To say it differently, $\kappa_0(W_u(\rho)$ intersects the axis $\{y=0 \}$ at $(0, \zeta(\rho) ) $. 
\begin{figure}[h]
\begin{center}
\includegraphics[scale=0.5]{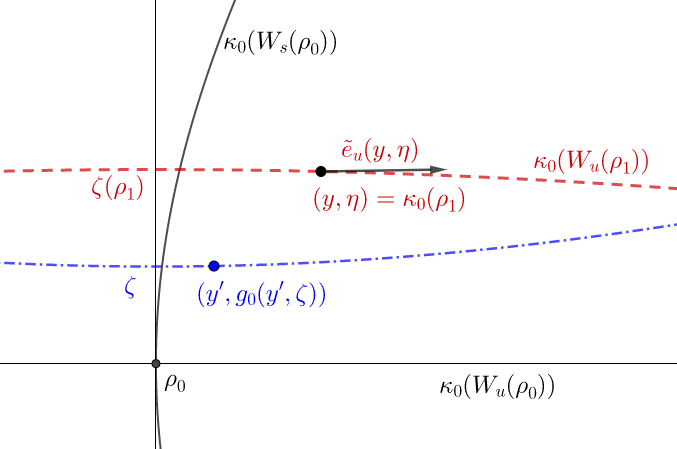}
\caption{The definitions of $g_0$ and $\zeta$ use the flow generated by $\tilde{e}_u$. }
\end{center}
\end{figure}

$\zeta$ and $g_0$ are $C^{1+ \beta}$, their $C^{1+ \beta}$ norms depend uniformly on $\rho_0$ and they satisfy : 

\begin{itemize}[nosep]
\item By definition, $\zeta$ is constant on the flow lines, and hence, on the unstable manifolds $W_u(\rho)$ if $\rho \in \mathcal{T} \cap U_{\rho_0}^\prime \cap \{ |y| < y_0 \} $ ; 
\item $\zeta(\rho_0) = 0$  ; 
\item Since $f_0(y,0)=0$, $Z^y(0, 0)  = 0$ and hence $g_0(y, 0) = 0$ ; 
\item Since $Z^0(0, \eta) = \eta$, $g_0(0,\eta) = \eta$.
\end{itemize}
However, at this stage, the last condition ($\partial_\zeta g_0(y,0) = 1$) is not satisfied by $g_0$ and we need to modify the chart. To do so, we'll make use of the following lemma, which is proved in the appendix \ref{appendix_regularity_partial}. 

\begin{lem}\label{Lemma_regularity_partial}
The map $y \in \{ |y| < y_0 \} \mapsto \partial_\eta f_0(y, 0)$ is smooth, with $C^N$ norms bounded uniformly in $\rho_0$. 
\end{lem}

We first show that this lemma implies that $y \in \{ |y| < y_0 \} \mapsto \partial_\eta g_0(y, 0)$ is smooth. Indeed, due to the $C^{1+ \beta}$ regularity of $E_u$, $(t, y, \eta) \mapsto Z^t (y,\eta)$ is $C^1$ and satisfies : 
$$ \frac{d}{dt} \partial_\eta Z^t(y,\eta) = \partial_\eta f_0 \left( y+t, Z^t (y, \eta) \right)$$
Specifying in $(y,\eta) = (0,0)$, we have 
$$ \frac{d}{dt}  \partial_\eta Z^t (0,0) = \partial_\eta f_0 (t,0) $$ 
This exactly says that $ y \mapsto \partial_\eta g_0(y,0)$ is $C^1$ and has $\partial_\eta f_0(y, 0)$ as derivative with respect to $y$ and hence $y \mapsto \partial_\eta g_0(y, 0)$ is smooth, as required.

Due to the relation $g_0(0, \eta) = \eta$, we have $\partial_\eta g_0 (0,0)= 1$. As a consequence, if $y_0$ is small enough, we can assume that 
$\partial_\eta g_0 (y,0) >0$ for $|y| < y_0$ and consider the smooth diffeomorphism defined in $\{ |y| < y_0 \}$ 
$$\psi : y \mapsto \int_0^y \partial_\eta g_0 (s,0) ds $$
We then use the canonical transformation 
$$ \Psi : (y, \eta) \mapsto \left( \psi(y) , \frac{\eta}{\psi^\prime(y) } \right)$$
We finally consider the chart $\kappa_{\rho_0} = \Psi \circ \kappa_0$ defined in $U_{\rho_0} = U^\prime_{\rho_0} \cap \{ |y| < y_0 \}$ and if $\varepsilon_0$ is small enough, we can ensure that $B(\rho_0, 2 \varepsilon_0) \subset U_{\rho_0}$. In this chart, the graph of $g_0 (\cdot, \zeta)$ is sent to the graph of the function 
$$ g : y \in \Omega \coloneqq \psi\left( (-y_0, y_0) \right) \mapsto \frac{g_0(\psi^{-1}(y), \zeta)}{\psi^\prime\left(\psi^{-1}(y) \right)} $$ 
We eventually check that 
\begin{itemize}[nosep]
\item $g(y,0) = 0$ since $g_0(y,0) = 0$ ; 
\item $g(0,\zeta) = \zeta$ since $\psi(0) = 0$, $\psi^\prime(0) = 1$ and $g(0,\zeta) = \zeta$ ; 
\item $\partial_\eta g(y,0) = 1$; 
\item The $C^{1+\beta}$ norm of $g$ is bounded uniformly in $\rho_0$ ; 
\item The $C^N$ norms of $\kappa_{\rho_0}$ are bounded uniformly in $\rho_0$. 
\end{itemize}
\end{proof}

\section{Construction of a refined quantum partition} \label{Section_first_ingredient}

We start the proof of Theorem \ref{Spectral_radius}. 
We consider $T=T(h) \in I_{0^+}(Y \times Y, F^\prime)$ a semi-classical Fourier integral operator associated with $F$, microlocally unitary in a neighborhood of $\mathcal{T}$, and a symbol $\alpha \in S_{0^+}(U)$. We want to show a bound for the spectral radius of $M(h) = T(h) \op(\alpha)$, independent of $h$. 

\subsection{Numerology}\label{Numerology}
We'll use the standard fact : 
$$ ||M^n||_{L^2 \to L^2}\leq \rho \implies \rho_{spec}(M) \leq \rho^{1/n}$$ 
The trivial lemma which follows reduces the theorem to the study of $||M^n||$ with $n=n(h) \sim  \delta |\log h|$. 
\begin{lem}
Let $\delta >0$ and $N(h) \in \N$ satisfy $N(h)\sim  \delta |\log h|$. Suppose that there exists $h_0 >0$ and $ \gamma>0$ such that 
\begin{equation}\label{Goal_thm}
\forall 0 < h < h_0 \quad , \quad ||M(h)^{N(h) } || \leq h^\gamma ||\alpha||_\infty^{N(h) } 
\end{equation} 
Then, for every $\varepsilon>0$, there exists $h_\varepsilon$ such that, for $h \leq h_\varepsilon$, 
$$\rho_{spec}(M(h) ) \leq e^{-\frac{\gamma}{\delta} + \varepsilon}||\alpha||_\infty$$
\end{lem}

\begin{proof}
It suffices to observe that under the assumption (\ref{Goal_thm}), we have $\rho_{spec}(M(h) ) \leq e^{ \frac{\gamma \log h}{N(h)}} ||\alpha||_\infty$ and use the equivalence for $N(h)$. 
\end{proof}

\begin{rem}
If we use the bound $||M|| \leq ||\alpha||_\infty + O(h^{1/2-\varepsilon})$, one get the obvious bound $||M^N|| \leq ||\alpha||_\infty^N (1+ o(1))$. Hence, (\ref{Goal_thm}) is a decay bound.
\end{rem}

The proof of Theorem \ref{Spectral_radius} is then reduced to the proof of the following proposition. 

\begin{prop}\label{Prop_thm}
There exists $\delta >0$, a family of integer $N(h) \sim \delta |\log (h)|$ and $\gamma>0$ such that, for $h$ small enough, (\ref{Goal_thm}) holds. 
\end{prop}

Actually, this proposition is enough to show Corollary \ref{Cor_Thm} concerning perturbed operators, in virtue of 

\begin{cor}\label{cor_thm_perturbation}
Suppose that $R(h) : L^2(Y) \to L^2(Y)$ is a family of bounded operators such that $R(h) = O(h^\eta)$ for some $\eta >0$. Then, there exists $\gamma^\prime = \gamma^\prime(\gamma, \eta)$ such that for $h$ small enough, 
$$\left| \left| (M(h) + R(h) ) ^{N(h)} \right|\right| \leq  h^{\gamma^\prime} ||\alpha||_\infty^{N(h) }  $$
\end{cor}

\begin{proof}
We write $$(M+R)^N = M^N + \sum_{\substack{ \epsilon \in \{0,1\}^N  \\ \varepsilon \neq (1,\dots,1)}} (\varepsilon_1 M +(1- \varepsilon_1) R) \dots (\varepsilon_N M +(1- \varepsilon_N) R)$$
Using this, we can estimate 
\begin{align*}
\left| \left| (M+R )^{N} \right|\right|& \leq  h^{\gamma} ||\alpha||_\infty^N + \left( \left( ||M|| + ||R|| \right)^N - ||M||^N \right) \\
&\leq h^{\gamma} ||\alpha||_\infty^N + N ||R|| \left( ||M|| + ||R|| \right)^{N-1} \\
&\leq h^{\gamma} ||\alpha||_\infty^N + C |\log h| h^\eta ||\alpha||_{\infty}^{N-1} ( 1 + O(h^\eta) ) \\
&= O( (h^{\gamma} + h^{\eta - }) ||\alpha||_\infty^N  )
\end{align*}
This gives the desired bound for any $\gamma^\prime < \min(\gamma, \eta)$. 
\end{proof}
Actually, the precise value of $N(h)$ we'll use is rather explicit and we now describe it. 
 We set \begin{equation}\label{alpha}
\mathfrak{b}= \frac{1}{1+\beta}
\end{equation}
where $\beta$ is the one appearing in Theorem \ref{Thm_regularity} concerning the regularity of the unstable distribution. 
We now choose $\delta_0 \in (0, \frac{1}{2})$ such that 
\begin{equation}\label{alpha+delta0}
\mathfrak{b} + \delta_0 < 1
\end{equation}
For instance, let us set $$\delta_0 = \frac{1-\mathfrak{b}}{2} = \frac{\beta}{2(1+ \beta)}$$
Recalling the definitions of the exponent $\lambda_0  \leq \lambda_1$ in (\ref{lamba0et1}) and (\ref{lambdao0et1bis}), we introduce the following notations 

\begin{equation}
 N(h) = N_0(h) + N_1(h) \quad  ; \quad N_0(h) = \left\lceil  \frac{\delta_0}{\lambda_1} |\log (h)| \right\rceil  \quad ; \quad N_1(h) = \left\lceil  \frac{1}{\lambda_0} |\log (h) | \right\rceil 
\end{equation}
 $N_0(h)$ (resp. $N_1(h)$) corresponds to a short (resp. long) logarithmic time. We will omit the dependence on $h$ in the following. 

To be complete with the numerology, we introduce another number $\tau <1$ such that 
\begin{equation}\label{condition_on_tau}
\mathfrak{b} < \tau < 1 \text{ and } \delta_0 \frac{\lambda_0}{\lambda_1} + \tau >1
\end{equation}
The meaning of these conditions will be clear in the core of the proof and we won't miss to recall where they are used. For instance, we set 
\begin{equation}\label{tau}
\tau = 1 - \frac{\lambda_0}{\lambda_1} \frac{1-\mathfrak{b}}{4}
\end{equation}

\paragraph{An important remark.} If two operators $M_1(h)$ and $M_2(h)$ are equal modulo $\hinf$, this is also the case for $M_1(h)^{N(h)}$ and $M_2(h)^{N(h)}$ as long as 
\begin{itemize}[label=-]
\item $N(h) = O(\log h)$.
\item $M_1(h) , M_2(h) = O(h^{-K})$ for some $K$. 
\end{itemize}
This will be widely used in the following. In particular, recall that we work with operators acting on $L^2(Y)$ but these operators take the form $ M_1(h) = \Psi_Y M_2(h) \Psi_Y$ where $\Psi_Y \in \cinfc(Y, [0,1])$  and $M_2(h)$ is a bounded operator on $\bigoplus_{j=1}^J L^2(\R)$ such that $M_2(h) = \Psi_Y M_2(h) \Psi_Y + \hinf_{L^2}$. As a consequence, modulo $\hinf$, it is enough to focus on $M_2(h)^{N(h)}$. For this reason, from now on and even if we keep the same notation, we work with  
$$M(h) = T(h)\op(\alpha): \bigoplus_{j=1}^J L^2(\R) \to \bigoplus_{j=1}^J L^2(\R)$$ where $T(h) = (T_{ij}(h) ) $ with $T_{ij} \in I_{0+}(\R \times \R, F_{ij}^\prime)$ and 
$$ \op(\alpha) = \Diag(\op(\alpha_1), \dots,\op( \alpha_J) ) $$

\subsection{Microlocal partition of unity and notations}
We consider some $\varepsilon_0 >0$, which is supposed small enough to satisfy all the assumptions which will appear in the following. \\
We consider a cover of $\mathcal{T}$ by a finite number of balls of radius $\varepsilon_0$ : 
$$ \mathcal{T} \subset \bigcup_{q=1}^Q B(\rho_q , \varepsilon_0)  \quad ; \quad \rho_q \in \mathcal{T}$$ and we assume that for all $ q \in \{1, \dots, Q \}$, there exists $j_q,l_q,m_q \in \{1, \dots, J \}$ such that 
$$B(\rho_q , 2 \varepsilon_0) \subset \widetilde{A}_{ j_q l_q } \cap \widetilde{D}_{m_q j_q} \subset U_{j_q}$$
We also assume that $T$ is microlocally unitary in $B(\rho_q, 4\varepsilon_0)$. 
We then note 
\begin{equation}\label{V_i}
\mathcal{V}_q = B(\rho_q , 2 \varepsilon_0)
\end{equation}
We complete this cover with
\begin{equation}\label{V_infini}
\mathcal{V}_\infty = U^\prime \setminus \bigcup_{q=1}^Q \overline{B(\rho_q, \varepsilon_0)} 
\end{equation}
$U^\prime \Subset U$ is an open set such that$ \WF(M) \Subset U^\prime \times U^\prime$. We note $U^\prime_j$ the component of $U^\prime$ inside $U_j$. 

\begin{figure}
\centering
\includegraphics[scale=0.5]{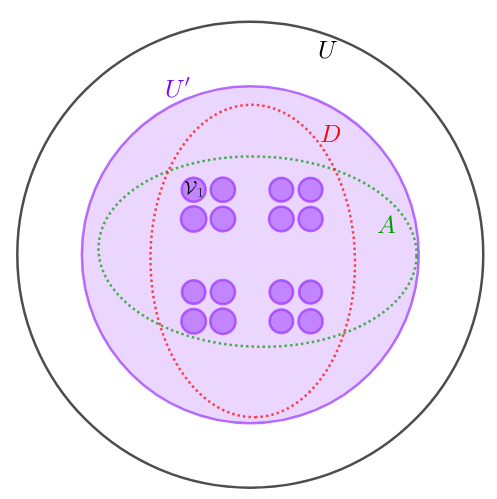}
\caption{The partition $(\mathcal{V}_q)_{q \in \mathcal{A}_\infty}$ is made by small neighborhoods of $\mathcal{T}$ (small purple disks) and a big open set included in $U^\prime$.  }
\end{figure}

We note $ \mathcal{A} = \{ 1, \dots, Q \}$ and $\mathcal{A}_\infty = \mathcal{A} \cup \{ \infty \}$. \\
We then consider a partition of unity associated with the cover $\mathcal{V}_1, \dots, \mathcal{V}_Q, \mathcal{V}_\infty$, namely a family of smooth functions $\chi_q \in \cinfc (U),$ for $q \in \mathcal{A}_\infty$ such that : 
\begin{itemize}
\item $\supp \chi_q \subset \mathcal{V}_q $
\item $0 \leq \chi_q \leq 1$
\item $1 = \sum_{q \in \mathcal{A}_\infty} \chi_q \text{ in }  \bigcup_{ q \in \mathcal{A}_\infty} \mathcal{V}_q$
\end{itemize}
More precisely, if $q \in \mathcal{A}, \chi_q \in \cinf(U_{j_q} )$ and for every $j \in \{1 , \dots , J\}$, there exists $b_j \in \cinfc(U_j)$ such that on $U_j^\prime$, $1 = b_j + \sum_{q \in \mathcal{A} , j_q =j } \chi_q $. Thus, $\chi_\infty = \sum_{j=1}^J b_j$. 

We can then quantize these symbols so as to get a pseudodifferential partition of unity. More precisely, to respect the matrix structure, we may write this quantization in a diagonal operator valued matrix, still denoted $\op$ : 
\begin{itemize}
\item for $q \in \mathcal{A}$, $A_q = \op(\chi_q)$ is the diagonal matrix $\Diag(0, \dots, \op(\chi_q) , 0, \dots ,0)$ where the block $\op(\chi_q)$ is in the $j_q$-ith position ; 
\item $\op (\chi_\infty) = \Diag(\op(b_1), \dots, \op(b_J))$. 
\end{itemize}
The family $(A_q)_{ q \in \mathcal{A}_\infty}$ satisfies the following properties : 
\begin{equation}\label{properties_Aq}
\sum_{ q \in \mathcal{A}_\infty} A_q = \Id \text{ microlocally in } U^\prime \quad ; \quad \forall q \in \mathcal{A}_\infty , ||A_q|| \leq 1 + O(h^{1/2} )
\end{equation}
Since $M = \sum_{ q \in \mathcal{A}_\infty} MA_q + \hinf$, we may write 
$$ M^n = \sum_{ \mathbf{q} \in \mathcal{A}^n_\infty} U_{\mathbf{q} } + \hinf$$ 
where for $\mathbf{q} = q_0 \dots q_{n-1}  \in \mathcal{A}^n_\infty$, 

\begin{equation}\label{U_q}
U_{\mathbf{q} }  \coloneqq MA_{q_{n-1}} \dots MA_{q_0}
\end{equation} 
For $\mathbf{q} = q_0 \dots, q_{n-1}  \in \mathcal{A}^n_\infty$, we also define a family of refined neighborhoods, forming a refined cover of $\mathcal{T}$, 

\begin{equation} \label{V_q}
 \mathcal{V}_{\textbf{q}}^{-} = \bigcap_{i=0}^{n-1} F^{-i} \left( \mathcal{V}_{q_i} \right) \quad ; \quad \mathcal{V}_{\textbf{q}}^+ = F^n \left( \mathcal{V}_{\textbf{q}}^{-}\right) = \bigcap_{i=0}^{n-1} F^{n-i}\left(\mathcal{V}_{q_i} \right)
  \end{equation}
This definition imply that a point $\rho \in \mathcal{V}_{\textbf{q}}^{-}$ lies in $\mathcal{V}_{q_i}$ at time $i$ (i.e $F^i(\rho) \in \mathcal{V}_{q_i}$) for $0 \leq i \leq n-1$ and a point $\rho \in \mathcal{V}_{\textbf{q}}^{+}$ lies in $\mathcal{V}_{q_{n-i}}$ at time $-i$, for $1 \leq i \leq n$. 
Roughly speaking, we expect that each operator $U_\mathbf{q}$ acts from $\mathcal{V}_\mathbf{q}^-$ to $\mathcal{V}_\mathbf{q}^+$ and is negligible (in some sense to be specified later on) elsewhere. Combining (\ref{properties_Aq}) and the bound on $M$ , the following bound is valid (for any $\varepsilon>0$) : 

\begin{equation}\label{bound_trivial_Uq}
||U_\mathbf{q}||_{L^2 \to L^2} \leq \left(||\alpha||_\infty + O(h^{1/2 -\varepsilon})\right)^n
\end{equation}
As soon as $|n| \leq C_0 |\log h|$, we have $||U_\mathbf{q}||_{L^2 \to L^2} \leq C||\alpha||_\infty^n$, for some $C$ depending on $C_0$ and a finite number of semi-norms of $\alpha$. 

\paragraph{Reduction to words in $\mathcal{A}$. } 
We can find a uniform $T_0 \in \N$ such that if $\rho \in \mathcal{V}_\infty$, there exists $k \in \{ - T_0, \dots, T_0 \}$ such that $F^k(\rho)$ "falls" in the hole. By standard properties of the Fourier integral operators, each component $(M^{T_0})_{i j }$ of $M^{T_0}$ is a Fourier integral operator associated with the component $(F^{T_0})_{i j }$ of $F^{T_0}$. In particular, $\WF^\prime (M^{T_0}) \subset \gr^\prime( F^{T_0})$. 
\vspace{0.5cm}

Let us study $M^{2T_0 + N(h)}$. If $\mathbf{q}=q_0 \dots q_{N-1} \in \mathcal{A}^N_\infty$ and if there exists an index $i \in \{0, \dots, N-1 \} $ such that $q_i = \infty$, one can isolate this index $i$ and trap $A_{q_i}$ between two Fourier integral operators $M_1, M_2$, belonging to a finite family of FIO associated with $F^{T_0}$, so that we can write 

$$M^{T_0} U_\mathbf{q} M^{T_0}  = B_1 M_1 A_{\infty} M_2 B_2$$ 
where $B_1, B_2$ satisfy the $L^2$-bound : 
$$ ||B_1|| \times ||B_2|| \leq (||\alpha||_\infty + O(h^{1/4}))^{N-1} =O(h^{-K})$$
for some integer $K$. Since, 
 $$ \WF^\prime (M_1 A_\infty M_2 ) \subset \{ \left( F^{T_0}( \rho) , F^{-T_0} (\rho) \right)  ; \rho \in \WF(A_\infty) \} = \emptyset$$
we have $M_1 A_\infty M_2 = \hinf $, with constants that can be chosen independent of $\mathbf{q}$. Hence, the same is true for $M^{T_0} U_\mathbf{q} M^{T_0}$.
$| \mathcal{A}^N | $ is bounded by a negative power of $h$. So, we can write : 
\begin{align*}
M^{N+ 2 T_0} = &\sum_{ \textbf{q} \in \mathcal{A}^N_\infty} M^{T_0} U_\mathbf{q} M^{T_0} \\
= &\sum_{ \textbf{q} \in \mathcal{A}^N} M^{T_0} U_\textbf{q} M^{T_0} + \hinf \\
= & M^{T_0} \left( \sum_{ \textbf{q} \in \mathcal{A}^N} U_\textbf{q} \right)M^{T_0} + \hinf \\
\end{align*}
We can then replace $M$ by 
\begin{equation}
\mathfrak{M} = \sum_{q \in \mathcal{A}} MA_q = M(\Id- A_\infty) + \hinf_{L^2 \to L^2}
\end{equation}
A decay bound 
\begin{equation}\label{Goal_thm_2}
||\mathfrak{M}(h)^{N(h)}||\leq h^\gamma ||\alpha||_\infty^{N(h)}
\end{equation}
will imply the required decay bound
(\ref{Goal_thm}) for $M$ with $N(h)$ replaced by $N(h) + 2T_0$. We are hence reduced to prove the decay bound (\ref{Goal_thm_2}). 

\subsection{Local Jabobian}\label{Local_jacobian}
\paragraph{A first definition.}
Following  \cite{NDJ19}, we introduce local unstable and stable Jacobians and we then state several properties.  For $n \in \N^*$ and $\textbf{q} \in \mathcal{A}^n$, let us define its local stable and unstable Jacobian. 

\begin{equation}\label{def_jacobian_1}
J^-_\textbf{q} \coloneqq \inf_{ \rho \in \mathcal{T} \cap \mathcal{V}^-_\textbf{q} } J^u_n (\rho) \; , \: J^+_\textbf{q} \coloneqq \inf_{ \rho \in \mathcal{T} \cap \mathcal{V}^+_\textbf{q} } J^s_{-n} (\rho)
\end{equation}
By the chain rule, we have for $\rho \in \mathcal{T} \cap \mathcal{V}^-_\textbf{q}$, 
$$ J^u_n (\rho) = \prod_{i=0}^{n-1} J_1^u\left( F^i (\rho) \right) $$ 
A similar formula is true for $\rho \in \mathcal{T} \cap \mathcal{V}^+_\textbf{q}$ : 
$$ J^s_{-n} (\rho) = \prod_{i=0}^{n-1} \left( J_1^s\left( F^{i-n} (\rho) \right)\right)^{-1} =  \prod_{i=0}^{n-1} J_{-1}^s\left( F^{-i} (\rho) \right)$$ 
Hence, we've got the basic estimates : 
\begin{align}
\mathcal{T} \cap \mathcal{V}^-_\textbf{q} \neq \emptyset \implies e^{\lambda_0 n } \leq J^-_\textbf{q}  \leq e^{\lambda_1 n } \\
 \mathcal{T} \cap \mathcal{V}^+_\textbf{q} \neq \emptyset \implies e^{\lambda_0 n } \leq J^+_\textbf{q}  \leq e^{\lambda_1 n }
\end{align}
If $\textbf{q} = q_0 \dots q_{n-1}$ and $\mathbf{q_-} = q_0 \dots q_{n-2}$, then $\mathcal{V}_\mathbf{q}^- \subset \mathcal{V}_\mathbf{q_-}^- $ and thus 

\begin{equation}
J_\mathbf{q}^- \geq e^{\lambda_0} J_\mathbf{q_-}^-
\end{equation}
Similarly, if $\mathbf{q_+} = q_1 \dots q_{n-1}$, $\mathcal{V}_\mathbf{q}^+ \subset \mathcal{V}_\mathbf{q_+}^+ $ and 
\begin{equation}
J_\mathbf{q}^+ \geq e^{\lambda_0} J_\mathbf{q_+}^+
\end{equation}
As a consequence of Corollary \ref{cor_control_jacobian}, if $\varepsilon_0$ is small enough, the local stable and unstable Jacobians give the expansion rate of the flow at every point of $\mathcal{T} \cap \mathcal{V}_{\mathbf{q}}^\pm$. If $\mathcal{T} \cap \mathcal{V}_{\mathbf{q}}^\pm \neq \emptyset$,  
\begin{align}
\forall \rho \in \mathcal{T} \cap \mathcal{V}_{\mathbf{q}}^- , \;  J^u_n(\rho) \sim J_{\mathbf{q}}^- \\
\forall \rho \in \mathcal{T} \cap \mathcal{V}_{\mathbf{q}}^+ , \;  J^s_{-n}(\rho) \sim J_{\mathbf{q}}^+
\end{align}

This definition is slightly not satisfactory since $J_\mathbf{q}^\pm = + \infty$ as soon as $\mathcal{V}_\mathbf{q}^\pm  \cap \mathcal{T}= \emptyset$. However, when $\mathcal{V}_\mathbf{q}^\pm \neq \emptyset$, this set can still stay relevant. For this purpose, we will give a definition of local stable and unstable Jacobian for such words with help of the Shadowing Lemma (\cite{KH} , Section 18.1). 

\paragraph{Enlarged definition.}
Let $n \in \N$ and $\textbf{q} = q_0 \dots q_{n-1} \in \mathcal{A}^n$. We focus on $\mathcal{V}_\mathbf{q}^-$, with the case of $\mathcal{V}_\mathbf{q}^+$ handled similarly by considering $F^{-1}$ instead of $F$.  

 If $\mathcal{V}_\mathbf{q}^-  \cap \mathcal{T} \neq \emptyset$, we keep the definition given in \ref{def_jacobian_1}. Assume now that $\mathcal{V}_\mathbf{q}^- \neq \emptyset$ but $\mathcal{V}_\mathbf{q}^-  \cap \mathcal{T}= \emptyset$. Fix $\rho \in \mathcal{V}_\mathbf{q}^-$. By definition of $\mathcal{V}_{q_i}$, for $i \in \{0, \dots, n-1\}$, we have $d(\rho_i, F^i(\rho) ) \leq 2 \varepsilon_0$. Hence, 

$$d(F(\rho_i), \rho_{i+1}) \leq d( F(\rho_i) , F^{i+1}(\rho) ) + d (F^{i+1}(\rho), \rho_{i+1} ) \leq C \varepsilon_0$$
for a constant $C$ only depending on $F$. That is to say, $(\rho_0, \dots, \rho_{n-1})$ is a $C\varepsilon_0$ pseudo orbit. Assume that $\delta_0 >0$ is a small fixed parameter. In virtue of the shadowing lemma, if $\varepsilon_0$ is sufficiently small, $(\rho_0, \dots, \rho_{n-1})$ is $\delta_0$ shadowed by an orbit of $F$ : there exists $\rho^\prime \in \mathcal{T}$ such that for $i \in \{0, \dots, n-1 \}$, $d(\rho_i, F(\rho^\prime) ) \leq \delta_0$. Consequently, 
$d(F^i (\rho) , F^i(\rho^\prime)) \leq \delta_0 + C\varepsilon_0$. If $\rho_2$ is another point in $\mathcal{V}_\mathbf{q}^-$, for $i= 0, \dots, n-1$, $d(F^i(\rho_2), F^i(\rho^\prime) ) \leq 2 \varepsilon_0 + C\varepsilon_0 + \delta_0$.  For convenience, set $\varepsilon_2 =2 \varepsilon_0+ \delta_0 + C\varepsilon_0$ and note that $\varepsilon_2$ can be arbitrarily small depending on $\varepsilon_0$. As a consequence, we have proven the following

\begin{lem}\label{help_def_jacobian_2}
If $\mathcal{V}_\mathbf{q}^- \neq \emptyset$, there exists $\rho^\prime \in \mathcal{T}$ such that $\forall i \in \{0, \dots, n-1 \}$ and $\rho \in \mathcal{V}_\mathbf{q}^-$, $d(F^{i}(\rho), F^i(\rho^\prime)) \leq \varepsilon_2$.
\end{lem}
Fix any $\rho^\prime$ satisfying the conclusions of this lemma and we arbitrarily set  
\begin{equation}\label{def_jacobian_2}
J_\mathbf{q}^- = J^u_n(\rho^\prime)
\end{equation}
If $\rho^\prime_1$ is another point satisfying this conclusion, we have $d(F^i(\rho^\prime), F^i(\rho^\prime_1)) \leq 2 \varepsilon_2$ for $i \in \{0, \dots, n-1 \}$ and in virtue of Corollary (\ref{cor_control_jacobian}), 
$$ J^u_n(\rho^\prime) \sim J^u_n(\rho^\prime_1)$$
Hence, up to global multiplicative constants, the definition of this unstable Jacobian is independent of the choice of $\rho^\prime$. Notice that if $\mathcal{V}_\mathbf{q}^-  \cap \mathcal{T} \neq \emptyset$, any $\rho^\prime \in \mathcal{T} \cap \mathcal{V}_\mathbf{q}^-$ satisfies the conclusions of Lemma \ref{help_def_jacobian_2} and $J_\mathbf{q}^- \sim J^u_n(\rho^\prime)$.

To define $J_\mathbf{q}^+$, we can argue similarly and show that there exists $\rho^\prime$ satisfying $d(F^i(\rho^\prime), F^i(\rho)) \leq  \varepsilon_2$ for $i \in \{-n, \dots, -1 \}$ and $\rho \in \mathcal{V}_\mathbf{q}^+$. We can assume that this is the same $\varepsilon_2$ as before and we set $J_\mathbf{q}^+ = J^s_{-n}(\rho^\prime)$ for any $\rho^\prime$. 
\begin{figure}[!h]
\includegraphics[scale=15]{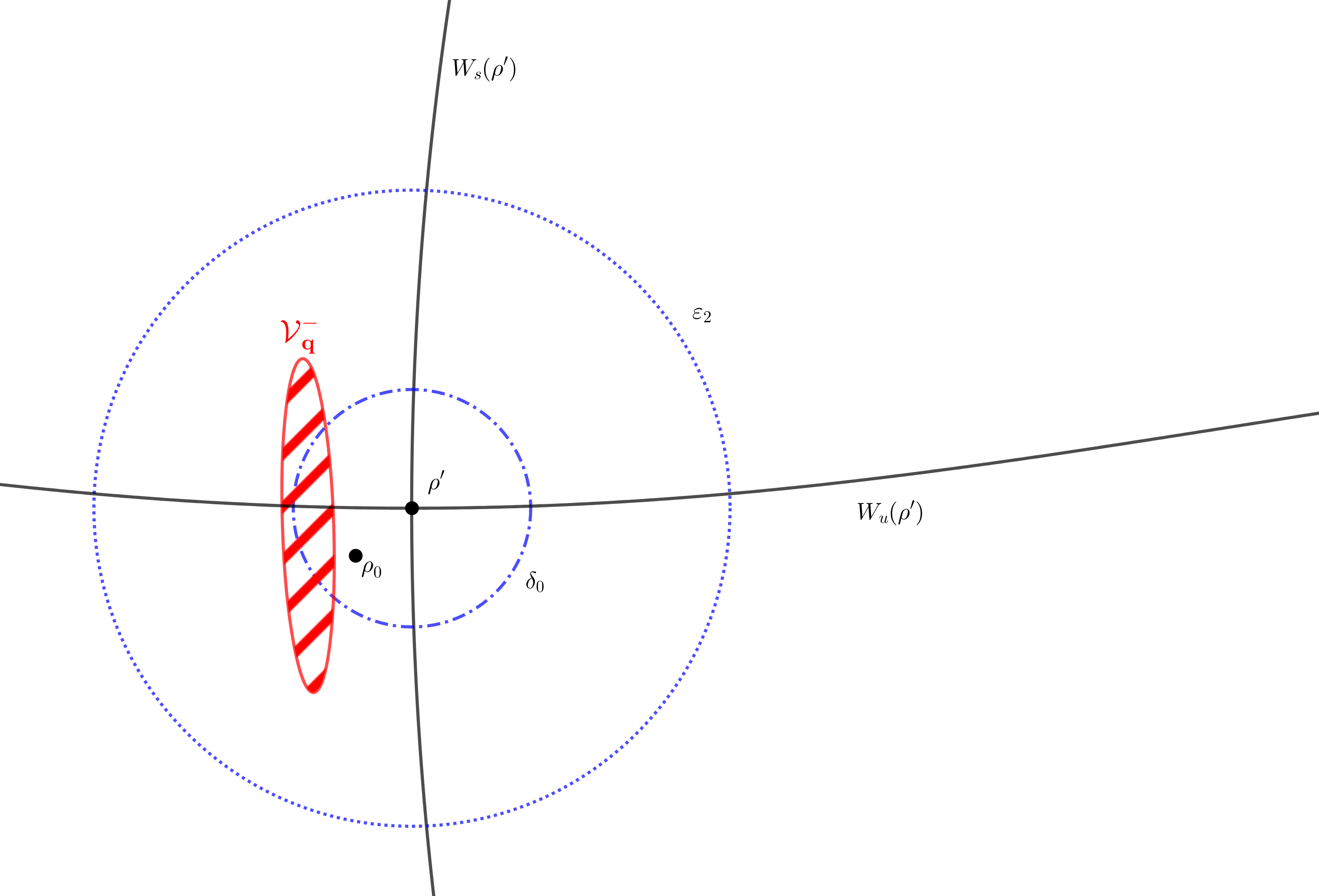}
\includegraphics[scale=15]{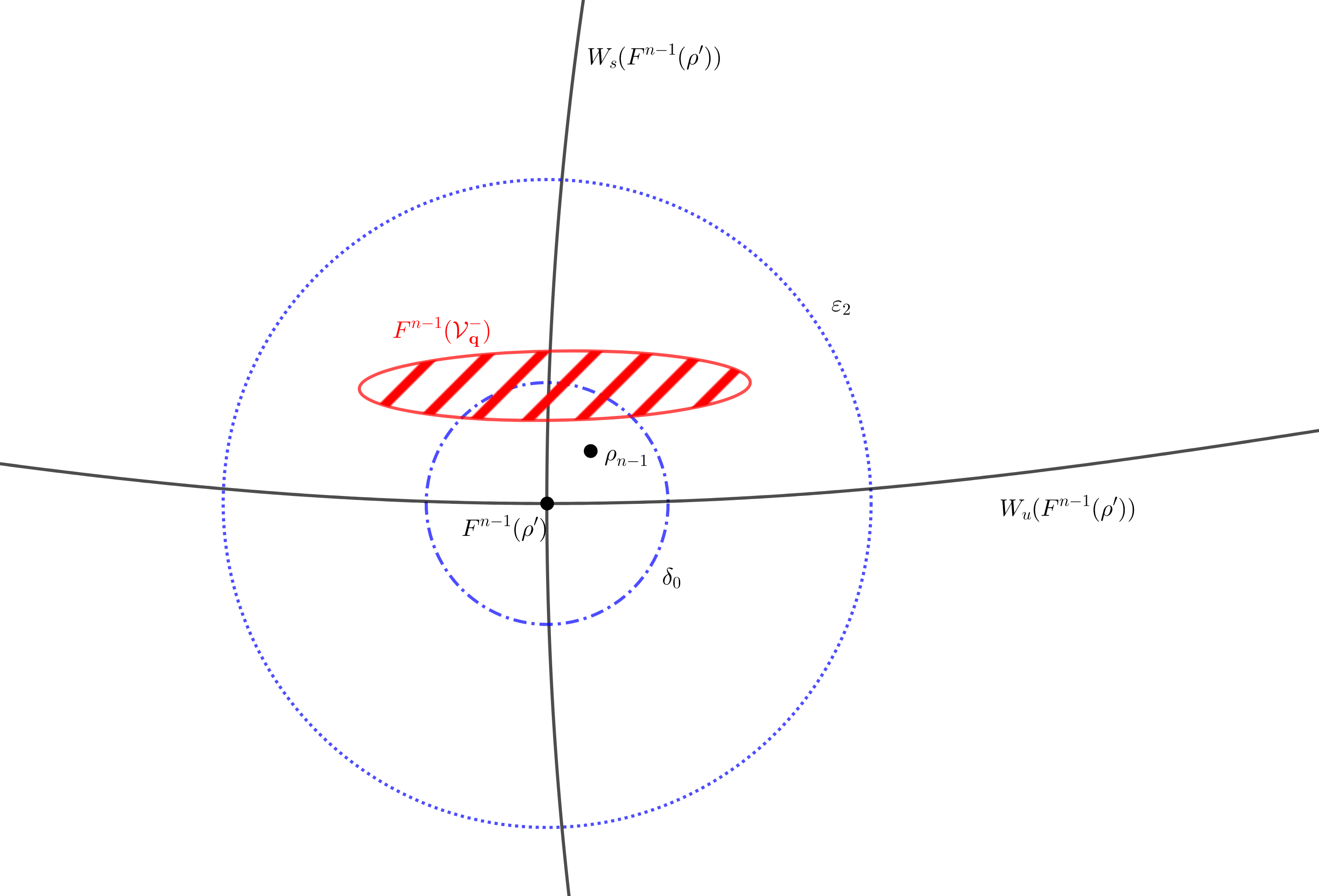}
\begin{center}
\caption{Evolution of the set $\mathcal{V}_\mathbf{q}^-$ (the red hatched set) at time $0$ and $n-1$. The points $\rho_i$, $F^i(\rho^\prime)$ are represented at these times, so as the balls $B(F^i(\rho^\prime), \varepsilon_2)$ and $B(F^i(\rho^\prime), \delta_0)$ (their boundaries are the blue dotted lines). We've also represented the stable (resp. unstable) manifold at $F^i(\rho^\prime)$ to show the directions in which $F$ contracts (resp. expands).}
\end{center}
\end{figure}

\paragraph{Behavior of the local Jacobian.}

 The following three lemmas are crucial to understand the behavior of the evolution of points in the sets $\mathcal{V}_{\mathbf{q}}^\pm$. The first one gives estimates to handle these quantities. 
  
\begin{lem}\label{manipulations_J_q}
Let $n \in \N$ and $\mathbf{q}, \mathbf{p}$ in $\mathcal{A}^n$. If $\varepsilon_0$ is chosen small enough, then the following holds 
\begin{enumerate}[label = \arabic*)]
\item $\mathcal{V}_{\mathbf{q}}^+ \neq  \emptyset \iff \mathcal{V}_{\mathbf{q}}^- \neq  \emptyset$ and in that case $J^-_\mathbf{q} \sim J^+_\mathbf{q}$. 
\item If two propagated neighborhoods intersects, the local Jacobians are comparable : 
\begin{equation}
\mathcal{V}_{\mathbf{q}}^\pm \cap \mathcal{V}_{\mathbf{p}}^\pm \neq \emptyset \implies J_{\mathbf{q}}^\pm \sim J_{\mathbf{p}}^\pm 
\end{equation}
\item If $\mathbf{q}$ can be written as the concatenation of $\mathbf{q_1}$ and $\mathbf{q_2}$ of lengths $n_1$ and $n_2$ such that $n_1 +n_2 = n$ and if $\mathcal{V}_\mathbf{q}^\pm \neq \emptyset$, then 
\begin{equation}\label{product_of_jacobian}
J_{\mathbf{q}}^\pm \sim J_{\mathbf{q_1}}^\pm  J_{\mathbf{q_2}}^\pm 
\end{equation}
\end{enumerate}
\end{lem}

\begin{nota}
The constants in $\sim$ are independent of $\rho$ and $n$. They depend on $F$ but also on the partition $(\mathcal{V}_q)_q$. In the following, we'll be lead to use constants with the same kind of dependence. These constants will be allowed to depend also on the partition of unity $(\chi_q)_q$ and on $M$. Constants with such dependence will be called \textbf{global} constants. 
\end{nota}

\begin{proof}
\begin{enumerate}[label = \arabic*)]
\item The equivalence is obvious. From the fact that $F$ is a volume-preserving canonical transformation, we have for some $C>0$, $$ \forall \rho \in \mathcal{T}, \forall n \in \N, C^{-1} \leq J^u_n(\rho) J^s_n(\rho) \leq C$$
and we write $J^u_n(\rho) \sim  J^s_n(\rho)^{-1}$. From $F^{-n} \circ F^n (\rho) = \rho$, we also get $J_n^s(\rho)^{-1} = J_{-n}^{s}(F^n(\rho) ) $. Eventually, if $\rho^\prime \in \mathcal{T}$ satisfies $d(F^i(\rho), F^i(\rho^\prime) \leq \varepsilon_2$ for $i \in \{0, \dots, n-1 \}$ and $\rho \in \mathcal{V}_\mathbf{q}^-$, $F^n(\rho^\prime) = \rho^+$ satisfies $d(F^i(\rho), F^i(\rho^+)) \leq \varepsilon_2$ for $i \in \{-n, \dots, -1 \}$ and $\rho \in \mathcal{V}_\mathbf{q}^+$. Hence $$J_\mathbf{q}^+ \sim J^s_{-n}(\rho^+ ) \sim J^u_n(\rho^\prime) \sim J_\mathbf{q}^-$$
Thanks to this first point, it is enough to show the remaining point only for $-$. 
\item Pick $\rho_\mathbf{q} \in \mathcal{T}$ (resp. $\rho_\mathbf{p}$) satisying the conclusions of lemma \ref{help_def_jacobian_2} for $\mathcal{V}_{\mathbf{q}}^-$ (resp.  $\mathcal{V}_{\mathbf{p}}^-$).\\ $d(F^i (\rho_\mathbf{q}), F^i(\rho_\mathbf{p})) \leq 2 \varepsilon_2$ and hence, in virtue of Corollary \ref{cor_control_jacobian},  $ J^u_n(\rho_\mathbf{q}) \sim J^u_n(\rho_\mathbf{p})$. This gives 2). 
\item Pick $\rho \in \mathcal{T}$ satisfying the conclusions of lemma \ref{help_def_jacobian_2} for  $\mathcal{V}_{\mathbf{q}}^-$. \\By the chain rule, $J^u_n(\rho) = J^u_{n_2}\left(F^{n_1} (\rho) \right)J^u_{n_1} (\rho)$. Remark that $$\mathcal{V}_{\mathbf{q}}^- = \mathcal{V}_{\mathbf{q_1}}^- \cap F^{-n_1} \left(  \mathcal{V}_{\mathbf{q_2}}^- \right)   $$ Hence, $\rho$ satisfies the conclusions of Lemma \ref{help_def_jacobian_2} for $\mathbf{q}_1$ with $\varepsilon_2$ and the same is true for $F^{n_1}(\rho)$ and $\mathbf{q_2}$. It follows that $J_\mathbf{q_1}^- \sim J^u_{n_1} (\rho) $ and $J_\mathbf{q_2}^- \sim J^u_{n_2} (F^{n_1}(\rho) ) $.
This gives 3). 
\end{enumerate}
\end{proof}

\begin{rem}
The first point of the previous lemma shows that we could consider only one of the two quantities. Nevertheless, we prefer keeping trace of it. The reason is that \textit{a priori} $J^+$ and $J^-$ support two different kind of information : $J^+_\mathbf{q}$ controls the growth of $F^n$ whereas $J^-_\mathbf{q}$ controls the growth of $F^{-n}$. The fact that the two dynamics (in the past and in the future) have similar behaviors is a consequence of the fact that $F$ is volume-preserving. 
\end{rem}

The next lemmas relate the local Jacobian to the expansion rates of the flow in the $\mathcal{V}_\mathbf{q}^\pm$. It will be important in our semiclassical study of operators microlocally supported in $\mathcal{V}_\mathbf{q}^\pm$.

\begin{lem}\textbf{Control of expansion rate by unstable Jacobian}. If $\varepsilon_0$ is small enough, there exists a global constant $C>0$ satisfying the following inequalities. \label{lemma_control_local_expansion_rate}\\
For every $n \in \N^*$ and $\mathbf{q} \in \mathcal{A}^n$ such that $\mathcal{V}_\mathbf{q}^-  \neq \emptyset$ we have : 
\begin{align}
\sup_{ \rho \in \mathcal{V}_\mathbf{q}^-} || d_\rho F^n ||\leq C J_{\mathbf{q}}^- \\
\sup_{ \rho \in \mathcal{V}_\mathbf{q}^+} || d_\rho F^{-n}  || \leq C J_{\mathbf{q}}^+
\end{align}
\end{lem}

\begin{proof}
This is a consequence of (\ref{control_of_jacobian}). Indeed, if $\mathcal{V}_\mathbf{q}^- \neq \emptyset$ and if $\rho^\prime \in \mathcal{T}$ satisfies the conclusions of lemma \ref{help_def_jacobian_2}, for every $\rho \in \mathcal{V}_\mathbf{q}^-$, $||d_\rho F^n  || \leq C J^u_n(\rho)$ with $C$ a global constant depending only on $\varepsilon_2$. 
\end{proof}

This third lemma emphasizes that $\mathcal{V}_\mathbf{q}^-$ lies in a small neighborhood  of a stable manifold and $\mathcal{V}_\mathbf{q}^+$ lies in a small neighborhood  of an unstable manifold, with the size of this neighborhood controlled by the local Jacobian.  It is a direct consequence of Lemma \ref{Local_hpyerbolic_2}. 
\begin{lem} \label{Localization_V_q} \textbf{Localization of the} $\mathcal{V}_\mathbf{q}^\pm$.
There exists a global constant $C>0$ such that for all $n \in \N$ and $\mathbf{q} \in \mathcal{A}^n$, 
\begin{enumerate}[label= (\arabic*)]
\item If $\mathcal{V}_\mathbf{q}^- \neq \emptyset$ and if $\rho^\prime \in \mathcal{T}$ satisfies the conclusion of lemma \ref{help_def_jacobian_2}, then, for all $\rho \in \mathcal{V}_\mathbf{q}^-$, 
\begin{equation}
d\left( \rho, W_s(\rho^\prime) \right) \leq  \frac{C}{J_\mathbf{q}^-} 
\end{equation} 
\item If $\mathcal{V}_\mathbf{q}^+ \neq \emptyset$ and if $\rho^\prime \in \mathcal{T}$ satisfies the conclusion of lemma \ref{help_def_jacobian_2} in the future (namely, $d(F^i(\rho), F^i(\rho^\prime) )\leq \varepsilon_2$ for all $\rho \in \mathcal{V}_\mathbf{q}^+$ and $i \in \{-n, \dots, -1 \}$), then for all $\rho \in \mathcal{V}_\mathbf{q}^+$, 
\begin{equation}
d\left( \rho, W_u(\rho^\prime) \right) \leq  \frac{C}{J_\mathbf{q}^+} 
\end{equation} 
\end{enumerate}
\end{lem}

\subsection{Propagation up to local Ehrenfest time}\label{section_propagation_Ehrenfest}
In this section, we show that under some control of the local Jacobian defined above, one can handle the operators $U_\mathbf{q}$ and prove the existence of symbols $a^\pm_{\mathbf{q}}$ (in exotic classes $S_\delta$) such that 
\begin{align}
U_\mathbf{q} = \op\left( a_\mathbf{q}^+ \right) T^{|\textbf{q}|} + \hinf\\
U_\mathbf{q} = T^{|\textbf{q}|} \op\left( a_\mathbf{q}^- \right) + \hinf
\end{align}
with symbols $a_\mathbf{q}^\pm$ supported in $\mathcal{V}_\mathbf{q}^\pm$. 
We recall that $U_\mathbf{q} = MA_{q_{n-1}}\dots MA_{q_0}$ with $M = T \op(\alpha)$. Let us state the precise statement we will prove.

\begin{prop}\label{prop_symbols_a_q}
Fix $0 < \delta<\delta_1 < \frac{1}{2}$ and $C_0>0$. 
\begin{enumerate}[label = (\arabic*)]
\item For every $n \in \N$ and for all $\mathbf{q} \in \mathcal{A}^n$ satisfying 

\begin{equation} \label{Control_local_jacobian}
J_\mathbf{q}^+ \leq C_0 h^{- \delta}
\end{equation}
there exists $ a_\mathbf{q}^+ \in ||\alpha||_\infty^{n} S_{\delta_1}^{comp}$ such that 
\begin{equation}
U_\mathbf{q} = \op\left( a_\mathbf{q}^+ \right) T^{n} + \hinf_{L^2 \to L^2}
\end{equation}
and \begin{equation}
\supp  a_\mathbf{q}^+ \subset \mathcal{V}_\mathbf{q}^+
\end{equation}

\item For every $n \in \N$ and for all $\mathbf{q} \in \mathcal{A}^n$ satisfying 

\begin{equation} \label{Control_local_jacobian_-}
J_\mathbf{q}^- \leq C_0 h^{- \delta}
\end{equation}
there exists $ a_\mathbf{q}^- \in ||\alpha||_\infty^{n}S_{\delta_1}^{comp}$ such that 
\begin{equation}
U_\mathbf{q} = T^{n} \op\left( a_\mathbf{q}^- \right) + \hinf_{L^2 \to L^2}
\end{equation}
 \begin{equation}
\supp  a_\mathbf{q}^- \subset \mathcal{V}_\mathbf{q}^-
\end{equation}
\end{enumerate}
\end{prop}

\begin{rem}\text{}
\begin{itemize}[nosep]
\item The implied constants appearing in the $\hinf$ are quasi-global : they have the same dependence as global constants but depend also on $C_0, \delta, \delta_1$.   What is important is that they are independent of $n$ and $\mathbf{q}$ as soon as the assumption (\ref{Control_local_jacobian}) is satisfied. 
\item (\ref{Control_local_jacobian}) implies that $ \mathcal{V}_\mathbf{q}^+ \neq \emptyset$. In particular, if $\mathbf{q}$ satisfies this assumption, there exists a sequence $(i_0, \dots, i_{n})$ such that for all $p \in \{ 0 , \dots, n-1\}, 
\mathcal{V}_{q_p} \subset \widetilde{D}_{i_{p+1}, i_p} \subset U_{i_p}$
\item In fact, $\supp  a_\mathbf{q}^+ \subset F \left(  \mathcal{V}_{q_{n-1}} \right) \subset U_{i_n}$. Hence, the operator $\op\left( a_\mathbf{q}^+\right)$ acting on $\bigoplus_{i=1}^J L^2 (\R)$ is the diagonal matrix $\Diag( 0, \dots,\op\left( a_\mathbf{q}^+ \right) , \dots, 0) $. 
\item The symbol $a_\mathbf{q}^+$ has an asymptotic expansion in power of $h$. The principal symbol is given by 

\begin{equation}\label{principal_symbol_aq+}
\left( a_\mathbf{q}^+ \right)_0 = \prod_{p=1}^{n} a_{q_{n-p}} \circ F^{-p} 
\end{equation} 
where $a_q = \chi_q \times \alpha$. 
Note that if the functions $a_{q_{n-p}} \circ F^{-p} $ are not necessarily well defined, the product is well defined thanks to the assumptions on the supports of $\chi_q$, namely $\supp \chi_q \Subset \mathcal{V}_q$. Indeed, such a symbol can be constructed inductively as the $n$-th term $b_n$ of the sequence of functions $b_1 = a_{q_0} \circ F^{-1} $ and $b_{i+1}$ is obtained from $a_i$ by the following 

$$ b_{i+1} = \left( a_{q_i} \times a_i \right) \circ F^{-1} $$
If we assume that $\supp b_{i} \Subset \mathcal{V}_{ q_0\dots q_{i-1}}^+$, then $\supp ( a_{q_i} \times b_{i}) \Subset F^{-1} \left(\mathcal{V}_{ q_0 \dots q_{i}}^+ \right) $. This property allows us to define $b_{i+1}$ and $\supp b_{i+1} \Subset \mathcal{V}_{ q_0 \dots q_{i}}^+ $.
\item The same hols for $a_{\mathbf{q}}^-$ with principal symbol 
\begin{equation}\label{principal_symbol_aq-}
\left( a_\mathbf{q}^- \right)_0 = \prod_{p=0}^{n-1} a_{q_p} \circ F^p
\end{equation}
\item Our proof follows the sketch of proof of \cite{NDJ19} (Section 5) and \cite{Ri10} (Section 7). 
\end{itemize}
\end{rem}

In the end of this section, we focus on proving this proposition. We only prove the first point. The second point can be proved similarly by using the same techniques. 

\subsubsection{Iterative construction of the symbols}
 Let us start by a lemma combining the precise versions of the expansion of the Moyal product (Lemma \ref{Moyal_produc_op}) and of Egorov theorem (Proposition \ref{prop_Egorov}). This lemma is the key ingredient for the iterative formulas below.
 
 \begin{lem}
Let $q \in \mathcal{A}$ and let $a \in S_{\delta_1}^{comp}$ such that $ \supp a \Subset U_j$ for some $j \in \{1, \dots, J \}$. 
 Then, there exists a family of differential operators $L_{k,q}$ of order $2k$, with smooth coefficients compactly supported in $\mathcal{V}_q$, such that for every $N \in \N$, we have the following expansion 
\begin{equation} \label{Egorov+moyal}
M A_q \op(a) = \op \left( \sum_{k=0}^{N-1} h^k( L_{k,q} a ) \circ F^{-1}\right) T + O \left(||a||_{C^{2N + 15}} h^N \right)_{L^2 \to L^2}
\end{equation}
Moreover, one has $L_{0,q} = \chi_q \times \alpha\coloneqq a_q $.
 \end{lem}
 
 \begin{rem}\text{}
 \begin{itemize}
 \item  Again, since $\supp a\subset U_j$,  $\op(a)$ is a diagonal matrix with only one non-zero block equal to $\op(a)$. 
 \item Recall that we've supposed that $\mathcal{V}_q \subset \widetilde{D}_{m_q j_q}$. As a consequence, the symbols  $$a^{(k)}_1 \coloneqq L_{k,q} a  \circ F^{-1}$$ are equal to $L_{k,q} a  \circ \left(F_{m_q j_q}\right)^{-1}$ and are supported in $U_{m_q}$ : $\op(a_1^{(k)})$ is still a diagonal matrix. 
 \end{itemize}
 \end{rem} 
 
 \begin{proof}
 Let us first work at the order of operators $L^2(\R) \to L^2(\R)$ and let us study : 
$$M_{m_q j_q} \op(\chi_q) \op(a) = T_{m_q j_q} \op(\alpha_{j_q})  \op(\chi_q) \op(a)$$ 
Using Lemma \ref{Moyal_produc_op}, we write

$$ \op(\chi_q) \op(a)  = \op \left( \sum_{k=0}^{N-1} \frac{i^k h^k}{k!} A(D)^{k} (\chi_q \otimes a)|_{\rho= \rho_1= \rho_2}  \right)  +O\left( h^N ||\chi_q \otimes a ||_{C^{2N+13}}\right)$$ 
the principal term of the expansion being $ \chi_q a$. Set $a_{q,k}(\rho) = A(D)^{k} (\chi_q \otimes a)|_{\rho= \rho_1=\rho_2}$ and use Lemma \ref{Moyal_produc_op} to write 
\begin{align*}
\op(\alpha_{j_q}) \op(\chi_q) \op(a) &= 
&= \sum_{k_1+k_2 <N} \frac{i^{k_1+k_2} h^{k_1+k_2}}{k_1!k_2!} \op \left( A(D)^{k_2} (\alpha_{j_q} \otimes  a_{q,k_1})|_{\rho= \rho_1= \rho_2} \right) +O\left( h^N || a ||_{C^{2N+13}}\right)
\end{align*}
The principal term in the expansion is $\alpha_{j_q} \chi_q a$. 
We note that $$a \mapsto \sum_{k_1+k_2 =k} A(D)^{k_2} (\alpha_{j_q} \otimes  a_{q,k_1})|_{\rho= \rho_1= \rho_2}$$ is a differential operator of order $2k$. 
Using the precise version of Egorov theorem in Lemma \ref{lem_precise_Egorov}, we see that for any $b$ with $\supp(b) \subset \mathcal{V}_q$, 
$$T_{m_q j_q} \op(b) = \op \left( b \circ (F_{m_q j_q})^{-1} + \sum_{k=1}^{N-1} h^k(D_k b) \circ (F_{m_q j_q})^{-1}\right) + O\left(h^N ||b||_{C^{2N+15}}\right)$$
where $D_k$ are differential of order $2k$ compactly supported in $\mathcal{V}_q$. 
Applying this to the previous expansion, we see that we can write : 
$$ T_{m_q j_q} \op(\alpha_{j_q}) \op(\chi_q) \op(a) = \op\left((\alpha_{j_q} \chi_q a) \circ F^{-1} + \sum_{k=1}^{N-1} k^k(L_{k,q} a) \circ F^{-1} \right) + O\left(h^N ||a||_{C^{2N+15}}\right)$$

We now come to the entire matrix operator. Note that the matrix $M \op(\chi_q) \op(a)$ is of the form 
$$ \left( \begin{matrix}
0& \dots & M_{1j_q} \op(\chi_q) & \dots& 0 \\
\vdots &\vdots & \vdots & \vdots &\vdots \\
0& \dots & M_{Jj_q} \op(\chi_q) & \dots& 0 \\
\end{matrix} \right) \op(a) $$
Recall that $\WF(\op(\chi_q)) \subset \widetilde{D}_{m_q j_q}$ and $\WF^\prime(M_{m_q j_q})  \subset \Gr^\prime( F_{m_q j_q})$. Hence, for $m \neq m_q$, $M_{m j_q}\op(\chi_q) = \hinf$ and the previous matrix can be written 
$$ \left( \begin{matrix}
0& \dots &  \hinf & \dots& 0 \\
\vdots &\vdots & \vdots & \vdots &\vdots \\
0& \dots & M_{m_q j_q} \op(\chi_q)  & \dots& 0  \\ 
\vdots &\vdots & \vdots & \vdots &\vdots \\
0& \dots & \hinf  & \dots& 0 \\
\end{matrix} \right)\op(a)   =
 \left( \begin{matrix}
0& \dots &  0 & \dots& 0 \\
\vdots &\vdots & \vdots & \vdots &\vdots \\
0& \dots & M_{m_qj_q} \op(\chi_q) \op(a) & \dots& 0  \\ 
\vdots &\vdots & \vdots & \vdots &\vdots \\
0& \dots & 0  & \dots& 0 \\
\end{matrix} \right) + \hinf ||\op(a)||_{L^2}$$
With constant in $\hinf$ depending on $\chi_q, M$ and $||\op(a)||_{L^2 \to L^2} = O(||a||_{C^8})$. 
Let's note $$a_1^{(k)} = L_{k,q} a \circ F^{-1}$$ 
and observe that $\supp(a_1^{(k)}) \subset F(\supp \chi_q) \Subset \widetilde{A}_{m_q j_q}$. Consider a cut-off function $\tilde{\chi}_q$ such that $\tilde{\chi}_q \equiv 1$ in a neighborhood of $F(\supp \chi_q)$ and $\supp \tilde{\chi}_q \subset \widetilde{A}_{m_q j_q}$. Using Lemma \ref{Moyal_produc_op} and the support properties of $\tilde{\chi}_q$, one has 
$$ \op(a_1^{(k)})= \op(a_1^{(k)}) \op(\tilde{\chi}_q)+O\left( h^{N-k} ||a_1^{(k)}||_{C^{2(N-k) +13}}\right) = \op(a_1^{(k)}) \op(\tilde{\chi}_q)+O\left(h^{N-k} ||a||_{C^{2N +13}} \right)$$
Then, one can write $\op(a_1^{(k)}) T$ on the form 
 $$\left( \begin{matrix}
0& \dots & 0 \\
\vdots &\vdots & \vdots  \\
\op(a_1^{(k)}) \op(\tilde{\chi}_q)T_{m_q1} &\dots & \op(a_1^{(k)})\op(\tilde{\chi}_q) T_{m_qJ}&  \\ 
\vdots &\vdots &\vdots \\
0& \dots & 0 \\
\end{matrix} \right)+O\left(h^{N-k} ||a||_{C^{2N +13}} \right) $$
and for $j \neq j_q, \op(\tilde{\chi}_q) T_{m_q j} = \hinf$. We can conclude that 
\begin{align*}
\op(a_1^{(k)}) T&= \left( \begin{matrix}
0& \dots & \dots & \dots & 0 \\
\vdots &\dots & \dots & \dots & \vdots  \\
0 &\dots &  \op(a_1^{(k)}) \op(\tilde{\chi}_q) T_{m_q j_q}&  \dots& 0   \\ 
\vdots &\dots & \dots & \dots &\vdots \\
0& \dots & \dots & \dots& 0 \\
\end{matrix} \right) + \hinf ||\op(a_1^{(k)}) ||_{L^2 \to L^2} + O\left(h^{N-k} ||a||_{C^{2N +13}} \right) \\ &=  \left( \begin{matrix}
0& \dots & \dots & \dots & 0 \\
\vdots &\dots & \dots & \dots & \vdots  \\
0 &\dots &  \op(a_1^{(k)})  T_{m_q j_q}&  \dots& 0   \\ 
\vdots &\dots & \dots & \dots &\vdots \\
0& \dots & \dots & \dots& 0 \\
\end{matrix} \right) + O\left(h^{N-k} ||a||_{C^{2N +13}} \right) 
\end{align*}
 Combining this with the version obtained with $M_{m_q j_q}$, we get (\ref{Egorov+moyal}). 
 \end{proof}

 Let us now start the iterative construction of the symbols. Fix $N \in \N$ which can be taken arbitrarily large. Recall that we want to write  
 \begin{equation}
U_\mathbf{q} = \op\left( a_\mathbf{q}^+ \right) T^{|\mathbf{q}|} + \hinf_{L^2 \to L^2}
\end{equation}
Note $U_r = U_{  q_0 \dots q_{r-1}}$. We want to write 
\begin{equation}
U_r = \op \left( \sum_{k=0}^{N-1} h^k a_r^{(k)} \right) T^r + R_r^{(N)}
\end{equation}
We start by writing 
\begin{equation}
U_1= \op \left( \sum_{k=0}^{N-1} h^k a_1^{(k)} \right) T + R_1^{(N)}
\end{equation}
which is possible in virtue of (\ref{Egorov+moyal}).
To pass form $U_r$ to $U_{r+1}$, we have the relation 
$$ U_{r+1} = M A_{q_r} U_r = \sum_{k=0}^{N-1} h^k M A_{q_r} \op \left( a_r^{(k)} \right) T^r +M A_{q_r} R_r^{(N)}$$
So, we will construct inductively our symbols by setting 
\begin{equation}\label{iteration_formula_1}
a_{r+1}^{(k)} = \sum_{p=0}^k \left( L_{p, q_r} a_r^{(k-p)} \right) \circ \left(F_{i_{r+1}, i_r}\right)^{-1} 
\end{equation}
and 

\begin{equation}
R_{r+1}^{(N)} = M A_{q_r} R_r^{(N)} + \sum_{k=0}^{N-1}  O \left( ||a^{(k)}_r ||_{C^{2(N-k) +15}} \right)
\end{equation}
The $O$ encompasses the remainder terms in Lemma \ref{Egorov+moyal}. The constants in the $O$ only depend on $M$ and the $\chi_q, q \in \mathcal{A}$, but not on $\mathbf{q}$.  

To make this construction work, we will have to prove that the symbols $a_r^{(k)}$ lie in a good symbol class $S_{\delta_1}^{comp}$. 

Before reaching this step, let us just note that by induction one sees that : \begin{itemize}[nosep]
\item \begin{equation}\label{control_remainder}
||R_r^{(N)} || \leq C_N h^N \left( 1 + \sum_{k=0}^{N-1} \sum_{l=0}^{r-1} || a_l^{(k)}||_{C^{2(N-k)+15}} \right) 
\end{equation}
with $C_N$ depending on $N$, $M$ and the $a_q$, but neither on $r$ nor $\mathbf{q}$. 
\item Since $L_{p,q_r}$ has coefficient supported in $\mathcal{V}_{q_r}$,  we see by induction  that $\supp a_{r+1}^{(k)} \subset \mathcal{V}^+_{ q_0 \dots q_r} $ as announced. 
\item $a_{r+1}^{(0)} = \prod_{p=1}^{r+1} a_{q_{r+1-p}} \circ F^{-p}$ 
\end{itemize}

\subsubsection{Control of the symbols}
We aim at estimating the semi-norms $||a^{(k)}_r||_{C^m}$ for $ k < N$, $1 \leq r \leq n$ and $m \in \N$. 
We will show the following : 
\begin{prop}\label{control_of_the_symbols}
For every $r \in \{ 1 , \dots, n \}$, $k \in \{ 0, \dots, N-1 \}$ and $m \in \N$, there exists $C(k,m)$, such that with $\Gamma_{k,m}=(k+1)(m+k+1)$,
\begin{equation}
||a^{(k)}_r ||_{C^m} \leq C(k,m) r^{\Gamma_{k,m}} \left(J_{q_0 \dots q_{r-1} }^+ \right)^{2k +m} ||\alpha||_\infty^r
\end{equation}
\end{prop}

\begin{rem}\text{}
\begin{itemize}[nosep]
\item What is important in this result is the way in which the bound depends on $r$ and $\mathbf{q}$.  Up to the term $r^{\Gamma_{k,m}}$, which  is supposed to behave like $O\left( | \log h|^{\Gamma_{k,m}}\right)$, the significant part of the estimate is that we can control the symbols by the local Jacobian. 
\item Since $\supp a^{(k)}_r \subset  \mathcal{V}^+_{ q_0 \dots q_{r-1}}$, we need to focus on points $\rho \in \mathcal{V}^+_{ q_0 \dots q_{r-1}}$. 
\item Our method is very close to the ones developed in \cite{Ri10} and \cite{NDJ19}. However, we've changed a few things at the cost of being less precise on the exponent $\Gamma_{k,m}$. Our aim was to treat our problem as if we wanted to control the product of $r$ triangular matrices. 
\end{itemize}
\end{rem}

Let us pick $\rho \in \mathcal{V}^+_{ q_0 \dots q_{r-1}}$. With (\ref{iteration_formula_1}), one sees that if $k, m \in \N$, $d^m a_{r+1}^{(k)}$ depends on $d^{m^\prime} a_{r}^{ (k^\prime)} ( F^{-1} (\rho))$ for several $m^\prime, k^\prime$. Before going deeper in the analysis of this dependence, let us note two obvious facts : 
\begin{itemize}[nosep]
\item This dependence is linear, with coefficients smoothly depending on $\rho$.
\item If $d^m a_{r+1}^{(k)}$ depends effectively on $d^{m^\prime} a_{r}^{ (k^\prime)} ( F^{-1} (\rho))$, then $k^\prime \leq k$ and $2k^\prime + m^\prime \leq 2k + m$. 
\end{itemize}

\paragraph{Precise analysis of the dependence.}
That being said, let us pick $m_0 , k_0 \in \N$. Set $N_0 = 2k_0 + m_0$ and consider the (column) vector 

\begin{equation}
A_r (\rho) \coloneqq \left( d^m a_r^{(k)} (\rho)\right)_{ k \leq k_0 , 2k + m \leq N_0} \in \bigoplus_{k \leq k_0, 2k +m \leq N_0} S^m T^*_\rho U 
\end{equation}
Here $S^m T^*_\rho U$ is the spaces of $m$-linear symmetric form on $T_\rho U$. To define a norm on the fibers $S^m T^*_\rho U$, we can use for $f \in S^m T^*_\rho U$,
\begin{equation}\label{def_norm_m}
 ||f||_{m,\rho} = \sup_{ v_1,\dots, v_m \in T_\rho U} \frac{f(v_1, \dots, v_m)}{||v_1||_\rho \dots ||v_m||_\rho}
\end{equation}
where $||v||_\rho$ for $v \in T_\rho U$ is the norm induced by the Riemannian metric used to define $J_1^u$ in \ref{Def_jacobian}. Note that for any fixed neighborhood of $\mathcal{T}$, there exists a global constant $C>0$ such that for each $a \in \cinfc(U)$ supported in this neighborhood, one has 
$$ C^{-1} ||a||_{C^m} \leq \sup_{m^\prime \leq m }\sup_{\rho \in U} ||d^{m^\prime} a||_{m^\prime,\rho} \leq  C ||a||_{C^m}$$
\begin{figure}
\centering
\includegraphics[scale=0.4]{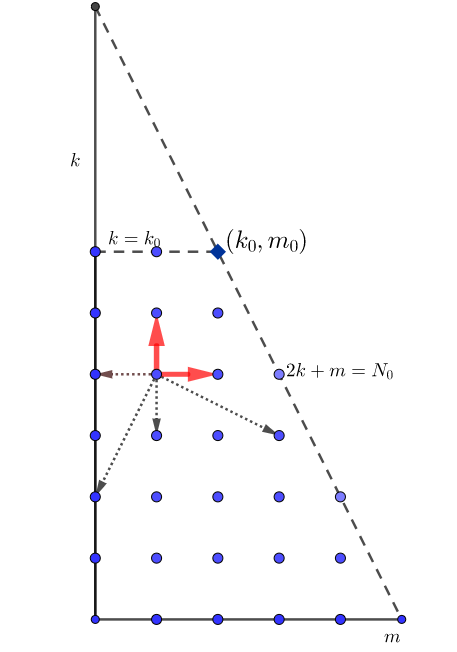}
\caption{The starting point $(k_0,m_0)$ is represented by a diamond. The set $\mathcal{I}$ corresponds to the couple $(k,m) \in \N^2$ in the region under the dotted lines $k=k_0$ and $2k+m=N_0$. We've represented a family of arrows starting from a point $\gamma_1 \in \mathcal{I}$. The dotted arrows points toward $\beta$ such that $\gamma_2 \prec \gamma_1$. The big red arrows points toward points $\gamma_2$ such that $P_{\gamma_1 \gamma_2}^{(r)}=0$. }
\label{descent}
\end{figure}
We will denote by $\gamma_1, \gamma_2$, etc. elements of $\mathcal{I} \coloneqq \mathcal{I}(k_0,m_0) = \{ (k, m) \in \N^2 ,k \leq k_0, 2k+m \leq N_0 \}$. We equip $\mathcal{I}$ with the lexicographic order $\prec$ and note $\# \mathcal{I}\coloneqq \Gamma_{k_0,m_0}$ (see Figure \ref{descent}). We order the indices of $A_r(\rho)$ with $\prec$. 
$A_r(\rho)$ depends linearly on $A_{r-1} ( F^{-1} (\rho))$ and this dependence can be made explicit by a matrix 
$$P^{(r)} (\rho)  = \left( P^{(r)}_{ \gamma_1 \gamma_2 } (\rho)\right)_{ \gamma_1, \gamma_2 \in \mathcal{I}}, \text{ where } P^{(r)}_{\gamma_1 \gamma_2} (\rho)  \in L \Big( S^{m^\prime} T^*_{F^{-1}(\rho)} U, S^m T^*_\rho U\Big)\} \text{ if } \gamma_1 =(k,m), \gamma_2=(k^\prime, m^\prime) $$
so that \begin{equation}\label{def_matrix_P}
 A_r(\rho) = P^{(r)}(\rho) \; A_{r-1} \left( F^{-1} (\rho) \right)
\end{equation}

\begin{nota}
If $\gamma_1=(k,m),\gamma_2 = (m^\prime, k^\prime), \rho , \rho^\prime \in U$ and if $A : S^{m^\prime}T^*_{\rho^\prime} U\to S^{m}T^*_{\rho}U$ is a linear operator, we will note $$||\cdot||_{\gamma_1,\rho,\gamma_2,\rho^\prime} $$ its subordinate norm for the norms defined by (\ref{def_norm_m}). 
\end{nota}

Analyzing (\ref{iteration_formula_1}), it turns out that if $\gamma_1 = (k,m), \gamma_2=(k^\prime,m^\prime) \in \mathcal{I}$, then 
\begin{itemize}[nosep]
\item if $ k^\prime > k$, $P_{\gamma_1 \gamma_2}^{(r)} (\rho) =0$ ; 
\item if $k=k^\prime$, the contribution to $d^m a^{(k)}_r (\rho)$ of $a_{r-1}^{(k)}$ comes from 
\begin{align*}
 &d^m \left( (a_{q_{r-1}} a^{(k)}_{r-1}) \circ F^{-1} \right)(\rho)  \\
&=a_{q_{r-1}} \left(F^{-1}(\rho)\right)  \times d^m \left( a^{(k)}_{r-1} \circ F^{-1} \right)(\rho)+ (\textit{derivatives of order stricly less than m for } a_{r-1}^{(k)}  )\\
&= a_{q_{r-1}} \left( F^{-1} (\rho) \right) \times \left({}^t dF^{-1} (\rho)  \right)^{ \otimes m} d^m a^{(k)}_{r-1} \left(F^{-1} (\rho) \right) + (\textit{idem}  )
\end{align*} 
In particular, if $\gamma_1=(k,m) \prec \gamma_2 = (k,m^\prime)$ doesn't hold, we see that $P_{\gamma_1 \gamma_2}^{(r)}(\rho)=0$. 
\item If $k^\prime < k$, we can have $P_{\gamma_1 \gamma_2}^{(r)}(\rho)\neq 0$ with $m^\prime > m$. But, the use of the lexicographic order ensures that $\gamma_1 \prec \gamma_2$ in that case. 
\end{itemize}

Hence, $P^{(r)} (\rho)$ is a lower triangular matrix and the diagonal coefficients for the index $\gamma_1 =(k,m)$ are given by 
\begin{equation}
P_{\gamma_1 \gamma_1}^{(r)} (\rho) : f \in S^{m}T^*_{F^{-1} (\rho)}U \mapsto a_{q_{r-1}} \left( F^{-1} (\rho) \right) \times \left({}^t dF^{-1} (\rho)  \right)^{ \otimes m}  f \in S^{m}T^*_{\rho} U
\end{equation}
Iterating (\ref{def_matrix_P}), we have 
$$ A_r(\rho) =  P^{(r)}(\rho)  P^{(r-1)} \left(F^{-1} (\rho) \right) \dots P^{(2)} \left(F^{-(r-2)}(\rho) \right) A_1 \left( F^{1-r} (\rho) \right)$$
For $\gamma \in \mathcal{I}$, we note $$\mathcal{E}_r( \gamma)  = \{ \overrightarrow{\gamma} = (\gamma_1, \dots, \gamma_r) \in \mathcal{I}^r ; \quad \gamma_r= \gamma, \gamma_{i} \prec \gamma_{i+1} \} $$
\begin{figure}[!h]
\centering
\includegraphics[scale=0.4]{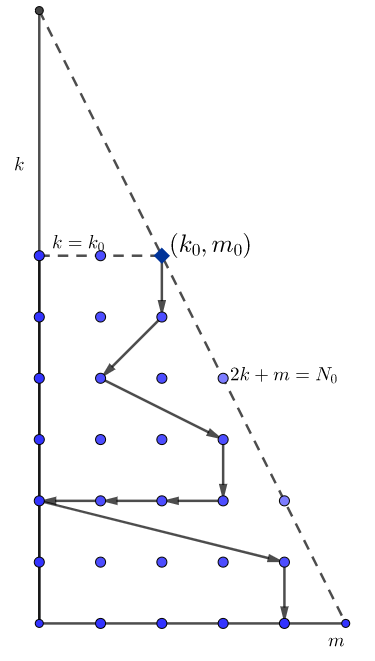}
\caption{We've represented the reduction of an element  $\protect \overrightarrow{\gamma} \in \mathcal{E}_r (k_0,m_0)$, i.e, the arrows between $\gamma_i $ and $ \gamma_{i+1}$ when $\gamma_i \neq \gamma_{i+1}$. During the descent, the value of $m$ can only increase when $k$ decreases strictly.}
\end{figure}
The triangular property of $P$ allows us to write : 
$$\left(A_r(\rho)\right)_\gamma = \sum_{ \overrightarrow{\gamma}\in \mathcal{E}_r(\gamma)} P^{(r)}_{ \gamma_r \gamma_{r-1} }(\rho)  \dots P^{(2)}_{\gamma_2 \gamma_1} \left(F^{-(r-2)}(\rho) \right) \left( A_1 \left( F^{1-r} (\rho) \right)\right)_{\gamma_1}$$

\paragraph{Control of individual terms.}
Let us fix $\gamma=(k,m)$ and pick $\overrightarrow{\gamma} \in \mathcal{E}_r( \gamma)$. We wish to analyze the operator  $$
P_{\overrightarrow{\gamma}}(\rho) \coloneqq 
P^{(r)}_{ \gamma_r \gamma_{r-1} }(\rho)  \dots P^{(2)}_{\gamma_2 \gamma_1} \left(F^{-(r-2)}(\rho) \right) $$
First of all, $\# \{ i \in \{ 1 , \dots, r-1 \} , \gamma_{i+1} \neq \gamma_i \} \leq \Gamma_{k_0,m_0}$.
So let us write $$ \{ i \in \{ 1 , \dots, r-1 \} , \gamma_{i+1} \neq \gamma_i \}  = \{ t_1 < \dots < t_d \}$$ with $d \leq \Gamma_{k_0,m_0}$. We can set $t_{d+1} = r, t_0 =0$ and we can rewrite 
$$ \overrightarrow{\gamma} = ( \underbrace{\beta_1, \dots, \beta_1}_{t_1}, \underbrace{\beta_2, \dots, \beta_2}_{t_2-t_1}, \dots, \underbrace{\beta_d, \dots, \beta_d}_{t_d - t_{d-1}} , \underbrace{\beta_{d+1}, \dots, \beta_{d+1}}_{t_{d+1}- t_d} ) $$
For $ p \in \{ 1, \dots, d+1 \}$, we introduce the operator 
\begin{align*}
D_p(\rho) &= P^{(t_p)}_{\beta_p \beta_p} \left( F^{- ( r-t_p) }(\rho) \right)  \dots P^{(t_{p-1} +2 )}_{\beta_p \beta_p} \left( F^{- ( r-t_{p-1} -2) } (\rho) \right) 
\end{align*}
and for $p \in \{1, \dots, d \} $
\begin{equation*}
T_p (\rho) = P_{ \beta_{p+1} \beta_p}^{t_p+1} \left( F^{- (r- t_p - 1)} (\rho) \right) 
\end{equation*}
so that we can write 

\begin{equation*}
P_{\overrightarrow{\gamma}} (\rho)= D_{d+1}(\rho) T_d(\rho) D_d (\rho)\dots T_1(\rho) D_1 (\rho)
\end{equation*}
For $p \in \{1 , \dots, d+1\}$, if $\beta_p = (k,m)$, we can see that 
\begin{align*}
D_p(\rho) &= \left[\prod_{j=t_{p-1} +1}^{t_p -1} a_{q_j} \circ F^{-(r-j)} (\rho)\right] \left[ \left( {}^t dF^{-1} \left(F^{-(r-t_p)} (\rho)  \right) \right)^{\otimes m} \circ \dots \circ \left( {}^t dF^{-1} \left(F^{-(r-t_{p-1}-2)} (\rho) \right) \right)^{\otimes m} \right] \\
&=\left[\prod_{j=t_{p-1} +1}^{t_p -1} a_{q_j} \circ F^{-(r-j)} (\rho)\right] \left( {}^t dF^{-(t_p -t_{p-1} -1)} \left(F^{-(r-t_p)} (\rho)  \right) \right)^{\otimes m}
\end{align*}
We introduce the word $$\mathbf{q_p} = q_{t_{p-1}} \dots q_{t_p -1}$$ and set $\rho_p = F^{-(r-t_p)} (\rho), \rho_p^\prime= F^{-(t_p-t_{p-1}-1)}(\rho_p) $.
To estimate the subordinate norm of $D_p(\rho)$, we use Lemma \ref{lemma_control_local_expansion_rate}. Since $ \rho \in \mathcal{V}_\mathbf{q}^+, \rho_p \in \mathcal{V}_\mathbf{q_p}^+$ and we have

\begin{align*}
 ||D_p(\rho) ||_{\beta_p, \rho_p, \beta_p,\rho_p^\prime} &\leq \left|\prod_{j=t_{p-1} +1}^{t_p -1} a_{q_j} \circ F^{-(r-j)} (\rho)\right| \sup_{\rho_p \in \mathcal{V}_\mathbf{q_p}^+ } || dF^{-(t_p-t_{p-1}-1)}(\rho_p) ||^m\\
 & \leq \left(CJ_\mathbf{q_p}^+\right)^m  \left|\prod_{j=t_{p-1} +1}^{t_p -1} a_{q_j} \circ F^{-(r-j)} (\rho)\right| \\
 &\leq C_{k_0,m_0} \left(J_\mathbf{q_p}^+ \right)^{N_0}  \left|\prod_{j=t_{p-1} +1}^{t_p -1} a_{q_j} \circ F^{-(r-j)} (\rho)\right|
\end{align*}

To estimate the norms of $T_p(\rho)$, we simply note that they depend smoothly on $\rho_p$ which lies in a compact set, so we can bound them by a uniform constant $C_1$. This is not a problem since they appear $d$ times in $P_{\overrightarrow{\gamma}}$ with $d \leq \Gamma_{k_0,m_0}$.
Consequently, we can estimate $||P_{\overrightarrow{\gamma}}(\rho)||_{\gamma,\rho, \gamma_1, F^{-(r-1)}(\rho)}$, 
\begin{equation}
||P_{\overrightarrow{\gamma}}(\rho) ||_{\gamma,\rho, \gamma_1, F^{-(r-1)}(\rho)} \leq C_{k_0,m_0} \left( J^+_\mathbf{q_1} \dots J^+_\mathbf{q_{d+1}} \right)^{N_0} | a_{\mathbf{q},\overrightarrow{\gamma}}(\rho) | \leq C_{k_0,m_0} \left(J^+_\mathbf{q} \right)^{N_0} | a_{\mathbf{q},\overrightarrow{\gamma}}(\rho) | 
\end{equation}
where 
\begin{equation}
a_{\mathbf{q},\overrightarrow{\gamma}}= \prod_{p=1}^{d+1} \prod_{j=t_{p-1} +1}^{t_p -1} a_{q_j} \circ F^{-(r-j)} 
\end{equation}
Here, the last inequality holds by applying $d$ times  (\ref{product_of_jacobian}), with $d \leq \Gamma_{k_0,m_0}$, once we've noted that $$\mathbf{q} = \mathbf{q}_1 \dots \mathbf{q_{d+1}}$$
Finally, if $\gamma_1 = (k_1,m_1)$, to estimate $ ||\left( A_1 \left( F^{1-r} (\rho) \right)\right)_{\gamma_1}||_{m_1, F^{1-r}(\rho)} $, we simply note that it depends smoothly on  $F^{1-r}(\rho)$, so that we can bound it by a uniform constant. Hence, we have 
\begin{equation}\label{control_of_P_gamma}
||P_{\overrightarrow{\gamma}} (\rho) A_1 \left( F^{1-r} (\rho)\right)||_{m,\rho}  \leq C_{k_0,m_0} \left(J^+_\mathbf{q} \right)^{N_0} | a_{\mathbf{q},\overrightarrow{\gamma}}(\rho) | 
\end{equation}

\paragraph{Cardinality of $\mathcal{E}_r(\gamma)$.}
The bound we will provide is far from being optimal but it will turn out to be enough for our purpose.
To count the number of elements in $\mathcal{E}_r(\gamma)$, we remark that it is similar than counting the number of decreasing sequences of length $r$ starting from $\gamma$. This number is smaller than the number of increasing sequences of length $r$ in $\{ 1, \dots, \Gamma_{k_0,m_0} \}$ . 
Recalling that the number of sequences $u_1 \leq u_2 \leq \dots \leq u_r$ satisfying $u_1= 1$ and $u_r=b$ is equal to ${  b+r-2 \choose r-2}$, one can estimate 
\begin{equation}\label{cardinality of E_r_alpha}
 \# \mathcal{E}_r(\gamma) \leq \sum_{b=1}^{\Gamma_{k_0,m_0}} { b+r-2 \choose r-2 } \leq \Gamma_{k_0,m_0} (r-1)^{\Gamma_{k_0,m_0}}
\end{equation}
Finally, we can compute explicitly $\Gamma_{k_0,m_0}$ and we find $\Gamma_{k_0,m_0} = (k_0+1) (m_0 + 1 + k_0)$. 

\paragraph{Conclusion.} We finally combine (\ref{cardinality of E_r_alpha}) and (\ref{control_of_P_gamma}) to prove Proposition \ref{control_of_the_symbols}. Recall that $|a_q| = |\alpha| \chi_q \leq ||\alpha||_\infty$. 

\begin{align*}
 \sup_{\rho \in  \mathcal{V}_{q_0 \dots q_{r-1} } } || d^{m_0} a_r^{(k_0)}||_{m_0,\rho} &= \sup_{ \rho \in \mathcal{V}_{q_0 \dots q_{r-1} }} ||\left( A_r(\rho) \right)_{(k_0,m_0)}||_{m_0,\rho} \\
&\leq \sum_{ \overrightarrow{\gamma} \in \mathcal{E}_r(k_0,m_0)} ||P_{\overrightarrow{\gamma}} (\rho) A_1 \left( F^{1-r} (\rho)\right)||_{m_0,\rho}  \\
& \leq \Gamma_{k_0,m_0} r^{\Gamma_{k_0,m_0}} C_{k_0,m_0} \left(J^+_\mathbf{q} \right)^{N_0} | a_{\mathbf{q},\overrightarrow{\gamma}}(\rho) | \\
& \leq C_{k_0,m_0} r^{\Gamma_{k_0,m_0}} \left(J^+_\mathbf{q} \right)^{N_0} ||\alpha||_\infty^r 
\end{align*}
Finally, we get as expected
$$||a_r^{(k_0)}||_{C^{m_0}}  \leq C_{k_0,m_0} r^{\Gamma_{k_0,m_0}} \left(J^+_\mathbf{q} \right)^{N_0} ||\alpha||_\infty^r $$

\subsubsection{End of proof of proposition \ref{prop_symbols_a_q}}
Armed with these estimates, we are now able to conclude the proof of Proposition \ref{prop_symbols_a_q} under the assumptions (\ref{Control_local_jacobian}). Assume that this assumption is satisfied and construct inductively the symbols $a_r^{(k)}$ with the formula (\ref{iteration_formula_1}). 
Since $ J_{\mathbf{q}}^+ \leq Ch^{-\delta}$, it implies that $n = O ( \log h)$. Hence, we have for $r \leq n$, 
$$ || a_r^{(k)} ||_{C^m} \leq C_{k,m} h^{- \delta m} h^{-2k \delta} |\log h |^{\Gamma_{k,m}} ||\alpha||_\infty^r\leq  C_{k,m} h^{- \delta_1 m} h^{-2k \delta_1} ||\alpha||_\infty^{r}$$
The symbol $ h^{2 \delta_1 k}  a_r^{(k)}$ lies in $||\alpha||_\infty^r S_{\delta_1}^{comp}(T^* \R)$. 
Using Borel's theorem with the parameter $h^\prime = h^{1- 2 \delta_1}$, we can construct a symbol 

$$ a^+_{q_0 \dots q_{r-1}} \sim \sum_{k=0}^\infty (h^\prime)^k  h^{2 \delta_1 k} a_r^{(k)} = \sum_{k=0}^\infty h^k  a_r^{(k)} \in ||\alpha||_\infty^r S_{\delta_1}^{comp} $$ 
that is, for every $N \in \N$, 
$$ a_{q_0 \dots q_{r-1}}^+ - \sum_{k=0}^{N-1} h^k  a_r^{(k)} = O \left( h^{(1- 2 \delta_1)N}||\alpha||_\infty^r \right)$$
By construction of the $a_r^{(k)}$, for every $N \in \N$, we have 
$$U_\mathbf{q}^+ -  \op(a_\mathbf{q}^+)T^{|\mathbf{q}|} = R_n^{(N)} + O \left( h^{(1- 2 \delta_1)}||\alpha||_\infty^r\right)$$
Fix some $K\geq 0$ such that $\min(1, ||\alpha||_\infty^n) = O(h^{-K})$, so that $||\alpha||_\infty^r = O(k^{-K})$. 
With (\ref{control_remainder}) and our estimates, we can control 
$$ ||R_n^{(N)}|| \leq C_N h^N \left(1 + |\log h|^{\Gamma_{k,m}+1} h^{- \delta (2N + 15) }h^{-K}\right) \leq C_N h^{-15 \delta_1 + N (1- 2 \delta_1)-K}$$ 
Since we can choose $N$ as large as we want, we have finally proved that 
$$U_\mathbf{q}^+ -  \op(a_\mathbf{q}^+)T^{|\mathbf{q}|} = \hinf$$
\begin{flushright}
\qed
\end{flushright}

\subsubsection{Norm of sums over many words}

We'll make use of the tools and notations developed in this subsection to prove the following proposition.
To state it, we introduce the notations 

\begin{equation}
\mathcal{Q}(n,\tau, C_0) \coloneqq \{ \mathbf{q} \in \mathcal{A}^n ; J_\mathbf{q}^+ \leq C_0h^{-\tau} \}
\end{equation}

\begin{prop}\label{sum_over_many_words}
There exists $C=C(C_0,\tau)$ such that for every $\mathcal{Q} \subset \mathcal{Q}(n,\tau,C_0)$, the following bound holds : 
\begin{equation}
\left| \left| \sum_{\mathbf{q} \in \mathcal{Q}} U_\mathbf{q} \right| \right|_{L^2 \to L^2} \leq C||\alpha||^n |\log h |
\end{equation}
\end{prop}

\begin{proof}
Throughout the proof, we'll denote by $C$ quasi-global constants, i.e. constants depending on $C_0, \tau$ and the same other parameters as global constants. We will also be lead to use a constant $C_1$: it has the same dependence. 

\vspace{0.5cm}
\emph{Step 1: } First note that since $J_\mathbf{q}^+ \leq C_0 h^{-\tau}$, $n$ satisfies the bound $n=O (\log h )$. 
\vspace{0.5cm}

\emph{Step 2 : } If $\mathbf{q} \in \mathcal{Q}(n,\tau, C_0)$, denote by $l(\mathbf{q}) = l$ the largest integer such that 
$$ J_{q_0 \dots q_{l-1}}^+ \leq h^{-\tau /2}$$
Since $ J_{q_0 \dots q_{l}} > h^{-\tau /2}$, $ J_{q_0 \dots q_{l-1}}^+ > Ch^{-\tau/2}$ and hence 
$$J_{q_l \dots q_{n-1}}^+ \leq C \frac{h^{-\tau}}{J_{q_0 \dots q_{l-1}}^+} \leq C_1 h^{-\tau /2}$$
We can then write $\mathbf{q} = \mathbf{s} \mathbf{r}$ with $ \mathbf{s} \in \mathcal{Q}(l,\tau/2, 1) , \mathbf{r} \in \mathcal{Q}(n-l,\tau/2, C_1)$. 
It follows that we can write 

$$  \sum_{\mathbf{q} \in \mathcal{Q}} U_\mathbf{q} = \sum_{l=1}^n  \sum_{\substack{ \mathbf{s} \in \mathcal{Q}(l,\tau/2, 1) \\ \mathbf{r} \in \mathcal{Q}(n-l,\tau/2, C_1)}} F_l (\mathbf{s},\mathbf{r}) U_\mathbf{r} U_\mathbf{s} $$
with $F_l (\mathbf{s},\mathbf{r}) = \mathds{1}_{\mathbf{s}\mathbf{r} \in \mathcal{Q}}$.
It is then enough to show the bound 
\begin{equation}
\max_{1 \leq l \leq n } \left| \left| \sum_{\substack{ \mathbf{s} \in \mathcal{Q}(l,\tau/2, 1) \\ \mathbf{r} \in \mathcal{Q}(n-l,\tau/2, C_1)}} F_l (s,r) U_\mathbf{r} U_\mathbf{s}  \right| \right| \leq C||\alpha||_\infty^n
\end{equation}
In the following, we fix some $1 \leq l \leq n$ and we'll simply note $\sum_{\mathbf{s}, \mathbf{r}}$ to alleviate the notations. Note that the number of terms in the sum is bounded by 
 
 $$ | \mathcal{Q}(l,\tau/2, 1) \times \mathcal{Q}(n-l,\tau/2, C_1)| \leq |\mathcal{A}|^l \times |\mathcal{A}|^{n-l} \leq |\mathcal{A}|^n \leq h^{-Q}$$ 
 where $Q = C \log |\mathcal{A}|$. 
 
\vspace{0.5cm}
\emph{Step 3: } We fix some large $N \in \N$ and $\delta_1 \in (\tau/2 , 1/2)$. Recall that we can write, 

\begin{align*}
&U_\mathbf{s} = \left( \op \left( \sum_{k=0}^{N-1} h^k a_{\mathbf{s}}^{(k)} \right) + O_{L^2 \to L^2} \left(h^{(1-2 \delta_1)N - 15 \delta_1}||\alpha||_\infty^l  \right)  \right)T^{l} \\
 &U_\mathbf{r} = T^{n-l} \left( \op \left( \sum_{k=0}^{N-1} h^k a_{\mathbf{r}}^{(k)} \right) + O_{L^2 \to L^2} \left( h^{(1-2 \delta_1)N - 15 \delta_1}||\alpha||_\infty^{n-l}  \right)  \right)
\end{align*}
with bounds on $a_{\mathbf{s}}^{(k)}$ and $a_{\mathbf{r}}^{(k)}$ given by Proposition \ref{prop_symbols_a_q}. 

We then use the formula for the composition of operators in $\Psi_{\delta_1}^{comp}(T^*\R)$ (Lemma \ref{Moyal_produc_op}) and for simplicity, we note $\mathcal{L}_k (a , b) (\rho) = \frac{i^k}{k!}(A(D))^k (a \otimes b)(\rho, \rho)$.  For $0 \leq k \leq N-1$, we set 
$$a_{\mathbf{s},\mathbf{r}, k} = \sum_{j + k_- + k_+=k} \mathcal{L}_j \left( a_{\mathbf{r}}^{(k_-)}, a_{\mathbf{s}}^{(k_+)} \right)$$
Note that if $j+k_- + k_+ \geq N$, 

\begin{align*}
 ||a_{\mathbf{r}}^{(k_-)} \otimes a_{\mathbf{s}}^{(k_+)}||_{C^{2j+13}} 
&\leq C_j\sup_{m_+ + m_- = 2j+13} ||a_{\mathbf{r}}^{(k_-)}||_{C^{m_-}} ||a_{\mathbf{s}}^{(k_+)}||_{C^{m_+}} \\
&\leq C_{j,k_-,k_+}  h^{-(2k_-+m_-)\delta_1} h^{-(2k_-+m_+)\delta_1} ||\alpha||_\infty^n\\
& \leq C_{j,k_-,k_+}  h^{-2\delta_1(j+k_-+k_+) -  13\delta_1 }||\alpha||_\infty^n \\
&\leq C_{j,k_-,k_+}  h^{-2 \delta_1 N - 13 \delta_1}||\alpha||_\infty^n
\end{align*}
and henceforth, $$O\left( h^{j+k_-+k_+}  ||a_{\mathbf{r}}^{(k_-)} \otimes a_{\mathbf{s}}^{(k_+)}||_{C^{2j+13}} \right) = O\left( h^{(1-2\delta_1)N - 15 \delta_1}||\alpha||_\infty^n \right)$$
As a consequence, we can write 

$$U_\mathbf{r} U_\mathbf{s} = T^{n-l} \left(\op \left(\sum_{k=0}^{N-1} h^k a_{\mathbf{s}, \mathbf{r}, k} \right) \right) T^l + O_{L^2 \to L^2} \left( h^{(1-2\delta_1)N - 15\delta_1} ||\alpha||_\infty^n\right)$$
It follows that $$ \sum_{\mathbf{s}, \mathbf{r}} U_\mathbf{r} U_\mathbf{s} =T^{n-l} \left(\op \left(\sum_{k=0}^{N-1} h^ka^{(k)} \right) \right) T^l + O_{L^2 \to L^2} \left( h^{(1-2\delta_1)N - 15 \delta_1 - Q}||\alpha||_\infty^n \right)$$
where 
\begin{equation}
a^{(k)} = \sum_{\mathbf{s}, \mathbf{r}}  F(\mathbf{s}, \mathbf{r})  a_{\mathbf{s}, \mathbf{r}, k}
\end{equation}
Suppose that $N$ has been chosen such that 
$$ (1- 2 \delta_1)N > 15 \delta_1 + Q$$
The remainder term is thus controlled by the desired bound since it is of order $O(||\alpha||_\infty^n )$.

\vspace{0.5cm}
\emph{Step 4: $C^0$ norm of $a^{(0)}$.}
$$a^{(0)} = \sum_{\mathbf{s}, \mathbf{r}} F(\mathbf{s}, \mathbf{r}) a_{\mathbf{s}}^{(0)} a_{\mathbf{r}}^{(0)} $$
where, in virtue of (\ref{principal_symbol_aq+}) and (\ref{principal_symbol_aq-}), 
$$ a_{\mathbf{s}}^{(0)}  = \prod_{p=1}^{l}a_{s_{l-p}} \circ F^{-p} \; ; \;  a_{\mathbf{r}}^{(0)}  = \prod_{p=0}^{n-l-1}a_{r_p} \circ F^{p}$$
As a consequence, we can estimate
\begin{align*}
|a^{(0)}| &\leq \sum_{\mathbf{s}, \mathbf{r}} | a_{\mathbf{s}}^{(0)}| |a_{\mathbf{r}}^{(0)}| \\
&\leq  \prod_{p=1}^{l} \left( \sum_{q \in \mathcal{A} } |a_q| \right) \circ F^{-p}  \times  \prod_{p=0}^{n-l-1} \left( \sum_{q \in \mathcal{A} } |a_q| \right) \circ F^{p}   \\
&\leq || \alpha||_\infty^n 
\end{align*}

\emph{Step 5 : $C^m$ norms of $a^{(k)}$.} 
We will show the following : there exists constants $C_{k,m}$ (depending only on $C_0, \delta_1, \tau$ and $m,k$) such that for all $0 \leq k \leq N-1$ and $m \in \N$, 

\begin{equation}\label{bound_on_symbol_sum}
 ||a^{(k)} ||_{C^m} \leq C_{k,m} h^{-(2k +m)\delta_1 }||\alpha||_\infty^n
\end{equation}
Let's compute : 
\begin{align*}
||a^{(k)}||_{C^m} &\leq \sum_{\mathbf{s}, \mathbf{r}}  ||a_{\mathbf{s}, \mathbf{r},k}||_{C^m} \\
&\leq  \sum_{\mathbf{s}, \mathbf{r}}  \sum_{j + k_+ + k_- = k} \left| \left| \mathcal{L}_j \left( a_{\mathbf{r}}^{(k_-)} ,  a_{\mathbf{s}}^{(k_+)} \right) \right| \right|_{C^m} \\
&\leq \sum_{\mathbf{s}, \mathbf{r}}  \sum_{j + k_+ + k_- = k} \left| \left|a_{\mathbf{r}}^{(k_-)} \otimes a_{\mathbf{s}}^{(k_+)} \right| \right|_{C^{2j+m}} \\
&\leq \sum_{\mathbf{s}, \mathbf{r}}  \sum_{\substack{j + k_+ + k_- = k \\m_+ + m_- \leq m+2j}}  \left| \left|a_{\mathbf{r}}^{(k_-)} \right| \right|_{C^{m_-}}   \left| \left|a_{\mathbf{s}}^{(k_+)} \right| \right|_{C^{m_+}} \\
\end{align*} and hence 

\begin{equation}\label{last}
||a^{(k)}||_{C^m} \leq C_{k,m} \sup_{\substack{j + k_+ + k_- = k \\ m_+ + m_- \leq m+2j}}  \sum_{\mathbf{s}, \mathbf{r}}   \left| \left|a_{\mathbf{r}}^{(k_-)} \right| \right|_{C^{m_-}}   \left| \left|a_{\mathbf{s}}^{(k_+)} \right| \right|_{C^{m_+}} 
\end{equation}
Let us fix $j,k_+,k_-,m_+,m_-$ satisfying $j+k_++k_- = k, m_-+m_+ \leq m +2j$ and let us estimate $$ \sum_{\mathbf{s}}   \left| \left|a_{\mathbf{s}}^{(k_+)} \right| \right|_{C^{m_+}} \times \sum_{\mathbf{r}}\left| \left|a_{\mathbf{r}}^{(k_-)} \right| \right|_{C^{m_-}}  $$
We estimate the sum over $\mathbf{s}$. The same kind of estimates will hold for $\mathbf{r}$ with the same methods.  We reuse the tools developed in the last subsections. Namely, we set $N_+ = 2k_+ + m_+$, $\gamma_+ = (k_+,m_+)$, $\mathcal{I}=\mathcal{I}(\gamma_+)$ and $$\left( A_\mathbf{s}(\rho) \right) = \left(d^m a_\mathbf{s}^{(k)} \right)_{ k \leq k_+, 2k +m \leq N_+}$$
We have shown that there exists a global constant $C >0$ such that 
\begin{align*}
|| a_\mathbf{s}^{(k_+)} ||_{C^{m_+}} \leq \sup_{\rho} \left| \left| A_\mathbf{s}(\rho)\right| \right| &\leq C \sum_{\overrightarrow{\gamma} \in \mathcal{E}_l(\gamma_+) } \left| \left|  P_{\overrightarrow{\gamma}}(\rho)\right|\right| \\
&\leq \sum_{\overrightarrow{\gamma} \in \mathcal{E}_l(\gamma_+) } C_{N_+,k_+} \left( J_\mathbf{s}^+ \right)^{N_+} |a_{\mathbf{s}, \overrightarrow{\gamma}} (\rho)| \\
&\leq C_{N_+,k_+} h^{-\tau N_+ /2} \sum_{\overrightarrow{\gamma} \in \mathcal{E}_l(\gamma_+) } |a_{\mathbf{s}, \overrightarrow{\gamma}} (\rho)|
\end{align*} 
where  $C_{N_+,k_+}$ depends on $C_0, \tau, N_+, k_+$ and global parameters. 
We hence have to estimate 
$$ \sum_\mathbf{s} \sum_{\overrightarrow{\gamma} \in \mathcal{E}_l(\gamma_+) } |a_{\mathbf{s}, \overrightarrow{\gamma}} (\rho)|$$
Fix $\overrightarrow{\gamma} \in \mathcal{E}_l(\alpha_+)$ and write it 
$$ \overrightarrow{\gamma} =  ( \underbrace{\beta_1, \dots, \beta_1}_{t_1}, \underbrace{\beta_2, \dots, \beta_2}_{t_2-t_1}, \dots, \underbrace{\beta_d, \dots, \beta_d}_{t_d - t_{d-1}} , \underbrace{\beta_{d+1}, \dots, \beta_{d+1}}_{t_{d+1}- t_d} ) \text{ where } d \leq \Gamma_{k_+,m_+}$$
and recall that 
$$a_{\mathbf{s},\overrightarrow{\gamma}}= \prod_{p=1}^{d+1} \prod_{j=t_{p-1} +1}^{t_p -1} a_{s_j} \circ F^{-(l-j)} $$
When one sums over $\mathbf{s} \in \mathcal{A}^l$, the values of $\mathbf{s}$ at the indices $t_i, 1 \leq i \leq d$ do not play a role and we write : 

\begin{align*}
\sum_\mathbf{s} |a_{\mathbf{s},\overrightarrow{\gamma}}| &= \sum_{s_{t_1} \in \mathcal{A} } \dots \sum_{s_{t_d} \in \mathcal{A} }  \prod_{p=1}^{d+1} \prod_{j=t_{p-1}+1}^{t_p -1} \left( \sum_{s \in \mathcal{A}} |a_s| \right) \circ F^{-(l-j)} \\
&\leq |\mathcal{A}|^d \sup_{\rho} \left( \sum_{s \in \mathcal{A}} |a_s |\right)^{l} \\
&\leq K^{\Gamma_{k_+,m_+}}||\alpha||_\infty^l \leq C_{k_+,m_+}||\alpha||_\infty^l
\end{align*}
As a consequence, 
$$  \sum_\mathbf{s} \sum_{\overrightarrow{\gamma} \in \mathcal{E}_l(\gamma_+) } |a_{\mathbf{s}, \overrightarrow{\gamma}}| \leq \# \mathcal{E}_l(\gamma_+) C_{k_+,m_+} ||\alpha||_\infty^l\leq C_{k_+,m_+}  (l-1)^{\Gamma_{k_+,m_+}} ||\alpha||_\infty^l$$
which gives 
$$\sum_{\mathbf{s}}   \left| \left|a_{\mathbf{s}}^{(k_+)} \right| \right|_{C^{m_+}} \leq C_{k_+,m_+} h^{- \tau N_+/2} (l-1)^{\Gamma_{k_+,m_+}}  ||\alpha||_\infty^l\leq  C_{k_+,m_+} h^{ -\delta_1 N_+} ||\alpha||_\infty^l$$
where the last inequality (with a different value of $C_{k_+,m_+}$) follows from the fact that $ l =O(\log h)$ and $\delta_1 > \frac{\tau}{2}$. The same kind of estimates holds for the sum over $\mathbf{r}$ : 

$$\sum_{\mathbf{r}}\left| \left|a_{\mathbf{r}}^{(k_-)} \right| \right|_{C^{m_-}} \leq C_{k_-,m_-} h^{ -\delta_1 N_-} ||\alpha||_\infty^{n-l} $$
Eventually, using (\ref{last}), we get (\ref{bound_on_symbol_sum}) since $$N_+ + N_- = 2k_+ + m_+ + 2k_- + m_- \leq 2(k_+ + k_- + j) +m  = 2k + m $$

\emph{Step 6 : Conclusion.} We can conclude the proof of the Proposition \ref{sum_over_many_words}. The bound (\ref{bound_on_symbol_sum}) shows that for $0 \leq k \leq N-1$,  $a^{(k)} \in h^{-2k \delta_1}||\alpha||_\infty^n S_{\delta_1}^{comp}$ and thus $\sum_{k=0}^{N-1}  h^k a^{(k)} \in S_{\delta_1}^{comp}||\alpha||_\infty^n$. From the $L^2$-boundedness of pseudodifferential operators with symbol in $S_{\delta_1} $,
 $$ \left| \left| \op\left(\sum_{k=0}^{N-1} h^k a^{(k)}  \right) \right|\right|  \leq \sum_{k=0}^{N-1} \sum_{ m \leq M} h^{k+m/2} || a^{(k)}||_{C^m }\leq \sum_{k=0}^{N-1} \sum_{ m \leq M} C_{k,m} h^{(k+2m)(1/2 - \delta_1)} ||\alpha||_\infty^n \leq C||\alpha||_\infty^n $$ 
where $C$ depends only on $C_0, \tau, \delta_1$. Since $||T||\leq 1$, we get 
 $$ \left| \left| \sum_{\mathbf{s}, \mathbf{r}} F(\mathbf{s}, \mathbf{r}) U_\mathbf{r} U_\mathbf{s} \right|\right| \leq C ||\alpha||_\infty^n $$ which concludes the proof of Proposition \ref{sum_over_many_words}.
\end{proof}

\subsection{Manipulations of the $U_\mathbf{q}$}
\subsubsection{First consequences}
We now make use of Proposition \ref{prop_symbols_a_q} to deduce several important facts. We go on following \cite{NDJ19}. In the whole subsection, we fix $0 \leq \delta < \delta_1 < \frac{1}{2}$ and $C_0 >0$. We denote $\mathcal{A}^{\rightarrow} = \bigcup_{n \in \N}  \mathcal{A}^n$. 

\begin{rem} 
The constants in $\hinf$ depend on $\mathbf{p}$ and $\mathbf{q}$ only through $C_0, \delta, \delta_1$, not on the precise value of $\mathbf{p}$ and $\mathbf{q}$. It will always be the case in the following and we won't precise it anymore. As already done, all the quasi-global constants (i.e. depending on global parameters and $C_0, \delta, \tau, \delta_1$) will be noted by the letter $C$. 
\end{rem}

\begin{lem}\label{empty_+-}
Let $\mathbf{q}, \mathbf{p} \in \mathcal{A}^{\rightarrow}$ satisfying $\mathcal{V}_\mathbf{q}^+ \cap \mathcal{V}_\mathbf{p}^- = \emptyset$ and $\max (J_\mathbf{q}^+, J_\mathbf{p}^- ) \leq C_0 h^{-\delta}$. Then $$U_{\mathbf{p}} U_\mathbf{q} = \hinf_{L^2 \to L^2}$$
\end{lem}

\begin{proof}
In virtue of Proposition \ref{prop_symbols_a_q}, we can write 
$$ U_\mathbf{p} = T^{|\mathbf{p}|} \op(a_\mathbf{p}^- ) + \hinf $$
 $$ U_\mathbf{q} = \op(a_\mathbf{q}^+)  T^{|\mathbf{q}|} + \hinf $$
 With $a_\mathbf{q}^+ \in ||\alpha||_\infty^{|\mathbf{q}|} S_{\delta_1}^{comp}, a_\mathbf{p}^- \in ||\alpha||_\infty^{|\mathbf{p}|}S_{\delta_1}^{comp}$ and $\supp a_\mathbf{p}^- \subset \mathcal{V}_\mathbf{p}^- , \supp a_\mathbf{q}^+ \subset \mathcal{V}_\mathbf{q}^+$. Since $\mathcal{V}_\mathbf{q}^+ \cap \mathcal{V}_\mathbf{p}^- = \emptyset$,  $\op(a_\mathbf{p}^- )\op(a_\mathbf{q}^+) = \hinf$ as a consequence of the composition of two symbols of $S_{\delta_1}$. The constants in $\hinf$ depend on semi-norms of these symbols, themselves depending on $C_0, \tau, \delta_1$. Since $T^n = O(1)$, the result is proved. 
\end{proof}

Lemma \ref{empty_+-} will have interesting consequences, starting with the following lemma which enables use to get rid (that is to say to control by $\hinf$) of words $\mathbf{q}$ where $\mathcal{V}_\mathbf{q}^\pm = \emptyset$, under some assumptions. In particular, it can be applied without trouble to words of "small" lengths $N \leq \frac{1}{2 \lambda_1} | \log h|$, what could also be deduced from applying Egorov's theorem up to the global Ehrenfest time $\frac{1}{2 \lambda_1} |\log h|$. 

\begin{lem}\label{empty_V_q}
Let $\mathbf{q} \in \mathcal{A}^{\rightarrow}$ such that $n=|\mathbf{q}| \leq C_0| \log h|$ and assume that $\mathcal{V}_\mathbf{q}^- = \emptyset$. We suppose that one of the above assumptions is satisfied : 
\begin{enumerate}[label=(\roman*)]
\item If $m = \max \{ k \in \{1, \dots, n \} , \mathcal{V}_{q_0 \dots q_{k-1}}^- \neq \emptyset \} $, $J_{q_0 \dots q_{m-1}}^- \leq C_0 h^{-2 \delta}$. 
\item If $m = \min \{ k \in \{0, \dots, n-1 \} , \mathcal{V}_{q_m \dots q_{n-1}}^- \neq \emptyset \} $, $J_{q_m \dots q_{n-1}}^- \leq C_0 h^{-2 \delta}$. 
\end{enumerate}
Then, $U_\mathbf{q} = \hinf$. 
\end{lem}

\begin{proof}
We prove this lemma under assumption (i). This is similar under (ii). We note $m= \max \{ k \in \{1, \dots, n \} , \mathcal{V}_{q_0 \dots q_{m-1}}^- \neq \emptyset \}$ and assume $J_{q_0 \dots q_{m-1}}^- \leq C_0 h^{-2 \delta}$. Due to (\ref{bound_trivial_Uq}), it is enough to show that $U_{q_0 \dots q_m} = \hinf$.  Let us denote $l = \max \{ k \in \{ 1, \dots, m \}, J_{q_0 \dots q_{l-1}}^- \leq h^{-\delta} \}$ and notice that $l <m$ (if $h$ is small enough). By maximality of $l$, it is clear that $J_{q_0 \dots q_l}^- \geq h^{-\delta}$. According to the third point of Lemma \ref{manipulations_J_q}, 
$$ J_{q_{l+1} \dots q_{m-1}}^- \sim \frac{J_{q_0 \dots q_{m-1}}^-}{J_{q_0 \dots q_l}^-} \leq Ch^{-\delta}$$
Set $\mathbf{p} = q_l \dots q_{m}$. We distinguish now between two cases
\begin{itemize}[label = \ding{228}]
\item  $\mathcal{V}_\mathbf{p}^- \neq \emptyset$ :  We set $\mathbf{r} = q_0 \dots q_{l-1}$.  It follows that $$\max( J_\mathbf{p}^-,J_\mathbf{r}^-)  \leq Ch^{-\delta}$$ Moreover, $$\mathcal{V}_\mathbf{p}^- \cap \mathcal{V}_\mathbf{r}^+ = F^l \left( \mathcal{V}_{q_0 \dots q_m}^- \right) = \emptyset$$
By Lemma \ref{empty_+-}, $U_\mathbf{p} U_\mathbf{r} = U_{q_0 \dots q_m} = \hinf$. 
\item $\mathcal{V}_\mathbf{p}^- = \emptyset$ : This time, we have $\max (J_{q_l \dots q_{m-1}}^-, J_{q_m}^-) \leq Ch^{-\delta}$ and $\mathcal{V}_{q_m}^- \cap \mathcal{V}_{q_l \dots q_{m-1}}^+ = \emptyset$. According to Lemma \ref{empty_+-}, $U_{q_l \dots q_m} = U_{q_m} U_{q_l \dots q_{m-1}} = \hinf$. It follows that $U_{q_0 \dots q_m} = \hinf$. 
\end{itemize}
\end{proof}

\subsubsection{Orthogonality of the $U_\mathbf{q}$}
We now focus on terms $U_\mathbf{q} U_\mathbf{p}^*$ and $U_\mathbf{q}^* U_\mathbf{p}$ when $\mathcal{V}_\mathbf{q}^+ $ and $\mathcal{V}_\mathbf{p}^+ $ are disjoint, under  growth conditions of the Jacobian. The following result shows that the operators $U_\mathbf{q}$ and $U_\mathbf{p}$ are (up to $\hinf$) orthogonal. These estimates will turn out to be important to apply Cotlar-Stein type estimates.

\begin{prop}\label{orthogoality_disjoint_support}
Assume that $\mathbf{q}, \mathbf{p} \in \mathcal{A}^{\rightarrow}$ are two words of same length $|\mathbf{q}| = | \mathbf{p}|=n$ satisfying $\mathcal{V}_\mathbf{q}^+ \cap \mathcal{V}_\mathbf{p}^+ = \emptyset$ and $\max ( J_\mathbf{q}^+, J_\mathbf{p}^+) \leq C_0 h^{-2 \delta}$. Then, 
\begin{align*}
U_\mathbf{q} U_\mathbf{p}^* = \hinf \\
U_\mathbf{q}^* U_\mathbf{p} = \hinf 
\end{align*}
\end{prop}

Before proving it, we need the following lemma, whose proof relies on the iterative construction of the symbols $a_\mathbf{q}^\pm$.

\begin{lem}
Assume that $\mathbf{q}, \mathbf{p} \in \mathcal{A}^{\rightarrow}$ are two words of same length $|\mathbf{q}| = | \mathbf{p}|=n$ satisfying $\max ( J_\mathbf{q}^+, J_\mathbf{p}^+) \leq C_0 h^{- \delta}$. Then, 
\begin{align*}
U_\mathbf{q} U_\mathbf{p}^* =  \op(a_\mathbf{q}^+ )  \op(a_\mathbf{p}^+ )^* + \hinf \\
U_\mathbf{q}^* U_\mathbf{p} = \op(a_\mathbf{q}^- )^*    \op(a_\mathbf{p}^- ) + \hinf \\
\end{align*}
\end{lem}

\begin{proof} (\textit{of the lemma})
We prove the first equality. The second one could be treated similarly. 
Recall the construction procedure of the subsection \ref{section_propagation_Ehrenfest}. We adopt the same notations. We will show by induction on $r \in \{0, \dots, n-1 \}$ that : 
$$ V_r \coloneqq U_{q_0 \dots q_{r-1}} U_{p_0 \dots p_{r-1}}^* = \op( a_{q_0 \dots q_{r-1}}^+ ) \op(a_{p_0 \dots p_{r-1} }^+ )^* + \hinf $$
The case $r=1$ follows from $$MA_{q_0}A_{p_0}^* M^* = \op(a_{q_0}^+ ) T T^* \op(a_{p_0}^+ )^* + \hinf = \op(a_{q_0}^+ )\op(a_{p_0}^+ )^* + \hinf$$
where we use the fact that $TT^* = I + \hinf$ microlocally in $\mathcal{V}_{p_0}^+$, 
Assume that the assumption is satisfied for $r$, namely : 
$$ V_r  = \op( a_{q_0 \dots q_{r-1}}^+ ) \op(a_{p_0 \dots p_{r-1} }^+ ) + \hinf $$
and let's prove it for $r+1$. 
\begin{align*}
V_{r+1} &= MA_{q_r} V_r A_{p_r}^*M^* \\
&= MA_{q_r} \op( a_{q_0 \dots q_{r-1}}^+ ) \op(a_{p_0 \dots p_{r-1} }^+ )^* A_{p_r}^* M^*r + \hinf \\
&= \op( a_{q_0 \dots q_{r}}^+ )TT^* \op(a_{p_0 \dots p_{r} }^+ )^* + \hinf\\
&= \op( a_{q_0 \dots q_{r}}^+ ) \op(a_{p_0 \dots p_{r} }^+ )^* + \hinf
\end{align*}
 The last equality follows from $TT^* = I + \hinf$ microlocally in $\mathcal{V}_{p_r}^+$ and the one before is due to the recursive construction of the symbols $a_{q_0 \dots q_{r}}^+$ in the subsection \ref{section_propagation_Ehrenfest}. 

\end{proof}

\begin{proof} (\textit{of the proposition})
Let us begin with the first equality. Consider the largest integer $l$ such that 
$$ \max (J_{ q_0 \dots q_{l-1}}^+ , J_{p_0 \dots p_{l-1} }^+ ) \leq h^{- \delta}$$ 
We set $\mathbf{q}_\leftarrow = q_0 \dots q_{l-1}$ and $\mathbf{q}_\rightarrow = q_l \dots q_{n-1}$, and the same notations for $\mathbf{p}$. We obviously have : $$U_\mathbf{q} U_\mathbf{p}^* = U_{\mathbf{q}_\rightarrow} U_{\mathbf{q}_\leftarrow} U_{\mathbf{p}_\leftarrow}^* U_{\mathbf{p}_\rightarrow}^*$$
We then consider two cases, 
\begin{itemize}[label = \ding{228}]
\item $\mathcal{V}_{ \mathbf{q}_\leftarrow}^+ \cap \mathcal{V}_{ \mathbf{p}_\leftarrow}^+= \emptyset$ : we may write $$ U_{\mathbf{q}_\leftarrow} U_{\mathbf{p}_\leftarrow}^* = T^l \op( a_{\mathbf{q}_\leftarrow}^- ) \op( a_{\mathbf{q}_\leftarrow}^-)^* T^l + \hinf$$
Since, $\mathcal{V}_{\mathbf{q}_\leftarrow}^- \cap \mathcal{V}_{\mathbf{p}_\leftarrow}^- = \emptyset$, we can use the composition formula in $S_{\delta_1}^{comp}$ to conclude that $\op( a_{\mathbf{q}_\leftarrow}^- ) \op( a_{\mathbf{q}_\leftarrow}^-)^*  = \hinf$, which gives the desire result, recalling that $U_\mathbf{q} = O(1)$. 
\item $\mathcal{V}_{ \mathbf{q}_\leftarrow}^+ \cap \mathcal{V}_{ \mathbf{p}_\leftarrow}^+\neq \emptyset$ : in this case, we use the previous lemma and we can write $$U_{\mathbf{q}_\leftarrow} U_{\mathbf{p}_\leftarrow}^* =  \op(a_{\mathbf{q}_\leftarrow}^+ )  \op(a_{\mathbf{p}_\leftarrow}^+ )^* + \hinf $$ 
In virtue of the second point of Lemma \ref{manipulations_J_q}, $J_{\mathbf{q}_\leftarrow}^+ \sim J_{\mathbf{p}_\leftarrow}^+$. 
Moreover, by maximality of $l$, either $J_{\mathbf{q}_\leftarrow q_l}^+ > h^{-\delta}$ or $J_{\mathbf{p}_\leftarrow p_l}^+ > h^{-\delta}$. But 
$$ J_{\mathbf{q}_\leftarrow q_l}^+ \sim J_{\mathbf{q}_\leftarrow}^+$$
Hence, $J_{\mathbf{q}_\leftarrow}^+ \sim h^{-\delta}$. Using now the third point of Lemma \ref{manipulations_J_q}, we conclude that 
$$ J_{\mathbf{q}_\rightarrow}^+ \sim J_{\mathbf{p}_\rightarrow}^+ \sim h^{-\delta}$$
This estimate allows us to write 
$$ U_\mathbf{q} U_\mathbf{p}^*  = T^{n-l}     \op(a_{\mathbf{q}_\rightarrow}^- ) \op(a_{\mathbf{q}_\leftarrow}^+ )  \op(a_{\mathbf{p}_\leftarrow}^+ )^*   \op(a_{\mathbf{p}_\rightarrow}^- )^*  ( T^*)^{n-l} + \hinf$$
with all the symbols in $h^{-M}S_{\delta_1}^{comp}$ for some $M >0$. To conclude, we use the composition formula in this symbol class, noting that 
$$ \mathcal{V}_{\mathbf{q}_\leftarrow}^+ \cap \mathcal{V}_{\mathbf{q}_\rightarrow}^- \cap \mathcal{V}_{\mathbf{p}_\leftarrow}^+  \cap \mathcal{V}_{\mathbf{p}_\rightarrow}^-  = F^l \left( \mathcal{V}_{\mathbf{q}}^-  \cap \mathcal{V}_{\mathbf{p}}^-  \right) = \emptyset$$ 
\end{itemize}
To deal with the second equality, we consider the smallest integer $l$ such that : 
$$ \max (J_{ q_l \dots q_{n-1}}^+ , J_{p_l \dots p_{n-1} }^+ ) \leq h^{- \delta}$$ 
As before, we write $\mathbf{q}_\leftarrow = q_0 \dots q_{l-1}$ and $\mathbf{q}_\rightarrow = q_l \dots q_{n-1}$, and the same notations for $\mathbf{p}$. We obviously have : $$U_\mathbf{q}^* U_\mathbf{p} = U_{\mathbf{q}_\leftarrow}^* U_{\mathbf{q}_\rightarrow}^* U_{\mathbf{p}_\rightarrow} U_{\mathbf{p}_\leftarrow}$$
We distinguish the cases $\mathcal{V}_{ \mathbf{q}_\rightarrow}^+ \cap \mathcal{V}_{ \mathbf{p}_\rightarrow}^+= \emptyset$ or not and argue similarly. 
\end{proof}

\subsection{Reduction to sub-words with precise growth of their Jacobian }

Recall that we are interested in a decay bound for $|| \mathfrak{M}^{N_0 + N_1} ||$ where $\mathfrak{M}= M (\Id- A_\infty) = \sum_{q \in \mathcal{A}} M A_q$. For this purpose, we decompose $\mathfrak{M}^{N_1} = \sum_{ \mathbf{q} \in \mathcal{A}^{N_1} } U_\mathbf{q}$.

If $\mathbf{q} \in \mathcal{A}^{N_1}$, either $\mathcal{V}_\mathbf{q}^+ = \emptyset$, and in this case $J_\mathbf{q}^+ = + \infty$, or $\mathcal{V}_\mathbf{q}^+ \neq \emptyset$, which implies that $J_\mathbf{q}^+  \geq e^{\lambda_1 N_1 } \geq h^{-1} \gg h^{-\tau}$. In both cases, 
the following integer is well defined : 

\begin{equation}
n(\mathbf{q}) = \max \{k \in \{1, N_1 \}, J_{q_{N_1 - k } \dots q_{N_1 -1} }^+ \leq h^{-\tau} \}
\end{equation}
We then set $\mathbf{q}_\tau = q_{N_1 - n(q) - 1 } \dots q_{N_1 -1 }$. The case $\mathcal{V}_{ \mathbf{q}_\tau} = \emptyset$ is irrelevant. Indeed, if $\mathbf{q} \in \mathcal{A}^{N_1}$ and if $\mathcal{V}_{ \mathbf{q}_\tau} = \emptyset$, then $U_\mathbf{q} = \hinf$, as an obvious consequence of Lemma \ref{empty_V_q}. Then, we set 
\begin{equation}
Q= \{ \mathbf{q} \in \mathcal{A}^{N_1} , \mathcal{V}_{ \mathbf{q}_\tau} \neq \emptyset \} 
\end{equation}
so that, due to the fact that $|\mathcal{A}^{N_1}| = O(h^{-M})$, for some $M >0$, we have 
$$\mathfrak{M}^{N_1} = \sum_{\mathbf{q} \in Q} U_\mathbf{q} + \hinf $$
We partition $Q$  in function of the length of $\mathbf{q}_\tau$ and the value of $q_{N_1-1}$. Namely, we set 

$$ Q_0(n,a) = \{ \mathbf{q} \in Q ; |\mathbf{q}_\tau| = n , q_{N_1-1} = a \}$$
We finally set $Q(n,a) = \{ \mathbf{q}_\tau , \mathbf{q} \in Q_0(n,a) \}$ which is simply the set of words $\mathbf{q} \in \mathcal{A}^n$ such that  $q_{n-1} = a$ and $J_{q_1 \dots q_{n-1}}^+ \leq h^{-\tau} < J_\mathbf{q}^+ $. Note that every word $\mathbf{q} \in Q_0(n,a)$ can be written in the form $\mathbf{q} = \mathbf{r}\mathbf{p} $ with $\mathbf{p} \in Q(n,a)$ and $\mathbf{r} \in \mathcal{A}^{N_1 - n}$. We deduce that, \emph{modulo $\hinf$}, 
\begin{align*}
\mathfrak{M}^{N_1} &= \sum_{n=1}^{N_1} \sum_{a \in \mathcal{A}} \sum_{ \mathbf{q} \in Q_0(n,a) } U_\mathbf{q} \\
&= \sum_{n=1}^{N_1} \sum_{a \in \mathcal{A}}  \sum_{\substack{ \mathbf{p} \in Q(n,a)\\\mathbf{r} \in \mathcal{A}^{N_1 - n } }} U_\mathbf{p} U_\mathbf{r} \\
&=  \sum_{n=1}^{N_1} \sum_{a \in \mathcal{A}} \left( \sum_{ \mathbf{q} \in Q(n,a) } U_\mathbf{q} \right) \mathfrak{M}^{N_1 - n }
\end{align*}
As a consequence, we get 
\begin{equation}
||\mathfrak{M}^{N_0 + N_1} || \leq C N_1 | \mathcal{A} | \sup_{\substack{ 1 \leq n \leq N_1 \\a \in \mathcal{A} }} || \mathfrak{M}^{N_0} U_{Q(n,a)} || \left(||\alpha||_\infty\right)^{N_1-n}
\end{equation}
where we've noted 
\begin{equation}
U_{Q(n,a)} = \sum_{\mathbf{q} \in Q(n,a) }  U_\mathbf{q}
\end{equation}
Since $N_1 = O(\log h)$, the proof of (\ref{Goal_thm_2}) is the reduced to prove : 

\begin{prop}
There exists $\gamma>0$ such that, for $h$ small enough, we have 
\begin{equation}
\sup_{\substack{ 1 \leq n \leq N_1 \\a \in \mathcal{A} }} \frac{||\mathfrak{M}^{N_0} U_{Q(n,a)} || }{||\alpha||_\infty^{n+N_0}}  \leq h^\gamma
\end{equation}
\end{prop}

\subsection{Partition into clouds }\label{Clouds}

We fix $1 \leq n \leq N_1$ and $a \in \mathcal{A}$. We aim at gathering pieces of $\mathfrak{M}^{N_0} U_{Q(n,a) }$ into clouds and we want these clouds to interact (with a meaning we will define further) with only a finite and uniform number of other clouds, so that the global norm of $||\mathfrak{M}^{N_0} U_{Q(n,a})|| $  can be deduced from a uniform bound for each cloud. 

Recall that $\delta_0$ and $\tau$ (see (\ref{alpha}), (\ref{alpha+delta0}) and (\ref{condition_on_tau})) have be chosen such that 
\begin{equation*}
\mathfrak{b}+ \delta_0 < 1 \; ; \; \mathfrak{b} < \tau 
\end{equation*}

We start by defining a notion of closeness between two words $\mathbf{q}, \mathbf{p} \in Q(n,a)$. We choose $\varepsilon_2$ as in Lemma \ref{help_def_jacobian_2}. 

\begin{defi}\label{close_to_each_other}
Let $\mathbf{q}, \mathbf{p} \in Q(n,a)$. We say that these two words are\emph{ close to each other} if there exists $\rho_0 \in \mathcal{T} \cap F\left(\mathcal{V}_a (\varepsilon_2)\right)$ such that :  
$$ \forall \rho \in \mathcal{V}_\mathbf{q}^+ \cup \mathcal{V}_\mathbf{p}^+ \; , \; d ( \rho, W_u(\rho_0) ) \leq h^\mathfrak{b}$$
Otherwise, we say that $\mathbf{q}$ and $\mathbf{p}$ are \emph{far from each other}. 
\end{defi}

\begin{rem}
By definition of $\mathcal{V}_\mathbf{q}^+$, if $\mathbf{q} \in \mathcal{Q}(n,a)$ and if $\rho \in \mathcal{V}_\mathbf{q}^+$, $\rho$ does not lie in $\mathcal{V}_a$, but $F^{-1}(\rho)$ does. Hence, we work with $F(\mathcal{V}_a)$ instead of $\mathcal{V}_a$. Moreover, the set $F\left(\mathcal{V}_a (\varepsilon_2)\right)$ is chosen to fit well in the computations below and in particular in the proof of Lemma \ref{Lemma_decom_far}. We could replace it by $\mathcal{V}_a^+(C\varepsilon_2)$, where $C$ is any Lipschitz constant for $F$. 
\end{rem}

\begin{figure}
\includegraphics[scale=0.5]{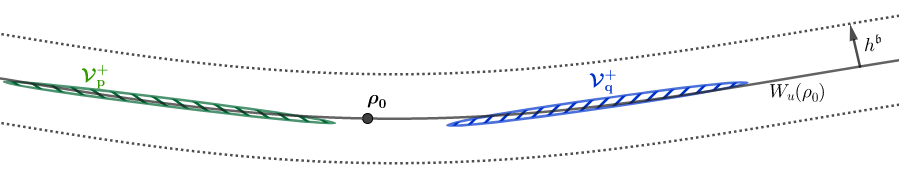}
\caption{Two words  $\mathbf{q}, \mathbf{p} \in Q(n,a)$ are close to each other if $\mathcal{V}_\mathbf{q}^+$ and  $\mathcal{V}_\mathbf{p}^+$ lie in the $h^\mathfrak{b}$-neighborhood of the same unstable leaves, as stated in Definition \ref{close_to_each_other}. }
\end{figure}
The important fact on words far from each other is that the associated operator $\mathfrak{M}^{N_0} U_\mathbf{q}$ are almost orthogonal : 

\begin{prop}
Assume that $\mathbf{q}, \mathbf{p} \in Q(n,a)$ are far from each other. Then, 
\begin{align}
 \left( \mathfrak{M}^{N_0} U_\mathbf{q} \right)^*\left( \mathfrak{M}^{N_0} U_\mathbf{p} \right)= \hinf \\
 \left( \mathfrak{M}^{N_0} U_\mathbf{q} \right) \left( \mathfrak{M}^{N_0} U_\mathbf{q} \right)^* = \hinf 
\end{align}
\end{prop}
We will need the following lemma. 

\begin{lem}\label{Lemma_decom_far}
If $\mathbf{q}, \mathbf{p} \in Q(n,a)$ are far from each other, there exist words $\mathbf{p_1}, \mathbf{q_1}, \mathbf{p_2}, \mathbf{q_2}$ such that 
\begin{itemize}[label = - , nosep]
\item $|\mathbf{p_1}|=|\mathbf{q_1}|, |\mathbf{p_2}|=|\mathbf{q_2}|$ ; 
\item $\mathbf{q} = \mathbf{q_1}\mathbf{q_2} , \mathbf{p}= \mathbf{p_1} \mathbf{p_2}$ ; 
\item $\mathcal{V}_{\mathbf{q_2}}^+ \cap \mathcal{V}_{\mathbf{p_2}}^+ = \emptyset$;
\item $\max( J_{\mathbf{q_2}}^+, J_\mathbf{p_2}^+ ) \leq C h^{-\mathfrak{b}}$ (for some global constant $C>0$). 
\end{itemize}
In particular, $\mathcal{V}_{\mathbf{q}}^+ \cap \mathcal{V}_{\mathbf{p}}^+ = \emptyset$
\end{lem}
Let's momentarily admit it and prove the proposition. 
\begin{proof}(\textit{of the proposition}). Fix $\mathbf{q}, \mathbf{p} \in Q(n,a)$ far from each other. 
Since $\mathcal{V}_{\mathbf{q}}^+ \cap \mathcal{V}_{\mathbf{p}}^+ = \emptyset, U_\mathbf{q} U_\mathbf{p}^* = \hinf$ in virtue of Proposition \ref{orthogoality_disjoint_support}. Hence, using the polynomial bounrds $||\mathfrak{M}^{N_0} ||=O(h^{-M})$ (for some $M>0$), we have  $$\left( \mathfrak{M}^{N_0} U_\mathbf{q} \right) \left( \mathfrak{M}^{N_0} U_\mathbf{p} \right)^* = \hinf $$
To prove the first point, we write 
\begin{align*}
\left( \mathfrak{M}^{N_0} U_\mathbf{q} \right)^*\left( \mathfrak{M}^{N_0} U_\mathbf{p} \right)= \sum_{\mathbf{s},\mathbf{t} \in \mathcal{A}^{N_0} } U_\mathbf{q_1}^\star U_\mathbf{q_2}^* U_\mathbf{s}^* U_\mathbf{t} U_\mathbf{p_2} U_\mathbf{p_1}
\end{align*}
Hence, it is enough to show that $U_\mathbf{q_2}^* U_\mathbf{s}^* U_\mathbf{t} U_\mathbf{p_2} =  \hinf$ uniformly in $\mathbf{s},\mathbf{t}$. 
To do so, we note that 
\begin{align*}
&\mathcal{V}^+_{\mathbf{q_2} \mathbf{s}} \cap \mathcal{V}^+_{\mathbf{p_2} \mathbf{t} }  \subset F^{N_0} \left( \mathcal{V}_{\mathbf{q_2}}^+ \cap \mathcal{V}_{\mathbf{p_2}}^+ \right) = \emptyset \\
&J^+_{\mathbf{q_2} \mathbf{s}} \leq C J_\mathbf{s}^+ J_\mathbf{q_2}^+ \leq Ce^{\lambda_1 N_0} h^{-\mathfrak{b}} \leq Ch^{ -(\delta_0 + \mathfrak{b})} \\
&J^+_{\mathbf{p_2} \mathbf{t} } \leq Ch^{ -(\delta_0 + \mathfrak{b})}
\end{align*}
and apply Proposition \ref{orthogoality_disjoint_support}, with $\delta = \frac{\delta_0 + \mathfrak{b}}{2} <1/2$ (here we use the condition (\ref{alpha+delta0})). 
\end{proof}

We now prove the lemma. 

\begin{proof}(\textit{of the lemma})
Consider $\mathbf{q}, \mathbf{p} \in Q(n,a)$ far from each other. 
Consider the smallest integer $m$ such that $\mathcal{V}_{q_m \dots q_{n-1} }^+  \cap \mathcal{V}_{p_m \dots p_{n-1} }^+ \neq \emptyset$. We will show that $m >0$ and set $\mathbf{p_2} = p_{m-1} \dots p_{n-1} , \mathbf{q_2} = q_{m-1} \dots q_{n-1}$. 
Pick $\rho \in \mathcal{V}_{q_m \dots q_{n-1} }^+  \cap \mathcal{V}_{p_m \dots p_{n-1} }^+ $. By choice of $\varepsilon_2$ after Lemma \ref{help_def_jacobian_2}, there exists $\rho_0 \in \mathcal{T}$ such that $d(F^{-i}(\rho) , F^{-i}(\rho_0) ) \leq \varepsilon_2$ for $i \in \{ 1, \dots, n-m\}$. In particular, $d(F^{-1}(\rho), F^{-1}(\rho_0)) \leq \varepsilon_2$ and $F^{-1}(\rho) \in \mathcal{V}_a$, so that $\rho_0 \in F\left(\mathcal{V}_a(\varepsilon_2) \right) $. Since, $\mathbf{q}, \mathbf{p}$ are far from each other, there exists $\rho_1 \in \mathcal{V}_\mathbf{q}^+ \cup \mathcal{V}_\mathbf{p}^+$ such that $d(\rho_1, W_u(\rho_0) ) > h^{\mathfrak{b}}$ (otherwise, it would contradict the definition \ref{close_to_each_other}).

Suppose for instance that $\rho_1 \in \mathcal{V}_\mathbf{q}^+ \subset  \mathcal{V}_{q_m \dots q_{n-1} }^+$. Hence, $d(F^{-i}(\rho_0), F^{-i} (\rho_1) ) \leq 2\varepsilon_0 + \varepsilon_2$ for $i \in \{1, \dots, n-m \}$. 
From (\ref{close_to_leaves}), $d(\rho_1, W_u(\rho_0) ) \leq C \left(J_s^{n-m}(\rho_0)\right)^{-1}$ and hence, $J_s^{n-m}(\rho_0) \leq C h^{-\mathfrak{b}}$.\\
But,  $ J_s^{n-m}(\rho_0) \sim J_{p_{m} \dots p_{n-1} }^+ \sim J_{q_{m} \dots q_{n-1} }^+$, which gives 
$$\max\left( J_{p_{m} \dots p_{n-1} }^+ , J_{q_{m} \dots q_{n-1} }^+ \right) \leq C h^{-\mathfrak{b}}$$
Since $\min (J_\mathbf{q}^+,J_\mathbf{p}^+) > h^{- \tau} \gg h^{-\mathfrak{b}}$ (here we use (\ref{condition_on_tau})), we cannot have $m=0$ (if $h$ small enough). Thus, we can set $\mathbf{p_2} = p_{m-1} \dots p_{n-1}$, $\mathbf{q_2} = q_{m-1} \dots q_{n-1}$ which satisfy the required properties by minimality of $m$. 
\end{proof}

We now decompose $U_{Q(n,a)}$ into a sum of operators, each of them corresponding to a \emph{cloud} of words. In the following, we'll use the term \emph{cloud} to mean a subset $\mathcal{Q} \subset Q(n,a)$ and we'll adopt the notation $$\mathcal{V}_{\mathcal{Q}}^+ = \bigcup_{\mathbf{q} \in \mathcal{Q}} \mathcal{V}_\mathbf{q}^+$$ and the definition : 
\begin{defi}
We say that two clouds $\mathcal{Q}_1, \mathcal{Q}_2$ \emph{do not interact} if for all couples $(\mathbf{q_1}, \mathbf{q_2} ) \in \mathcal{Q}_1 \times\mathcal{Q}_2$, $\mathbf{q_1}$ and $\mathbf{q_2}$ are far from each other. 
\end{defi}

The existence of such a decomposition follows from the key proposition : 

\begin{prop}\label{partition_into_clouds}
Suppose $\varepsilon_0$ is small enough. 

There exists a partition of $Q(n,a)$ into clouds $\mathcal{Q}_1, \dots, \mathcal{Q}_r$ and a global constant $C>0$ such that, for $i = 1 , \dots , r$, 
\begin{enumerate}[label=\roman*)]
\item there exists $\rho_i \in \mathcal{T}$ such that for all $\rho \in \mathcal{V}^+_{\mathcal{Q}_i}$, $d(\rho, W_u(\rho_i) ) \leq Ch^\mathfrak{b}$ ; 
\item if $\mathcal{Q}_i$ interacts with exactly $c_i$ clouds, then $c_i \leq C$. 
\end{enumerate}
\end{prop}

\begin{rem}
Actually, $r$ and the clouds $\mathcal{Q}_i$ depend on $n$ and $a$. We do not write this dependence explicitly here to make the notations lighter. The second point is relevant since \textit{a priori}, the only obvious bound on $r=r(n,a)$ is $|r| \leq |\mathcal{A}|^n$, where $n = O(\log h)$. 
\end{rem}

\begin{proof}
Keeping in mind that for all $\mathbf{q} \in Q(n,a)$, $\mathcal{V}_\mathbf{q}^+ \subset \mathcal{V}_a^+$, we fix $\rho_a \in \mathcal{V}_a^+$. If $\varepsilon_0$ is small enough, $\mathcal{V}_a^+$ do not intersect the boundaries of $W_s(\rho_a)$ and $W_u(\rho_a)$. 

For $\mathbf{q} \in Q(n,a)$, there exists $\rho_\mathbf{q} \in \mathcal{T}$ such that $d(F^{-i}(\rho), F^{-i}(\rho_\mathbf{q} ) ) \leq \varepsilon_2$ for all $\rho \in \mathcal{V}_\mathbf{q}^+$ and for $i = 1 , \dots, n$, according to Lemma \ref{help_def_jacobian_2} and since $J_\mathbf{q}^+ \sim h^{\tau}$, 

$$d(\rho, W_u(\rho_\mathbf{q}) ) \leq C h^{-\tau}$$
$d(\rho_a, \rho_\mathbf{q} ) \leq C(\varepsilon_2 + \varepsilon_0)$ and hence, if $\varepsilon_0$ is small enough, $ z_\mathbf{q} \coloneqq H^u_{\rho_a} (\rho_\mathbf{q})$ (here, $H^u_{\rho_a} :B(\rho_a, \varepsilon_0^\prime) \to W_s(\rho_a) $) is the unstable holonomy map defined before Lemma \ref{regularity_holonomy_maps}) is well defined, and depends Lipschitz-continuously on $\rho_\mathbf{q}$ (with global Lipschitz constant). 

Next, consider a maximal subset $\{ z_1 , \dots, z_r \} \subset \{ z_\mathbf{q} , \mathbf{q} \in Q(n,a) \}$ which is $h^\mathfrak{b}$ separated. 
By maximality, for every $\mathbf{q} \in Q(n,a)$, there exists $i \in \{1,\dots, r \}$ such that $|z_i - z_\mathbf{q}| \leq h^{\mathfrak{b}}$ and we use these $z_i$ to partition $Q(n,a)$ into clouds $\mathcal{Q}_i$ where for $ i \in \{ 1, \dots, r \}$, $|z_i - z_\mathbf{q} | \leq h^\mathfrak{b}$ for all $\mathbf{q} \in \mathcal{Q}_i$. We now show that this partition satisfies the required properties. 

Let $i \in \{1, \dots, r \}$, $\mathbf{q} \in \mathcal{Q}_i$ and $\rho \in \mathcal{V}_\mathbf{q}^+$. By local uniqueness of the unstable leaves, we may assume that $\varepsilon_0$ is small enough so that $W_u(\rho_\mathbf{q}) \cap \mathcal{V}_a^+ = W_u(z_\mathbf{q}) \cap \mathcal{V}_a^+$. Hence, 
$$ d(\rho, W_u(z_\mathbf{q}) ) \leq Ch^{-\tau}$$
Since the unstable leaves depend  Lipschitz-continuously on $\rho \in \mathcal{T}$, we have 
$$ d(\rho, W_u(z_i) ) \leq C|z_i - z_\mathbf{q}| +  C d(\rho, W_u(z_\mathbf{q})) \leq Ch^\mathfrak{b} + Ch^\tau \leq Ch^\mathfrak{b}$$
This gives i). 

To show ii), suppose that $\mathcal{Q}_i$ and $\mathcal{Q}_j$ interact. Then, there exists $(\mathbf{q}, \mathbf{p} ) \in \mathcal{Q}_i \times \mathcal{Q}_j$ and $\rho_0 \in \mathcal{T}$ such that for all $\rho \in \mathcal{V}_\mathbf{q}^+ \cup \mathcal{V}_\mathbf{p}^+$, $d(\rho, W_u(\rho_0) ) \leq h^\mathfrak{b}$.  It follows that $d(z_\mathbf{q}, W_u(\rho_0)) \leq Ch^\tau + h^\mathfrak{b} \leq Ch^\mathfrak{b}$ and if we note $z_0 = H^u_{\rho_a}(\rho_0)$ the unique point in $W_u(\rho_0) \cap W_s(\rho_a)$ then $|z_0 - z_\mathbf{q}| \leq Ch^\mathfrak{b}$. The same is true for $\mathbf{p}$ and we have $|z_\mathbf{q}- z_\mathbf{p}| \leq Ch^\mathfrak{b}$ and eventually, $|z_i - z_j | \leq Ch^\mathfrak{b}$. Since $ z_1, \dots, z_r$ are $h^\mathfrak{b}$ separated, we see after rescaling that the number of $j$ such that $\mathcal{Q}_i$ and $\mathcal{Q}_j$ interact is smaller than the maximal number of
points in $B(0,C)$ which are $1$-separated (one can for instance bound it by $(2C+1)^2$, but what matters is that it is a global constant). 
\end{proof}

\begin{figure}
\includegraphics[scale=0.5]{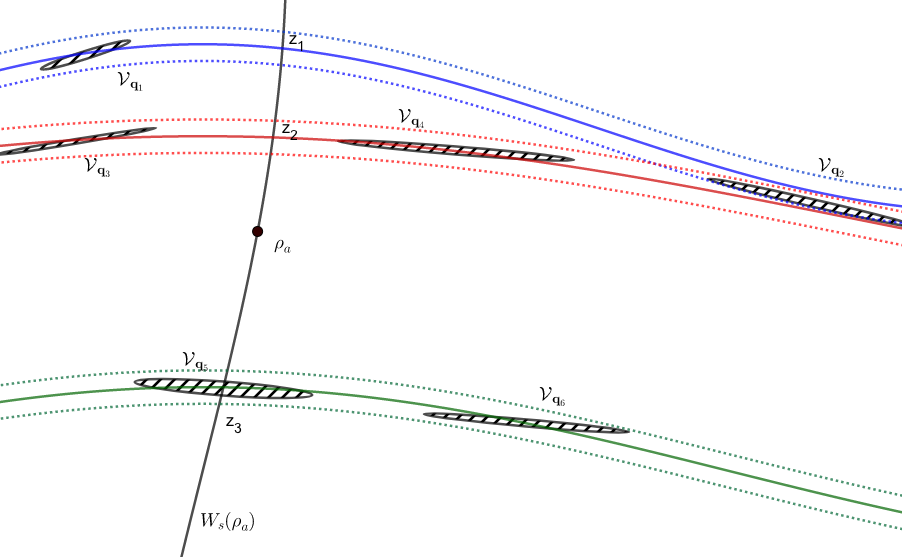}
\caption{We gather the 6 small sets $\mathcal{V}_\mathbf{q}$ into 3 clouds corresponding to $z_1, z_2$ and $z_3$. Here, $\mathcal{Q}_1 = \{ \mathbf{q}_1\} , \mathcal{Q}_2 = \{ \mathbf{q}_2, \mathbf{q}_3, \mathbf{q}_4 \}, \mathcal{Q}_3 = \{ \mathbf{q}_5 ,\mathbf{q}_6 \}$. The clouds $\mathcal{Q}_1$ and $\mathcal{Q}_2$ interact. The dotted lines draw tubes of width $Ch^\mathfrak{b}$ around the unstable leaves $W_u(z_i)$. The sets $\mathcal{V}_\mathbf{q}$ have width of order $h^\tau$. }
\end{figure}

This partition into clouds allows us to decompose $\mathfrak{M}^{N_0} U_{Q(n,a)}$ into a sum of operators
\begin{equation}
B_i =\mathfrak{M}^{N_0} U_{\mathcal{Q}_i}= \sum_{\mathbf{q} \in \mathcal{Q}_i} \mathfrak{M}^{N_0} U_\mathbf{q} \quad ; \quad 
\mathfrak{M}^{N_0} U_{Q(n,a)} = \sum_{i=1}^r B_i 
\end{equation}

The use of Cotlar-Stein theorem (\cite{ZW}), Theorem C.5) reduces the control of the sum by the control of individual clouds : 

\begin{lem} With the above notations, there exists a global constant $C>0$ such that 
\begin{equation}
|| \mathfrak{M}^{N_0} U_{Q(n,a)} || \leq C \sup_{1 \leq i \leq r } ||B_i|| + \hinf
\end{equation}
\end{lem}

\begin{proof}
Cotlar-Stein theorem reduces to control 
\begin{align*}
\max_i \sum_j ||B_i^* B_j||^{1/2} \\
\max_i \sum_j ||B_j B_i^* ||^{1/2} 
\end{align*}
Fix $i \in \{1, \dots, r \}$. \\
If $\mathcal{Q}_i $ and $\mathcal{Q}_j$ do not interact, $||B_i^* B_j||^{1/2}$ (resp. $ ||B_j B_i^* ||^{1/2}$) is a sum of terms of the form  $\left( \mathfrak{M}^{N_0} U_\mathbf{q} \right)^*\left( \mathfrak{M}^{N_0} U_\mathbf{p} \right)$ (resp. $\left( \mathfrak{M}^{N_0} U_\mathbf{q} \right) \left( \mathfrak{M}^{N_0} U_\mathbf{p} \right)^*$) where $\mathbf{p}$ and $\mathbf{q}$ are far from each other. In virtue of Proposition \ref{orthogoality_disjoint_support}, these terms are uniformly $\hinf$ and since the number of terms in the sum grows at most polynomially with $h$, we can gather all these terms in a single uniform $\hinf$. As a consequence, we have 
\begin{align*}
\sum_j ||B_i^* B_j||^{1/2} & \leq \sum_{ \mathcal{Q}_i \text{ and } \mathcal{Q}_j \text{ interact}} ||B_i^* B_j||^{1/2} + \hinf  \\
&\leq \sum_{ \mathcal{Q}_i \text{ and } \mathcal{Q}_j \text{ interact}} \max_k ||B_k|| + \hinf \\
&\leq C \max_k||B_k|| + \hinf
\end{align*}
and the same holds for the second sum. This gives the desired inequalities. 
\end{proof}

The proof of (\ref{Goal_thm_2}) and, as a consequence, of Proposition \ref{Prop_thm} is then reduced to the proof of 

\begin{prop}\label{Prop_on_clouds} There exists $\gamma >0$ such that the following holds for $h$ small enough. 
Assume that $\mathcal{Q} \subset \mathcal{Q}(n,a)$ satisfies, for some global constant $C >0$, 
$$ \exists \rho_0 \in \mathcal{T} , \quad  \forall \rho \in \mathcal{V}^+_\mathcal{Q} , \quad d(\rho, W_u(\rho_0) ) \leq Ch^\mathfrak{b}$$
where $\mathfrak{b} = \frac{1}{1+\beta}$ is defined in (\ref{alpha}). 
Then, 
$$\frac{|| \mathfrak{M}^{N_0} U_\mathcal{Q} ||}{ || \alpha||_\infty^{N_0 + n}} \leq h^\gamma$$
\end{prop}

\section{Reduction to a fractal uncertainty principle via microlocalization properties }\label{Section_reduction_to_FUP}
In this section, we reduce the proof of Proposition \ref{Prop_on_clouds} to a fractal uncertainty principle. To do so, we aim at showing microlocalization properties of the operators involved. The disymmetry between $N_0$ and $N_1$ in the decomposition $N=N_0 + N_1$ will appear clearly in this section. Since $N_0$ is below the Ehrenfest time, we can actually use semiclassical tools. By contrast, things are more complicated for operators $U_\mathbf{q}$, with $\mathbf{q} \in \mathcal{Q}(n,a)$ and we'll use methods of propagation of Lagrangian leaves. These methods are inspired by \cite{AnN01}, \cite{AnN02} and \cite{NZ09} and are also used in \cite{NDJ19}.

\subsection{Microlocalziation of $\mathfrak{M}^{N_0}$}

We first state a microlocalization result for $\mathfrak{M}^{N_0}$. This is the easiest one to obtain since $N_0$ is below the Ehrenfest time. 
We recall the definition of $\mathcal{T}_-$ the set of the future trapped points $$\mathcal{T}_- = \bigcap_{ n \in \N} F^{-n} \left( U \right)$$ and focus on $\mathcal{T}_{-}^{\text{loc}} \coloneqq \mathcal{T}_- \cap \mathcal{T}(4\varepsilon_0)$. $\mathcal{T}_-$ is laminated by the weak global stable leaves. Hence, if $\varepsilon_0$ is small enough, ensuring that the boundaries of the local stable leaves $W_s(\rho), \rho \in \mathcal{T}$ do not intersect $\mathcal{T}(4 \varepsilon_0)$, we have $$\mathcal{T}_{-}^{\text{loc}}   \subset \bigcup_{ \rho \in \mathcal{T}} W_s(\rho)$$
When $\mathbf{q} \in \mathcal{A}^{N_0}$ and $\mathcal{V}_\mathbf{q}^- \neq \emptyset$,$\mathcal{V}_\mathbf{q}^- $ lies in a $O\left(h^{\delta_0 \frac{\lambda_0}{\lambda_1}} \right)$ neighborhood of a stable leaves, as stated in the following lemma. In the following, we write 
\begin{equation}
\delta_2 = \delta_0 \frac{\lambda_0}{\lambda_1}
\end{equation}
We recall that we have defined $\mathfrak{b}$ in (\ref{alpha}) and $\tau$ in (\ref{tau}) such that $\alpha < \tau < 1$ and $\delta_2 + \tau >1$ (see \ref{condition_on_tau}). Moreover, $N_0 = \lceil \frac{\delta_0}{ \lambda_1} |\log h | \rceil$. 
\begin{lem}
There exists a global constant $C_2>0$ such that for all $ \mathbf{q} \in \mathcal{A}^{N_0}$ satisfying $\mathcal{V}_\mathbf{q}^- \neq \emptyset$, 
$$ d \left( \mathcal{V}_\mathbf{q}^-, \mathcal{T}_{-}^{\text{loc}} \right) \leq C_2 h^{ \delta_2} $$
\end{lem}

\begin{rem}
In the end of this section, the use of $C_2$ will always refer to the constant appearing in this lemma. On other places, we keep our convention on global constants, noting them always $C$. 
\end{rem}

\begin{proof}
We already know by Lemma \ref{Localization_V_q} that there exists $C>0$ such that if $\mathcal{V}_\mathbf{q}^- \neq \emptyset$, there exists $\rho_0 \in \mathcal{T}$ such that $$d\left(\mathcal{V}_\mathbf{q}^-, W_s(\rho_0)\right) \leq \frac{C}{J_\mathbf{q}^-}$$ 
But $J_\mathbf{q}^- \geq e^{\lambda_0 N_0} \geq C^{-1} h^{- \delta_0 \frac{\lambda_0}{\lambda_1}}$. Finally, $d(\mathcal{V}_\mathbf{q}^-,\mathcal{T}_{-}^{\text{loc}}) \leq Ch^{ \delta_2} $, as required. 
\end{proof}

The following lemma allows us to construct symbols in nice symbol classes with supports in $h^\delta$ neighborhood. Its proof can be found in \cite{DZ16} (Lemma 3.3). 

\begin{lem}\label{lemma_construction_cut_off}
Let $\varepsilon>0$ and $\delta \in [0, \frac{1}{2}[$. Let $V_0(h)\subset V_1(h) \subset \R^d$ be sets depending on $h$ and assume that for $ 0 \leq h \leq 1, d(V_0(h) , V_1(h)^c ) > \varepsilon h^\delta$. Then, there exist a family $\chi_h \in \cinfc(\R^d)$ such that, for all $h \leq 1$, 
\begin{itemize}[nosep]
\item $\chi_h =1$ on $V_0(h)$ ; 
\item $\supp \chi \subset V_1(h) $.
\item For every $\alpha \in \N^d$, there exists $C_\alpha$ depending only on $\varepsilon$ such that for all $x \in \R^d$ and for all $0 < h \leq 1$, 
$$ |\partial^\alpha \chi_h (x) | \leq C_\alpha h^{-\delta |\alpha|}$$
\end{itemize}
\end{lem}

Applying this lemma with $V_0(h) = \mathcal{T}_-^{loc}\left(2C_2h^{\delta_2} \right)$, $V_1(h) = \mathcal{T}_-^{loc}\left(4C_2h^{\delta_2} \right)$ with $\varepsilon = 2C_2$, we consider a family of smooth cut-offs $\chi_h \in S_{\delta_2}^{comp}$ and we can consider it as an element of $S_{\delta_2}^{comp}(U)$ since (at least for $h$ small enough) the support of $\chi_h$ is included in $U$. We are now ready to state the microlocalization property of $\mathfrak{M}^{N_0}$.

\begin{prop}\label{prop_micro_left}
\begin{equation}
\mathfrak{M}^{N_0} = \mathfrak{M}^{N_0} \op(\chi_h) + \hinf_{L^2(Y) \to L^2(Y)}
\end{equation}
\end{prop}

\begin{proof}
We need to show that $\mathfrak{M}^{N_0} ( \op(1 - \chi_h) ) = \hinf$. To do so, we decompose $\mathfrak{M}^{N_0} = \sum_{ \mathbf{q} \in \mathcal{A}^{N_0} } U_\mathbf{q}$. Since the number of terms in this sum grows polynomially with $h$, it is enough to show that 
$$ \forall \mathbf{q} \in \mathcal{A}^{N_0} , U_\mathbf{q}( \op(1 - \chi_h) ) = \hinf$$
with bounds uniform in $\mathbf{q}$. We then consider two cases : 
\begin{itemize}[label = \ding{228}]
\item $\mathcal{V}_\mathbf{q}^- = \emptyset$ :  Lemma \ref{empty_V_q} applies. Indeed, if $m \leq N_0$ and $\mathcal{V}_{q_0 \dots q_{m-1}}^- \neq \emptyset$, we have 
$$ J^-_{q_0 \dots q_{m-1}} \leq e^{m \lambda_1} \leq e^{N_0 \lambda_1} \leq Ch^{-\delta_0}$$Hence, $U_\mathbf{q} = \hinf$, with global constants in the $\hinf$. 
\item $\mathcal{V}_\mathbf{q}^- \neq \emptyset$ : We apply Proposition \ref{prop_symbols_a_q}. Since $J_\mathbf{q}^- \leq C e^{\lambda_1 N_0} \leq C h^{-\delta_0}$, we take some $\delta_1 \in ]\delta_0 , \frac{1}{2} [$ (in particular, $\delta_2 < \delta_1$) and we can write $U_\mathbf{q} = T^{N_0} \op(a_\mathbf{q}^- ) + \hinf$ with $a_\mathbf{q}^- \in S_{\delta_1}^{comp}(U)$ and $\supp a_\mathbf{q}^- \subset \mathcal{V}_\mathbf{q}^-$.
Noticing that $\chi_h =1$ on $\mathcal{V}_\mathbf{q}^- \subset \mathcal{T}_-^{loc} \left( 2C_2h^{\delta_2}  \right) $, the composition formula in $S_{\delta_1}^{comp}$ implies that $\op(a_\mathbf{q}^- ) \op(1- \chi_h) = \hinf$. Since the seminorms of $a_\mathbf{q}^-$ are uniformly bounded in $\mathbf{q}$, the constants appearing in $\hinf$ are uniform in $\mathbf{q}$. 
\end{itemize}
This concludes the proof. 
\end{proof}

\subsection{Propagation of Lagrangian leaves and Lagrangian states}
So as to study the microlocalization of $U_\mathbf{q}$, we'll use the same strategy as in \cite{NDJ19}, themselves inspired by \cite{AnN01}, \cite{AnN02} and \cite{NZ09}. We cannot show that $U_\mathbf{q}$ is a Fourier integral operator since the propagation goes behind the Ehrenfest time. Instead, we show a weaker result which will be enough for our purpose. The idea is to decompose a state $u$ in a sum of Lagrangian states associated with Lagrangian leaves almost parallel to unstable leaves, what we will call horizontal leaves (because we will consider them in charts where the unstable leaves are close to be horizontal).  Studying the precise behavior of these states, we can get fine information on the microlocalization of  $U_\mathbf{q} u$. Roughly speaking, we'll show that if $u$ is a Lagrangian state associated with an original horizontal Lagrangian $\mathcal{L}_{q_0, \theta} \subset \mathcal{V}_{q_0}$, then $U_\mathbf{q} u $ is a Lagrangian state associated with the piece of the evolved Lagrangian $F^n \left( \mathcal{L}_{q_0, \theta} \right)$ inside $\mathcal{V}_\mathbf{q}^+$.

To define "horizontal" Lagrangian leaves, we need to work in adapted coordinate charts in which the notion of horizontality (thinking $W_u(\rho)$ as the reference) makes sens. For this purpose, for $q \in \mathcal{A}$, we consider charts centered around the points $\rho_q$, associated with the fixed macroscopic partition of $\mathcal{T}$ by the $\mathcal{V}_q = B(\rho_q, 2 \varepsilon_0)$. First, we consider symplectic maps

\begin{equation*}
\kappa_q  : W_q \subset U_{k_q} \to V_q\subset \R^2
\end{equation*}
satisfying (we note $(x,\xi)$ the variable in $U$ and $(y,\eta)$ in $\R^2$) : 
\begin{enumerate}[label=(\arabic*), nosep]
\item $B(\rho_q, C \varepsilon_0) \subset W_q$  for some global constant $C \gg 2$; 
\item $\kappa(\rho_q) = 0$ , $d\kappa(\rho_q) (E_u(\rho_q) ) = \R \times \{ 0 \} ; d\kappa(\rho_q) (E_s(\rho_q) ) = \{0\} \times \R $  ;
\item The image of the unstable leave $W_u(\rho_q)$ is exactly $\{ (y,0) , y \in \R \} \cap \tilde{V}_q$.
\end{enumerate}

Theses charts are for instance given by Lemma \ref{weak_adapted_chart} (at this stage, the strong straightening property is not necessary). In these adapted charts where $W_u(\rho_q)$ coincides with $\R \times \{0 \}$, the horizontal Lagrangian leaves will be the of the form

\begin{equation}
\mathcal{C}_\theta \coloneqq \{ (y, \theta ), y \in \R \} 
\end{equation}

Finally, we fix unit vectors on $E_u(\rho_q)$ and $E_s(\rho_q)$, $e_u(\rho_q)$ and $e_s(\rho_q)$, used to defined the unstable and stable Jacobians in section \ref{section_hyperbolic}. Let's write
$$ d\kappa_q(e_u (\rho_q)) = (\lambda_{q,u} ,0) \quad ;  \quad d\kappa_q(e_s (\rho_q)) = (0,\lambda_{q,s})$$
Note $D_q =  \left(\begin{matrix}
\lambda_{q,u} & 0 \\
0& \lambda_{q,s} 
\end{matrix} \right)$. 
 We dilate the chart $\tilde{\kappa}_q$ and define
$$ \tilde{\kappa}_q : \rho \in W_q \mapsto D_q \kappa_q(\rho) \in \tilde{V}_q \coloneqq D_q \left(V_q \right) $$

\subsubsection{Horizontal Lagrangian and their evolution}
Let us fix a word $\mathbf{q} \in \mathcal{A}^n$ and let us define 
\begin{equation}
\mathcal{L}_{q_0, \theta} = \kappa_{q_0}^{-1} \left( \mathcal{C}_\theta \cap V_{q_0} \right) \cap \mathcal{V}_{q_0} \\
\end{equation}
Then, let's define inductively 
\begin{equation}
\mathcal{L}_{q_0 \dots q_j , \theta} = F \left( \mathcal{L}_{q_0 \dots q_{j-1}, \theta}\right) \cap  \mathcal{V}_{q_j} 
\end{equation}
which allows to define $ \mathcal{L}_{\mathbf{q}, \theta}$. One can check that 
\begin{equation}
\mathcal{L}_{\mathbf{q}, \theta} = F^{-1}\left( \mathcal{V}_\mathbf{q}^+ \right) \cap F^{n-1} \left(\mathcal{L}_{q_0, \theta}  \right)
\end{equation}
The term $F^{-1}$ comes from the definition of $\mathcal{V}_{\mathbf{q}}^+$ : $$\rho \in \mathcal{V}_{\mathbf{q}}^+ \iff \forall 1 \leq i \leq n, F^{-i}(\rho) \in \mathcal{V}_{q_{n-i}}$$
Finally, let's define
\begin{align}
\mathcal{C}_{\mathbf{q}, \theta} = \kappa_{q_{n-1}} \left(\mathcal{L}_{\mathbf{q}, \theta}  \right) 
\end{align}
We first focus on one step of the iterative process. 

In $\tilde{V}_q \subset \R^2$, we use the notations 
$\tilde{B}_q(0, r)$ for the cube$ ]-r, r \times ]-r,r[$ . We keep the subscript $q$ to keep trace of the chart in which this cube is supposed to live. Finally, we set $$B_q(0,r) = D_q^{-1} \left( \tilde{B}_q(0, r) \right) \subset V_q$$ $B_q(0,r)$ is simply a rectangle centered at zero with size only depending on $q$ (this is also a ball for some norm in $\R^2$). The advantage of $\tilde{B}_q$ and $\tilde{\kappa}_q$ compared with $B_q$ and $\kappa_q$ will appear below. However, $\tilde{\kappa}_q$ is not symplectic, and for further use, it is not possible to use $\tilde{\kappa}_q$ as a symplectic change of coordinates.

Let $q,p \in \mathcal{A}$ and suppose that $\mathcal{V}_q \cap F^{-1}(\mathcal{V}_p) \neq \emptyset$. As a consequence there exists a global constant $C^\prime >0$ such that $d(F(\rho_q), \rho_p) \leq C^\prime \varepsilon_0$ and if $C$ in (1) of Lemma \ref{weak_adapted_chart} is large enough, we can assume that for some global constant $C_1>0$, 
\begin{equation}\label{inclusion_needed}
\kappa_q \left( \mathcal{V}_q \right) \subset B_q(0, C_1 \varepsilon_0) \subset V_q \quad  \; \quad \kappa_p \circ F \circ \kappa_q^{-1} \left(  B_q(0, C_1 \varepsilon_0) \right)  \subset V_p
\end{equation} 
The following map is hence well defined $$ \tau_{p,q} \coloneqq  \kappa_p \circ F \circ \kappa_q^{-1} :  B_q(0, C_1 \varepsilon_0)  \to \tau_{p,q}(  B_q(0, C_1 \varepsilon_0) ) \subset V_p$$
$\tau_{p,q}$ is nothing but the writing of $F$ between the charts $V_q$ and $V_p$. 
Note that since the number of possible transitions is finite, we can assume that $C_1$ is uniform for all $q,p \in \mathcal{A}$ such that $\mathcal{V}_q \cap F^{-1}(\mathcal{V}_p) \neq \emptyset$. 

We also adopt the following definitions and notations : 

\begin{defi}
Let $G_q : ]-C_1 \varepsilon_0 , C_1\varepsilon_0[ \to ]-C_1 \varepsilon_0 , C_1\varepsilon_0[ $ be a smooth map. It represents the horizontal Lagrangian 
$$ \mathcal{L}_{G_q} \coloneqq D_q^{-1} \big( \left\{ (y, G_q(y) , y \in  ]-C_1 \varepsilon_0 , C_1\varepsilon_0[  \right\} \big)  \subset B_q ( 0,C_1 \varepsilon_0) \subset V_q$$
We say that such a Lagrangian lies in the $\gamma$-unstable cone if 
$$||G_q^\prime||_\infty \leq \gamma$$
and we note $G_q \in \mathcal{C}^u_q(C_1\varepsilon_0, \gamma)$. 
\end{defi}
\vspace{0.2cm}
\begin{rem}
This is where the use of $\tilde{\kappa}_q$ and $\tilde{B}_q$ turns out to be useful : to represent horizontal Lagrangian in $V_q$, we use the  cube $\tilde{B}_q(0,C_1\varepsilon_0) \subset \tilde{V}_q$ of fixed size. 
\end{rem}

With this definition, we show in the following lemma an invariance property of the $\gamma$-unstable cones :

\begin{lem}\label{lemma_propagation_lagrangian_leaves}
There exist global constants $C >0,C_1 >0$ such that if $\varepsilon_0$ is sufficiently small, then the following holds. 

For every $G_q \in \mathcal{C}^u_q(C_1\varepsilon_0, C\varepsilon_0)$, there exists $G_p \in \mathcal{C}^u_p(C_1\varepsilon_0, C\varepsilon_0)$ such that
\begin{enumerate}[label=(\roman*)]
\item $\tau_{p,q}\left( \mathcal{L}_{G_q} \right)\cap B_p(0,C_1 \varepsilon_0) =  \mathcal{L}_{G_p} $ ; 

\item For some  global constants $C_l, l \geq 2$, $||G_q||_{C^l} \leq C_l \implies ||G_p||_{C^l} \leq C_l$ ; 
\end{enumerate}
Moreover, let's define $\phi_{qp} : ]-C_1\varepsilon_0, C_1 \varepsilon_0[ \to \R$ by $$
y_q= \phi_{qp} (y_p) \iff (y_p,G_p(y_p)) = D_p \circ \tau_{pq} \circ D_q^{-1} \Big( (\phi_{qp}(y_p), G_q \circ \phi_{qp} (y_p)  \Big) $$
Then, $\phi_{pq}$ is smooth contracting diffeomorphism onto its image. In particular, there exists a global constant $\nu<1$ such that 
$||\phi_{pq}^\prime||_\infty \leq \nu $. 
\end{lem}

\begin{proof}
Take $C_1$ large but fixed (with conditions further imposed) and assume that $\varepsilon_0$ is small enough so that (\ref{inclusion_needed}) holds.
Let us note $\lambda_q = J^u_1(\rho_q) > 1$ and $\mu_q = J^s_1(\rho_q) < 1$ and let us fix some global $\nu$ satisfying
$$ \forall q \in \mathcal{A},  \max( \lambda_q^{-1}, \mu_q) < \nu < 1$$
Recall that $e_u$ and $e_s$ are $C^{1, \varepsilon}$ in $\rho$.  We write $\partial_y$ and $\partial_\eta$ to denote the unit vector of $\R \times \{0 \} $ and $\{0 \} \times \R$ respectively. We fix a constant $C >0$ with conditions imposed further and we assume that $||G_p^\prime||_\infty \leq C \varepsilon_0$. 
We note $\tilde{\tau} = D_p \circ \tau_{p,q} \circ D_q^{-1}$ (we drop the subscript for $\tilde{\tau}$ to alleviate the notations). In the computations below, the implied constants in the $O$ are global constants (depending also on the choices on $\kappa_q$):
\begin{itemize}[label = *]
\item $\tilde{\tau}(0) = \tilde{\kappa}_p \circ F (\rho_q) = O (\varepsilon_0) $; 
\item $d\tilde{\tau}(0)  = d \tilde{\kappa}_p(F(\rho_q) )  \circ dF(\rho_q) \circ \left[ d\tilde{\kappa}_q (\rho_q) \right]^{-1} $ ; 
\item $d\tilde{\tau}(0) (\partial_y) = d \tilde{\kappa}_p (F(\rho_q) ) ( \lambda_q e_u(F(\rho_q) ) )  =  \lambda_q \left( d \tilde{\kappa}_p (\rho_p) + O(\varepsilon_0) \right) \left( e_u(\rho_p) + O(\varepsilon_0) \right) = \lambda_q \partial_y + O(\varepsilon_0)$, where we use the Lipschitz regularity of $\rho \mapsto e_u(\rho) $ in the second equality ; 
\item Similarly, $d\tilde{\tau}(0) (\partial_\eta) = \mu_q \partial_\eta + O(\varepsilon_0)$; 
\end{itemize}
(this is here that we use the renormalization of $\kappa_q$ into $ \tilde{\kappa}_q$).
Eventually, we use the fact that $\tilde{\tau} - \tilde{\tau}(0) - d\tilde{\tau}(0) =O(C_1\varepsilon_0)_{C^1 ( B (0, C_1 \varepsilon_0) ) }$ and we get that 
\begin{equation}
\tilde{\tau}(y,\eta) = (\lambda_q y  + y_r(y, \eta) , \mu_q \eta + \eta_r(y,\eta) )  , (y,\eta) \in \tilde{B}_q(0, C_1\varepsilon_0)
\end{equation}
where $y_r(y, \eta)$ and $\eta_r(y, \eta)$ are $O(C_1\varepsilon_0)_{C^1}$. 
Before going further, let us show that we can fix $C_1$ such that  
\begin{equation}\label{condition_C1}
(y,\eta) \in \tilde{B}_q(0, C_1\varepsilon_0) \implies  \left|\mu_q \eta + \eta_r(y,\eta) )\right|\leq C_1 \varepsilon_0 
\end{equation}
To do so, let us note that in fact $\tilde{\tau} - \tilde{\tau}(0) - d\tilde{\tau}(0) =O\left( (C_1\varepsilon_0)^2 \right)_{C^0 ( B (0, C_1 \varepsilon_0) ) }$ and hence if $(y,\eta) \in \tilde{B}_q(0, C_1\varepsilon_0)$ we have : 
$$ |\eta_r(y,\eta)| = O(\varepsilon_0) + O\left( (C_1\varepsilon_0)^2 \right)_{C^0 ( B (0, C_1 \varepsilon_0) ) } \leq C^\prime \varepsilon_0 \left ( 1 + C_1^2 \varepsilon_0\right)$$ 
Assume that $C_1$ is large enough such that $\nu C_1 +  C^\prime < C_1 \frac{\nu +1}{2}$. If $(y,\eta) \in \tilde{B}_q(0, C_1\varepsilon_0)$, we have 
$$\left|\mu_q \eta + \eta_r(y,\eta) )\right| \leq \nu C_1 \varepsilon_0 + C^\prime \varepsilon_0 \left ( 1 + C_1^2 \varepsilon_0\right)  \leq \left(  C_1 \frac{\nu +1}{2} + C_1^2 \varepsilon_0 \right) \varepsilon_0 $$
This fixes $C_1$. Since $C_1$ is now a global fixed parameter, we can remove it from the $O$ in the estimates. 
If $\varepsilon_0$ is small enough, depending on our choice of $C_1$, (\ref{condition_C1}) holds. \\
To write the image of the leaf as a graph, we observe that, if $\varepsilon_0$ is small enough (depending only on global parameters) the map $$\psi : y \in ]-C_1 \varepsilon_0, C_1 \varepsilon_0 [ \mapsto \lambda_q y + y_r(y, G_q(y) ) $$ is expanding and we can impose $|\psi^\prime| \geq  \nu^{-1}$. 
In particular, $\im \psi $ contains an interval of size $2 \nu^{-1} C_1 \varepsilon_0$. Moreover, $\psi(0) = y_r(0,G_q(0)) \leq ||y_r||_{C^1} |G_q(y) | =O(\varepsilon_0^2) $. We claim that if $\varepsilon_0$ is small enough, $\im \psi$ contains $ ]-C_1 \varepsilon_0, C_1 \varepsilon_0[$.  Indeed, it suffices to have
$$ \nu^{-1} C_1 \varepsilon_0 - |\psi(0)| \geq C_1 \varepsilon_0$$
But we have 
$$C_1 \varepsilon_0 + |\psi(0)| \leq C_1 \varepsilon_0 (1 + O( \varepsilon_0) ) \leq C_1 \varepsilon_0 \nu^{-1} $$
if $1 + O( \varepsilon_0)  \leq \nu^{-1}$, condition that can be satisfied if $\varepsilon_0$ is small enough. 
Hence, $\phi \coloneqq \phi_{pq}  = \psi^{-1}_{| ]-C_1 \varepsilon_0, C_1 \varepsilon_0[}$ is well defined and we set 
\begin{equation}\label{Gq_ti_Gp}
G_p ( y) = \mu_q G_q( \phi(y) ) + \eta_r \Big( \phi(y) , G_q(\phi(y) )\Big)   , y \in  ]-C_1 \varepsilon_0, C_1 \varepsilon_0[
\end{equation}
By definition, it is clear that $\tau_{p,q}\left( \mathcal{L}_{G_q} \right)\cap B_p(0,C_1 \varepsilon_0) =  \mathcal{L}_{G_p} $  and $(y, G_p(y) ) = \tilde{\tau} \Big( \phi(y) , G_q(\phi(y) ) \Big) $. $\phi$ is obviously a smooth contracting diffeomorhpism and $||\phi^\prime|| \leq \frac{1}{\inf |\psi^\prime(y) |} \leq \nu$. 
Moreover, due to (\ref{condition_C1}), $ |G_p(y) | \leq C_1 \varepsilon_0$.
To prove that $G_p \in \mathcal{C}^u_p(C_1 \varepsilon_0, C\varepsilon_0)$, we compute : 
\begin{align*}
&G_p^\prime(y) = \mu_q G_q^\prime(\phi(y) ) \times \phi^\prime(y) + \left(\partial_y \eta_r + \partial_\eta \eta_r \times G_q^\prime(\phi(y) )  \right)  \phi^\prime(y)\\
&|G_p^\prime(y)| \leq \nu^2 C \varepsilon_0 + O ( \varepsilon_0 (1+ C \varepsilon_0)) \nu \leq [\nu^2 C + \nu C^\prime ( 1 + C \varepsilon_0)] \varepsilon_0
\end{align*}
for some global $C^\prime>0$. If we assume $\nu^2 + \varepsilon_0 C^\prime \nu <1$, which is possible if $\varepsilon_0$ is small enough, then we can choose $C$ large enough satisfying 
$$ C \times \left( \nu^2 + \nu  C^\prime \varepsilon_0 \right) + \nu C^\prime \leq C$$ 
This ensures that $||G_p^\prime||_\infty \leq C \varepsilon_0$.
 
Finally, we prove (ii) by induction on $l$ : the case $l=1$ is done. Assume that there exists a constant $C_l$ such that $||G_q||_{C^l} \leq C_l \implies ||G_p||_{C^l} \leq C_l$. We want to find a constant $C_{l+1}$ fitting for the $C^{l+1}$ norm. Using (\ref{Gq_ti_Gp}), we see by induction that the $(l+1)$ derivatives of $G_p$ has the form 
$$ G_p^{(l+1)}(y) = \phi^\prime(y)^{l+1} \times G_q^{(l+1)}(y)  \times \Big(1 + \partial_\eta \eta_r (y, \phi(y) ) \Big) + P_y\left(G_q(y), \dots, G_q^{(l)}(y) \right)$$
where $P_y(\tau_0, \dots, \tau_l)$ is a polynomial with smooth coefficients in $y$. Hence, there exists a constant $M(C_l)$ such that for $y \in ]-C_1\varepsilon_0, C_1 \varepsilon_0[$, $\left|P_y\left(G_q(y), \dots, G_q^{(l)}(y) \right)\right| \leq M(C_l)$. Since 
$$ \left| \phi^\prime(y)^{l+1} \Big( 1 + \partial_\eta \eta_r (y, \phi(y) ) \Big) \right| \leq \nu(1 + \varepsilon_0 C^\prime) \coloneqq \nu_1$$
 if $\varepsilon_0$ is small enough ensuring that $\nu_1 <1$, we can take 
$$C_{l+1} = \max \left( C_l, \frac{M(C_l)}{1 - \nu_1}\right) $$
Indeed, with such a constant, assuming that $||G_q||_{C^{l+1}} \leq C_{l+1}$, we have
$$ |G_p^{(l+1)}(y)| \leq C_{l+1}  \nu_1 + M(C_l) \leq C_{l+1}$$
\end{proof}

Armed with this lemma, we can now iterate the process and get the following proposition describing the evolution of the Lagrangian $\mathcal{C}_{\mathbf{q}, \theta}$. 
\begin{prop}\label{lagrangian_leaves_evolved}
Assume that $\varepsilon_0$ is small enough. Then, for every $n \in \N^*$, $\mathbf{q} \in \mathcal{A}^n$ , and $\theta \in \R$, there exists an open subset $I_{\mathbf{q}, \theta} \subset \R$ and a smooth map $G_{\mathbf{q}, \theta}$ such that : 
\begin{itemize}
\item $ \mathcal{C}_{\mathbf{q}, \theta} = \Big\{ (y, G_{\mathbf{q}, \theta}(y) ) , y \in I_{\mathbf{q}, \theta} \Big\}$ ;
\item $||G_{\mathbf{q}, \theta}^\prime||_\infty \leq C \varepsilon_0$ for some global constant $C$; 
\item For every $l \geq 2$, $||G_{\mathbf{q}, \theta}||_{C^l} \leq C_l$ for some global $C_l$; 
\item If $\phi_{\mathbf{q}, \theta} : I_{\mathbf{q}, \theta} \to \R$ is defined by 
$$ \kappa_{q_{n-1}} \circ F^{n-1} \circ \kappa_{q_0}^{-1} \left( \phi_{\mathbf{q}, \theta}(y), \theta\right) = (y, G_{\mathbf{q}, \theta}(y) ) $$
Then, for some global constants $C>0$ and $0<\nu<1$, $||\phi_{\mathbf{q}, \theta}^\prime || \leq C\nu^{n-1}$ . 
\end{itemize} 
\end{prop}

\begin{proof}
Assume that $\mathcal{L}_{\mathbf{q}, \theta} \neq \emptyset$, otherwise, there is nothing to prove. In particular, we can restrict our attention to small $\theta$, $|\theta| \leq C_1 \varepsilon_0$. As a consequence, for every $i \in \{1, \dots, n \}$, $F (\mathcal{V}_{q_{i-1}} ) \cap \mathcal{V}_{q_i} \neq \emptyset$. Hence, we can consider the maps $\tau_i \coloneqq \tau_{q_i, q_{i-1} }$ and since we assume that $\kappa_{q_i} \left(\mathcal{V}_{q_i} \right) \subset B_{q_i}(0,C_1 \varepsilon_0)$, 
$$C_{q_0 \dots q_{i} , \theta} = \tau_i \left( C_{q_0 \dots q_{i-1} , \theta} \right) \cap \kappa_{q_i} (\mathcal{V}_{q_i} )  $$
We start with a constant function $G_0  \in \mathcal{C}^u_0(C_1 \varepsilon_0, 0)$ such that $\mathcal{L}_{G_0} =\mathcal{C}_\theta$ (it suffices to take $G_0= \lambda_{q_0,s} \theta$) and we inductively apply the previous lemma to show the existence of a family $G_j \in \mathcal{C}^u_{q_j}(C_1 \varepsilon_0, C \varepsilon_0),  0 \leq j \leq n-1$, such that 
\begin{enumerate}[label=(\roman*)]
\item $\tau_i \left( \mathcal{L}_{G_i} \right) \cap B_{q_i}(0,C_1 \varepsilon_0)=  \mathcal{L}_{G_{i+1}}$ ; 
\item$||G_i||_{C^l} \leq C_l$; 
\item If we define $\phi_i : ]-C_1\varepsilon_0, C_1 \varepsilon_0[ \to ]-C_1\varepsilon_0, C_1 \varepsilon_0[ $ by 
$$ (y,G_i(y)) =  D_{q_i} \circ \tau_i \circ D_{q_{i-1}}^{-1}  \Big( \phi_{i}(y), G_{i-1} \circ \phi_{i} (y)  \Big) $$
then there exists $\nu<1$ such that
$||\phi_{i}^\prime||_\infty \leq \nu $. 
\item $\mathcal{C}_{q_0 \dots q_i, \theta}$ is an open subset of $\mathcal{L}_{G_i}$. 
\end{enumerate}
We have 
$$ \mathcal{L}_{G_{n-1}} = D_{q_{n-1}}^{-1} \Big( \{ (y, G_{n-1}(y) ) , y \in ]-C_1 \varepsilon_0, C_1 \varepsilon_0 [ \} \Big) $$
This can be also written 
$$  \mathcal{L}_{G_{n-1}} =  \left\{ \left( y, \lambda_{q_{n-1},s}^{-1}G_{n-1}(\lambda_{q_{n-1},u}y) \right) , |y| <\lambda_{q_{n-1},u}^{-1} C_1 \varepsilon_0 \right\} $$
It suffices to consider 
\begin{align*}
&G_{\mathbf{q}, \theta}(y) = \lambda_{q_{n-1},s}^{-1} G_{n-1}(\lambda_{q_{n-1},u} y )\\
&I_{\mathbf{q}, \theta}  = \Big\{ y \in ]- \lambda_{q_{n-1},u}^{-1}C_1 \varepsilon_0,  \lambda_{q_{n-1},u}^{-1}C_1 \varepsilon_0 [ , (y, G_{\mathbf{q}, \theta}(y) ) \in \mathcal{C}_{\mathbf{q}, \theta} \Big\}\\ &\phi_{\mathbf{q}, \theta}(y) = \lambda_{q_1,u}^{-1} \phi_1 \circ \dots \circ \phi_{n-1}( \lambda_{q_{n-1},u} y ) 
\end{align*}
\end{proof}

\subsubsection{Evolution of Lagrangian states}
Once we've studied the evolution of the Lagrangian leaves starting from $\mathcal{C}_\theta$, we can study the evolution of the corresponding Lagrangian states. In our case, since the leaves stay rather horizontal, the form of the Lagrangian states we'll consider is the simplest : $$ a (x) e^{i \psi(x)/h}$$ 
where $a$ is an amplitude and $\psi$ a generating phase function. It is associated with the Lagrangian, 
$$ \mathcal{L}= \{ (y, \psi^\prime(y) ) ,  y \in \supp a \}$$

For $q \in \mathcal{A}$, we quantize $\kappa_q$. Remind that we denoted $k_q$ the integer such that $\mathcal{V}_q \Subset U_{k_q}$. There exist Fourier integral operators $B_q, B_q^\prime \in I_0^{comp}(\kappa_q) \times I_0^{comp}(\kappa_q^{-1})$, 

\begin{align*}
B_q : L^2 (Y_{k_q}) \to L^2(\R) ; \\
B_q^\prime : L^2 (\R) \to L^2 (Y_{k_q} ) 
\end{align*}
such that they quantize $\kappa_q$ in a neighborhood of $\kappa_q \left( \overline{ \mathcal{V}_q}  \right) \times \overline{\mathcal{V}_q}$. Moreover, we impose that $\WF (B_q B^\prime_q)$ is a compact subset of $\R^2$.
We will still denoted $B_q$ and $B_q^\prime$ the operators 

\begin{align*}
B_q = ( 0 ,\dots, \underbrace{B_q}_{k_q}, \dots ,0) : L^2 (Y) \to L^2 (\R) \quad ; \quad  
B_q^\prime ={}^t (0 ,\dots ,\underbrace{B_q^\prime}_{k_q}, \dots,0) : L^2 (\R) \to L^2(Y) 
\end{align*}
If $\supp (c_q) \subset \mathcal{V}_q$ and if $C$ denotes the operator valued matrix with only one non zero entry $\op(c_q)$ in position $(k_q,k_q)$, then as operators $L^2(Y) \to L^2(Y)$, 
$$ B_q^\prime B_q C = C + \hinf  \; ; \; CB_q^\prime B_q= C + \hinf$$

The proposition we aim at proving in the following :
\begin{prop}\label{Prop_evolution_lagrangian_state}
Fix $C_0 >0$. For every $n \in \N,\mathbf{q} \in \mathcal{A}^n$ and $\theta \in \R$ satisfying 
\begin{equation}
n \leq C_0 |\log h| \; ; \; |\theta| \leq C_0
\end{equation}
and for every $N \in \N$, there exists a symbol $a_{\mathbf{q}, \theta, N} \in \cinfc(I_{\mathbf{q}, \theta})$ such that : 
\begin{enumerate}[label= (\roman*)]
\item $U_\mathbf{q} \left( B_{q_0}^\prime e^{i \frac{\theta \cdot}{h} } \right) =  MA_{q_{n-1}}B^\prime_{q_{n-1}}  \left( e^{i \frac{\psi_\mathbf{q} }{h} } a_{\mathbf{q}, \theta, N} \right) + O(h^N)_{L^2}$
\item$ || a_{\mathbf{q}, \theta, N}||_{C_l} \leq C_{l,N} h^{-C_0 \log B} $
\item There exists $\delta >0$ such that $ d\left( \supp(a_{\mathbf{q},\theta,N}), \R \setminus I_{\mathbf{q}, N, \theta} \right) \geq \delta $
\end{enumerate}
where $\psi_{\mathbf{q}, \theta}$ is a primitive of $G_{\mathbf{q}, \theta}$ and $B>0$ is a global constant. 
\end{prop}

\begin{rem}
\begin{itemize}
\item  As usual, $\delta, C_{l,N}$ and $C_N$ depend only on $F, A_q, B_q, B^\prime_q, \kappa_q$ and the indices indicated in their notations. 
\item In other words, the Lagrangian state $e^{i\frac{\theta \cdot}{h}}$ is changed to a Lagrangian state associated with $\mathcal{C}_{\mathbf{q}, \theta}$. 
\end{itemize}
\end{rem} 

The end of this subsection is devoted to the proof of Proposition \ref{Prop_evolution_lagrangian_state}. In the rest of this section, we fix a constant $C_0>0$ and we work with a fixed word $\mathbf{q} \in \mathcal{A}^n$ with length $n \leq  C_0 |\log h|$ and a fixed momentun $|\theta| \leq C_0$. From now on and until the end of the proof, the constants below will always be  uniform in $\mathbf{q}, \theta$ satisfying the previous assumption. They will depend on global parameters and on $C_0$. If they depend on other parameters, we will specify it with subscripts.  This is also the case for implicit constants in $O$ (such as in $\hinf$). 

\paragraph{Preparatory work.}
We first note the following fact : if  $\mathcal{V}_q \cap F^{-1}(\mathcal{V}_p) = \emptyset$, $A_p M A_q = \hinf$. As a consequence, if $ \mathcal{V}_{q_{i-1}} \cap F^{-1}(\mathcal{V}_{q_i}) =\emptyset$ for some $i$, then $U_\mathbf{q} = \hinf$. In the sequel, it is enough to consider words $\mathbf{q}$ for which $ \mathcal{V}_{q_{i-1}} \cap F^{-1}(\mathcal{V}_{q_i}) \neq \emptyset$ for $1 \leq i \leq n-1$. 

We consider symbols $\tilde{a}_q$ such that $\supp(\tilde{a}_q) \subset \mathcal{V}_q$ and $\tilde{a}_q \equiv 1$ on $\supp(\chi_q)$. We denote $\tilde{A}_q = \op(\tilde{a}_q)$ (as usual thought as a diagonal operator valued matrix). 
The following computations holds since $n = O(\log h) $ and $ ||MA_q|| \leq ||\alpha||_\infty + o(1)$ uniformly in $q$ : 

\begin{align*}
U_\mathbf{q} B_{q_0}^\prime= &M A_{q_{n-1}} \tilde{A}_{q_{n-1}} M A_{q_{n-2}} \tilde{A}_{q_{n-2}} \dots M A_{q_1} \tilde{A}_{q_1} M A_{q_0} B_{q_0}^\prime+ \hinf  \\
&=M   A_{q_{n-1}} B^\prime_{q_{n-1}} B_{q_{n-1} } \tilde{A}_{q_{n-1}} M \dots M  A_{q_1} B^\prime_{q_1} B_{q_1}  \tilde{A}_{q_1} M A_{q_0} B_{q_0}^\prime  + \hinf 
\end{align*}
We set $T_{p,q} = B_p \tilde{A}_p M A_q  B_q^\prime$ and $M_q = MA_q B^\prime_q$, which allows us to write 

$$ U_\mathbf{q} B_{q_0}^\prime = M_{q_{n-1}} T_{q_{n-1}, q_{n-2}} \dots T_{q_1, q_0} + \hinf $$
For $p,q \in \mathcal{A}$ with $ \mathcal{V}_{q} \cap F^{-1}(\mathcal{V}_{p}) \neq \emptyset$, $T_{q,p} \in I_0^{comp}(\tau_{p,q})$. Moreover, the previous computations have shown that $\tau_{p,q}$ has the form
$$\tau_{p,q} (y,\eta) =  (\lambda_{p,q} y  + y_r(y, \eta) , \mu_{p,q} \eta + \eta_r(y,\eta) )  , (y,\eta) \in B_q(0, C_1\varepsilon_0)$$
where $y_r(y, \eta)$ and $\eta_r(y, \eta)$ are $O(\varepsilon_0)_{C^1}$.
This time, $\lambda_{p,q}, \mu_{p,q}$ are simply constants uniformly bounded from below and from above for $p, q \in \mathcal{A}$ (recall that $B_q(0,C_1 \varepsilon_0)$ is a rectangle in $\R^2$, built from the cube $\tilde{B}_q(0,C_1 \varepsilon_0)$ adapted to the definition of the unstable Jacobian). If $\varepsilon_0$ small enough, the projection  $\pi : (y,\eta,x,\xi) \in  \mathcal{L}_{q,p} \mapsto (y,\xi) \in \R^2$ is a diffeomorphism onto its image. where 
$$ \mathcal{L}_{q,p} = \Big\{ (\tau_{q,p}(x,\xi) , x, -\xi) ,( x, \xi) \in B_q(0, C_1 \varepsilon_0) \Big\}$$
is  the twisted graph of $\tau_{p,q}$. As a consequence, there exists a smooth phase function $S_{p,q}$ defined in an open set $\Omega_{p,q}$ of $\R^2$, generating $  \mathcal{L}_{p,q} $ locally i.e. 
$$  \mathcal{L}_{p,q}  \cap \tau_{p,q}\left(  B_q(0,C_1 \varepsilon_0) \right)  \times B_q(0,C_1 \varepsilon_0) = \Big\{ (y, \partial_y S_{p,q}(y,\xi), \partial_\xi S_{p,q}(y,\xi), - \xi), (y,\xi) \in \Omega_{q,p} \Big\} $$
Hence, $T_{p,q}$ can be written in the following form, up to a $\hinf$ remainder and for some symbol $\alpha_{p,q}(\cdot ; h) \in \cinfc(\Omega_{p,q} )$:

\begin{equation}\label{fio_in_nice_form}
T_{p,q} u(y) = \frac{1}{2\pi h}\int_{\R^2} e^{\frac{i}{h}( S_{p,q}(y,\xi) - x\xi)} \alpha_{p,q}(y,\xi ; h) u(x) dx d \xi
\end{equation}
Moreover, due to the operators $\tilde{A}_p$ and $A_q$ in the definition of $T_{p,q}$, we can assume that 

$$ (y,\xi) \in \supp (\alpha_{p,q} ) \implies (\partial_\xi S_{p,q}(y,\xi), \xi) \in \kappa_q (\supp a_q ) , (y, \partial_yS_{p,q}(y,\xi) ) \in \kappa_p (\supp \tilde{a}_p ) $$ 
In the sequel, we write 
$$\mathcal{C}_i  = \mathcal{C}_{q_0 \dots q_i, \theta}  $$
and we change the subscripts $(q_{i-1}, q_i)$ to $i$ in all the objects $T, \alpha, S, \tau$. Due to the previous results, we can write 
$\mathcal{C}_i = \Big\{ (y, G_i(y) ) , y \in I_i  \Big\}$ with $I_i \coloneqq I_{q_0 \dots q_i, \theta}$ and  $G_i \coloneqq G_{q_0 \dots q_i, \theta}$. We also have projection maps $\Phi_i :I_i \to \R$ defined by : 
$$ \tau_i \circ \dots \circ \tau_1 (\Phi_i(y), \theta) = (y, G_i(y) ) $$
satisfying $||\Phi_i^\prime||_\infty \leq C\nu^i < 1$. Moreover, if we note the intermediate corresponding projection $\phi_i \coloneqq \Phi_i \circ \Phi_{i-1}^{-1} : I_i \to I_{i-1}$, we observe that $\phi_i$ is constructed using the properties of $F$ and $G_{i-1}$ (see the proof of Lemma \ref{lagrangian_leaves_evolved}) and hence, for every $l$, $||\phi_i||_{C^l} \leq C_l$ for some $C_l$ not depending on $\mathbf{q}, \theta$ nor $i$. \\
 For $0 \leq i \leq n-1$, we consider a primitive $\psi_i$ of $G_i$ so that $\mathcal{C}_i$ is generated by $\psi_i$ i.e. 
 $$\mathcal{C}_i = \Big\{ (y, \psi_i^\prime(y) , y \in I_i \Big\}$$
 The following lemma can be found in \cite{NZ09} (Lemma 4.1). We state it without proof, since it is the reference but it is a direct application of the stationary phase theorem. 
 
 \begin{lem}\label{iteration_lemma}
 Pick $i \in \{1, \dots ,n-1\}$. \\
 For any $a \in \cinfc(I_{i-1})$, the application of $T_{i}$ to the Lagrangian state $a e^{i\frac{ \psi_{i-1}}{h}}$
 associated with $\mathcal{C}_{i-1}$ gives a Lagrangian state associated with $\mathcal{C}_{i}$ and satisfies 
 \begin{equation}
 T_i \left( a e^{i\frac{ \psi_{i-1}}{h}} \right) (y) = e^{i\frac{ \beta_i}{h}} e^{i\frac{ \psi_{i}(y)}{h}} \left(\sum_{j=0}^{N-1} b_j(y) h^j+ h^N r_N(y;h)  \right)
 \end{equation}
 where, if we note $x= \phi_i(y)$, 
 $b_j(y) = (L_{j,i}(x,D_x)  a) (x)$ for some differential operator $L_{j,i}$ of order $2j$ with  smooth coefficients supported in $I_{i-1}$ and $\beta_i \in \R$. Moreover, one have : 
 \begin{itemize}
 \item $b_0(y) = \frac{\alpha_i(y,\xi) }{|\det D^2_{y,\xi} S_i(y,\xi) |^{1/2}} \left| \phi_i^\prime(y) \right|^{1/2}  a (x) $ with $\xi = \psi^\prime_{i-1}(x)$;
 \item $||b_j||_{C^l(I_i)} \leq C_{l,j} ||a||_{C^{l+2j} (I_{i-1})} , l \in \N, 0 \leq j \leq N-1$ ; 
 \item $||r_N ||_{C^l(I_i)} \leq C_N ||a||_{  C^{l+1 + 2N}(I_{i-1})} $
 \end{itemize}
 The constants $C_N$ and $C_{l,j}$ depend on $\tau_i, \alpha_i, ||\psi_i^{(m)}||_{\infty, I_i}$.
 \end{lem}

\begin{rem}\text{ } \\
\begin{itemize}
\item  In particular, in virtue of Proposition \ref{lagrangian_leaves_evolved}, the constant $C_{l,j}$ and $C_N$ can be chosen uniform in $\mathbf{q}, \theta$ as soon as they satisfy the required assumptions. $|q| \leq C_0 |\log h| , \theta \leq C_0$. 
\item Without loss of generality, we can replace $\psi_i$ by $\beta_i + \psi_i$ (this actually corresponds to fixing an antiderivative on  $\mathcal{C}_{i+1}$) and hence we can assume that $\beta_i = 0$. 
\item The properties on the support of $\alpha_i$ imply the following ones on the support of the differential operators $L_{j,i}$ : 
\begin{equation}\label{support_L_i_j}
 y \in \supp L_{j,i} \implies (y, \psi^\prime_i(y) ) \in \kappa_{q_i} ( \supp \tilde{a}_{q_i} )  \cap \tau_{i-1} \circ \kappa_{q_{i-1}} ( \supp a_{q_{i-1}} )
\end{equation}
\end{itemize}
\end{rem}

\paragraph{Iteration formulas and analysis of the symbols}
Then, we iterate this lemma starting from $\psi_0 (x) = x \cdot \theta$, in the spirit of Proposition 4.1 in \cite{NZ09}. 
In the sequel, we adopt the following convention : we note $x_k$ the variable in $I_k$ and we naturally denote $(x_k,x_{k-1}, \dots, x_1, x_0)$ the sequence defined by $x_{i-1}= \phi_i(x_i)$. 
We also note 
$$\beta_i(x_i) =   \frac{\alpha_i(x_i,\xi) }{|\det D^2_{x_i,\xi} S_i(x_i,\xi)|^{1/2}} \quad ; \quad  \xi = \psi_{i-1}^\prime(x_{i-1})$$

$$f_i(x_i)=\beta(x_i) \left| \phi_i^\prime(x_i) \right|^{1/2}$$   
 We fix a constant $B >0$ (depending only on $F,A_q, B_q, B^\prime_q,C_0)$ satisfying for all $1\leq i \leq n-1$, 

\begin{align*} 
\sup_{x_i\in I_{i}} |\beta_i(x_i) |\leq B  \\
||T_i ||  \leq B 
\end{align*} 
Roughly speaking, $B$ is of order $||\alpha||_\infty$, but in this part, the precise value of $B$ is not relevant. Finally, note that there exists $\nu <1$ (again depending only on $F,A_q, B_q, B^\prime_q$) such that $|\phi_i^\prime(x_i) | \leq \nu$ for $x_i \in I_i$. 
Fix $N \in \N$ and denote 
\begin{equation} \tilde{N} = 1 + \lceil N + C_0 \log B \rceil 
\end{equation}
We iteratively define a sequence of symbols $a_{i,j}$, $0 \leq i \leq n-1$, $0 \leq j \leq \tilde{N} -1$ by $a_{0,0}=1, a_{0,j} = 0$ and for $0 \leq j \leq \tilde{N} -1$ 

\begin{equation}\label{iteration_formula}
a_{i,j}(x_i) = \sum_{p=0}^j L_{j-p, i} ( a_{i-1,p} )  (x_{i-1})
\end{equation}

The following lemma controls the growth of the symbols. The proof is a precise analysis of the iteration formula (\ref{iteration_formula}) and is rather technical. We write the detailed proof in the appendix (See subsection \ref{Proof_of_lemma_bound_symbol}) and refer the reader to \cite{NZ09} (Proposition 4.1), where the author lead the same analysis (but in the case $B=1$). 
\begin{lem}\label{Lemma_bound_symbol}
For all $j \in \{0, \dots , \tilde{N} -1 \}, l \in \N$, there exists $C_{j,l} >0$ such that for all $i \in \{0, \dots, n-1 \}$, one has
\begin{equation}\label{bound_symbol}
||a_{i,j}||_{C^l(I_i)} \leq C_{j,l} \left(B \nu^{1/2} \right)^i (i+1)^{l + 3j}
\end{equation}
\end{lem}

\begin{rem}
Again, what is important is the fact that $C_{j,l}$ does \emph{not} depend on $\mathbf{q}, n, \theta$ nor $i$ :  it depends on $C_0$ and global parameters.
\end{rem}

\paragraph{Control of the remainder}
Let us call $r_{i,N}(a)$ the remainder appearing in Lemma \ref{iteration_lemma}. Define inductively $(R_{i,\tilde{N}})$ by $R_{0,\tilde{N}}=0$ and 
\begin{equation}
R_{i+1,\tilde{N}}= e^{-\frac{i \psi_{i+1}}{h}} T_{i+1} \left( e^{ \frac{i \psi_i}{h}} R_{i,\tilde{N}} \right) + \sum_{j=0}^{\tilde{N}-1} r_{i+1,\tilde{N}-j} (a_{i,j})
\end{equation}
This definition ensures that for all $1 \leq i \leq n$, 
\begin{equation}
T_i \dots T_1  \left( e^{ \frac{i \psi_0}{h} } \right) = e^{i\frac{ \psi_{i}(y)}{h}} \left(  \sum_{j=0}^{\tilde{N}-1} h^j a_{i,j} + h^{\tilde{N}} R_{i,\tilde{N}} \right)
\end{equation}

\begin{lem}
There exists $C_{\tilde{N}}$ depending only on $\tilde{N}$, $C_0$ and global parameters such that for all $1 \leq i \leq n-1$, 
$$ ||R_{i,\tilde{N}} ||_{L^2(\R)} \leq C_{\tilde{N}} B^i $$
\end{lem}

\begin{proof}
Recalling that $||T_i||_{L^2 \to L^2} \leq B$ and the bound on the remainder in Lemma \ref{iteration_lemma}, the recursive definition of $R_{i,\tilde{N}}$ gives the following bound: 
$$ ||R_{i,\tilde{N}} ||_{L^2} \leq B ||R_{i-1,\tilde{N}}||_{L^2} + \sum_{j=0}^{\tilde{N} - 1} C_{\tilde{N} -j} ||a_{i-1, j} ||_{C^{1+ 2(\tilde{N} - j) }} $$
By induction and using the previous bounds on $||a_{i,j}||_{C^l}$, we get 
\begin{align*}
||R_{\tilde{N},i}||_{L^2} &\leq \sum_{p=0}^{i-1} B^{i-1-p} \sum_{j=0}^{\tilde{N} - 1} C_{\tilde{N}-j} ||a_{p,j}||_{C^{1+2(\tilde{N}-j)}} \\
&\leq  \sum_{p=0}^{i-1} B^{i-1-p} \sum_{j=0}^{N_1 - 1} C_{\tilde{N}-j} C_{\tilde{N}-j, 0} (B \nu^{1/2} )^p (p+1)^{1+ 2\tilde{N} + j} \\
&\leq C_{\tilde{N}} B^i \sum_{p=0}^{i-1} \nu^{p/2} (p+1)^{1 + 3N_1} \\
&\leq C_{\tilde{N}}B^i
\end{align*}
using that the sum is absolutely convergent. 
\end{proof}

\paragraph{End of proof of Proposition \ref{Prop_evolution_lagrangian_state}.}
We've got now all the elements to conclude the proof. We set $$a_{\mathbf{q},\theta , N} \coloneqq \sum_{j=0}^{\tilde{N} -1} h^j a_{n-1, j}$$
We know that 
$$ U_\mathbf{q} B_{q_0}^\prime \left( e^{i \frac{\theta}{h} } \right) =  M_{q_{n-1}}  \left( e^{i \frac{\psi_\mathbf{q} \cdot}{h} } a_{\mathbf{q}, \theta, N} \right) + M_{q_{n-1}}  (h^{\tilde{N}} R_{n-1,\tilde{N}} ) $$
Since $M_q$ are uniformly bounded in $q$ and $R_{n-1,\tilde{N}} \leq C_{\tilde{N}} B^{n-1} \leq C_{N_1} h^{-C_0 \log B}$, we have : 
$$ || M_{q_{n-1}}  (h^{\tilde{N}} R_{n-1,\tilde{N}} ) ||_{L^2} \leq C_N h^{\tilde{N} - C_0 \log B} \leq C_N h^N$$
Concerning the bounds on $a_{\mathbf{q},\theta , N}$, we have 
\begin{align*}
||a_{\mathbf{q},\theta , N}||_{C^l}& \leq \sum_{j=0}^{\tilde{N}-1} h^j ||a_{n-1,j}||_{C^l} \\
&\leq  \sum_{j=0}^{\tilde{N}-1} C_{j,l} \left( B\nu^{1/2} \right)^{n-1} n^{l+3j} h^j  \\
&\leq C_{l,N} n^{l+3\tilde{N}}\left( B\nu^{1/2} \right)^{n-1} \\
&\leq C_{l,N} h^{-C_0 \log B} n^{l+3\tilde{N}}\nu^{\frac{n-1}{2}} \\&\leq C_{l,N} h^{-C_0 \log B} 
\end{align*}
where we use the fact that $n \leq C_0 |\log h|$  and bound $n^{l+3\tilde{N}}\nu^{\frac{n-1}{2}}$ by some $C_{l,\tilde{N}}$ since $\nu < 1$. 

Finally, we need to prove the property on the support of $a_{\mathbf{q},\theta , N}$. To do so, let us introduce, for $q \in \mathcal{A}$, an open set $\mathcal{W}_q$ satisfying 

$$ \supp \tilde{a}_q \Subset \mathcal{W}_q \subset \mathcal{V}_q $$ 
This allows us to define new objects replacing $\mathcal{V}_q$ by $\mathcal{W}_q$ in the definitions : 
\begin{align*}
&\mathcal{W}_\mathbf{q}^+ = \bigcap_{i=0}^{n-1} F^{n-i}(\mathcal{W}_{q_i} )\Subset \mathcal{V}_\mathbf{q}^+ \\
&\mathcal{D}_{\mathbf{q}, \theta}  = \kappa_{q_{n-1}} \Big( F^{-1}\left( \mathcal{W}_\mathbf{q}^+ \right) \cap F^{n-1} \left(\mathcal{L}_{q_0, \theta}  \right) \Big) \Subset \mathcal{C}_{\mathbf{q}, \theta}
\end{align*} and the associated subinterval $J_{\mathbf{q}, \theta} \Subset I_{\mathbf{q}, \theta}$ built thanks to Proposition \ref{lagrangian_leaves_evolved} such that 
$$ \mathcal{D}_{\mathbf{q}, \theta} = \Big\{ (y, G_{\mathbf{q}, \theta}(y)) ; \; y \in J_{\mathbf{q}, \theta} \Big\}$$
Let us fix $\delta>0$ small (with further conditions imposed). We will show the following stronger statement $$ d\left( \supp(a_{\mathbf{q},\theta,N}), \R \setminus J_{\mathbf{q},\theta} \right) \geq \delta $$
 Suppose that this is not the case. We can find $x_{n-1} \in \supp a_{\mathbf{q},\theta,N}, y_{n-1} \in I_{\mathbf{q},\theta} \setminus J_{\mathbf{q},\theta}$ such that $|x_{n-1} - y_{n-1} | \leq \delta$. 
As already done, we denote by $x_i$ (resp. $y_i$) the points defined by $x_{i-1} = \phi_i(x_i)$ (resp. $y_{i-1} = \phi_i(y_i)$). Since $\phi_i$ are contractions, we have $ |x_i - y_i | \leq \delta$ for $1 \leq i \leq n-1$. If we note $$\rho_i = \kappa_{q_i}^{-1}(x_i, \psi_i^\prime(x_i)) \quad ; \quad  \zeta_i =  \kappa_{q_i}^{-1} (y_i, \psi_i^\prime(y_i))$$
 we have for some $C >0$ : 
$d(\rho_i, \zeta_i) \leq C \delta$. 
By definition, one also has 
$$ F^{-i} (\rho_{n-1} ) = \rho_{n-1-i}\quad  ; \quad F^{-i} (\zeta_{n-1} ) = \zeta_{n-1-i} $$
By the support property (\ref{support_L_i_j}) of the operators $L_{j,i}$, 
$\rho_i \in \supp \tilde{a}_{q_i}$ for $0 \leq i \leq n-1$. 
Let's assume that $\delta$ is small enough so that for all $q \in \mathcal{A}$, 
$$  d\big( \supp \tilde{a}_q , \left( \mathcal{W}_q \right)^c \big) \geq 2C \delta$$
Hence, 
$$\rho_i \in \supp \tilde{a}_{q_i} \text{ and } d(\rho_i, \zeta_i) \leq C \delta \implies \zeta_i \in \mathcal{W}_{q_i} $$
As a consequence, for all $0 \leq i \leq n-1$, $F^{i+1-n}(\zeta_{n-1} ) \in \mathcal{W}_{q_i}$, or equivalently $\zeta_{n-1} \in F^{-1} \left(\mathcal{W}_\mathbf{q}^+ \right)$. Hence, $$(y_{n-1}, \psi^\prime_{n-1} (y_{n-1}) ) \in \mathcal{C}_{\mathbf{q}, \theta} \cap \kappa_{q_{n-1} } \left(  F^{-1} \left(\mathcal{W}_\mathbf{q}^+\right) \right) \subset \mathcal{D}_{\mathbf{q}, \theta}$$
 showing that $y_{n-1} \in J_{\mathbf{q}, \theta}$, and giving a contradiction with $y_{n-1} \in  I_{\mathbf{q},\theta} \setminus J_{\mathbf{q},\theta}$.

\subsection{Microlocalization of $U_\mathcal{Q}$}\label{Micro_of_U_Q}
We now fix a cloud $\mathcal{Q} \subset \mathcal{Q}(n,a)$, centered at a point $\rho_0 \in \mathcal{T}$, namely satisfying the condition of Proposition \ref{Prop_on_clouds} : 

$$\forall \rho \in \bigcup_{ \mathbf{q} \in \mathcal{Q}} \mathcal{V}_\mathbf{q}^+ , d(\rho, W_u(\rho_0) ) \leq Ch^\mathfrak{b}$$
Let us note 
\begin{equation}U_\mathcal{Q} = \sum_{ \mathbf{q} \in \mathcal{Q}} U_\mathbf{q}
\end{equation}
 and 
 
 \begin{equation}
 \mathcal{V}_\mathcal{Q}^+ = \bigcup_{ \mathbf{q} \in \mathcal{Q}} \mathcal{V}_\mathbf{q}^+
 \end{equation}
 
 We fix an adapted chart $\kappa \coloneqq \kappa_{\rho_0} : U_0 \to V_0$ around $\rho_0$ as permitted by the Lemma \ref{Existence_of_adapated_charts}. We can assume that $\mathcal{V}_a^+ \Subset U_0$ (if $\varepsilon_0$ is small enough and since the local unstable leaf $W_u(rho_0)$ is close to points in $\mathcal{V}_a^+)$). We consider a cut-off function $\tilde{\chi}_a \in \cinfc(U_0)$ such that $\tilde{\chi}_a \equiv 1$ on $F(\supp \chi_a)$ and $\supp \tilde{\chi}_a \subset \mathcal{V}_a^+$. Let us note $\Xi_a = \op(\tilde{\chi}_a)$. Since $\Xi_a MA_a = MA_a+  \hinf$, $| \mathcal{Q}| = O(h^{-K})$ and $||U_\mathbf{q}|| = O(h^{-K})$ for some $K >0$, we have 
 
 \begin{equation*}
 \mathfrak{M}^{N_0} U_\mathcal{Q} =  \mathfrak{M}^{N_0} \Xi_a U_\mathcal{Q} + \hinf
 \end{equation*}
 Let us introduce Fourier integral operators $B, B^\prime$ quantizing $\kappa$ in $\supp(\chi_a)$ : 
 $$ B^\prime B = I + \hinf \text{ microlocally in } \supp(\chi_a) $$
  Hence : 
   \begin{equation*}
 \mathfrak{M}^{N_0} U_\mathcal{Q} =  \mathfrak{M}^{N_0} \Xi_a B^\prime B  U_\mathcal{Q} + \hinf
 \end{equation*}
 
 We introduce the following sets : 
 \begin{equation}\label{def_Omega_+}
 \Gamma^+ = \eta \left( \kappa\left( \mathcal{V}_\mathcal{Q}^+ \right) \right) \; ; \;  \Omega^+ = \Gamma^+ (h^\tau)
 \end{equation}
 and for $\mathbf{q} \in \mathcal{Q}$, 
 \begin{equation}
  \Gamma^+_\mathbf{q} = \eta \left( \kappa\left( \mathcal{V}_\mathbf{q}^+ \right) \right)
 \end{equation}

\begin{figure}[h]
\centering
\includegraphics[scale=0.5]{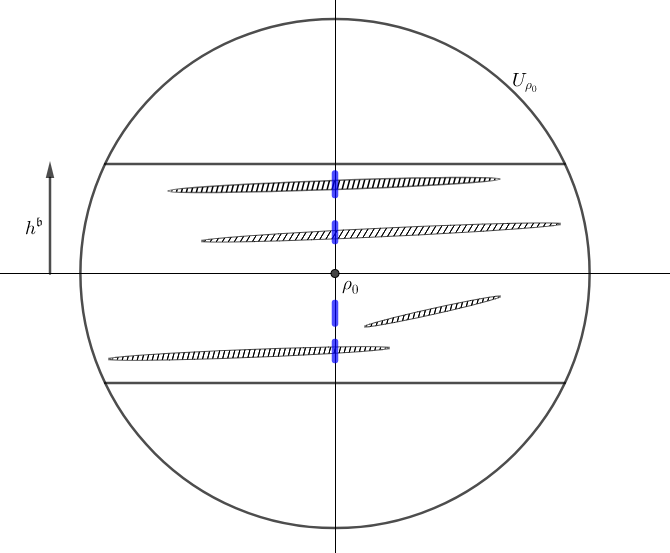}
\caption{The definition of the sets $\Gamma_\mathbf{q}^+$. They are represented by the blue segments on the $\eta$-axis and are the projections on the $\eta$ variable of the sets $\mathcal{V}_\mathbf{q}^+$ (the hached sets). They are of width of order $h^\tau$. }
\label{figure_gamma_+}
\end{figure}

 We will prove in the following lemma that the pieces $U_\mathbf{q}$ are microlocalized in thin horizontal rectangles (see Figure \ref{figure_gamma_+}).
 \begin{lem}\label{micro_U_q}
 For every $\mathbf{q} \in \mathcal{Q}$, 
 \begin{equation}
  \mathds{1}_{ \Gamma^+_\mathbf{q}(h^\tau) } \left( h D_y \right) B U_\mathbf{q} = B U_\mathbf{q} + \hinf_{L^2 \to L^2}
 \end{equation}
 with uniform bounds in the $\hinf$. 
 \end{lem}
 Using the polynomial bounds $| \mathcal{Q}| = O(h^{-C})$ and $||U_\mathbf{q}|| = O(h^{-C})$), we immediately deduce the 
 \begin{prop}\label{prop_micro_right}
\begin{equation}
\mathds{1}_{\Omega^+} (hD_y)B U_\mathcal{Q} = B U_\mathcal{Q} + \hinf_{L^2 \to L^2}
\end{equation}
 \end{prop}

 \subsubsection{Proof of Lemma \ref{micro_U_q}}
We fix a word $\mathbf{q}= q_0 \dots q_{n-2} a \in \mathcal{Q}$. Since $\WF(A_{q_0})$ is compact, we can find $\chi \in \cinfc(\R)$ such that 
$$ A_{q_0} = A_{q_0} B^\prime_{q_0} \chi(hD_y) B_{q_0} + \hinf$$
Since there is a finite number of symbols in $\mathcal{A}$, we can choose one single $\chi$ for all the possible symbols $q_0$.
We are hence reduced to prove that 
\begin{equation}
\underbrace{\mathds{1}_{\R \setminus  \Gamma^+_\mathbf{q}(h^\tau) } (hD_y) B U_\mathbf{q} B^\prime_{q_0}}_{T} \chi(hD_y) = \hinf_{L^2 \to L^2}
\end{equation} 
If $u \in L^2(\R)$, writing 
$$ \left( \chi(hD_y) u \right)(y) = \frac{1}{(2 \pi h)^{1/2} } \int_\R \chi(\theta) \mathcal{F}_h u (\theta) e^{i \frac{\theta y}{h}} d\theta$$
we have
$$ T \big( \chi(hD_y) u \big) = \frac{1}{(2 \pi h)^{1/2} } \int_\R \chi(\theta) \mathcal{F}_h u (\theta)  \left( Te^{i \frac{\theta \cdot}{h} } \right)d\theta$$
Hence, 
\begin{align*}
 ||T \big( \chi(hD_y) u \big) ||_{L^2} &\leq \frac{1}{(2 \pi h)^{1/2} } \int_\R |\chi(\theta) \mathcal{F}_h u (\theta)|  \left|\left|Te^{i \frac{\theta \cdot}{h}}  \right|\right|_{L^2} d\theta \\
 & \leq \frac{1}{(2 \pi h)^{1/2}  } \int_\R |\chi(\theta) \mathcal{F}_h u (\theta)| \sup_{\theta \in \supp \chi} \left|\left|Te^{i \frac{\theta \cdot}{h}}  \right|\right|_{L^2}\\
 &\leq \frac{C_\chi}{h^{1/2}} ||\mathcal{F}_h u ||_{L^2} \sup_{\theta \in \supp \chi} \left|\left|Te^{i \frac{\theta \cdot}{h}}  \right|\right|_{L^2}\\ 
 & \leq \frac{C_\chi}{h^{1/2}} ||u||_{L^2} \sup_{\theta \in \supp \chi} \left|\left|Te^{i \frac{\theta \cdot}{h}}  \right|\right|_{L^2} \\
\end{align*}
As a consequence, we are lead to estimate $\sup_{\theta \in \supp \chi} \left|\left|Te^{i \frac{\theta \cdot}{h}}  \right|\right|_{L^2} $. We fix $\theta \in \supp \chi$. 
Writing that $\supp \chi \subset [-C_0, C_0]$ and recalling $|\mathbf{q}| = n \leq C_0 |\log h|$ for some global $C_0$,  we are in the framework of Proposition \ref{Prop_evolution_lagrangian_state}. 

We fix $N \in \N$ and we aim at proving that $T e^{i \frac{\theta \cdot}{h}} = O(h^N)$. By Proposition \ref{Prop_evolution_lagrangian_state}, there exists $a _{ \mathbf{q}, N, \theta} \in \cinfc(I_{\mathbf{q}, \theta} ) $ such that 

$$U_\mathbf{q} B^\prime_{q_0}  \left( e^{i \frac{\theta \cdot}{h}}  \right) = MA_aB_a^\prime \left( a_{\mathbf{q}, N, \theta} e^{i\frac{\Phi_{\mathbf{q}, \theta}}{h}}\right)  + O(h^N)$$ 
Set $S \coloneqq B MA_aB_a^\prime$. $S$ is a Fourier integral operator associated with $s \coloneqq \kappa \circ F \circ \kappa_a^{-1}$. Recall that the definitions and the description of the Lagrangian 
\begin{align*}
 C_{\mathbf{q}, \theta} &= \kappa_a \left( F^{-1}\left(\mathcal{V}_\mathbf{q}^+ \right) \cap F^{n-1} \left(\mathcal{L}_{q_0, \theta} \right)  \right) \\
 &= \{ (y, \Phi^\prime_{\mathbf{q} , \theta} (y) ) , y \in I_{\mathbf{q}, \theta} \} \\
\end{align*}
with 
$\Phi_{\mathbf{q},\theta} \in \cinf(I_{\mathbf{q}, \theta}) \; ; \;  ||\Phi_{\mathbf{q},\theta}||_{C^1} \leq C\varepsilon_0 \; ; \; ||\Phi_{\mathbf{q},\theta}||_{C^l} \leq C_l  $. \\
Assuming that $\varepsilon_0$ is small enough, we can assume that : 
\begin{itemize}
\item $s$ is well defined on $B_a(0, C_1 \varepsilon_0)$ and satisfies the conclusion of Lemma \ref{lemma_propagation_lagrangian_leaves}. As a consequence, the Lagrangian line 
$$s ( \mathcal{C}_{\mathbf{q}, \theta} ) = \kappa \left( \mathcal{V}_\mathbf{q}^+\right) \cap \kappa \circ F^n \left( \mathcal{L}_{q_0, \theta} \right) $$ can be written $\{ (y, \Psi^\prime(y) ) , y \in I \}$ for some open $I \subset \R$ and some function $\Psi \in \cinf(I)$ satisfying 
$$ ||\Psi||_{C^1} \leq C\varepsilon_0 \; ; \; ||\Psi||_{C^l} \leq C_l$$
with global constants $C$ and $C_l$. 
\item $S$ has the form (\ref{fio_in_nice_form}) with a phase function and a symbol having $C^l$ norms bounded by global constants (depending on $l$). 
\end{itemize}
Hence, we can apply Lemma \ref{iteration_lemma} to see that there exists $b \in \cinfc(I)$ such that 

\begin{equation*}
S \left( a_{\mathbf{q}, N ,\theta} e^{i  \frac{\Phi_{\mathbf{q},\theta}}{h}} \right) = b e^{i \frac{\Psi}{h}} + O(h^N)_{L^2}
\end{equation*}
 $b$ satisfies the same type of bounds as $a_{\mathbf{q}, N, \theta}$, namely : 
$$ ||b||_{C^l} \leq C_{l,N} h^{-C_0 \log B}$$
 Moreover, since  $d( \supp a_{\mathbf{q}, N, \theta},\R \setminus  I_{\mathbf{q}, \theta} ) \geq \delta$, there exists $\delta^\prime >0$ such that $d(\supp b, \R \setminus I) \geq \delta^\prime$. The constants $C_{l,N}$ and $\delta^\prime$ are global constants. 
Since $N$ is arbitrary, to conclude the proof of Lemma \ref{micro_U_q}, it  remains to show that 
\begin{equation}\label{micro_U_q_reduced}
\mathds{1}_{\R \setminus \Gamma_\mathbf{q}^+(h^\tau) } (hD_y) \left(b e^{i \frac{\Psi}{h}} \right) = O(h^N)
\end{equation}

To do so, we make use of the fine Fourier localization statement from Proposition 2.7 in \cite{NDJ19}. We state it for convenience but refer the reader to the quoted paper for the proof. 

\begin{prop}
Let $U \subset \R^n$ open, $K \subset U$ compact, $\Phi \in \cinf(U)$ and $a \in \cinfc(U)$ with $\supp a \subset K$. Assume that there is a constant $C_0$ and constants $C_N, N \in \N^*$ such that : 
\begin{align}
&\text{vol}(K) \leq C_0 \label{point1}\\
&d(K, \R^n \setminus U) \geq C_0^{-1} \label{point2}\\
&\max_{ 0 < |\alpha| \leq N } \sup_{U} |\partial^\alpha \Phi | \leq C_N ; N \geq 1 \label{point3} \\
&\max_{ 0 \leq |\alpha| \leq N } \sup_{U} |\partial^\alpha a | \leq C_N ; N \geq 1  \label{point4}\\
\end{align}
Finally, assume that the projection of the Lagrangian $\{ (x, \Phi^\prime(x)), x \in U \}$ on the momentum variable has a diameter of order $h^\tau$, namely : 
\begin{equation}\label{point5}
\text{diam}( \Omega_\Phi ) \leq C_0 h^\tau  \text{ where } \Omega_\Phi = \{ \Phi^\prime(x), x \in U \}
\end{equation}
Define the Lagrangian state 
\begin{equation}
u(x) = a(x) e^{i \frac{\Phi(x)}{h}} \in \cinfc(U) \subset \cinfc(\R^n)
\end{equation}
Then, for every $N \geq 1$, there exists $C^\prime_N$ such that 
\begin{equation}
|| \mathds{1}_{ \R^n \setminus \Omega_\Phi(h^\tau)} u || \leq C^\prime_N h^N 
\end{equation}
$C^\prime_N$ depends on $\tau, n, N, C_0, C_{N^\prime}$ for some $N^\prime(n,N,\tau)$. 
\end{prop}
\vspace{0.5cm}

When $U = I$, $K = \supp b$, $a = h^{C_0 \log B} b$ , $\Phi = \Psi$, the assumptions  (\ref{point1}) to (\ref{point4}) are satisfied for some global constants $C_0, C_N$. In this case, 
$$\Omega_\Psi = \{ \Psi^\prime(y), y \in I \} = \eta \left(  \kappa \left( \mathcal{V}_\mathbf{q}^+\right) \cap \kappa \circ F^n \left( \mathcal{L}_{q_0, \theta} \right) \right) $$
Since $\Omega_\Psi  \subset  \Gamma_\mathbf{q}^+$, to prove (\ref{micro_U_q_reduced}), it is enough to prove it with $ \Gamma_\mathbf{q}^+$ replaced by $\Omega_\Psi$ and to apply the last proposition, it remains to check that the last point (\ref{point5}) is satisfied. Since who can do more, can do less, we will show that 
$$ \text{diam} \left( \Gamma_\mathbf{q}^+\right) \leq C_0 h^\tau$$
 This is where the strong assumption on the adapted charts will play a role. To insist on this role, we state the following lemma :

\begin{lem}\label{straigh_Wu}
Let $C_0 >0$. Assume that $\rho_1 \in \mathcal{T} \cap U_{\rho_0}$ satisfies $d(\rho_1, W_u(\rho_0)) \leq C_0 h^\mathfrak{b}$. If $\rho_2 \in W_u(\rho_1)$, then for some global constant $C>0$, 
\begin{equation}
| \eta ( \kappa(\rho_1) ) - \eta (\kappa(\rho_2) )| \leq C C_0^{1+\beta}  h
\end{equation}
\end{lem}

\begin{proof}
Recall that the chart $(\kappa, U_{\rho_0})$ is the one centered at $\rho_0$, given by Lemma \ref{Existence_of_adapated_charts}. In this chart, $\kappa(W_u(\rho_1))$ is almost horizontal : we have 
$$ \kappa(W_u(\rho_1) ) = \{ y , g(y, \zeta(\rho_1)) , y \in \Omega \}$$ 
where $\Omega$ is some open bounded set of $\R$, with $g$ and $\zeta$ satisfying the properties of Lemma \ref{Existence_of_adapated_charts}. Hence, to prove the lemma, it is enough to estimate $|g(y, \zeta(\rho_1)) - g(0, \zeta(\rho_1)) |, y \in \Omega$. 
Since $\zeta(\rho_0) =0$ and $\zeta$ is Lipschitz, $|\zeta(\rho_1) | \leq C_0 h^\mathfrak{b}$. Indeed, if $\rho_0^\prime \in W_u(\rho_0)$ satisfies $d(\rho_0^\prime, \rho_1) \leq 2 C_0 h^\mathfrak{b}$, $$|\zeta(\rho_1)| = |\zeta(\rho_1) - \zeta(\rho_0^\prime)|\leq C d(\rho_1, \rho_0^\prime) \leq C C_0 h^\mathfrak{b} $$
Then, we have
\begin{align*}
|g(y, \zeta(\rho_1)) - g(0, \zeta(\rho_1))| &= |g(y, \zeta(\rho_1)) - g(y, 0) - \partial_\zeta g(y, 0) \zeta(\rho_1) | \\
&= \left| \int_0^{\zeta(\rho_1)}  \left(\partial_\zeta g(y, \zeta)-\partial_\zeta g(y, 0) \right) d \zeta \right |\\
& \leq  \left| \int_0^{\zeta(\rho_1)} C \zeta^{\beta} d \zeta \right| \\
& \leq C \zeta(\rho_1)^{1 + \beta} \leq C  C_0^{1 + \beta} h^{\mathfrak{b}(1+\beta)}
\end{align*}
In the first equality, we've used the facts that $g(0, \zeta) = \zeta$, $\partial_\zeta g(y,0) = 1$ and $g(y,0)=0$.
This concludes the proof since, by definition (see (\ref{alpha})), $\mathfrak{b}(1+ \beta)= 1$. 
\end{proof}

\begin{rem}
This lemma explains our definition of $\mathfrak{b}$. 
\end{rem}

From this lemma, we can deduce (\ref{point5}). Indeed, recall that there exists $\rho_\mathbf{q} \in \mathcal{T}$ such that $\mathcal{V}_\mathbf{q}^+ \subset W_u(\rho_\mathbf{q}) (Ch^\tau)$. If $\rho_1, \rho_2 \in \mathcal{V}_\mathbf{q}^+$, there exists $\rho^\prime_1, \rho^\prime_2 \in W_u(\rho_\mathbf{q})$ such that 
$$ d(\rho_i, \rho_i^\prime) \leq C h^\tau \; ;\; i=1,2$$
Hence, one can estimate 

\begin{equation*}
|\eta(\kappa(\rho_1) ) - \eta(\kappa(\rho_2) ) | \leq  
\underbrace{ |\eta(\kappa(\rho_1) ) - \eta(\kappa(\rho_1^\prime) ) |}_{\leq Ch^\tau}
+ \underbrace{ |\eta(\kappa(\rho_1^\prime) ) - \eta(\kappa(\rho_2^\prime) ) |}_{\leq C h}+
\underbrace{ |\eta(\kappa(\rho_2) ) - \eta(\kappa(\rho_2^\prime) ) |}_{\leq Ch^\tau}
\end{equation*}
The inequality in the middle is a consequence of the previous lemma. Indeed,  $\rho_1^\prime, \rho_2^\prime \in W_u(\rho_1^\prime)$ where (recall that $\tau > \mathfrak{b}$) 
$$d(\rho_1^\prime, W_u(\rho_0)) \leq d(\rho_1, \rho_1^\prime) + d(\rho_1, W_u(\rho_0)) \leq Ch^\tau + Ch^\mathfrak{b} \leq 2Ch^\mathfrak{b}$$

\subsection{Reduction to a fractal uncertainty principle}
We go on the work started in the last subsection and we keep the same notations. In virtue of Proposition \ref{prop_micro_left} and Proposition \ref{prop_micro_right}, we can write 

\begin{equation}
\mathfrak{M}^{N_0} U_\mathcal{Q} = \mathfrak{M}^{N_0} B^\prime B \op(\chi_h) \Xi_a B^\prime \mathds{1}_{\Omega^+} (hD_y)B  U_\mathcal{Q} + \hinf_{L^2 \to L^2}
\end{equation}
where \begin{itemize}
\item $\chi_h \in S_{\delta_2}^{comp}$, $\chi_h \equiv 1$ on $\mathcal{T}_-^{loc}(2C_2h^{\delta_2})$ and $\supp \chi_h \in \mathcal{T}_-^{loc}(4C_2h^{\delta_2})$ (see Proposition \ref{prop_micro_left} and before); 
\item $\Xi_a = \op(\tilde{\chi}_a)$ where $\tilde{\chi}_a \in \cinfc(U_0)$ is a cut-off function such that $\tilde{\chi}_a \equiv 1$ on $F(\supp \chi_a)$ and $\supp \tilde{\chi}_a \subset \mathcal{V}_a^+$ (see the beginning of subsection \ref{Micro_of_U_Q}) ;  
\item $\Omega^+ = \eta \big( \kappa \left( \mathcal{V}_\mathcal{Q}^+ \right) \big)\big(h^\tau \big)$ (see \ref{def_Omega_+} and Proposition \ref{prop_micro_right}).
\end{itemize}
In $V_{\rho_0}$, $U_\mathcal{Q}$ is microlocalized in a region  $\{ |\eta| \leq Ch^\mathfrak{b}\}$. To work with symbols in usual symbol classes, we will rather consider a bigger region $\{ |\eta| \leq h^{\delta_0} \}$. For this purpose, let us denote 
\begin{equation}\label{def_Omega_-}
\Gamma^- = y \Big( \kappa\left( \mathcal{V}_a^+ \cap \mathcal{T}_-^{loc}(4C_2h^{\delta_2}) \right) \cap \{ |\eta| \leq h^{\delta_0} \} \Big) \; ; \; \Omega^- = \Gamma^- \left(h^{\delta_0} \right)
\end{equation}
Since $\mathcal{V}_\mathcal{Q}^+ \subset W_u(\rho_0)(Ch^\mathfrak{b})$, $\Omega_+ \subset [-C_0h^\mathfrak{b}, C_0 h^\mathfrak{b}] \subset [-h^{\delta_0}, h^{\delta_0}]$ for $h$ small enough. By Lemma \ref{lemma_construction_cut_off}, there exists $\chi_+(\eta) \coloneqq \chi_+ (\eta ; h) \in \cinfc(\R)$ such that : 
\begin{itemize}[nosep]
\item $\chi_+ \equiv 1$ on $\Omega^+$;
\item  $ \supp \chi_+ \subset [-h^{\delta_0}, h^{\delta_0}]$ ; 
\item $\forall k \in \N$ and $\eta \in \R$, $|\chi_+^{(k)}(\eta)| \leq C_k h^{-\delta_0 k }$ for some global constants $C_k$. 
\end{itemize}
$\chi_+$ satisfies : 
$$\mathds{1}_{\Omega^+} (hD_y) = \chi_+(hD_y) \mathds{1}_{\Omega^+} (hD_y)$$
Let's now consider the following subset of $\Gamma^-$ : 
$$\widetilde{\Gamma}^- = y \Big( \kappa\left( \mathcal{V}_a^+ \cap \mathcal{T}_-^{loc}(4C_2h^{\delta_2}) \right) \cap \{ \eta \in \supp \chi_+ \} \Big)$$
The inclusion $\widetilde{\Gamma}^- \subset \Gamma^-$ comes from the support property of $\chi_+$. 

\begin{figure}[h]
\centering
\includegraphics[scale=0.5]{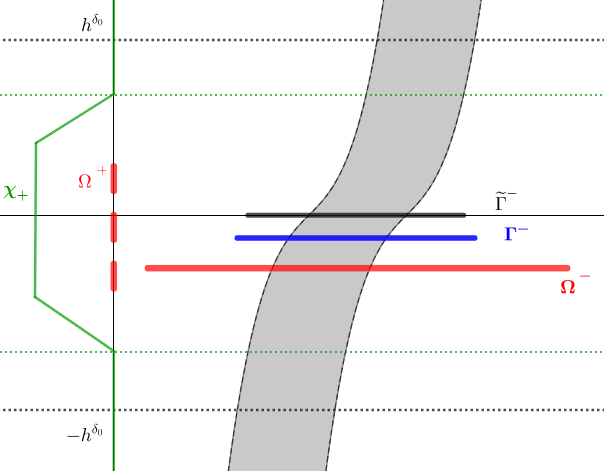}
\caption{The set $\Omega^+$ is represented on the $\eta$-axis, with the support of the function $\chi_+$. On the $y$-axis, we project the gray set $\kappa\left( \mathcal{V}_a^+ \cap \mathcal{T}_-^{loc}(4C_2h^{\delta_2}) \right) $ to obtain both $\Gamma^-$ and $\widetilde{\Gamma}^-$ depending on the size of the $\eta$-window. The larger set $\Omega^-$ is also represented in red. }
\end{figure}

Using again Lemma \ref{lemma_construction_cut_off}, we construct a family 
$\chi_-(y) \coloneqq \chi_- (y ; h) \in \cinfc(\R)$ such that : 
\begin{itemize}[nosep]
\item $\chi_- \equiv 1$ on $\widetilde{\Gamma}^-$;
\item  $ \supp \chi_- \subset \Omega^- = \Gamma^- (h^{\delta_0})$ ; 
\item $\forall k \in \N$ and $y \in \R$, $|\chi_-^{(k)}(y)| \leq C_k h^{-\delta_0 k }$. 
\end{itemize} 
and $\chi_-$ allows to write 
$$ \chi_- (y)\mathds{1}_{\Omega^-}(y) = \chi_-(y)$$
We now claim that

\begin{equation}\label{reduction_to_FUP}
\mathfrak{M}^{N_0} U_\mathcal{Q} = \mathfrak{M}^{N_0} \op(\chi_h) \Xi_a B^\prime \chi_- (y) \mathds{1}_{\Omega^-}(y) \mathds{1}_{\Omega^+} (hD_y)B  U_\mathcal{Q} + \hinf_{L^2 \to L^2}
\end{equation}
 Due to the polynomial bounds on $||\mathfrak{M}^{N_0}||$ and $||U_\mathcal{Q}||$, it is then enough to show that 

$$ \op(\chi_h) \Xi_a B^\prime (1-\chi_- (y)) \chi_+(hDy) = \hinf$$ 
Using Egorov's theorem in $\Psi_{\delta_2}(\R)$, we see that $\Xi_0 \coloneqq B  \op(\chi_h) \Xi_a B^\prime$ is in $\Psi_{\delta_2}(\R)$ and $\WF(\Xi_0) \subset  \kappa ( \supp \chi_a \cap \supp \chi_h)$. 
We now observe that 
\begin{align*}
(y,\eta) \in \WF (\Xi_0) \cap \WF\left(1-\chi_-(y) \right) \cap \WF \left( \chi_+ (hD_y) \right) \implies \\
(y,\eta) \in  \kappa ( \supp \chi_a \cap \supp \chi_h) , \eta \in \supp \chi_+ , y \not \in \widetilde{\Gamma}^-, \\
\end{align*}
But the first two conditions imply that $y \in \widetilde{\Gamma}^-$. Hence, $$\WF (\Xi_0) \cap \WF\left(1-\chi_-(y) \right) \cap \WF \left( \chi_+ (hD_y) \right) = \emptyset$$
 By the composition formulas in $\Psi_{\delta_0}(\R)$, $\Xi_0(1-\chi_-(y) ) \chi_+ (hD_y) = \hinf$. Note that the constants in $\hinf$ depend on the semi-norms of $\chi_\pm$ ,$\chi_h$ and $\chi_a$. Due to their construction, the semi-norms of $\chi_\pm$ and $\chi_h$ are bounded by global constants. As a consequence, the constants $\hinf$ are global constans.

This proves the claim \ref{reduction_to_FUP}. Recalling the bound 
$$ || \mathfrak{M}^{N_0} ||_{L^2 \to L^2} \leq ||\alpha||^{N_0} (1+ o(1)) \quad , \quad \left| \left|U_\mathcal{Q} \right|\right|_{L^2 \to L^2} \leq C |\log h| ||\alpha||_\infty^{N_1} $$  we see that the proof of Proposition \ref{Prop_on_clouds} and hence of Proposition  \ref{Prop_thm}, has been reduced to proving the following proposition. 

\begin{prop} \label{FUP_to_Omega}
With the above notations, 
There exists $\gamma >0$ and $h_0 >0$ such that : 
\begin{equation}
\forall h \leq h_0, \;  ||\mathds{1}_{\Omega^-}(y) \mathds{1}_{\Omega^+}(hD_y) ||_{L^2 \to L^2} \leq h^\gamma
\end{equation}
\end{prop}

\begin{rem}
$\gamma$ and $h_0$ are global : they do not depend on the particular $\mathcal{Q} \subset \mathcal{Q}(n,a)$ satisfying the conditions of Proposition \ref{Prop_on_clouds}, nor on $n$. 
\end{rem}

The proof of this proposition is the aim of the next section and relies on a fractal uncertainty principle.

\section{Application of the fractal uncertainty principle}\label{Section_FUP}

The fractal uncertainty principle, first introduced by Dyatlov-Zahl in \cite{DZ16} and further proved in full generality by Bourgain-Dyatlov in \cite{BD18}, is the key tool for our decay estimate. We'll use the slightly more general version proved and used in \cite{NDJ19}.

\subsection{Porous sets}\label{subsection_porous_sets}
We start by recalling the definition of porous sets and then we state the version of the fractal uncertainty principle we'll use.  

\begin{defi}
Let $\nu \in (0,1)$ and $0 \leq \alpha_0 \leq \alpha_1$. We say that a subset $\Omega \subset \R$ is $\nu$-porous on scales $\alpha_0$ to $\alpha_1$ if for every interval $I \subset \R$ of size $|I| \in [\alpha_0,\alpha_1]$, there exists a subinterval $J \subset I$ of size $|J| = \nu |I|$ such that $J \cap \Omega = \emptyset$. 
\end{defi}

\begin{figure}[h]
\begin{center}
\includegraphics[scale=0.6]{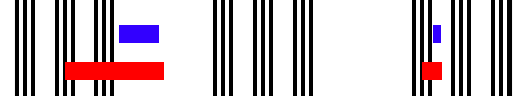}
\caption{Example of a porous set. Its construction is based on a Cantor-like set. Red intervals correspond to choices of $I$, blue ones correspond to $J$. }
\end{center}
\end{figure}

The following simple lemma shows that when one fattens a porous set, one gets another porous set. For its (very elementary) proof, see \cite{NDJ19} (Lemma 2.12). 

\begin{lem}\label{fattenning_porous_set}
Let $\nu \in (0,1)$ and $0 \leq \alpha_0 < \alpha_1$. Assume that $\alpha_2 \in (0, \frac{\nu}{3} \alpha_1]$ and $\Omega \subset \R$ is $\nu$-porous on scales $\alpha_0$ to $\alpha_1$. Then, the neighborhood $\Omega(\alpha_2) = \Omega + [-\alpha_2, \alpha_2]$ is $\frac{\nu}{3}-$ porous on scale $\max(\alpha_0, \frac{3}{\nu} \alpha_2)$ to $\alpha_1$. 
\end{lem}
\vspace{0.2cm}
The notion of porosity can be related to the different notions of fractal dimensions. Let us recall the definition of the upper box dimension of a metric space $(X,d)$. We denote by $N_X(\varepsilon)$ the minimal number of open balls of radius $\varepsilon$ needed to cover  $X$. Then, the upper box dimension of $X$ is defined by : 

\begin{equation}
\overline{\dim} X \coloneqq \limsup_{\varepsilon \to 0} \frac{\log N_X(\varepsilon)}{- \log \varepsilon}
\end{equation}

In particular, if $\delta > \overline{\dim}_X $, there exists $\varepsilon_0>0$ such that for every $\varepsilon \leq \varepsilon_0$, $N_X(\varepsilon) \leq \varepsilon^{-\delta}$. This observation motivates the following lemma :

\begin{lem}\label{from_fractal_to_porous} Let $\Omega \subset \R$. Suppose that there exist $0 < \delta < 1$, $C >0$ and $\varepsilon_0 >0$ such that 
$$\forall \varepsilon \leq \varepsilon_0 , \; N_\Omega(\varepsilon) \leq C \varepsilon^{-\delta}$$ 
Then, there exists $\nu = \nu(\delta, \varepsilon_0, C)$ such that $\Omega$ is $\nu$-porous on scale $0$ to $1$. 
\end{lem}

\begin{rem}
The proof will give an explicit value of $\nu$. This quantitative statement will be important in the sequel to ensure the same porosity for all the sets $W_{u/s}(\rho_0) \cap \mathcal{T}$. 
\end{rem}

\begin{proof}
Let us set $T = \lfloor \max\left((6\varepsilon_0)^{-1}, (6^\delta C)^{\frac{1}{1-\delta}}\right) \rfloor +1$ and $\nu = (3T)^{-1}$. We will show that $\Omega$ is $\nu$-porous on scale $0$ to $1$. Let $I \subset \R$ be an interval of size $|I| \in (0,1]$. Cut $I$ into $3T$ consecutive closed intervals of size $\nu$: $J_0, \dots, J_{3T-1}$. We argue by contradiction and assume that each of these intervals does intersect $\Omega$.  Let us show that 
\begin{equation}\label{bound_N}
N_\Omega(\nu/2) \geq T
\end{equation}
Assume that $U_1, \dots, U_k$ is a family of open intervals of size $\nu$ covering $\Omega$. For $i = 0, \dots, T-1$, there exists $x_i \in J_{3i+1}$ and $j_i \in \{1, \dots, k \}$ such that $x_i \in U_{j_i}$. It follows that $U_{j_i} \subset J_{3i} \cup J_{3i+1} \cup J_{3i+2}$ and hence $i \neq l \implies U_{j_i} \cap U_{j_l} = \emptyset$. The map $i \in \{0, \dots, T-1\}  \mapsto j_i \in \{1, \dots, k \}$ is one-to-one, and it gives (\ref{bound_N}). 
Since $T \geq \frac{1}{6 \varepsilon_0}$, $\nu/2 \leq \varepsilon_0$. As a consequence , 
$$ T \leq N(\nu/2) \leq C(6T)^\delta$$
which implies $T^{1-\delta} \leq C6^\delta$. This contradicts the definition of $T$. 
\end{proof}

In the appendix \ref{from_porous_to_dim}, we give a result in the other way, namely, porous sets down to scale 0 have an upper box dimension strictly smaller than one. 

For further use, we also record the easy lemma : 
\begin{lem}\label{Lipshictz and fractal sets}
Assume that $(X,d)$, $(Y,d^\prime)$ are metric spaces and $f : X \to Y$ is $C$-Lipschitz. Then, for every $\varepsilon >0$, 
$$N_{f(X)}(\varepsilon) \leq N_X(\varepsilon/C)$$
In particular, if $N_X(\varepsilon) \leq C_1^\delta \varepsilon^\delta$ for $\varepsilon \leq \varepsilon_0$, then for $\varepsilon \leq C \varepsilon_0$, $N_{f(X)} (\varepsilon) \leq (C_1C)^\delta \varepsilon^{-\delta}$. 
\end{lem}

\subsection{Fractal uncertainty principle}

We state here the version of the fractal uncertainty principle we'll use. This version is stated in Proposition 2.11 in \cite{NDJ19}. The difference with the original version in \cite{BD18} is that it relaxes the assumption regarding the scales on which the sets are porous. We refer the reader to the review of Dyatlov \cite{DyFUP} to an overview on the fractal uncertainty principle with other references and applications. 

\begin{prop}\label{FUP} \textbf{Fractal uncertainty principle.}
Fix numbers $\gamma_0^\pm, \gamma_1^\pm$ such that 
$$ 0 \leq \gamma_1^\pm < \gamma_0^\pm \leq 1, \; \gamma_1^+ + \gamma_1^- < 1 < \gamma_0^+ + \gamma_0^-$$ 
and define 
$$ \gamma \coloneqq \min (\gamma_0^+, 1- \gamma_1^-) - \max(\gamma_1^+, 1- \gamma_0^-) $$ 
Then for each $\nu >0$ , there exists $\beta = \beta(\nu) >0$ and $C = C(\nu)$ such that the estimate 
\begin{equation}
||\mathds{1}_{\Omega_-} \mathcal{F}_h \mathds{1}_{\Omega_+} ||_{L^2(\R) \to L^2(\R) } \leq Ch^{\gamma \beta} 
\end{equation}
holds for all $0 < h \leq 1$ and all $h$-dependent sets $\Omega_\pm \subset \R$ which are $\nu$-porous on scale $h^{\gamma_0^\pm}$ to $h^{\gamma_1^\pm}$.
\end{prop}

\begin{rem}
In the sequel, we will use this result with $\gamma_1^\pm =0$. In this case, the condition on $\gamma_0^\pm$ becomes $\gamma_0^- + \gamma_0^+ >1$ and the exponent $\gamma$ is $\gamma_0^- + \gamma_0^+ -1$. This condition can be interpreted as a condition of saturation of the standard uncertainty principle : a rectangle of size $h^{\gamma_0^+} \times h^{\gamma_0^-} $ will be subplanckian. 
\end{rem}

\subsection{Porosity of $\Omega^+$ and $\Omega^-$}

Since we want to apply Proposition \ref{FUP} to prove Proposition \ref{FUP_to_Omega}, we need to show the porosity of the sets $\Omega^\pm$ defined in (\ref{def_Omega_+}) and (\ref{def_Omega_-}). The main tool is the following proposition. 

\begin{prop}\label{upper_box_dim}
There exist $\delta \in [0,1[$, $C>0$ and $\varepsilon_0 >0$ such that for every $\rho_0 \in \mathcal{T}$, if $X= W_{u/s}(\rho_0) \cap \mathcal{T} \cap U_{\rho_0}$, 
\begin{align*}
N_X(\varepsilon) \leq C \varepsilon^{-\delta} ; \; \forall \varepsilon \leq \varepsilon_0
\end{align*}
\end{prop}

\begin{rem}
Recall that $W_{u/s}(\rho_0)$ is a local unstable (resp. stable) manifold at $\rho_0$, and in particular a single smooth curve. $U_{\rho_0}$ is the domain of the chart adapted $\kappa_{\rho_0}$ (see \ref{Existence_of_adapated_charts}). 
\end{rem}

Roughly speaking, this proposition says that the upper box dimension of the sets $W_{u/s}(\rho) \cap \mathcal{T}$, the trace of $\mathcal{T}$ along the stable and unstable manifolds, is strictly smaller than one. This condition on the upper box dimension is a fractal condition. In our case, we need uniform estimates on the numbers $N_X(\varepsilon)$ for $X = W_{u/s}(\rho) \cap \mathcal{T}$. This uniformity is a consequence of the fact that the holonomy maps are $C^1$ with uniform $C^1$ bounds (and thus Lipschitz, which is enough to conclude). 
This result is clearly linked with Bowen's formula which has been proved in different contexts and links the dimension of $X$ with the topological pressure of the map $\phi_u = - \log |J_u^{1} |$. This is where the assumption $\ref{Fractal}$ is used. This proposition is proved in the Appendix \ref{upper-box dimension} where we borrow the arguments of \cite{Bar} (Section 4.3) to get the required bounds. 

From the Proposition \ref{upper_box_dim}, we get 

\begin{cor}
There exists $\nu>0$ such that for every $\rho_0 \in \mathcal{T}$, the sets $y \circ \kappa \left( W_u(\rho_0) \cap \mathcal{T} \cap U_{\rho_0} \right)$ and $ \zeta \left( W_s(\rho_0) \cap \mathcal{T} \cap U_{\rho_0} \right)$ are $\nu$-porous on scale $0$ to $1$. 
\end{cor}

\begin{proof}
The maps $y \circ \kappa$ and $\zeta$ are $C$-Lipschitz for a global constant $C$. As a consequence, the previous lemma and Lemma \ref{Lipshictz and fractal sets} give $$\forall \varepsilon \leq \varepsilon_0 /C  , N_\Omega(\varepsilon) \leq C^\delta \varepsilon^{-\delta} \quad, \text{ where }  \Omega = y \circ \kappa \left( W_u(\rho_0) \cap \mathcal{T} \cap U_{\rho_0} \right) \text{ or } \zeta \left( W_s(\rho_0) \cap \mathcal{T} \cap U_{\rho_0} \right)$$
 Applying Lemma \ref{from_fractal_to_porous}, the $\nu$-porosity is proved for some $\nu = \nu(\delta, C, \varepsilon_0)$. 
\end{proof}

To conclude, we use this corollary to show the porosity of $\Omega^\pm$. We start by studying $\Omega^+$. 
\begin{lem}
There exists a global constant $C>0$ such that 
$$ \Omega^+ \subset \zeta \left( W_s(\rho_0) \cap \mathcal{T} \cap U_{\rho_0} \right) (Ch^\tau)$$
\end{lem}

\begin{proof}
Since $\Omega^+ = \Gamma^+ (h^\tau)$, it is enough to show the same statement for $\Gamma^+ = \eta \circ \kappa_{\rho_0} \left(\mathcal{V}_\mathcal{Q}^+ \right) $. \\
Let $\rho \in \mathcal{V}_\mathcal{Q}^+$. By assumption on $\mathcal{Q}$ and $\rho_0$, $d(\rho, W_u(\rho_0) ) \leq Ch^\mathfrak{b}$. Since $\rho \in \mathcal{V}_\mathbf{q}$ for some $\mathbf{q} \in \mathcal{Q}$, there exists $\rho_1 \in \mathcal{T}$ such that $ d (\rho, W_u(\rho_1) ) \leq \frac{C}{J_\mathbf{q}^+(\rho_1)} \leq Ch^\tau$. Fix $\rho_2 \in W_u(\rho_1)$ such that $d(\rho,\rho_2) \leq Ch^\tau$. 
$$ |\eta \circ \kappa(\rho) - \zeta(\rho_1) | = |\eta \circ \kappa(\rho) - \zeta(\rho_2) |  \leq  |\eta \circ \kappa(\rho) -\eta \circ \kappa(\rho_2) | + |\eta \circ \kappa(\rho_2) - \zeta(\rho_2) | $$
Since $\eta \circ \kappa$ is Lipschitz, we can control the first term by 
$$  |\eta \circ \kappa(\rho) -\eta \circ \kappa(\rho_2) | \leq C d(\rho,\rho_2) \leq Ch^\tau$$ 
To estimate the second term, the same arguments used after Lemma \ref{straigh_Wu} show that $$ |\eta \circ \kappa(\rho_2) - \zeta(\rho_2) | \leq \text{diam} \big[ \eta \circ \kappa \left( W_u(\rho_2) \cap U_{\rho_0} \right)\big] \leq Ch$$ 
It gives $|\eta \circ \kappa(\rho) -\zeta(\rho_1) | \leq Ch^\tau$. To conclude, note that there exists a unique point $\rho_1^\prime \in W_s(\rho_0) \cap W_u(\rho_1)$ and $\zeta(\rho_1) = \zeta(\rho_1^\prime)$. 
\end{proof}

As a simple corollary of this lemma and of Lemma \ref{fattenning_porous_set}, we get 
\begin{cor}
$\Omega^+$ is $\nu/3$-porous on scale $\frac{3}{\nu}Ch^\tau$ to $1$. 
\end{cor}

We now turn to the study of $\Omega^-$. We can state and prove similar results with different scales of porosity. Recall that $\delta_2 = \frac{\lambda_0}{\lambda_1} \delta_0$. 

\begin{lem}
There exists a global constant $C>0$ such that 
$$ \Omega^- \subset y \circ \kappa \left( W_u(\rho_0) \cap \mathcal{T} \cap U_{\rho_0} \right) (Ch^{\delta_2})$$
\end{lem}

\begin{proof}
Since $\Omega^- = \Gamma^- (h^{\delta_0})$ with $\delta_0 > \delta_2$, it is enough to prove if for 
$$ \Gamma^- = y \circ \kappa \left( \mathcal{V}_a^+ \cap \mathcal{T}_-^{loc}\left( 4C_2h^{\delta_2} \right) \cap \{ |\eta| \leq h^{\delta_0} \}  \right) $$
Recall that $\mathcal{T}_-^{loc} \subset \bigcup_{\rho \in \mathcal{T}} W_s(\rho)$. Since in $\mathcal{V}_a^+$, all the local stable leaves intersect $W_u(\rho_0)$, we have 
$$ \mathcal{V}_a^+ \cap \mathcal{T}_-^{loc}(4C_2h^{\delta_2}) \subset \bigcup_{ \rho \in W_u(\rho_0) \cap \mathcal{T} } W_s(\rho) (4C_2 h^{\delta_2}) $$ 
Fix $\rho \in W_u(\rho_0) \cap \mathcal{T}$. Since $d\kappa(E_s(\rho_0)) = \R \partial_\eta$, if $\varepsilon_0$ is small enough, we can write $\kappa (W_s(\rho) ) = \left\{ (G_\rho(\eta), \eta) , \eta \in O \right\}$ where $O$ is some open subset of $\R$ and $G_\rho : O \to \R$ is $\cinf$. In particular, it is Lipschitz with a global Lipschitz constant $C$. If $|\eta| \leq h^{\delta_0}$, $|G_\rho(\eta) - G_\rho(0) | \leq C h^{\delta_0}$.
Recall that $\kappa(W_u(\rho_0) \cap U_{\rho_0}) \subset \R \times \{0\}$ and hence, $G_\rho(0) = y \circ \kappa (\rho)$. As a consequence, if $\rho_1 \in W_s(\rho) \cap \{ |\eta| \leq h^{\delta_0} \} $, writing $\kappa(\rho_1) =(G_\rho(\eta), \eta)$, we have 
$$ | y \circ \kappa (\rho_1) - y \circ \kappa (\rho) | = |G_\rho(\eta) - G_\rho(0) | \leq C h^{\delta_0}$$
 Then, if $\rho_2 \in W_s(\rho) (4C_2 h^{\delta_2})$, since $\kappa$ is Lipschitz with global Lipschitz constant , 
$$ |y \circ \kappa(\rho_2) - y \circ \kappa(\rho) | \leq Ch^{\delta_2} + Ch^{\delta_0} \leq Ch^{\delta_2} $$
This shows that  $y \circ \kappa(\rho_2)  \in y \circ \kappa (W_u(\rho_0) \cap \mathcal{T} ) (Ch^{\delta_2})$ and concludes the proof. 
\end{proof}

As a corollary, using Lemma \ref{fattenning_porous_set}, we get 
\begin{cor}
$\Omega^-$ is $\nu/3$-porous on scale $\frac{3}{\nu}Ch^{\delta_2}$ to $1$. 
\end{cor}

We can now prove the last Proposition \ref{FUP_to_Omega} needed to end the proof of Proposition \ref{Prop_thm}. This is a consequence of the porosity of $\Omega^\pm$ and the fractal uncertainty principle.  To apply Proposition \ref{FUP}, we need to ensure that the scale condition is satisfied, that is to say 
$$ \delta_2 + \tau >1$$
which has been supposed when defining $\tau$ in (\ref{condition_on_tau}) and (\ref{tau}). 
Proposition \ref{Prop_thm} then comes with any $ 0< \gamma < (\delta_2 + \tau  -1) \beta(\nu/3)$. 

\pagebreak{}
\appendix
\section{}
\subsection{Holder regularity for flows}

\begin{lem}\label{flow_regularity}
Let $U \subset \R^n$ be open and $Y : U \to \R^n$ be a complete $C^{1+\beta}$ vector field. We note $\phi^t(x)$ the flow generated by $Y$. Then, for any $T \in \R$ and $ K \subset U$ compact, the map 
$$ (t,x) \in [-T,T] \times K \mapsto \phi^t(x)$$
is $C^{1+\beta}$. 
\end{lem}

\begin{proof}
We fix $T, K$ as in the statement. We'll use the same constants $C,C^\prime$ at different places, with different meaning. In addition to $Y$, they will depend on $T,K$ . 

Since $Y$ is $C^1$, Cauchy-Lipschitz theorem gives the local existence and uniqueness of the flow. It is standard that the flow is also $C^1$ and satisfies 
\begin{equation} \label{ODE_diff}
\partial_t d \phi^t(x) = d Y (\phi^t(x)) \circ d \phi^t(x) 
\end{equation}
Let's note $A^t(x) = d \phi^t(x)$ and $\Xi(t,x) = dY (\phi^t(x) )$.
The assumption on $Y$ implies that $\Xi$ is $\beta$-Hölder. \\
Fix $(t_0,x_0), (t_1,x_1) \in [-T, T] \times K$ and let's estimate $ || A^{t_1}(x_1) - A^{t_0}(x_0)||$. We split it into two pieces and control it with the triangle inequality : 
$$ || A^{t_1}(x_1) - A^{t_0}(x_0)|| \leq || A^{t_1}(x_1) - A^{t_0}(x_1)|| + || A^{t_0}(x_1) - A^{t_0}(x_0)||  $$ 
It is not hard to control the first term of the right hand side using (\ref{ODE_diff}) since 
$$  ||A^{t_1}(x_1) - A^{t_0}(x_1)|| = \left| \left| \int_{t_0}^{t_1} \Xi(s,x_1) \circ A^s(x_1) ds  \right| \right| \leq C |t_1- t_0|$$
To estimate the second term, 
we estimate \begin{align*}
|| \partial_t (A^t(x_1) - A^t(x_0)) || &\leq || \left(\Xi(t,x_1) - \Xi(t,x_0)  \right) \circ A^t(x_1) + \Xi(t,x_0) \circ (A^t(x_1) - A^t(x_0)) || \\
&\leq C d(x_0,x_1)^\beta + C^\prime ||A^t(x_1) - A^t(x_0)||
\end{align*}
By Gronwall's lemma, $$||A^{t_0}(x_1) - A^{t_0}(x_0)|| \leq C d(x_0,x_1)^\beta e^{C^\prime t_0} \leq Cd(x_0,x_1)^\beta$$
This concludes the proof. 
\end{proof}

\subsection{Proof of Lemma \ref{Lemma_regularity_partial}}\label{appendix_regularity_partial}
We give the missing proof of Lemma \ref{Lemma_regularity_partial} and widely use the notations of the subsection \ref{Adapted_charts}. Its proof uses the construction of $e_u$ in the proof of Theorem \ref{Thm_regularity}. It is inspired by techniques usually used to show the unstable manifold's theorem (see for instance \cite{Dy18}). In fact, the smoothness of $y \mapsto f_0(y,0)$ is a direct consequence of the smoothness of the unstable manifold $W_u(\rho_0)$. It was not clear for us if it was possible to easily deduce from this the required smoothness of $y \mapsto \partial_\eta f_0(y,0)$. This is why we decided to give a proof of this proposition. It uses the fact that $e_u$ has been constructed to satisfy $\R d_{\rho} F(e_u(\rho) ) = \R e_u (F(\rho) )$ for $\rho$ in a small neighborhood of $\mathcal{T}$. 
To show the lemma, we need information along all the orbit of $\rho_0$. For this purpose, we introduce the following, for $m \in \Z$, 

\begin{itemize}[nosep]
\item $\rho_m = F^m (\rho_0)$ ; 
\item $\kappa_m : U_m \to V_m \subset \R^2$ the chart given by Lemma \ref{weak_adapted_chart} centered at $\rho_m$ and we assume that the relation $\R d_{\rho}F(e_u(\rho))  = \R e_u (F(\rho) )$ holds for $\rho \in U_m$. We will note $(y_m, \eta_m)$ the variable in $V_m$ ; 
\item $G_m = \kappa_{m+1} \circ F \circ \kappa_m^{-1} : V_m \to V_{m+1}$ ; 
\item A reparametrization of the vector field $(\kappa_m)_* e_u$  : $\R (\kappa_m)_* e_u = \R e_m$ where  $e_m (y_m, \eta_m) = {}^t ( 1, s_m(y_m,\eta_m) ) $ where $s_m$ is a slope function which is known to be $C^{1+ \beta}$. 
\end{itemize} 
Note that $s_m(y_m,0) = 0$ due to the fact that $\kappa_m(W_u(\rho_m) ) \subset \R \times \{ 0 \}$. 
The hyperbolicity assumption on $F$ and the properties of $\kappa_m$ allow us to write 
$$ G_m(y_m,\eta_m) = \Big( \lambda_m y_m + \alpha_m(y_m,\eta_m) , \mu_m \eta_m + \beta_m(y_m,\eta_m) \Big)$$ 
where 
\begin{itemize}[nosep]
\item For some $\nu <1$, $0 \leq |\mu_m| \leq \nu , |\lambda_m| \geq \nu^{-1}$ for all $m \in \N$ ;
\item $\alpha_m(0,0) = \beta_m (0,0) = 0$; 
\item $\beta_m(y_m,0)=0$ for $(y_m,0) \in V_m$
\item $d\alpha_m(0,0) = d \beta_m(0,0) = 0$ ; 
\item We can assume that $U_m$ are sufficiently small neighborhoods of $\rho_m$ so that $\beta_m, \alpha_m = O(\delta_0)_{C^1(U_m)}$ for some small $\delta_0 >0$. 
\end{itemize}

The property $d_\rho F(e_u(\rho) )  \in \R e_u (F(\rho) )$ implies that $ d_{(y_m,\eta_m)} G_m \big(e_m(y_m,\eta_m) ) \big) \in \R e_{m+1} \big(G_m (y_m,\eta_m) \big)$. As a consequence, the transformation of the slopes gives an equation satisfied by the family of slopes $(s_m)_{m \in \Z}$ : 

\begin{equation}\label{equation_for_s_m}
s_{m+1} \left( G_m(y_m,\eta_m) \right) = Q_m  \big( y_m, \eta_m, s_m(y_m,\eta_m) \big)
\end{equation}
where $Q_m$ is the smooth function 
$$ Q_m(y_m, \eta_m ,s) = \frac{s \times \big(\mu_m + \partial_{\eta_m} \beta_m(y_m,\eta_m) \big) + \partial_{y_m} \beta_m(y_m,\eta_m)  }{ \lambda_m + \partial_{y_m} \alpha_m (y_m,\eta_m) + s \times \partial_{\eta_m} \alpha_m(y_m,\eta_m)} $$

Writing $G_m(y_m,\eta_m) = (y_{m+1}, \eta_{m+1})$, we deduce by differentiation of  (\ref{equation_for_s_m}) with respect to $\eta_{m+1}$: (we omit the point of evaluation of the maps involved in the right hand side to alleviate the line)
\begin{multline}\label{equation_sigma_m}
\partial_{\eta_{m+1}} s_{m+1} (y_{m+1}, \eta_{m+1}) = \partial_{y_m} Q_m \times \partial_{\eta_{m+1}} y_m + \partial_{\eta_m} Q_m \times \partial_{\eta_{m+1}} \eta_m \\
+ \partial_s Q_m \times \big( \partial_{y_m} s_m \times \partial_{\eta_{m+1}} y_m + \partial_{\eta_m} s_m \times \partial_{\eta_{m+1}} \eta_m \big) 
\end{multline}
This last equation gives the transformation of vertical derivative of the slope. 
We now evaluate this identity at the point $(y_{m+1},0)$.  In the following lines, when the variable $y_m$ and $y_{m+1}$ appear in the same equation, we implicitly assume that they are related by $ (y_{m+1},0) = G_m(y_m,0)$, namely $y_{m+1}  = \lambda_m y_m + \alpha_m(y_m,0)$. We remark that due to the fact that $\beta_m(y_m,0)=0$, $Q_m(y_m,0,0)=0$ and the first term of the right hand side vanishes. The term $\partial_{y_m}s_m$ also vanishes at $(y_m, 0)$.  We will note 
\begin{align*}
&\sigma_m(y_m) = \partial_{\eta_m} s_m(y_m,0) \\
&  h_m(y_m) = \partial_{\eta_m} Q_m(y_m,0,0) \times \partial_{\eta_{m+1}} \eta_m(y_{m+1},0) \\
&  c_m(y_m) = \partial_s  Q_m(y_m,0,0) \times \partial_{\eta_{m+1}} \eta_m(y_{m+1},0) 
\end{align*}
These notations allow us to rewrite (\ref{equation_sigma_m}) at $(y_{m+1},0)$ : 
\begin{equation}\label{relation_sigma_m}
\sigma_{m+1}(y_{m+1}) = h_m(y_m) + c_m(y_m) \times \sigma_m(y_m)
\end{equation}
 We observe that $|\partial_{\eta_{m+1}} \eta_m (y_m,0)| =| \mu_m^{-1} + O(\delta_0)_{C^0}| $ and after some computations, we see that 
 $$ \partial_s Q_m(y_m,0,0) = \frac{\mu_m}{\lambda_m} + O(\delta_0)_{C^0}$$
  As a consequence, 
\begin{equation}
 |c_m(y_m)| = |\lambda_m^{-1}| + O(\delta_0)_{C^0} \leq \nu_1
\end{equation} 
 where, if $\delta_0$ is small enough, we can fix $\nu_1 <1$. 
Moreover, $c_m$ and $h_m$ are smooth functions and their $C^N$ norms are bounded uniformly in $m$, and actually by global constants depending only on $F$.  Furthermore, $y_m \mapsto y_{m+1}$ is given by $y_m \mapsto \lambda_m y + \alpha_m(y_m,0)$ and is an expanding diffeomorphism provided $\delta_0$ is small enough. 

We fix some small $\varepsilon$ such that $ (-\varepsilon,\varepsilon) \times \{0 \} \subset U_m$ for all $m$. Let's note $I = (-\varepsilon,\varepsilon) $. We will make use of the Fiber Contraction Theorem to show that $y_m \in I \mapsto \sigma_m(y_m)$ is smooth for every $m$, with uniform $C^N$ norms. For this purpose, let us introduce the following notations : 
\begin{itemize}
\item $C_0 \leq C_1  \leq \dots C_N \leq \dots$ a family of constant which will be specified in the sequel ; 
\item The complete metric space $X_N = \{ \gamma \in C^N(I) ; ||\gamma||_{C^k} \leq C_k, 0 \leq k \leq N \}$ equipped with the $C^N$ norm ; 
\item The auxiliary metric space $X_N^{aux} = \{ \gamma \in C^0(I); ||\gamma||_\infty \leq C_N \}$  equipped with the $C^0$ norm ; 
\item The complete metric space $E_N = \left(X_N \right)^{\Z}$ equipped with the metric $$d(\gamma_1,\gamma_2) = \sup_{m \in \Z} ||(\gamma_1)_m -(\gamma_2)_m||_{C^N}$$ 
\item Its auxiliary counterpart $E_N^{aux} = \left(X_N^{aux} \right)^{\Z}$ equipped with the metric $$d(\gamma_1,\gamma_2) = \sup_{m \in \Z} ||(\gamma_1)_m -(\gamma_2)_m||_{C^0}$$
\end{itemize}
For $ \gamma \in E_N$, let's define $T \gamma $ with the formula (\ref{relation_sigma_m}) : 
$$ (T \gamma)_{m+1}(y_{m+1}) = \left( h_m + c_m \gamma_m \right) (y_m)$$
Since $y_m \mapsto y_{m+1}$ is expanding, we see that $y_{m+1} \in I \implies y_m \in I$. Hence, $(T\gamma)_{m+1}$ is well defined on $I$. 
Our aim is to show by induction on $N$ that for every $N \in \N$, $\sigma \coloneqq (\sigma_m)_{m \in \Z}$ is in $E_N$ and is an attractive fixed point of $T : E_N \to E_N$. 

We start with the case $N=0$. We need to check that $T(E_0) \subset E_0$. It will be the case as soon as 
$$ C_0 \nu_1 + \sup_{m} ||h_m||_\infty \leq C_0$$
For instance, take $C_0 = \frac{2 \sup_m ||h_m||_\infty}{1-\nu_1}$. Due to the fact that $||c_m||_{C^0(I)} \leq \nu_1$, $T$ is a contraction with contraction rate $\nu_1$ and hence $T : E_0 \to E_0$ has a unique attractive fixed point. This fixed point is necessarily $\sigma$ since $\sigma$ satisfies (\ref{relation_sigma_m}). 

Arguing by induction, we assume that $\sigma \in E_N$, $T (E_N) \subset E_N$ and $\sigma$ is an attractive fixed point for $T$ and we want to show that the same is true for $N+1$. For this purpose, suppose that $\gamma \in E_N$ is of class $C^{N+1}$. Analyzing the formula defining $T$, we see that can can write, for $m \in \Z$, 
\begin{multline}\label{equation_T_g_m_N}
(T\gamma)_m^{(N+1)} (y_{m+1} )   =   h_m^{(N+1)}(y_m) + c_m(y_m) \times \left( \frac{\partial y_{m+1}}{\partial y_m}(y_m) \right)^{-N-1} \times \gamma_m^{(N+1)}(y_m) \\
+ R_{N,m} \left(y_m, \gamma_m(y_m), \dots, \gamma^{(N)}_m(y_m) \right)
\end{multline}
where $R_{N,m} : I \times [-C_0,C_0] \times \dots \times [-C_N,C_N] \to \R$ is a polynomial in the last $N+1$ variables with smooth coefficients in $y_m$, uniformly bounded in $m$. As a consequence, there exists a global constant $C^\prime_{N+1}$ such that 
$$\sup_m  \sup_{I \times [-C_0,C_0] \times \dots \times [-C_N,C_N]} |R_{N,m}(y_m, \tau_0, \dots, \tau_N) | \leq C_{N+1}^\prime$$
We can then choose $C_{N+1} \geq C_N$ such that $$\sup_m ||h_m||_{C^{N+1}} + C^\prime_{N+1}+  \nu_1 C_{N+1} \leq C_{N+1}$$ which ensures that $T : E_{N+1} \to E_{N+1}$. We now wish to use the Fiber Contraction Theorem (Theorem \ref{Fiber_Contraction_Theorem}). If $\gamma \in E_N$, we define the map 
$S_\gamma :   E^{aux}_{N+1} \to E^{aux}_{N+1}$ by $$
\left(S_\gamma \theta \right)_{m+1} (y_{m+1})    =   h_m^{(N+1)}(y_m) + c_m(y_m) \times \left(  \frac{\partial y_{m+1}}{\partial y_m}(y_m) \right)^{-N-1} \times \theta_m(y_m) + R_{N,m} \left(y_m, \gamma_m(y_m), \dots, \gamma^{N}_m(y_m) \right) $$
Due to the choice of $C_{N+1}$, we see that $S_\gamma$ is well defined and since we have  $$\left| \frac{\partial y_{m+1}}{\partial y_m}(y_m)\right| \geq 1$$
and $||c_m||_{C^0(I)} \leq \nu_1$, $S_\gamma$ is a contraction with contraction rate $\nu_1$, for every $\gamma \in E_N$. In particular, the map $S_\sigma$ has a unique fixed point $\sigma_{N+1} \in E_{N+1}^{aux}$. 

The Fiber Contraction Theorem (Theorem \ref{Fiber_Contraction_Theorem}) applies to the continuous map $$T_N : (\gamma, \theta) \in E_N \times E_{N+1}^{aux} \mapsto (T\gamma , S_\gamma \theta )\in  E_N \times E_{N+1}^{aux} $$ and $(\sigma, \sigma_{N+1})$  is an attractive fixed point of $T_{N}$ in $E_N \times E_{N+1}^{aux}$. \\
In particular, if $\gamma \in E_{N+1}$, then $\tilde{\gamma} \coloneqq (\gamma, \gamma^{(N+1})) \in E_N \times E_{N+1}^{aux}$ and 
$$\lim_{ p \to +\infty} T_{N}^{ p} \tilde{\gamma} = (\sigma, \sigma_{N+1}) \text{ in } E_N \times E_{N+1}^{aux}$$
 However, by definition of $S_\gamma$, 
$$T_{N}^{ p} \tilde{\gamma} = \left( T^{p} \gamma, \left(T^{p} \gamma\right)^{(N+1)} \right)$$ Hence, for every fixed $m$, $\left(T^{p} \gamma\right)_m$ converges to $\sigma_m$ in $X_N$ and $ \left(T^{p} \gamma\right)_m^{(N+1)} $ converges uniformly on $I$ to $\sigma_{N+1}$. This proves that $\sigma$ is $C^{N+1}$ and $\sigma^{(N+1)} = \sigma_{N+1}$. We conclude that $\sigma \in E_{N+1}$ is then an attractive fixed point of $T : E_{N+1} \to E_{N+1}$, which proves the induction and concludes the proof of Lemma \ref{Lemma_regularity_partial}.

\subsection{Proof of Lemma \ref{Lemma_bound_symbol}}\label{Proof_of_lemma_bound_symbol}

We give the missing proof of Lemma \ref{Lemma_bound_symbol}. The proof is a precise analysis of the iteration formula (\ref{iteration_formula}). We adopt the notations introduced for Lemma \ref{Lemma_bound_symbol}. We argue by induction on $J$ to show the property $\mathcal{P}_{J} $:" the bound (\ref{bound_symbol}) is valid for all $j \leq J$ and for all $1 \leq i \leq n-1 , l \in \N$ with some constants $C_{j,l}$". 

\paragraph{1. Base case.} Let us start with $\mathcal{P}_0$. 
The iteration formula (\ref{iteration_formula}) implies that
$$ a_{i,0}(x_i) = \prod_{l=1}^i f_l(x_l)$$ 
Hence, the bound $||a_{i,0}||_{C^0} \leq \left(B \nu^{1/2} \right)^i $ is obvious and we can set $C_{0,0}=1$. We now argue by induction on $i$ and prove the property $\mathcal{P}_{0,i} $:"the bound (\ref{bound_symbol}) is valid for $j=0$, $i$ and for all $ l \in \N$ for some constants $C_{j,l}$".Theses bounds are trivially true for $i=0$ and are direct consequences of Lemma \ref{iteration_lemma} for $i=1$.  Suppose that the property holds for $i-1$ for some $i \geq 1$ and let's show it for $i$. 
\subparagraph{1.1. Case $l=1$. }
Let us first deal with $l=1$ and compute the derivative of $a_{i,0}$, using the formula : $a_{i,0}(x_i) = f_i(x_i) a_{i-1,0}(x_{i-1})$. 
$$ a_{i,0}^\prime (x_i) = f^\prime(x_i) a_{i-1,0}(x_{i-1}) + f_i(x_i) a_{i-1,0}^\prime(x_{i-1}) \left( \frac{\partial x_{i-1} }{\partial x_i} \right)$$
We use the (weak) bound  $\left| \frac{\partial x_{i-1} }{\partial x_i} \right| \leq 1$ and the property $\mathcal{P}_{0,i-1}$ to show that : 
$$||a_{i,0}||_{C^1} \leq C \left(B \nu^{1/2} \right)^{i-1} + C_{0,1} \left(B \nu^{1/2} \right) \times \left(B \nu^{1/2} \right)^{i-1} i  \leq C_{0,1} \left(B \nu^{1/2} \right)^i (i+1)$$
assuming that $C_{0,1} > C\left(B \nu^{1/2} \right)^{-1}$. 

\subparagraph{1.2. General case for $l >0$.  }
We now come back to the general case $l>0$. By using the formula $a_{i,0}(x_i) = f_i(x_i) a_{i-1,0}(x_{i-1})$, one sees that we can write $a_{i,0}^{(l)}$ on the form : 
$$a_{i,0}^{(l)}(x_i) = f_i(x_i) a_{i-1,0}^{(l)} (x_{i-1}) \left( \frac{\partial x_{i-1} }{\partial x_i} \right)^{l} + O \left( ||a_{i-1,0}||_{C^{l-1}} \right)$$
The constants appearing in the $O$ depend on $C^l$ norms of $f_i$ and $\phi_i$, which, by assumption are controlled by some uniform $C^\prime_l$. Hence, using the assumption $\mathcal{P}_{0,i-1}$,

\begin{align*}
|a_{i,0}^{(l)}(x_i)| & \leq \left( B\nu^{1/2} \right) ||a_{i-1,0}||_{C^l} \left( \frac{\partial x_{i-1} }{\partial x_i} \right)^{l} + C^\prime_l ||a_{i-1,0}||_{C^{l-1}} \\
& \leq C_{0,l} \left( B\nu^{1/2} \right) \left( B\nu^{1/2} \right)^{i-1} i^l  + C^\prime_l C_{0,l-1}\left( B\nu^{1/2} \right)^{i-1} i^{l-1} \\
&\leq  C_{0,l} \left( B\nu^{1/2} \right)^i (i+1)^l
\end{align*}
assuming that $C_{0,l}$ is chosen bigger than $\frac{1}{l} C^\prime_l C_{0,l-1} \left( B\nu^{1/2} \right)^{-1}$. As a consequence, we can build constants satisfying these conditions by defining inductively $C_{0,l} = \max \left( C_{0,l-1},  \frac{1}{l} C^\prime_l C_{0,l-1} \left( B\nu^{1/2} \right)^{-1} \right)$.
This ends the proof of $\mathcal{P}_{0,i}$ and hence of $\mathcal{P}_0$. 

\paragraph{2. Induction step.}
We now assume that $\mathcal{P}_{j-1}$ is true for some $j \geq 1$ and aim at proving $\mathcal{P}_j$.  Again, we do it by induction on $i$ by proving the properties $\mathcal{P}_{j,i}$ : "the bound (\ref{bound_symbol}) is true for $j$,$i$ and all $l \in \N$ ". Theses bounds are trivially true for $i=0$ and are direct consequences of Lemma \ref{iteration_lemma} for $i=1$. Suppose that the property holds for $i-1$ for some $i \geq 2$ and let's show it for $i$.

\subparagraph{2.1. Case $l=0$.}
Let's start with $l=0$. The iteration formula shows that 
$$a_{i,j}(x_i) = f_i(x_i)a_{i-1,j}(x_{i-1} ) + \sum_{p=0}^{j-1} L_{j-p,i}(a_{i-1,p})  (x_{i-1}) $$ 
By Lemma \ref{iteration_lemma}, there exists constants $C^\prime_{p,m} >0$ such that 
$$ ||L_{p,i} a||_{C^m(I_i)} \leq C^\prime_{p,m} ||a||_{C^{2p+m}(I_{i-1}) }$$

Hence, assuming that (\ref{bound_symbol}) holds for $a_{i-1,j}$ with $l=0$. 
\begin{align*}
||a_{i,j}||_\infty & \leq C_{j,0}  \left( B\nu^{1/2} \right)  \left( B\nu^{1/2} \right)^{i-1} i^{3j}  + \sum_{p=0}^{j-1} C^\prime_{j-p,0}  ||a_{i-1,p}||_{C^{2(j-p)}} \\
&  \leq C_{j,0}  \left( B\nu^{1/2} \right)^i i^{3j}  + \sum_{p=0}^{j-1} C^\prime_{j-p,0}  C_{p,2(j-p)}  \left( B\nu^{1/2} \right)^{i-1} i^{2(j-p) + 3p} \\
&\leq   C_{j,0}  \left( B\nu^{1/2} \right)^i i^{3j}+ i^{2j}\left( B\nu^{1/2} \right)^{i-1}\sum_{p=0}^{j-1} C^\prime_{j-p,0}  C_{p,2(j-p)}   i^p \\
&\leq  C_{j,0}  \left( B\nu^{1/2} \right)^i i^{3j}+ i^{2j}\left( B\nu^{1/2} \right)^{i-1} \left[\sup_{0 \leq p \leq j-1} C^\prime_{j-p,0}  C_{p,2(j-p)} \right] \frac{i^j -1}{i-1} \\ 
&\leq C_{j,0}  \left( B\nu^{1/2} \right)^i i^{3j}+ i^{3j-1}\left( B\nu^{1/2} \right)^{i-1} \left[\sup_{0 \leq p \leq j-1} C^\prime_{j-p,0}  C_{p,2(j-p)} \right] \tilde{C}_j \text{ where } \frac{i^j-1}{i-1} \leq \tilde{C}_j i^{j-1}  \\ 
&\leq C_{j,0}  \left( B\nu^{1/2} \right)^i (i+1)^{3j}
\end{align*} 
assuming that $C_{j,0} $ is chosen bigger than $K_j \coloneqq  \frac{1}{3j}\left( B\nu^{1/2} \right)^{-1} \left[\sup_{0 \leq p \leq j-1} C^\prime_{j-p,0}  C_{p,2(j-p)} \right] \tilde{C}_j $. As a consequence, the bounds hold for $l=0$ and $i,j$ if we set $C_{j,0} = \max(1, K_j)$. 

\subparagraph{2.2. Case $l>0$. }
Consider now $l >0$. As already done, one can write

$$a_{i,j}^{(l)}(x_i) = f_i(x_i)a_{i-1,j}^{(l)} (x_{i-1}) \left( \frac{\partial x_{i-1}}{\partial x_i}\right)^l + O\left( ||a_{i-1,j}||_{C^{l-1}}\right) + \sum_{p=0}^{j-1} \left(L_{j-p,i}(a_{i-1,p}) \right)^{(l)} (x_{i-1})$$
As usual, the constants in $O$ depend on $l,j$ but not on $i$ and we note $C^{\prime\prime}_{l,j}$ the constant in this $O$. 
Hence, we can control : 

\begin{align*}
||a_{i,j}^{(l)}||_\infty& \leq C_{j,l}  \left( B\nu^{1/2} \right)  \left( B\nu^{1/2} \right)^{i-1} i^{l+3j} + C^{\prime\prime}_{l,j}  C_{j,l-1}  \left( B\nu^{1/2} \right)^{i-1} i^{l+3j-1}  + \sum_{p=0}^{j-1} ||L_{j-p,i} ( a_{i-1,p} ) ||_{C^l} \\
&\leq C_{j,l}  \left( B\nu^{1/2} \right)^i i^{l+3j} + C^{\prime\prime}_{l,j}  C_{j,l-1}  \left( B\nu^{1/2} \right)^{i-1} i^{l+3j-1}  + \sum_{p=0}^{j-1} C^\prime_{j-p,l} || a_{i-1,p}  ||_{C^{l+2(j-p)}}\\
&\leq  C_{j,l}  \left( B\nu^{1/2} \right)^i i^{l+3j} + C^{\prime\prime}_{l,j}  C_{j,l-1}  \left( B\nu^{1/2} \right)^{i-1} i^{l+3j-1}  + \sum_{p=0}^{j-1} C^\prime_{j-p,l} C_{p,l+2(j-p)}  \left( B\nu^{1/2} \right)^{i-1} i^{l+2(j-p)+3p}   \\
&\leq C_{j,l} \left( B\nu^{1/2} \right)^i \left( i^{l+3j} +i^{l+3j-1} \frac{1}{ C_{j,l}} \underbrace{ \left( B\nu^{1/2} \right)^{-1}  \left( C^{\prime\prime}_{l,j}  C_{j,l-1}   + \sup_{0 \leq p \leq j-1} C^\prime_{j-p,l} C_{p,l+2(j-p)} \tilde{C}_j\right) }_{\tilde{C}_{j,l}} \right) \\
&\leq C_{j,l} \left( B\nu^{1/2} \right)^i  (i+1)^{l+3j}
\end{align*}
if $C_{j,l} \geq  \tilde{C}_{j,l}$. Eventually, we define by induction on $l$ the constants $C_{j,l}$ by setting $C_{j,l} =  \max \left( C_{j,l-1}, \tilde{C}_{j,l}\right)$, achieving the proof of $\mathcal{P}_j$. This concludes the proof of the lemma. 

\subsection{Upper-box dimension for hyperbolic set}\label{upper-box dimension}

This subsection is devoted to the proof of Proposition \ref{upper_box_dim}. We will simply recall some arguments which lead to give an upper bound to the upper box dimension. We borrow this arguments from \cite{Bar} (Section 4.3) and refer the reader to this book for the definitions and properties of topological pressure (definition 2.3.1), Markov partition (definition 4.2.6) and other references on this theory.

We'll show that the pressure condition (\ref{Fractal}) implies Proposition \ref{upper_box_dim}. We prove it for the unstable manifolds. The proof is similar in the case of stable manifolds by changing $F$ into $F^{-1}$. We first begin by fixing a Markov partition for $\mathcal{T}$ with diameter at most $\eta_0$. This is possible in virtue of Theorem 18.7.3 in \cite{KH}. We note $R_1, \dots , R_p \subset \mathcal{T}$ this Markov partition. Here,  $\eta_0$ is smaller than the diameter of the local stable and unstable manifolds and the holonomy maps $H_{\rho,\rho^\prime}^{u/s}$ are well defined for $d(\rho,\rho^\prime) \leq \eta_0$ : 
$$H_{\rho,\rho^\prime}^{u/s} : W_{s/u}(\rho) \to W_{s/u}(\rho^\prime) , \zeta \mapsto \text{ the unique point in }W_u(\zeta) \cap W_s(\rho^\prime) $$ 
Due to our results on the regularity of the stable and unstable distributions, these maps are Lipschitz with global Lipschitz constants. In particular, if an inequality of the kind $$N_{W_u(\rho) \cap \mathcal{T}} (\varepsilon) \leq C \varepsilon^{-\delta}$$ holds for some $\rho$, it holds for $\rho^\prime$ if $d(\rho,\rho^\prime) \leq \eta_0$ with $C$ replaced by $K^\delta C$ where $K$ is a Lipschitz constant for the holonomy maps. We fix $(\rho_1, \dots, \rho_p)$ in $(R_1, \dots, R_p)$ and we set $V  = \bigcup_{i=1}^p W_u(\rho_i) \cap R_i$. It is then enough to show that

$$ \overline{\dim} V   <1$$ 
Indeed, if $\overline{\dim} V <1$, for $\delta \in ( \overline{\dim} V , 1)$, there exists $\varepsilon_0 >0$  such that 
$$ \forall \varepsilon \leq \varepsilon_0, \; N_{ V}(\varepsilon) \leq \varepsilon^{-\delta}$$
and we conclude the proof of Proposition \ref{upper-box dimension} with the above considerations on the holonomy maps. 

$\delta \coloneqq \overline{\dim} V$ satisfies the equation $P(\delta \phi_u)=0$. We will actually show that $P( \delta\phi_u) \geq 0$. Since $s \mapsto P(s \phi_u)$ is strictly decreasing and has a unique root, the assumption $P(\phi_u)<0$ will give $\delta <1$.  
We will note $$ R_{i_0, \dots, i_n} = \bigcap_{k=0}^n F^{-i} (R_{i_k} )  \; ; \; V_{i_0, \dots, i_n}  = R_{i_0, \dots, i_n} \cap V$$
the elements of the refined partition at time $n$. Similarly to the definitions of $J_\mathbf{q}^+$, we will note 
$$J_{i_0, \dots, i_n} = \inf \{ J_u^n(\rho), \rho \in R_{i_0, \dots, i_n}\}  $$
and write
$$c_n(s) = \sum_{i_0, \dots, i_n} J_{i_0, \dots, i_n}^{-s} =  \sum_{i_0, \dots, i_n}  \exp \max_{R_{i_0, \dots, i_n}} \left( s \sum_{k=0}^{n-1} \phi_u \circ F^k  \right) $$
(the last equality follows from the chain rule). Properties of Markov partitions ensure that 
$$P(s\phi_u) = \lim_{n \to \infty} \frac{1}{n} \log c_n(s)$$ 
Fix $s > \delta$. Hence, there exists $\varepsilon_1$ such that $\forall \varepsilon \leq \varepsilon_1$, $N_{V } (\varepsilon) \leq \varepsilon^{-s}$. \\
Fix $n \in \N^*$. By writing $V= \bigcup_{i_0, \dots, i_n} V_{i_0, \dots, i_n} $ we have 
$$ N_{V} (\varepsilon) \leq \sum_{i_0, \dots, i_n} N_{ V_{i_0, \dots, i_n}  } (\varepsilon)$$ 
Note that 
$$F^n (V_{i_0, \dots, i_n}  ) \subset W_u(F^n(\rho_{i_0} )) \cap R_{i_n} $$ 
and $$H^s_{F^n(\rho_{i_0}), \rho_{i_n} } \left(  F^n (V_{i_0, \dots, i_n}   ) \right) \subset  V_{i_n}$$
Hence, if we cover $V_{i_n} $ by $N$ sets of diameter at most $\varepsilon$, $U_1, \dots, U_N$, the sets  $F^{-n} \circ H^s_{ \rho_{i_n}, F^n (\rho_{i_0}) } (U_i) , 1 \leq i \leq N$ cover $V_{i_0, \dots, i_n} $ and have diameters at most $ K \varepsilon J_{i_0, \dots, i_n}^{-1}$.  
Hence, $$N_{V_{i_n}} (\varepsilon) \geq N_{V_{i_0, \dots, i_n} } (K \varepsilon J_{i_0, \dots, i_n}^{-1} ) $$
which gives 
$$N_{V }(\varepsilon) \leq \sum_{i_0, \dots, i_n} N_{V_{i_n}} (\varepsilon K^{-1} J_{i_0, \dots, i_n})$$
As a consequence, if $ \varepsilon < \varepsilon_1 K J_n^{-1}$, where $J_n = \sup_{i_0, \dots , i_n} J_{i_0, \dots, i_n}$, we have 
$$ N_{V}(\varepsilon) \leq \sum_{i_0, \dots, i_n} K^{s} J^{-s}_{i_0, \dots , i_n} \varepsilon^{-s} = K^s \varepsilon^{-s} c_n(s)$$
By iterating this process, we see that for all $m \in \N$, if $\varepsilon < \varepsilon_1 (KJ_n^{-1})^m$, 
$$N_{V } (\varepsilon) \leq \varepsilon^{-s} K^{m s} c_n(s)^m$$
Hence, 
$$ \frac{\log N_{V  } (\varepsilon) } {- \log \varepsilon} \leq s  + m \frac{\log (K^s c_n(s))}{- \log \varepsilon} \leq s + m \frac{\log (K^s c_n(s))}{- \log \left(\varepsilon_1 (KJ_n^{-1})^m\right)}  $$
We then take the $\limsup$ as $\varepsilon \to 0$ first and then pass to the limit as $m \to + \infty$ and find that 
$$ \overline{\dim} V \leq s + \frac{\log K^s c_n(s)}{- \log KJ_n^{-1}}$$
Then, we pass to the limit 
$s  \to \delta$ and find that $\log (K^\delta c_n(\delta) ) \geq 0$. 
Hence, 
$$P(\delta \phi_u) = \lim_{n \to \infty} \frac{1}{n} \log c_n(\delta) \geq \lim_{n \to \infty} \frac{-\delta \log K}{n}  =0$$
This ends the proof of the required inequality and gives that $\overline{\dim} V <1$. 

\subsection{From porosity to upper box dimension}\label{from_porous_to_dim}

We have shown that sets with upper box dimension stricly smaller than one are porous. In this appendix, we show a result in the other way, namely, porous sets down to scale 0 have an upper box dimension strictly smaller than one.  The following lemma gives a quantitative version of this statement. This is not useful for our use (we only needed the first implication) but we found that it could be of independent interest. Our proof is based on the proof of Lemma 5.4 in \cite{DyJ18}. We adopt the same notations as in \ref{subsection_porous_sets}. 

\begin{lem}
Let $M \in \N, \nu >0 , \alpha_1 >0$. 
Let $X \subset [-M,M]$ be a closed set and assume that $X$ is $\nu$-porous on scale $0$ to $\alpha_1$. Then, there exists $C=C(\nu, \alpha_1, M) >0$, $\varepsilon_0 = \varepsilon_0(\nu, \alpha_1, M)$ and $\delta = \delta(\nu) \in [0,1)$ such that 
$$ \forall \varepsilon \leq \varepsilon_0 \; ; \; N_X(\varepsilon) \leq C\varepsilon^{-\delta}$$
In particular, 
$$ \overline{\dim} X \leq \delta$$ 
\end{lem}

\begin{proof}

We note $L = \lceil \frac{2}{\nu} \rceil$ and $k_0$ the unique integer such that 
$$ L^{-k_0} \leq \alpha_1 < L^{-k_0+1}$$
We will note $I_{m,k} = [mL^{-k}, (m+1)L^{-k}]$ for $k \in \N, m \in \Z$. \\
We now show by induction on $k \geq k_0$ that there exists $Y_k \subset \Z$ such that : 
\begin{equation}
\# Y_k \leq 2ML^{k_0} (L-1)^{k-k_0} \; ; \; \Omega  \subset \bigcup_{ m \in Y_k} I_{m,k} 
\end{equation}
namely, at each level $k \geq k_0$, one new interval $I_{m,k}$ does not intersect $\Omega$. \\
The case $k=k_0$ is trivial since we simply cover $\Omega$ by the intervals $I_{m,k_0}$, for $ML^{k_0} \leq m <ML^{k_0}$. \\
We now assume that the result is proved for $k \geq k_0$ and we prove it for $k+1$. Fix $m \in Y_k$. 
We write $I= \bigcup_{j=0}^{L-1} I_{mL+j, k+1}$. We claim that among the intervals $I_{mL+j,k+1}$, at least one does not intersect $\Omega$. Indeed, since $|I| \leq L^{-k_0} \leq \alpha_1$, the porosity of $\Omega$ implies the existence of an interval $J \subset I$ of size $\nu|I| = \nu L^{-k} \geq 2L^{-k-1}$ such that $J \cap \Omega = \emptyset$. Since $|J| \geq 2L^{-k-1}$, $J$ contains at least one of the intervals $I_{mL+j,k+1}$. We note this index $j_m$. 
We now set $$Y_{k+1} = \bigcup_{m \in Y_k} \{ mL+j , j \in \{0, \dots, L_1 \} \setminus j_m \}$$
By the property of $j_m$, $\Omega \subset \bigcup_{m \in Y_{k+1}} I_{m,k+1}$ and $ \# Y_{k+1} \leq (L-1) \# Y_k \leq (L-1)^{k+1-k_0} 2ML^{k_0}$. 

\begin{figure}[h]
\begin{center}
\begin{tikzpicture}[xscale=1,yscale=1]
\tikzstyle{fleche}=[->,>=latex,thick]
\tikzstyle{noeud}=[rectangle,draw]
\tikzstyle{missing}=[fill=red,rectangle,draw]
\tikzstyle{feuille}=[rectangle,draw]
\tikzstyle{etiquette}=[midway]
\def\DistanceInterNiveaux{1}
\def\DistanceInterFeuilles{2}
\def\NiveauA{(-0)*\DistanceInterNiveaux}
\def\NiveauB{(-1.6666666666666665)*\DistanceInterNiveaux}
\def\NiveauC{(-3)*\DistanceInterNiveaux}
\def\NiveauD{(-4)*\DistanceInterNiveaux}
\def\InterFeuilles{(1)*\DistanceInterFeuilles}
\node[noeud] (R) at ({(3)*\InterFeuilles},{\NiveauA}) {$I_{0,k_0}$};
\node[missing] (Ra) at ({(0)*\InterFeuilles},{\NiveauB}) {$I_{0,k_0+1}$};
\node[noeud] (Rb) at ({(2)*\InterFeuilles},{\NiveauB}) {$I_{1,k_0+1}$};
\node[noeud] (Rba) at ({(1)*\InterFeuilles},{\NiveauC}) {$I_{3,k_0+2}$};
\node[feuille] (Rbaa) at ({(1)*\InterFeuilles},{\NiveauD}) {$\dots$};
\node[missing] (Rbb) at ({(2)*\InterFeuilles},{\NiveauC}) {$I_{4,k_0+2}$};
\node[noeud] (Rbc) at ({(3)*\InterFeuilles},{\NiveauC}) {$I_{5,k_0+2}$};
\node[feuille] (Rbca) at ({(3)*\InterFeuilles},{\NiveauD}) {$\dots$};
\node[noeud] (Rc) at ({(5)*\InterFeuilles},{\NiveauB}) {$I_{2,k_0+1}$};
\node[noeud] (Rca) at ({(4)*\InterFeuilles},{\NiveauC}) {$I_{6,k_0+2}$};
\node[feuille] (Rcaa) at ({(4)*\InterFeuilles},{\NiveauD}) {$\dots$};
\node[noeud] (Rcb) at ({(5)*\InterFeuilles},{\NiveauC}) {$I_{7,k_0+2}$};
\node[feuille] (Rcba) at ({(5)*\InterFeuilles},{\NiveauD}) {$\dots$};
\node[missing] (Rcc) at ({(6)*\InterFeuilles},{\NiveauC}) {$I_{8,k_0+2}$};
\draw[fleche] (R)--(Ra) node[etiquette] {$$};
\draw[fleche] (R)--(Rb) node[etiquette] {$$};
\draw[fleche] (Rb)--(Rba) node[etiquette] {$$};
\draw[fleche] (Rba)--(Rbaa) node[etiquette] {$$};
\draw[fleche] (Rb)--(Rbb) node[etiquette] {$$};
\draw[fleche] (Rb)--(Rbc) node[etiquette] {$$};
\draw[fleche] (Rbc)--(Rbca) node[etiquette] {$$};
\draw[fleche] (R)--(Rc) node[etiquette] {$$};
\draw[fleche] (Rc)--(Rca) node[etiquette] {$$};
\draw[fleche] (Rca)--(Rcaa) node[etiquette] {$$};
\draw[fleche] (Rc)--(Rcb) node[etiquette] {$$};
\draw[fleche] (Rcb)--(Rcba) node[etiquette] {$$};
\draw[fleche] (Rc)--(Rcc) node[etiquette] {$ $ };
\end{tikzpicture}
\end{center}
\caption{It illustrates the tree structure of the family of intervals $I_{k,m}$ with $L=3$. The porosity allows us to withdraw at least one child to any parent. The missing children are drawn in red. }
\end{figure}
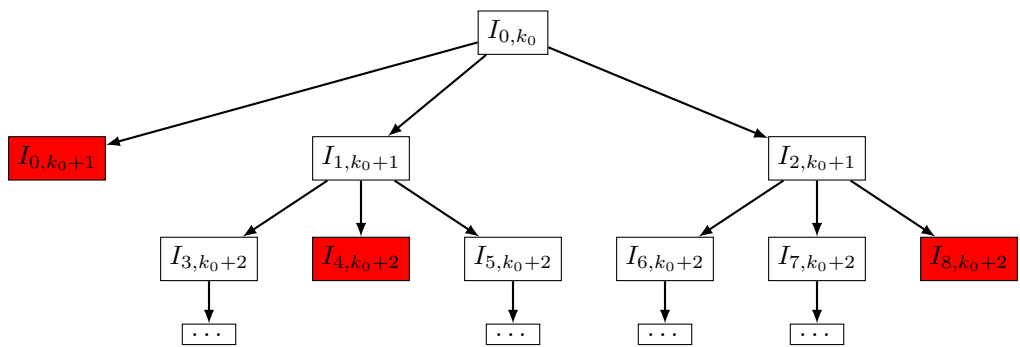

We now consider $\varepsilon \leq \frac{1}{2} L^{-k_0}$ and write $k$ the unique integer such that 
$$L^{-k} \leq 2\varepsilon < L^{-k+1} \quad i.e.\quad  k = \left\lceil \frac{- \log (2\varepsilon)}{\log L} \right\rceil$$
Since we can cover $\Omega$ by $2ML^{k_0} (L-1)^{k-k_0}$ closed intervals of size $L^{-k}$, we can cover $\Omega$ by $4ML^{k_0} (L-1)^{k-k_0}$ open intervals of size $2 \varepsilon$. 
Hence, 
$$N_\Omega(\varepsilon) \leq 4ML^{k_0}(L-1)^{k-k_0} \leq 4M \left( \frac{L}{L-1} \right)^{k_0} (L-1)^{ -\frac{\log (2\varepsilon)}{ \log L} +1 } \leq C \varepsilon^{-\delta}$$
with $\delta = \frac{\log (L-1)}{\log L} \in [0,1)$ and $C = 4M \left( \frac{L}{L-1} \right)^{k_0} (L-1)^{1- \frac{\log 2}{\log L}} $. 
\end{proof}

\newpage
\bibliographystyle{alpha}
\bibliography{biblio_these}

\end{document}